\documentclass{amsart}
\usepackage{epsfig}
\usepackage{amssymb}
\usepackage{diagrams}
\usepackage[usenames,dvipsnames]{pstricks}
\usepackage{pst-grad} 
\usepackage{pst-plot} 

\usepackage[pagebackref, colorlinks, citecolor=red, linkcolor=blue]{hyperref}

\title{Homotopy theory for algebras over polynomial monads}

\date{February 20, 2017}

\author{M. A. Batanin}

\address{Macquarie University, North Ryde, 2109 Sydney, Australia.}

\email{michael.batanin@mq.edu.au}

\author{C. Berger}

\address{Universit\'e de Nice, Lab. J.-A. Dieudonn\'e, Parc Valrose, 06108 Nice, France.}

\email{cberger@math.unice.fr}



\def\NN{\mathbb{N}}
\def\Ee{\mathcal{E}}
\def\uEe{\underline{\mathcal{E}}}
\def\Ob{\mathrm{Ob}}

\def\Alg{\mathrm{Alg}}

\def\lra{\leftrightarrows}
\def\Pp{\mathcal{P}}

\def\colim{\mathrm{colim}}
\def\CC{\mathbb{C}}
\def\DD{\mathbb{D}}
\def\Set{\mathrm{Set}}
\def\FinSet{\mathrm{FinSet}}
\def\Poly{\mathbf{Poly}}

\def\Bq{\mathbf{Bouq}}
\def\Coll{\mathbf{Coll}}
\def\Oo{\mathcal{O}}
\def\op{\mathrm{op}}
\def\Hom{\mathrm{Hom}}
\def\Cat{\mathrm{Cat}}
\def\Int{\mathrm{Int}}
\def\sg{\sigma}
\def\inc{\hookrightarrow}
\def\into{\rightarrowtail}
\def\Aa{\mathcal{A}}
\def\Bb{\mathcal{B}}
\def\Uu{\mathcal{U}}
\def\Ff{\mathcal{F}}
\def\AA{\mathbb{A}}
\def\BB{\mathbb{B}}
\def\LL{\mathbb{L}}
\def\hocolim{\mathrm{hocolim}}

\def\eqv{\overset{\sim}{\longrightarrow}}
\def\Ord{\mathrm{Ord}}

\def\RR{\mathbb{R}}
\def\dim{\mathrm{dim}}
\def\Ch{\mathrm{Ch}}
\def\Arr{\mathrm{Arr}}

\newarrow{Into}{littlevee}--->

\theoremstyle{plain}
\newtheorem{theorem}[subsection]{Theorem}

\newtheorem{defin}[subsection]{Definition}
\newtheorem{pro}[subsection]{Proposition}
\newtheorem{lem}[subsection]{Lemma}
\newtheorem{corol}[subsection]{Corollary}
\theoremstyle{remark}
\newtheorem{remark}[subsection]{Remark}
\newtheorem{example}[subsection]{Example}
\newtheorem{examples}[subsection]{Examples}

\newcommand{\HS}{\mbox{${\bf T}^{\scriptstyle \tt S}$}}

\newcommand{\HT}{\mbox{${\bf T}^{{\scriptstyle \tt   T}}$}}
\newcommand{\HPhi}{\mbox{${\bf T}^{{\scriptstyle \tt   \Phi}}$}}
\newcommand{\bc}{\mbox{${\bf b}$}}
\newcommand{\B}{\mbox{$\mathfrak{B}$}}

\newcommand{\Z}{\mbox{$\mathbb{S}$}}

\newcommand{\h}{\mbox{${\bf T}^{\scriptstyle \tt T_{f,g}}$}}
\newcommand{\ho}{\mbox{${\bf T}^{\scriptstyle \tt T+{2}}$}}
\newcommand{\hi}{\mbox{${\bf T}^{\scriptstyle \tt T_{f}}$}}
\newcommand{\hj}{\mbox{${\bf T}^{\scriptstyle \tt T_{g}}$}}
\newcommand{\cop}{\mbox{${\bf T}^{\scriptstyle \tt T+{1}}$}}

\newcommand{\nn}{\mbox{${\bf NO(n)}^{\scriptstyle \tt NO(n)}$}}
\newcommand{\nnt}{\mbox{$({\bf NO(n)}^{\scriptstyle \tt NO(n)})_T$}}

\newcommand{\copn}{\mbox{${\bf NO(n)}^{\scriptstyle \tt NO(n)+1}$}}

\newcommand{\ts}{\mbox{${\bf t}$}}
\newcommand{\tso}{\mbox{${\bf t}_{0}$}}

\newcommand{\LTrees}{\mbox{${\mathtt{LinearTrees}}$}}
\newcommand{\RTrees}{\mbox{${\mathtt{PlanarRootedTrees}}$}}
\newcommand{\OTrees}{\mbox{${\mathtt{OrderedTrees}}$}}
\newcommand{\ORTrees}{\mbox{${\mathtt{OrderedRootedTrees}}$}}
\newcommand{\OGraphs}{\mbox{${\mathtt{OrderedGraphs}}$}}
\newcommand{\OPTrees}{\mbox{${\mathtt{OrderedPlanarTrees}}$}}
\newcommand{\nPTrees}{\mbox{${\mathtt{nPlanarRootedTrees}}$}}

\newcommand{\wa}{\mbox{${\bf w}$}}
\newcommand{\qa}{\mbox{${\bf q}$}}
\newcommand{\la}{\mbox{${\bf l}$}}

\newcommand{\boxx}{
{\unitlength=0.2mm
\begin{picture}(11,2)(440,0)
\put(440,10){\line(0,-1){10}}
\put(440,0){\line(1,0){10}}
\put(450,0){\line(0,1){10}}
\put(450,10){\line(-1,0){10}}
\end{picture}}}

\newcommand{\coeq}{ {\unitlength=0.25mm
\begin{picture}(35,15)(-10,1)
\put(5.4,1.5){\makebox(0,0){\mbox{$\longrightarrow$}}}
\put(5.4,7.5){\makebox(0,0){\mbox{$\longrightarrow$}}}
\put(6.3,13){\makebox(0,0){\mbox{$\scriptscriptstyle d_1 $}}}
\put(6.3,-3){\makebox(0,0){\mbox{$\scriptscriptstyle d_0 $}}}
\end{picture}}}

\newcommand{\T}{\mbox{$T_{1,1}$}}
\newcommand{\F}{\mbox{$\mathcal F$}}

\begin{document}

\setcounter{tocdepth}{1}

\maketitle

\begin{abstract}We study the existence and left properness of transferred model structures for ``monoid-like'' objects in monoidal model categories. These include genuine monoids, but also all kinds of operads as for instance symmetric, cyclic, modular, higher operads, properads and PROP's. All these structures can be realised as algebras over polynomial monads.

We give a general condition for a polynomial monad which ensures the existence and (relative) left properness of a transferred model structure for its algebras. This condition is of a combinatorial nature and singles out a special class of polynomial monads which we call tame polynomial. Many important monads are shown to be tame polynomial.\end{abstract}

{\footnotesize \tableofcontents}

\section*{Introduction}

This text emerged from model-theoretical properties needed in our approach to the stabilisation hypothesis of Baez-Dolan \cite{BD}.  Indeed, our proof \cite{BBC} of this
hypothesis relies on a careful homotopical analysis of \emph{$n$-operads} and their \emph{symmetrisation}. This analysis was initiated in \cite{SymBat,EHBat} but for the proof of the stabilisation hypothesis, the formulas obtained in \emph{loc. cit.} had to be extended to a much broader context. Moreover, our study of stabilisation phenomena in \cite{BBC2} required the existence of a certain left Bousfield localisation of the transferred model structure on $n$-operads which turns $n$-operads into higher categorical analogues of $E_n$-operads \cite{LocBat,cis}. However, the available techniques for left Bousfield localisation (cf. \cite{Hirschhorn}) are based on left properness, and it turned out to be surprisingly difficult to verify this property for $n$-operads. These two problems motivated us to reconsider foundational aspects of the homotopy theory of operads from a new categorical perspective.

In this article we address among others the general problem of \emph{preservation of left properness under transfer}. We show that, under some specific conditions on the base category, a certain form of preservation is given for an interesting class of transfers, including those for the known model structures on monoids \cite{SS}, reduced symmetric operads  \cite{BergerMoerdijk}, general non-symmetric operads \cite{FM}, reduced $n$-operads \cite{SymBat}.

The common feature of all these transfers is that in each case the algebraic structure is governed by a \emph{polynomial monad} in sets. Building on $2$-categorical techniques developed in \cite{EHBat,W,W2} we show that for this kind of algebraic structure, the \emph{existence of a transferred model structure} and its \emph{left properness} are intimately related. Both rely on a careful analysis of \emph{free algebra extensions}. In the case of free monoid extensions this analysis has been done by Schwede-Shipley \cite{SS} in an exemplary and prototypical way. It is among the main results of this article that an analogous analysis is available for free algebra extensions over a general polynomial monad, provided the latter satisfies an extra-condition. This condition is of a combinatorial nature: it requires a certain category (attached to the polynomial monad) to be a coproduct of categories with terminal object. A polynomial monad fulfilling this extra-condition will be called \emph{tame}. This article provides a \emph{combinatorial toolkit} to construct free algebra extensions over any tame polynomial monad.

All operads above are algebras over tame polynomial monads. In particular, we recover Muro's \cite{FM} recent construction of free non-symmetric operad extensions. Although non-symmetric operads can be viewed as monoids for a certain circle-product, the construction of these free operad extensions is highly non-trivial, since the circle-product commutes with colimits only on one side, while the Schwede-Shipley construction of free monoid extensions is based on a commutation with colimits on both sides. The availability of a Schwede-Shipley type construction for free algebra extensions depends on the behaviour of what we call \emph{semi-free coproducts}. These are coproducts of an algebra with a free algebra. Any \emph{tame} polynomial monad induces a ``polynomial expansion'' for semi-free coproducts. E. g., the underlying object of the coproduct $M\vee F_T(K)$ of a monoid $M$ with a free monoid $F_T(K)$ can be computed as follows (cf. \cite{SS} and Section \ref{SSM}):\begin{equation}\label{semfreemon} M\vee F_T(K) = \coprod_{n\ge 0} M\otimes (K\otimes M)^{\otimes n}.\end{equation}

The existence of an analogous functorial ``polynomial expansion'' for semi-free coproducts of $T$-algebras over a tame polynomial monad $T$ is the main ingredient for a transferred model structure on $T$-algebras with good properties.

Interestingly, the question of left properness has been dealt with in literature only recently in some special cases, cf. Cisinski-Moerdijk \cite[Theorem 8.7]{CM} and Muro \cite[Theorem 1.11]{FM2}. Although the \emph{monoid axiom} of Schwede-Shipley \cite{SS} gives a quite precise criterion for the existence of a transfer, the monoid axiom alone does not guarantee preservation of left properness under transfer. We propose in this article a common strengthening of the monoid axiom and of left properness which ensures that the transferred model structure on $T$-algebras is left proper if $T$ is tame polynomial. We also show (cf. \cite{HRY,Sp}) that a relative form of left properness (namely, weak equivalences between $T$-algebras with \emph{cofibrant underlying object} are closed under pushout along cofibrations) already follows from the monoid axiom.

This strengthening crucially involves a model-theoretical concept of \emph{$h$-cofibration}. We call a Quillen model category \cite{Quillen} \emph{$h$-monoidal} if it is a monoidal model category in the sense of Hovey \cite{Hovey} and the tensor product of a (trivial) cofibration with an arbitrary object is a (trivial) $h$-cofibration. An $h$-monoidal model category, which is \emph{compactly generated} \cite{BergerMoerdijk2}, satisfies the monoid axiom of Schwede-Shipley. Most of the model categories of algebraic topologists are compactly generated $h$-monoidal. If in addition the class of weak equivalences is closed under tensor product (e.g. all objects are cofibrant) then the model category is called \emph{strongly $h$-monoidal}.\vspace{1ex}

Our main theorem can now be stated as follows:

\begin{theorem}For any tame polynomial monad $\,T$ in sets and any compactly generated monoidal model category $\,\Ee$ fulfilling the monoid axiom, the category of $\,T$-algebras in $\,\Ee$ admits a relatively left proper transferred model structure.

The transferred model structure is left proper provided $\,\Ee$ is strongly $h$-monoidal.\end{theorem}

Examples of $T$-algebras in $\,\Ee$ for tame polynomial monads $T$ include monoids, non-symmetric operads, reduced symmetric operads, reduced $n$-operads, reduced  cyclic operads, as well as higher opetopic extensions of all these structures. In several of these cases, existence results for a transferred model structure (under some conditions on $\Ee$) were known before. It seems however that even in the known cases, our assumptions on $\Ee$ are weaker than those which appeared in literature. Moreover, the discussion of left properness { in the generality considered here} seems to be new. Even more importantly, the uniformity of our approach allows us to give explicit formulas for the \emph{total left derived functors} induced by morphisms of  polynomial monads. These explicit formulas have often concrete applications.

Failure of tameness for a polynomial monad very often produces obstructions for the existence of transfer. However, these obstructions can be removed by imposing more restrictive conditions on the base category $\Ee$. For instance, if $\Ee$ is the category of chain complexes over a field of characteristic $0,$ or the category of simplicial sets, resp. compactly generated topological spaces with the Quillen model structure, then a transferred model structure for algebras in $\Ee$ over any (tame or not tame) polynomial monad exists. The question of left properness is more subtle, yet. \vspace{1ex}

The article is subdivided into four rather independent parts.\vspace{1ex}

Part 1 develops basic properties of $h$-monoidal model categories and relates this notion to the monoid axiom of Schwede-Shipley \cite{SS}. We recall the concept of compact generation \cite{BergerMoerdijk2} of a monoidal model category and give general ``admissibility'' conditions on a monad, sufficient for the existence and relative left properness of the transfer. Two themes are treated in some detail: monoids in $h$-monoidal model categories  (closely following Schwede-Shipley \cite{SS} but adding left properness) and stability of $h$-monoidality under passage to ``convoluted'' diagram categories (here we extend some of the results of Dundas-{\O}stv{\ae}r-R\"ondigs \cite{DOR}).

Part 2 is devoted to polynomial monads. This part relies on $2$-categorical techniques developed in \cite{EHBat,W2}, but we have tried to keep the presentation as self-contained as possible. These techniques are used to reformulate the construction of a pushout along a free $T$-algebra map as a left Kan extension of a certain functor attached to $\,T$. More generally, given a morphism of polynomial monads $S\to T$, the induced functor $Alg_S(\Ee)\to Alg_T(\Ee)$, left adjoint to restriction, is expressible as a left Kan extension.  Our main theorem then follows by combining this construction with the results of Part 1, since the explicit formula for free algebra extensions over a tame polynomial monad implies the admissibility of the latter.

At the end we study the Quillen adjunction induced by a morphism of tame polynomial monads. We show that in good cases the total left derived functor can be calculated as a homotopy colimit. Instances of this appear in \cite{SymBat}, where a higher-categorical Eckmann-Hilton argument is used to show that the derived symmetrisation of the terminal $n$-operad is homotopy-equivalent to the Fulton-MacPherson operads of compactified  point configurations in $\mathbb{R}^n$; and in \cite{G}, where Giansiracusa computes the derived modular envelope of several cyclic operads. Doing so he closely follows Costello \cite{Cos} who suggested that the derived modular envelope of the terminal planar cyclic operad is homotopy equivalent to the modular operad of nodal Riemann spheres with boundary. Notice that the main obstacle for Giansiracusa and Costello in rendering such a statement precise was the missing model structure on cyclic operads. This is by now not anymore the case.

Part 3 studies examples. We first show that the polynomial monads based on contractible graphs (i.e. trees) tend to be tame, at least their normalised or reduced versions\footnote{See Remarks \ref{Cavigliared} and \ref{Giovanni} regarding a subtlety  of the definitions of reduced operad and PROP.}. We then show that most of the polynomial monads for operads which are based on graphs rather than trees (such as modular operads, properads or PROP's) are not tame, even if we consider just their normalised versions. For all these operad types there is no transferred model structure for chain operads in positive characteristics. Nevertheless, a transfer exists in characteristic $0$. We get the surprising result that for \emph{any} polynomial monad $T$, its Baez-Dolan $+$-construction $T^+$ is a tame polynomial monad. This can be used to define a homotopy theory of homotopy $T$-algebras for any polynomial monad $T$.

We finally study in detail the monad for normalised $n$-operads. In \cite{EHBat,SymBat}, the latter was shown to be polynomial. Here we show that it is tame polynomial, which is quite a bit harder. This particular example is of special interest for us for reasons explained at the beginning of this introduction. In fact, it was this example which motivated the whole project. We also show that the polynomial monads for general (non-reduced) symmetric, cyclic and $n$-operads for $n\ge 2$, are not tame.

Part 4 contains a concise combinatorial definition of the notions of graph, tree and graph insertion. We decided to include this material here because the tameness of a given polynomial monad often relies on subtle properties of a canonically associated class of structured graphs. The monad multiplication reflects the operation of \emph{insertion of a graph into the vertex of another graph of the same class}. This close relationship between graph insertion and operad types is actually the central idea of Markl's article \cite{Markl}, and has been further developed in the recent book by Johnson-Yau \cite{JY}, cf. also chapter 13.14 in the book \cite{Loday} of Loday-Vallette. Our study of algebras over polynomial monads subsumes all these examples. Each class of graphs which is suitably closed under graph insertion defines a polynomial monad for which the question of tameness can be raised.

For the reader's convenience we present here two tables  which summarise various results obtained in this article. The first table presents monoidal model categories considered in Part 1 for which we were able to establish (strong)  $h$-monoidality.

\begin{figure}[h]
\centering 
\resizebox*{!}{8.5cm}{
\begin{tabular}{|| l || c | c | c ||  } 
 \hline \hline

 & & & \\
   &all objects  & strongly&

   \\ [-1ex]
 \raisebox{2ex}{Monoidal model category} &\raisebox{0.5ex}{cofibrant} & \raisebox{0.5ex}{ $h$-monoidal} & \raisebox{1.5ex}{h-monoidal}
\\ [1ex]
   \hline \hline

{Simplicial sets} (Quillen) & yes  &yes &{yes}
 \\
    \hline

Small categories (groupoids) (Joyal-Tierney)  &  yes  & yes& {yes}
 \\
   \hline
 Complete $\Theta_n$-spaces (Rezk) & yes & yes &{yes}
 \\ 

   \hline
Chain complexes over a field & yes & yes & yes
 \\
 \hline
C.g. topological spaces (Str\o m) & yes & yes & {yes} \\

 \hline

Modules over a commutative monoid & & & \\
in a monoidal category with cofibrant objects & \raisebox{1ex}{yes}  & \raisebox{1ex}{yes}&\raisebox{1ex}{yes}\\
    \hline

C.g. topological spaces (Quillen) & no  & yes & {yes}

 \\

   \hline
 Small $2$-categories ($2$-groupoids) with        &      &          & \\
 Gray tensor product  (Lack)    & \raisebox{1ex}{no} & \raisebox{1ex}{yes}  & \raisebox{1ex}{yes}\\

  \hline

Modules over a commutative monoid & & & \\
in a strongly $h$-monoidal model category  & \raisebox{1ex}{no}  & \raisebox{1ex}{yes}&\raisebox{1ex}{yes}
\\
  \hline

Chain complexes over a commutative ring  & & & \\
with the projective model structure & \raisebox{1ex}{no}  & \raisebox{1ex}{no}&\raisebox{1ex}{yes}\\

 \hline

Symmetric spectra in simplicial sets with & & & \\
levelwise or stable projective  model structures & \raisebox{1ex}{no}  & \raisebox{1ex}{no}&\raisebox{1ex}{yes}\\

 \hline

Modules over commutative monoid & & & \\
in an $h$-monoidal model category  & \raisebox{1ex}{no}  & \raisebox{1ex}{no}&\raisebox{1ex}{yes}\\
 \hline\hline
\end{tabular}}

\caption{$h$-monoidal model categories}
\end{figure}

\newpage

The second table contains the polynomial monads treated in Part 3 for which the question of tameness has been settled. This table refines a similar table contained in the aforementioned article \cite{Markl} by Markl.

\begin{figure}[h]

\centering 
\resizebox*{!}{14cm}{
\begin{tabular}{|| l | c || c || c ||  } 
 \hline \hline
   &  & &

   \\ [-1ex]
 \raisebox{1.5ex}{polynomial monad for} &\raisebox{1.5ex}{ \ \ type} & \raisebox{1.5ex}{insertional class of graphs} & \raisebox{1.5ex}{tame}
\\ [1ex]
   \hline \hline

{$\CC$-diagrams} & & {$\CC$-coloured corollas}&{yes}
 \\
    \hline

{monoids} &  & {linear rooted trees}&{yes}
 \\
   \hline

enriched categories  &  & $I$-coloured linear &
 \\
 with object-set $I$  & & rooted trees & \raisebox{1ex} {yes}                   \\

   \hline

{non-symmetric operads} & general & {planar rooted trees}&{yes}
 \\ 

   \hline
{symmetric operads} & general & rooted trees &{no}
 \\

   \cline{2-4}
 &reduced & {non-degenerate rooted trees}&{yes}
    \\
    \cline{2-4}

& constant-free  & {regular rooted trees}&{yes}
 \\
  \cline{2-4}
 & normal & {normal rooted trees}&{yes}
 \\
 \hline

{planar cyclic operads}& general  & planar trees &{no}
 \\
   \cline{2-4}

{ }& reduced & non-degenerate planar trees & yes

\\
   \cline{2-4}

{ }& constant-free& regular planar trees & yes

\\
   \cline{2-4}

{ }& normal & normal planar trees & yes

 \\
   \hline

cyclic operads & general & trees &{no}

 \\
   \cline{2-4}

& reduced   & non-degenerate trees &{yes}

 \\
   \cline{2-4}
& constant-free & regular  trees &{yes}

\\
   \cline{2-4}

{ }& normal & normal  trees & yes

 \\
   \hline

{$n$-operads for $n\ge 2$}& general & $n$-planar trees   &no

 \\
   \cline{2-4}

& reduced & non-degenerate $n$-planar trees &yes
 \\
   \cline{2-4}

& constant-free& regular $n$-planar trees &yes

 \\
   \cline{2-4}

&normal& normal $n$-planar trees &yes

 \\
 \hline

dioperads & general& directed trees &no

 \\
   \cline{2-4}

& normal& normal directed trees &yes

 \\
 \hline

 $\frac{1}{2}$PROP's& general &  $\frac{1}{2}$graphs &no

 \\
   \cline{2-4}

& normal  & normal $\frac{1}{2}$graphs &yes

\\
\hline
{modular operads} & general & {connected  graphs (with genus)}&no

\\
   \cline{2-4}
& normal & normal {connected (stable) graphs}&no

 \\

   \hline

   {properads} & general & {loop-free connected directed  graphs }&no

\\
   \cline{2-4}
 {} & normal &normal  {loop-free connected directed  graphs }&no

 \\  \hline

{PROP's} & general & {loop-free directed graphs}&no
 \\

    \cline{2-4}

& normal& normal {loop-free directed graphs}&no\\
\hline 

{wheeled operads} & general & {wheeled rooted trees }&{no}
 \\

    \cline{2-4}

& normal& normal  {wheeled rooted trees }& {yes}\\
\hline

  {wheeled properads} & general & {  connected directed  graphs }&no

\\
   \cline{2-4}
 {} & normal &normal  { connected directed  graphs }&no

 \\  \hline

{wheeled PROP's} & general  &directed graphs&no
 \\

    \cline{2-4}

& normal & normal {directed graphs}&no
\\
\hline
\hline
\end{tabular}}

\caption{Polynomial monads based on graphs}\label{table2}\end{figure}

\newpage

\part{Model structure for algebras over admissible monads}

\section{Homotopy cofibrations and $h$-monoidal model categories}\label{Hcof}

We introduce and investigate here a model-theoretical concept of \emph{$h$-cofibration}\footnote{Mike Hopkins suggested the term \emph{flat} morphism for $h$-cofibration, cf. \cite{HHR}. Our terminology is in conflict with the terminology of Peter May and Kate Ponto \cite[Chapter 17]{MayPonto} who use the term $h$-cofibration (resp. $q$-cofibration) to designate a cofibration in Str{\o}m's  (resp. Quillen's) model structure on topological spaces. We argue that this terminological conflict is not peculiar, because our $h$-cofibrations are defined for \emph{all} model categories, while the $h$-cofibrations of May and Ponto are defined for \emph{topological} model categories only. In those cases where the two definitions telescope each other, they do so in a nice way, cf. the proof of Proposition \ref{kspaces}, especially its footnote.} which seems interesting in itself. The dual concept of an \emph{$h$-fibration} has been studied by Rezk \cite{Rezk1} under the name of sharp map. The definition of an $h$-cofibration only depends on the class of weak equivalences and on the existence of pushouts. In left proper model categories, the class of $h$-cofibrations can be considered as the closure of the class of cofibrations under \emph{cofibre equivalence}, cf. Proposition \ref{hcofchar}. We prove a useful \emph{gluing lemma} for weak equivalences between pushouts along $h$-cofibrations in left proper model categories, cf. Proposition \ref{gluing}. We mainly use $h$-cofibrations in order to formulate a strengthening of left properness (\emph{$h$-monoidality}) well adapted to monoidal model categories. A similar concept has been developed by Dundas-{\O}stv{\ae}r-R\"ondigs \cite[Definition 4.6]{DOR}.

In the course of studying basic properties of $h$-monoidal model categories, we establish in Propositions \ref{twin1}, \ref{twin2} and \ref{detect} below some useful recognition criteria for $h$-monoidality. Since a compactly generated $h$-monoidal model categories fulfills the monoid axiom of Schwede-Shipley (Proposition \ref{couniv0}) and is left proper (Lemma \ref{mainimplications}), these criteria are helpful tools in establishing the monoid axiom of Schwede-Shipley and/or the left properness of a given monoidal model category.

We have been advertised by Maltsiniotis that the notion of $h$-cofibration already appears in some of Grothendieck's unpublished manuscripts. In recent work, Ara and Maltsiniotis use $h$-cofibrations to prove a version of the transfer theorem \cite[Proposition 3.6]{AM} which is related to our Theorem \ref{transfer}.\vspace{1ex}

For an object $X$ of a category $\Ee$ the undercategory $X/\Ee$ has as objects morphisms with domain $X$, and as morphisms ``commuting triangles'' under $X$. If $\Ee$ carries a model structure then so does $X/\Ee$. A map $X\to A\to B$ in $X/\Ee$ is a cofibration, weak equivalence, resp. fibration if and only if the underlying map $A\to B$ in $\Ee$ is.

\begin{defin}\label{h-cofibration}A morphism $f:X \rightarrow Y$ in a model category $\Ee$ is called an \emph{$h$-cofibration} if the functor $f_!:X/\Ee\to Y/\Ee$ (given by cobase change along $f$) preserves weak equivalences.\end{defin}

In more explicit terms, a morphism $f:X\to Y$ in $\Ee$ is an \emph{$h$-cofibration} if and only if in each commuting diagram of pushout squares in $\Ee$
\begin{equation}\label{hcof}\begin{diagram}[small]X&\rTo&A&\rTo^w&B\\\dTo^f&&\dTo&&\dTo\\Y&\rTo&A'&\rTo_{w'}&B'\end{diagram}\end{equation}
$w'$ is a weak equivalence as soon as $w$ is.

\begin{lem}\label{leftproper}A model category is left proper if and only if each cofibration is an $h$-cofibration.\end{lem}

\begin{proof}If in diagram (\ref{hcof}) above $\Ee$ is left proper, $f$ is a cofibration and $w$ a weak equivalence, then $w'$ is a weak equivalence as well, thus $f$ is an $h$-cofibration. Conversely, if every cofibration is an $h$-cofibration, then (\ref{hcof}) shows that weak equivalences are stable under pushout along cofibrations, so that $\Ee$ is left proper.\end{proof}

\begin{lem}\label{coprodhcof}The class of $\,h$-cofibrations is closed under composition, cobase change, retract, and under formation of finite coproducts.\end{lem}

\begin{proof}Closedness under composition, cobase change and retract follows immediately from the definition. For closedness under finite coproducts, it is enough to show that for each $h$-cofibration $f:X\to Y$ and each object $Z$, the coproduct $f\sqcup 1_Z$ is an $h$-cofibration. This follows from the fact that the commutative square

\begin{diagram}[small]X&\rTo  &X\sqcup Z \\\dTo^f&&\dTo_{f\sqcup 1_Z}\\Y&\rTo&Y\sqcup Z\end{diagram}
is a pushout.\end{proof}

\begin{lem}\label{charhcof}--
\begin{itemize}
\item[(i)]An object $Z$ is $h$-cofibrant if and only if $-\sqcup Z$ preserves weak equivalences.
\item[(ii)]The class of weak equivalences is closed under finite coproducts if and only if all objects of the model category are $h$-cofibrant.
\item[(iii)]The class of weak equivalences is closed under arbitrary coproducts whenever all objects are $h$-cofibrant and weak equivalences are closed under filtered colimits along coproduct injections.
\end{itemize}\end{lem}

\begin{proof}The first statement expresses the fact that a pushout along the map from an initial object to $Z$ is the same as taking the coproduct with $Z$. The second statement follows from the first and from the identity $f\sqcup g=(1_{\mathrm{codomain}(f)}\sqcup g)\circ(f\sqcup 1_{\mathrm{domain}(g)}).$ The third statement follows from the second and the fact that any coproduct can be calculated as a filtered colimit of finite coproducts with structure maps being coproduct injections.\end{proof}

\subsection{Homotopy pushouts and cofibre equivalences}--\vspace{1ex}

A commutative square in a general model category
\begin{diagram}[small]A&\rTo&B\\\dTo&&\dTo\\C&\rTo&D\end{diagram}will be called a \emph{formal homotopy pushout} if for any factorisation $A\into B'\eqv B$ (resp. $A\into C'\eqv C$) of $A\to B$ (resp. $A\to C$) into a cofibration followed by a weak equivalence, the induced map $B'\cup_AC'\to D$ is a weak equivalence.

In left proper model categories, it suffices to choose a particular factorisation of $A\to B$ and check that $B'\cup_AC\to D$ is a weak equivalence. We will see in Corollary \ref{formal=real} below that in left proper model categories, formal homotopy pushouts are homotopy pushouts (i.e. homotopy colimits). Yet, in order to avoid any confusion, we maintain the terminological distinction.\vspace{1ex}

The following proposition gives several useful characterisations of $h$-cofibrations in left proper model categories. Left properness is essential here because formal homotopy pushouts are easier to recognise in left proper model categories than in general model categories. For instance, in a left proper model category any pushout along a cofibration is a formal homotopy pushout, which is not the case in general model categories.

A weak equivalence $w$ in $X/\Ee$ is called a \emph{cofibre equivalence} if for each morphism $g:X\to B$ the functor $g_!:X/\Ee\to B/\Ee$ takes $w$ to a weak equivalence in $B/\Ee$.

\begin{pro}\label{hcofchar}In a left proper model category $\Ee$, the following four properties of a morphism $f:X\to Y$ are equivalent:\begin{itemize}\item[(i)]$f$ is an $h$-cofibration;\item[(ii)]every pushout along $f$ is a formal homotopy pushout;\item[(iii)]for every factorisation of $f$ into a cofibration followed by a weak equivalence, the weak equivalence is a cofibre equivalence;\item[(iv)]there exists a factorisation of $f$ into a cofibration followed by a cofibre equivalence.\end{itemize}\end{pro}

\begin{proof}(i)$\implies$(ii) For a given outer pushout rectangle like in diagram (\ref{hcof}) above with an $h$-cofibration $f$, factor the given map $X\to B$ as a cofibration $X\into A$ followed by a weak equivalence $w:A\eqv B$, and define $Y\to A'$ as the pushout of $X\into A$ along $f$. Since the right square is then a pushout, $w':A'\to B'$ is a weak equivalence, whence the outer rectangle is a formal homotopy pushout.

(ii)$\implies$(iii) The pushout
\begin{diagram}[small]X&\rTo^g&B\\\dTo^f&&\dTo\\Y&\rTo&B'\end{diagram}
is factored vertically according to a factorisation of $f$ into a cofibration $X\into Z$ followed by a weak equivalence $v:Z\eqv Y$:\begin{diagram}[small]X&\rTo^g&B\\\dInto&&\dInto\\Z&\rTo& Z'\\\dTo^v&&\dTo_{v'}\\Y&\rTo&B.\end{diagram}Since the outer rectangle is a formal homotopy pushout by assumption, $v'$ is a weak equivalence, whence $v$ is a cofibre equivalence.

(iii)$\implies$(iv) This is obvious.

(iv)$\implies$(i) Consider a commutative diagram like in (\ref{hcof}) above, and factor $f$ into a cofibration $X\into Z$ followed by a cofibre equivalence $v:Z\eqv Y$. This induces the following commuting diagram of pushout squares\begin{diagram}[small]X&\rTo&A&\rTo^w&B\\\dInto&&\dInto&&\dInto\\Z&\rTo&W&\rTo^{w''}&Z'\\\dTo^v&&\dTo_{v''}&&\dTo_{v'}\\Y&\rTo&A'&\rTo^{w'}&B'\end{diagram}
in which $v''$ and $v'$ are weak equivalences because $v$ is a cofibre equivalence, and $w''$ is a weak equivalence by left properness of $\Ee$. Therefore, the $2$-out-of-$3$ property of the class of weak equivalences implies that $w'$ is a weak equivalence as well, and hence $f$ is an $h$-cofibration as required.\end{proof}

\begin{lem}\label{formal}Let $\Ee$ be a left proper model category.
\begin{itemize}\item[(a)]A commutative square\begin{diagram}[small]A&\rTo&A'\\\dTo^w_\sim&&\dTo_{w'}\\B&\rTo&B'\end{diagram}with a left vertical weak equivalence $w:A\eqv B$ is a formal homotopy pushout if and only if $w':A'\to B'$ is a weak equivalence;\item[(b)]In a commutative diagram\begin{diagram}[small]A&\rTo&B&\rTo&C\\\dTo&&\dTo&&\dTo\\A'&\rTo&B'&\rTo&C'\end{diagram}in which the left square is a formal homotopy pushout, the right square is a formal homotopy pushout if and only if the outer rectangle is a formal homotopy pushout.\end{itemize}\end{lem}

\begin{proof}(a) Factoring $A\to A'$ into a cofibration $A\into A''$ followed by a weak equivalence $v:A''\eqv A'$ and taking a pushout we get the following commutative square\begin{diagram}[small]A&\rInto&A''&\rTo^v&A'\\\dTo^w&&\dTo^{w''}&&\dTo_{w'}\\B&\rInto&B''&\rTo_{v'}&B'\end{diagram}in which $w''$ is a weak equivalence by left properness. Therefore, by the $2$-out-of-$3$ property of weak equivalences, $v'$ is a weak equivalence (i.e. the given square is a formal homotopy pushout) if and only if $w'$ is a weak equivalence.

(b) Factoring $A\to A'$ into a cofibration $A\into A''$ followed by a weak equivalence $v:A''\eqv A'$ and taking pushouts we get the following commutative diagram\begin{diagram}[small]A&\rTo&B&\rTo&C\\\dInto&&\dInto&&\dInto\\A''&\rTo&\NWpbk B''&\rTo&\NWpbk C''\\
\dTo^v_\sim&&\dTo^{v''}_\sim&&\dTo_{v'}\\A'&\rInto&B'&\rTo&C'\end{diagram}in which $v''$ is a weak equivalence because the left square above is a formal homotopy pushout. Now $v'$ is a weak equivalence if and only if either the right square above is a formal homotopy pushout or, equivalently, the outer rectangle above is a formal homotopy pushout.\end{proof}

\begin{pro}[cf. Proposition B.12 in \cite{HHR}]\label{gluing}--

In a model category $\Ee$, consider the following commutative diagram
\begin{diagram}[small]C&\lTo&A&\rTo^u&B\\\dTo^{f_\gamma}&&\dTo^{f_\alpha}&&\dTo_{f_\beta}\\C'&\lTo&A'&\rTo_{u'}&B'\end{diagram}in which $f_\alpha,f_\beta,f_\gamma$ are supposed to be weak equivalences.\vspace{1ex}

Then the induced map on pushouts $f_\gamma\cup_{f_\alpha}f_\beta:C\cup_AB\to C'\cup_{A'}B'$ is a weak equivalence provided one of the following two conditions is satisfied:
\begin{itemize}\item[(a)]all objects of the diagram are cofibrant and the maps $u,u'$ are cofibrations;\item[(b)]the model category $\Ee$ is left proper and the maps $u,u'$ are $h$-cofibrations.\end{itemize}\end{pro}

\begin{proof}Case (a) is classical, see e.g. Goerss-Jardine \cite[Lemma 9.8]{GoeJa}. Alternatively, (a) can be deduced from (b) by the following trick: the full subcategory of $\Ee$ spanned by the cofibrant objects is a \emph{category of cofibrant objects} in the sense of Brown and the proof of (b) holds in such categories (in which every cofibration is an $h$-cofibration because all objects are cofibrant).

To prove (b) let us write $D=C\cup_AB$ and $D'=C'\cup_{A'}B'$. Then the induced map $f_\delta:D\to D'$ is part of the following commutative cube

\begin{gather}\label{cube0}\begin{diagram}[size=1.6em,silent,UO]&&C&\rTo&&&D\\&\ruTo&\vLine_{f_\gamma}&&&\ruTo&\dTo_{f_\delta}\\A&&&\rTo_u&B&&\\
\dTo_{f_\alpha}&&\dTo&&\dTo_{f_\beta}&&\\&&C'&\hLine&\VonH&\rTo&D'\\&\ruTo&&&&\ruTo&\\A'&\rTo_{u'}&&&B'&&\end{diagram}\end{gather}
in which top and bottom squares are formal homotopy pushouts by Proposition \ref{hcofchar}ii, and the left square is a formal homotopy pushout by Lemma \ref{formal}a. It follows then from Lemma \ref{formal}b that the right square is also a formal homotopy pushout. This implies by Lemma \ref{formal}a that $f_\delta$ is a weak equivalence as required.\end{proof}

\begin{corol}\label{formal=real}In left proper model categories, formal homotopy pushouts are homotopy pushouts in the homotopical sense.\end{corol}

\begin{proof}In cube (\ref{cube0}) assume that $f_\alpha:A\eqv A'$ is a cofibrant replacement of $A'$, and define $B$ (resp. $C$) by factoring $A\eqv A'\to B'$ (resp. $A\eqv A'\to C'$) as a cofibration $A\into B$ (resp. $A\into C$) followed by a weak equivalence $f_\beta:B\eqv B'$ (resp. $f_\gamma:C\eqv C'$). The top square is then a formal homotopy pushout and $D=B\cup_AC$ realises the homotopy colimit of the lower diagram $C'\leftarrow A'\to B'$.

If the bottom square is a formal homotopy pushout then the proof of Proposition \ref{gluing}b shows that the canonical map $f_\delta:D\to D'$ from the homotopy pushout to the formal homotopy pushout is a weak equivalence.\end{proof}

A weak equivalence which is an $h$-cofibration will be called a \emph{trivial $h$-cofibration}. A weak equivalence which remains a weak equivalence under any cobase change will be called \emph{couniversal}. For instance, trivial cofibrations are couniversal weak equivalences.

\begin{lem}\label{couniv}In general, couniversal weak equivalences are trivial $h$-cofibrations. In left proper model categories, the converse holds: trivial $h$-cofibrations are couniversal weak equivalences.\end{lem}

\begin{proof}The $2$-out-of-$3$ property of the class of weak equivalences implies that couniversal weak equivalences are trivial $h$-cofibrations. In a left proper model category, the pushout of a trivial $h$-cofibration is a trivial $h$-cofibration in virtue of Lemmas \ref{coprodhcof}, \ref{hcofchar}ii and \ref{formal}a.\end{proof}

\begin{defin}\label{h-monoidal}A model category is called \emph{$h$-monoidal} if it is a monoidal model category \cite{Hovey} such that for each (trivial) cofibration $f:X\to Y$ and each object $Z$, the tensor product $f\otimes 1_Z:X\otimes Z\to Y\otimes Z$ is a (trivial) $h$-cofibration.

It is called \emph{strongly $h$-monoidal} if moreover the class of weak equivalences is closed under tensor product.\end{defin}
In particular, each cofibration is an $h$-cofibration so that, by Lemma \ref{leftproper}, $h$-monoidal model categories are \emph{left proper}. Moreover, in virtue of Lemma \ref{couniv}, the condition on trivial cofibrations can be considered as a weak form of the \emph{monoid axiom} of Schwede-Shipley \cite{SS}, cf. Proposition \ref{couniv0} and Corollary \ref{couniv1} below.

\begin{lem}\label{mainimplications}For monoidal model categories the following implications hold:\begin{center}all objects cofibrant $\implies$ strongly $h$-monoidal  $\implies$  $h$-monoidal $\implies$ left proper.\end{center}\end{lem}
\begin{proof}For the first implication, it suffices to observe that, by a well-known argument of Rezk, if all objects are cofibrant then the model structure is left proper, i.e. (by \ref{leftproper}) cofibrations are $h$-cofibrations. Moreover, the pushout-product axiom implies that tensoring a (trivial) cofibration with an arbitrary object yields again a (trivial) cofibration. Therefore, the model structure is $h$-monoidal. The class of weak equivalences is closed under tensor product, since by Brown's Lemma (if all objects are cofibrant) each weak equivalence factors as a trivial cofibration followed by a retraction of a trivial cofibration. The other two implications are obvious.\end{proof}

It is in general difficult to describe explicitly the class of $h$-cofibrations of a model category. The following three propositions are useful since they are applicable even if such an explicit description is unavailable.

\begin{pro}\label{twin1}Let $\Ee$ be a closed symmetric monoidal category with two model structures, called resp. \emph{injective} and \emph{projective}, and sharing the same class of weak equivalences. We assume that the following three properties hold:

\begin{itemize}\item the projective model structure is a monoidal model structure;
\item the injective model structure is left proper;\item tensoring a (trivial) cofibration of the projective model structure with an arbitrary object yields a (trivial) cofibration of the injective model structure.\end{itemize}
Then the projective model structure is $h$-monoidal.\end{pro}

\begin{proof}Observe that the notion of $h$-cofibration only depends on the class of weak equivalences, hence both model structures have the same class of $h$-cofibrations. The statement then follows directly from Lemma \ref{leftproper}.\end{proof}

\begin{pro}\label{twin2}Let $\Ee$ be a symmetric monoidal category with two monoidal model structures such that each cofibration (resp. weak equivalence, resp. fibration) of the first model structure is an $h$-cofibration (resp. weak equivalence, resp. fibration) of the second. If all objects of the first model structure are cofibrant then both model structures are $h$-monoidal.\end{pro}

\begin{proof}Since all objects of the first model structure are cofibrant, the first model structure is (strongly) $h$-monoidal by Lemma \ref{mainimplications}. Since the trivial fibrations of the first structure are among the trivial fibrations of the second, the cofibrations of the second are among the cofibrations of the first. The latter class is closed under tensor product and contained in the class of $h$-cofibrations of the second structure. This yields the first half of $h$-monoidality for the second model structure. Similarly, since the fibrations of the first model structure are among the fibrations of the second, the trivial cofibrations of the second are among the trivial cofibrations of the first. The latter class is closed under tensor product and contained in the class of trivial $h$-cofibrations of the second model structure. This shows that the second half of $h$-monoidality holds for the second model structure as well.\end{proof}

\begin{pro}\label{detect}Let $\,\Ee$ be a monoidal model category in which all objects are fibrant. Then $\,\Ee$ is $h$-monoidal provided the internal hom of $\,\Ee$ detects weak equivalences in the following sense: a map $f:X\to Y$ is a weak equivalence whenever $\uEe(f,W)$ is a weak equivalence for all objects $W$.\end{pro}
\begin{proof}Let $f: X\to Y $ be a cofibration. We have to show that in
\begin{diagram}[small]X\otimes Z&\rTo&A&\rTo^w&B\\\dTo^{f\otimes Z}&&\dTo&&\dTo\\Y\otimes Z&\rTo&A'&\rTo_{w'}&B'\end{diagram}
$w'$ is a weak equivalence if $w$ is. For this it suffices to show that $\uEe(w',W)$ is a weak equivalence for all $W$, which follows from the pushout-product axiom and the hom-tensor adjunction. If $f$ is a trivial cofibration then $\uEe(f\otimes Z,W)\cong\uEe(f,\uEe(Z,W))$ is a trivial fibration for each object $W$. Hence $f\otimes Z$ is a weak equivalence.\end{proof}

\begin{remark}We implicitly used in the preceding proof that $\Ee$ is \emph{right proper} because all of its objects are fibrant, and that therefore weak equivalences in $\Ee$ are stable under pullback along fibrations. Proposition \ref{detect} deduces from this and the good behaviour of the internal hom of $\Ee$ that $\Ee$ is $h$-monoidal and hence in particular \emph{left proper}. This explicit relationship between right and left properness in monoidal model categories does not seem to have been observed before.\end{remark}

\begin{corol}\label{kspaces}The category of compactly generated topological spaces is strongly $h$-monoidal with respect to Str\o m's and Quillen's cartesian model structures.\end{corol}

\begin{proof}Recall that in Str\o m's model structure the weak equivalences and fibrations are homotopy equivalences and Hurewicz fibrations respectively; the corresponding classes in Quillen's model structure are weak homotopy equivalences and Serre fibrations. These classes verify the inclusion relations required by Proposition \ref{twin2}. The cofibrations of Str\o m's model structure are the closed cofibrations (i.e. NDR-pairs) in the topologist's classical sense. Closed cofibrations are $h$-cofibrations for Quillen's model structure\footnote{In order to show that weak homotopy equivalences are preserved under cobase-change along closed cofibrations, it is enough to consider weak homotopy equivalences which are closed cofibrations. In other words, for any NDR-triad $(X,X_1,X_2)$ we have to show that if $(X_1,X_1\cap X_2)$ is $\infty$-connected then so is $(X,X_2)$. This holds if all spaces are $CW$-complexes. $CW$-approximation reduces thus the problem to a gluing lemma for weak homotopy equivalences between NDR-triads. This follows from the analogous statement for excisive triads \cite[Theorem 6.7.9]{Dieck} because each NDR-triad is homotopy equivalent to an excisive triad by a double mapping cylinder construction.}. In Str\o m's model structure all objects are cofibrant so that it is strongly $h$-monoidal by Lemma \ref{mainimplications}. Proposition \ref{twin2} implies that Quillen's model structure is $h$-monoidal. It is strongly $h$-monoidal since the product of two weak homotopy equivalences is again a weak homotopy equivalence.\end{proof}

\begin{corol}\label{chaincomplexes}The following two examples are $h$-monoidal model categories:

\begin{itemize}\item the category of chain complexes over a commutative ring with the projective model structure;
\item the category of symmetric spectra (in simplicial sets) with the stable projective model structure.\end{itemize}\end{corol}

\begin{proof}We use in both cases Proposition \ref{twin1}. Recall that the cofibrations of the injective (resp. projective) model structure on chain complexes are the monomorphisms (resp. monomorphisms with degreewise projective quotient). In particular, a projective cofibration $f_\bullet:X_\bullet\to Y_\bullet$ is degreewise split so that $f_\bullet\otimes Z_\bullet$ is degreewise a monomorphism, and hence a cofibration in the injective model structure. If $f_\bullet$ is trivial (i.e. a quasi-isomorphism), its degreewise projective quotient $Y_\bullet/X_\bullet$ is acyclic, and hence contractible. Therefore $(Y_\bullet/X_\bullet)\otimes Z_\bullet$ is contractible as well, and hence $f_\bullet\otimes Z_\bullet$ is trivial as required. The statement about symmetric spectra follows by an analogous argument from Proposition III.1.11i, Lemma III.1.4 and Theorem III.2.2 of Schwede's book project \cite{Schwede}.\end{proof}

\begin{examples}Below a list of frequently used monoidal model categories in which all objects are cofibrant. By Lemma \ref{mainimplications} they are thus strongly $h$-monoidal.\vspace{1ex}

\begin{itemize}\item simplicial sets;
\item small categories with the folklore model structure, cf. \cite{Lack,BergerMoerdijk2};
\item Rezk's model for $(\infty,n)$-categories \cite{Rezk};
\item compactly generated spaces with Str\o m's model structure;
\item chain complexes over a field with the projective model structure.
\end{itemize}

There are strongly $h$-monoidal model categories in which not all objects are cofibrant, e.g.

\begin{itemize}\item compactly generated spaces with Quillen's model structure, cf. Corollary \ref{kspaces};

\item small $2$-categories (or $2$-groupoids) with the Gray tensor product, cf. Proposition \ref{detect} and \cite{Lack,BergerMoerdijk2}.\end{itemize}

Corollary \ref{chaincomplexes} treats two examples of $h$-monoidal model categories which are not strongly $h$-monoidal.
\end{examples}

\subsection{Resolution axiom and strong unit axiom}\label{strongunit}
It can be verified by direct inspection that in the preceding examples all objects are \emph{$h$-cofibrant}. In Proposition \ref{allhcof} below we give a sufficient criterion for this to hold. This is the only place in this article where we make explicit use of Hovey's \emph{unit axiom} \cite{Hovey} under a \emph{strengthened} form. We are indebted to Muro for clarifying comments, who uses in  \cite{FM2} another strengthening of Hovey's unit axiom, weaker than ours, but with a similar purpose.

The unit $e$ of a monoidal model category $\Ee$ plays an important role. The minimal requirement for the existence of a unit in the homotopy category of $\Ee$ has been formulated by Hovey \cite{Hovey} as the so-called \emph{unit axiom}: cofibrant resolutions of the unit should remain weak equivalences under tensor with \emph{cofibrant} objects.

It is often the case that the following (stronger) \emph{resolution axiom} holds: general cofibrant resolutions are stable under tensor with cofibrant objects. The resolution axiom implies (by $2$-out-of-$3$) that the class of weak equivalences is stable under tensor with cofibrant objects, and actually implies (again by $2$-out-of-$3$) that cofibrant resolutions of the unit remain weak equivalences under tensor with \emph{arbitrary} objects. We call the latter property the \emph{strong unit axiom}. We thus have for any monoidal model category the following chain of implications$$\text{strongly $h$-monoidal}\implies\text{resolution axiom}\implies\text{strong unit axiom}$$The reader should observe that the strong unit axiom also holds if the unit $e$ of $\Ee$ is already cofibrant, in which case the following proposition has a much easier proof.

\begin{pro}\label{allhcof}In any $h$-monoidal model category, in which the strong unit axiom holds, all objects are $h$-cofibrant.\end{pro}

\begin{proof}We shall use Lemma \ref{charhcof}(i) for recognising $h$-cofibrant objects. The strong unit axiom requires the existence of a \emph{cofibrant resolution} $Q(e)\to e$ for the \emph{unit} $e$, which remains a weak equivalence after tensoring with arbitrary objects. For each object $X$ of $\Ee$, we thus have a weak equivalence $Q(e)\otimes(e\sqcup X)\to e\sqcup X$. Since the tensor commutes with coproducts, this weak equivalence can be rewritten as$$Q(e)\sqcup(Q(e)\otimes X)  \to  Q(e)\sqcup X \to  e\sqcup X$$where the first map is the coproduct of $Q(e)$ with $Q(e)\otimes X\to X$. Therefore, since $Q(e)$ is $h$-cofibrant by Lemma \ref{leftproper}, the first map above is a weak equivalence. By the $2$-of-$3$ property of weak equivalences, the second map $Q(e)\sqcup X\to  e\sqcup X$ is a weak equivalence as well. But then, for each weak equivalence $X\to Y$, the commutative diagram\begin{diagram}[small]Q(e)\sqcup X&\rTo & Q(e)\sqcup Y\\\dTo & &\dTo\\  e\sqcup X &\rTo& e\sqcup Y\end{diagram}
implies that $e\sqcup X\to e\sqcup Y$ is a weak equivalence, which shows that $e$ is $h$-cofibrant.

In an $h$-monoidal model category we have the stronger property that $Q(e)\otimes Z$ is $h$-cofibrant for each object $Z$. Therefore, factoring the weak equivalence $$Q(e)\otimes (Z\sqcup X)=(Q(e)\otimes Z)\sqcup(Q(e)\otimes X)\to Z\sqcup X$$ through $(Q(e)\otimes Z)\sqcup X$ yields a weak equivalence $(Q(e)\otimes Z)\sqcup X\to Z\sqcup X$ for all objects $Z$ and $X$. This implies as above that all objects $Z$ are $h$-cofibrant.\end{proof}

\begin{remark}If all objects of a (monoidal) model category are $h$-cofibrant then each weak equivalence factors as a trivial $h$-cofibration followed by a retraction of a trivial $h$-cofibration, using the same argument as for Brown's Lemma. In this case, the resolution axiom amounts thus to the property that tensoring a trivial $h$-cofibration with a cofibrant object yields a weak equivalence. In particular, the two examples of Corollary \ref{chaincomplexes} fulfill the resolution axiom.\end{remark}

The following lemma is also useful to retain:

\begin{lem}\label{wecof}In an $h$-monoidal model category, tensoring a weak equivalence between cofibrant objects with an arbitrary object yields again a weak equivalence.\end{lem}

\begin{proof}By Brown's Lemma, a weak equivalence between cofibrant objects factors as a trivial cofibration followed by a retraction of a trivial cofibration. Both factors yield a weak equivalence when tensored with an arbitrary object.\end{proof}

\section{Admissible monads on compactly generated model categories}

It is well-known that the class of (trivial) cofibrations in an arbitrary model category is closed under cobase change, transfinite composition and retract. Classes of morphisms with these three closure properties will be called \emph{saturated}.

\begin{defin}With respect to a saturated class of morphisms $K$ in a model category $\,\Ee$, the class $W$ of weak equivalences of $\,\Ee$ is called \emph{$K$-perfect} if $\,W$ is closed under filtered colimits along morphisms in $K$.\end{defin}

\begin{remark}Our terminology is borrowed from Lurie \cite{Lurie} who calls a class of weak equivalences perfect if it is $K$-perfect with respect to the saturated class $K$ of \emph{all} morphisms. By Hovey's argument (see the proof of \cite[Corollary 7.4.2]{Hovey} or \cite[Lemma 3.5]{DOR} for a more recent treatment) a sufficient condition for the $K$-perfectness of the class of weak equivalences is the existence of a \emph{generating set of cofibrations} whose domain and codomain are \emph{finite} with respect to $K$.\end{remark}

\begin{lem}\label{transfinite}If the class $W$ of weak equivalences is $K$-perfect then the intersection $W\cap K$ is closed under transfinite composition.\end{lem}

\begin{proof}Any transfinite composition of maps can be identified with the colimit of a natural transformation from a constant diagram to the given sequence of maps. If the given maps belong to $W\cap K$ this colimit is a filtered colimit of weak equivalences along morphisms in $K$. By assumption such a colimit is a weak equivalence.\end{proof}

We shall say that a class of morphisms is \emph{monoidally saturated} if it is saturated and moreover closed under tensoring with \emph{arbitrary objects} of the monoidal model category. Accordingly, the \emph{monoidal saturation} of a class $K$ is the least monoidally saturated class containing $K$. For instance, in virtue of the pushout-product axiom, the class of (trivial) cofibrations of a monoidal model category is monoidally saturated whenever all objects of the model category are cofibrant.

We are mainly interested in the monoidal saturation of the class of cofibrations. This monoidal saturation will be denoted $I^\otimes$ since it suffices to monoidally saturate a generating set of cofibrations which traditionally is denoted $I$. For brevity we shall call \emph{$\otimes$-cofibration} any morphism in $I^\otimes$. An object will be called \emph{$\otimes$-small} (resp. \emph{$\otimes$-finite}) if it is small (resp. finite) with respect to $I^\otimes$. The class of weak equivalences will be called \emph{$\otimes$-perfect} if it is $I^\otimes$-perfect.

\begin{defin}[cf. \cite{BergerMoerdijk2}] A model category is called \emph{$K$-compactly generated} if it is cofibrantly generated, its class of weak equivalences is $K$-perfect, and each object is small with respect to $K$.

A monoidal model category is called \emph{compactly generated}, if the underlying model category is $I^\otimes$-compactly generated.\end{defin}

For instance, any monoidal model category whose underlying model category is \emph{combinatorial}, and whose class of weak equivalences is \emph{closed under filtered colimits}, is an example of a compactly generated monoidal model category. The majority of our examples are of this kind. However, compactly generated topological spaces form a monoidal model category which is neither combinatorial nor does it have a class of weak equivalences which is closed under filtered colimits. Yet, every compactly generated space is $\otimes$-small, and the class of weak equivalences is $\otimes$-perfect, hence the monoidal model category of compactly generated spaces is compactly generated in the aforementioned model-theoretical sense, cf. \cite{Hovey,BergerMoerdijk2}.\vspace{1ex}

\begin{pro}\label{couniv0}In any compactly generated $h$-monoidal model category, the monoid axiom of Schwede-Shipley holds and each $\otimes$-cofibration is an $h$-cofibration.

If in addition the strong unit axiom holds (cf. Section \ref{strongunit}) then the class of weak equivalences is closed under arbitrary coproducts.\end{pro}

\begin{proof}The monoid axiom of Schwede-Shipley \cite{SS} requires the monoidal saturation of the class of trivial cofibrations to stay within the class of weak equivalences. In a cofibrantly generated monoidal model category this monoidal saturation can be constructed by choosing a generating set $J\,$ for the trivial cofibrations, and saturating the class $\{f\otimes 1_Z\,|\,f\in J,\,Z\in\Ob\Ee\}$ under cobase change, transfinite composition and retract. Since, by Lemma \ref{couniv}, each $f\otimes 1_Z$ is a couniversal weak equivalence and a $\otimes$-cofibration, and both classes are closed under cobase change and retract, it remains to be shown that the class of maps, which are simultaneously weak equivalences and $\otimes$-cofibrations, is closed under transfinite composition. This is precisely Lemma \ref{transfinite} for $K=I^\otimes$.

For the second statement, we have to show that the monoidal saturation of the class of cofibrations stays within the class of $h$-cofibrations. As before, this monoidal saturation can be constructed by choosing a generating set $I\,$ for the cofibrations, and saturating the class $\{f\otimes 1_Z\,|\,f\in I,\,Z\in\Ob\Ee\}$ under cobase change, transfinite composition and retract. Since each $f\otimes 1_Z$ is an $h$-cofibration and a $\otimes$-cofibration, and both classes are closed under cobase change and retract, it remains to be shown that the class of maps, which are simultaneously $h$-cofibrations and $\otimes$-cofibrations, is closed under transfinite composition. This follows from the definition of an $h$-cofibration, since a vertical transfinite composition of diagrams of the form (\ref{hcof}) (all vertical maps being $h$-cofibrations and $\otimes$-cofibrations) yields a diagram of the same form (\ref{hcof}).

For the last statement, note first that by Proposition \ref{allhcof} all objects are $h$-cofibrant so that by Lemma \ref{charhcof}(iii) it remains to be shown that weak equivalences are closed under filtered colimits along coproduct injections. By the strong unit axiom, a cofibrant resolution $Q(e)\eqv e$ of the monoidal unit $e$ yields a resolution $Q(e)\otimes Z\eqv Z$ of each object $Z$ by a $\otimes$-cofibrant object $Q(e)\otimes Z$. This resolution functor commutes with colimits so that each colimit can be replaced with a weakly equivalent colimit of $\otimes$-cofibrant objects. Any coproduct injection with respect to a $\otimes$-cofibrant object is a $\otimes$-cofibration. By compact generation weak equivalences are closed under filtered colimits along $\otimes$-cofibrations as required.\end{proof}

\begin{corol}\label{couniv1}In a monoidal model category with $\otimes$-perfect class of weak equivalences, the monoid axiom of Schwede-Shipley holds if and only if the tensor product of a trivial cofibration with an arbitrary object is a couniversal weak equivalence.\end{corol}

\begin{proof}This follows from the argument of first paragraph of the preceding proof.\end{proof}

\begin{remark}The preceding proposition and corollary (together with \ref{twin1}, \ref{twin2} or \ref{detect}) may be an efficient tool to establish the monoid axiom and left properness. For instance, Lack's original proofs \cite[Theorems 6.3 and 7.7]{Lack} of these properties for the category of small $2$-categories are quite a bit more involved.\end{remark}

\subsection{Admissible monads}Recall that a \emph{monad} $T$ on $\Ee$ is called \emph{finitary} if $T$ preserves filtered colimits, or what amounts to the same, if the forgetful functor $U_T:\Alg_T\to\Ee$ preserves filtered colimits. Here, $\Alg_T$ denotes the category of $T$-algebras and $$F_T:\Ee\lra\Alg_T:U_T$$the free-forgetful adjunction. Thus $T=U_TF_T$ and $F_T(X)=(TX,\mu_X)$ where $\mu:T^2\to T$ is the multiplication of the monad $T$.
\begin{defin}Let $\Ee$ be a model category, $W$ its class of weak equivalences, and $K$ be an arbitrary saturated class in $\Ee$. A monad $T$ on $\Ee$ is said to be \emph{$K$-admissible} if for each cofibration (resp. trivial cofibration) $u:X\to Y$ and each map of $T$-algebras $\alpha:F_T(X)\to R$, the pushout in $\Alg_T$
\begin{gather}
\label{extension}
\begin{diagram}[small]F_T(X)&\rTo^\alpha&R\\\dTo^{F_T(u)}&&\dTo_{u_\alpha}\\F_T(Y)&\rTo&\NWpbk R[u,\alpha]
\end{diagram}
\end{gather}
yields a $T$-algebra map $u_{\alpha}: R\to R[u,\alpha]$ whose underlying map $U_T(u_{\alpha})$ belongs to $K$ (resp. to $W\cap K$).\end{defin}

For all the sequel a \emph{weak equivalence $f:R\to S$ of $T$-algebras} means a map of $T$-algebras such that the underlying map $U_T(f):U_T(R)\to U_T(S)$ is a weak equivalence in $\Ee$. A $T$-algebra $R$ will be called \emph{$U_T$-cofibrant} if the underlying object $U_T(R)$ is cofibrant in $\Ee$.

Note that the notion of \emph{$h$-cofibration} (cf. Definition \ref{h-cofibration}) makes sense for any category with pushouts and a specified class of weak equivalences. Accordingly, a map of free $T$-algebras $F_T(u):F_T(X)\to F_T(Y)$ will be called an \emph{$h$-cofibration} if for any diagram of pushouts of $T$-algebras
 \begin{gather}\label{extension2}\begin{diagram}[small]F_T(X)&\rTo^\alpha&R&\rTo^f&S\\\dTo^{F_T(u)}&&\dTo&&\dTo\\F_T(Y)&\rTo&\NWpbk R[u,\alpha]&\rTo&\NWpbk S[u,f\alpha]\end{diagram}\end{gather}
 in which $f:R\to S$ is a weak equivalence, the induced map $R[u,\alpha]\to S[u,f\alpha]$ is again a weak equivalence. We shall say that $F_T(u)$ is a \emph{relative $h$-cofibration} if\begin{itemize}\item[(*)]for any weak equivalence $f:R\to S$ between \emph{$U_T$-cofibrant} $T$-algebras and any $\alpha:F_T(X)\to R$, the induced map $R[u,\alpha]\to S[u,f\alpha]$ is again a weak equivalence between \emph{$U_T$-cofibrant} $T$-algebras.\end{itemize}If all objects of $\Ee$ are cofibrant, any $h$-cofibration $F_T(u)$ is a relative $h$-cofibration. In general there might be $h$-cofibrations which are not relative $h$-cofibrations.

\begin{defin}A model structure on $T$-algebras will be called  \emph{relatively left proper} if for any weak equivalence $f:R\to S$ between \emph{$U_T$-cofibrant} $T$-algebras and any cofibration $R\to R'$ of $\,T$-algebras, cobase change along the latter yields a weak equivalence $R'\to R'\cup_RS$ between \emph{$U_T$-cofibrant} $T$-algebras.\end{defin}

Again, if all objects of $\Ee$ are cofibrant, then a left proper model structure on $\Alg_T$ is automatically relatively left proper. In general however this might be wrong.

\begin{theorem}\label{transfer}For any finitary $K$-admissible monad $\,T$ on a $K$-compactly generated model category $\Ee$, the category of $T$-algebras admits a transferred model structure. This model structure is (relatively) left proper if and only if the free $T$-algebra functor takes cofibrations in $\Ee$ to (relative) $h$-cofibrations in $\Alg_T$.\end{theorem}

\begin{proof}By definition of a transfer, a map of $T$-algebras $f$ is defined to be a weak equivalence (resp. fibration) precisely when $U_T(f)$ is a weak equivalence (resp. fibration) in $\Ee$. Cofibrations of $T$-algebras are defined by the left lifting property with respect to trivial fibrations. In order to show that these three classes define a model structure on $\Alg_T$, the main difficulty consists in proving the existence of cofibration/trivial fibration (resp. trivial cofibration/fibration) factorisations. For this we apply Quillen's small object argument to the image $F_T(I)$ (resp. $F_T(J)$) of a generating set $I$ (resp. $J$) for the cofibrations (resp. trivial cofibrations) of $\Ee$. The following two points have to be shown:\begin{enumerate}\item[(i)]The domains of the maps in $F_T(I)$ (resp. $F_T(J)$) are small with respect to the saturation of $F_T(I)$ (resp. $F_T(J)$) under cobase change and transfinite composition in $\Alg_T$;\item[(ii)]The saturation of $F_T(J)$ under cobase change and transfinite composition in $\Alg_T$ stays within the class of weak equivalences.\end{enumerate}
Since the forgetful functor $U_T$ preserves filtered colimits, an adjunction argument and the $K$-smallness of the objects of $\Ee$ yield (i). Moreover, Lemma \ref{transfinite} and the $K$-perfectness of the weak equivalences in $\Ee$ yield (ii).

If the transferred model structure on $\Alg_T$ is (relatively) left proper then the left Quillen functor $F_T$ takes cofibrations in $\Ee$ to (relative) $h$-cofibrations in $\Alg_T$ by Lemma \ref{leftproper}. Conversely, assume that $F_T(u)$ is a (relative) $h$-cofibration for each generating cofibration $u$. Note first that the forgetful functor $U_T$ preserves transfinite compositions since it preserves filtered colimits. It follows then from the $K$-perfectness of the class of weak equivalences and the $K$-admissibility of $T$ that cobase change along a transfinite composition of free $T$-algebra extensions of the form $R\to R[u,\alpha]$ preserves weak equivalences (between $U_T$-cofibrant $T$-algebras). But any cofibration in $\Alg_T$ is retract of such a transfinite composition. Thus, $\Alg_T$ is (relatively) left proper.

The reader should observe that in the relatively left proper case, we need the full condition (*) of a relative $h$-cofibration in order to ensure that, at each step of the transfinite composition, the underlying objects are cofibrant.\end{proof}

\begin{pro}\label{rlp}The free $T$-algebra functor takes cofibrations to relative $h$-cofibrations if it takes cofibrations with cofibrant domain to relative $h$-cofibrations.\end{pro}

\begin{proof}Suppose that $u:X\to Y$ is a cofibration. We have to show that for a weak equivalence $f:R\to S$ between $U_T$-cofibrant $T$-algebras, the morphism
$R[u,\alpha] \to S[u,f\alpha]$ in the diagram (\ref{extension2}) is a weak equivalence. Let $\alpha': X \to U_T(R)$ be the composite   $$X\stackrel{\epsilon}{\to} U_TF_T(X)\stackrel{U_T(\alpha)}{\to} U_T(R)$$and consider the following pushout in $\Ee:$
\begin{diagram}[small]X&\rTo^{\alpha'}&U_T(R)\\\dTo^u&&\dTo_v\\Y&\rTo&\NWpbk P\end{diagram}
The given map $\alpha$ factors as
$$F_T(X)\stackrel{F_T(\alpha')}{\to} F_TU_T(R)\stackrel{k}{\to} R$$
where $k$ is the structure map of the $T$-algebra $R.$ Therefore, by the universal property of pushouts, the right-hand square of the following commutative diagram
\begin{diagram}[small]F_T(X)&\rTo^{F_T(\alpha')}&F_TU_T(R)&\rTo^k&R\\\dTo^{F_T(u)}&&\dTo&&\dTo\\F_T(Y)&\rTo&\NWpbk F_T(P)&\rTo&\NWpbk R[u,\alpha]\end{diagram}
is a pushout. Hence, we get the following pushout diagram in $\Alg_T:$
\begin{diagram}[small]F_TU_T(R)&\rTo^k&R&\rTo^f&S\\\dTo^{F_T(v)}&&\dTo&&\dTo\\F_T(P)&\rTo&\NWpbk R[u,\alpha]&\rTo&\NWpbk S[u,f\alpha]\end{diagram}

Since $v$ is a cofibration with cofibrant domain, $F_T(v)$ is a relative $h$-cofibration by assumption, so that $R[u,\alpha]\to S[u,f\alpha]$ is a weak equivalence as required.\end{proof}

\begin{defin}A monad $T$ is \emph{$K$-adequate} if the underlying map of any free $T$-algebra extension $u_\alpha:R\to R[u,\alpha]$ admits a functorial factorisation
$$U_T(R)= R[u]^{(0)}\to R[u]^{(1)}\to \ldots \to R[u]^{(n)}\to \ldots \to \colim_n R[u]^{(n)}= U_T(R[u,\alpha]);$$
such that for a (trivial) cofibration $u$, each map of the sequence belongs to $K$ (resp. $W\cap K$), and for a weak equivalence $f:R\to S$, the induced morphisms $R[u]^{(n)}\to S[u]^{(n)}$ are weak equivalences for all $n\ge 0.$

A monad $T$ is \emph{relatively $K$-adequate} if the second condition only holds for cofibrations $u$ with cofibrant domain and for weak equivalences $f:R\to S$ between $U_T$-cofibrant $T$-algebras, and the individual maps of the horizontal sequences are cofibrations.\end{defin}

\begin{theorem}\label{adequate}Any finitary (relatively) $K$-adequate monad $T$ on a $K$-compactly generated  model category $\Ee$ is $K$-admissible, and the associated free $T$-algebra functor takes cofibrations to (relative) $h$-cofibrations. Hence, the category of $T$-algebras has a transferred model structure which is (relatively) left proper.\end{theorem}

\begin{proof}The second statement follows from the first and from Theorem \ref{transfer}.

That a $K$-adequate monad is $K$-admissible follows from the saturation of $K$ and from Lemma \ref{transfinite}. Moreover, given a cofibration $u:X\to Y$ and a weak equivalence $f:R\to S$ of $T$-algebras, the underlying map of morphism $R[u,\alpha]\to S[u, f\alpha]$ of diagram (\ref{extension2}) is a sequential colimit of a ladder in $\Ee$\begin{diagram}[small]R[u]^{(0)}&\rTo&R[u]^{(1)}&\rTo&\cdots&\rTo&R[u]^{(n)}&\rTo&\cdots&\rTo&\colim_n R[u]^{(n)}&=U_T(R[u,\alpha])\\
\dTo&&\dTo&&&&\dTo&&&&\dTo\\S[u]^{(0)}&\rTo&S[u]^{(1)}&\rTo&\cdots&\rTo&S[u]^{(n)}&\rTo&\cdots&\rTo&\colim_n S[u]^{(n)}&=U_T(S[u, f\alpha])\end{diagram}
in which the vertical maps are weak equivalences and the horizontal maps belong to $K.$ Since $\Ee$ is $K$-compactly generated this colimit is a weak equivalence. This shows that the free $T$-algebra functor takes cofibrations to $h$-cofibrations.

For the relative version, $K$-admissibility is proved as before. For the relative $h$-cofibration property of $F_T(u)$ we can assume, according to Proposition \ref{rlp}, that $u$ has a cofibrant domain. Then, by assumption, the horizontal maps of the diagram above are cofibrations so that their composite is a cofibration as well, whence $U_T(R[u,\alpha])$ is cofibrant, given that $R$ is $U_T$-cofibrant by assumption. Moreover, the colimit of the ladder is a weak equivalence by the same argument as before (or by invoking a Reedy telescope lemma). Therefore, $F_T(u)$ is a relative $h$-cofibration.\end{proof}

\section{Monoids in $h$-monoidal model categories}\label{SSTHeorem}

This section presents the main result of Schwede-Shipley \cite{SS} concerning the existence of a model structure on monoids if the monoid axiom holds. We add a discussion of left properness of the transferred model structure, cf. Muro \cite{FM2}.

Recall that $I^\otimes$ denotes the monoidal saturation of the class of cofibrations, and that any morphism in $I^\otimes$ is called a $\otimes$-cofibration. Accordingly, we say $\otimes$-admissible (resp. $\otimes$-adequate) instead of $I^\otimes$-admissible (resp. $I^\otimes$-adequate).

\begin{theorem}
\label{SSformonoids}For any compactly generated monoidal model category $\Ee$ the free monoid monad $~T$ on $\,\Ee$ is :
\begin{itemize}\item[(a)]relatively $\otimes$-adequate if the monoid axiom holds;
\item[(b)]$\otimes$-adequate if $\,\Ee$ is strongly $h$-monoidal.
\end{itemize}
And hence
\begin{itemize}
\item[(a$'$)]there is a relatively left proper transferred model structure on monoids if the monoid axiom holds;
\item[(b$'$)]the model structure on monoids is left proper if $\,\Ee$ is strongly $h$-monoidal.
\end{itemize}
\end{theorem}

\begin{proof}(a$'$), (b$'$) follow from (a), (b) and Theorem \ref{adequate}.

Let $R$ be a monoid in $\Ee$, and let $u:Y_0\to Y_1$ be a map in $\,\Ee$ equipped with a map of monoids $F_T(Y_0)\to R$.
We shall exhibit the pushout in the category of monoids as a sequential colimit in $\Ee.$

 Let $R[u]^{(0)}=R$ and define inductively $R[u]^{(n)}$ by the following pushout
\begin{gather}\label{SSformula}\begin{diagram}[small]Y^{(n)}_-&\rTo&R[u]^{(n-1)}\\\dTo&&\dTo\\Y^{(n)}&\rTo&\NWpbk R[u]^{(n)}\end{diagram}\end{gather}
where
$$Y^{(n)}=R\otimes \overbrace{Y_1\otimes R\otimes\cdots\otimes Y_1\otimes R}^n$$
and $Y_-^{(n)}$ is the colimit of a diagram over a punctured $n$-cube $\{0,1\}^n-\{(1,\dots,1)\}$ in which the vertex $(i_1,\dots,i_n)$ takes the value
 \begin{equation*}
 R\otimes Y_{i_1}\otimes \cdots\otimes R\otimes Y_{i_n}\otimes R
 \end{equation*}
and the edge-maps are induced by $u.$ The map $Y^{(n)}_-\to Y^{(n)}$ is the comparison map from the colimit of this diagram to the value at $(1,\dots,1)$ of the extended diagram on the whole $n$-cube. The map $Y^{(n)}_-\to R[u]^{(n-1)}$ is defined inductively, using the fact that the construction of $R[u]^{(n-1)}$ involves $n-1$ tensor factors only.

Since the tensor $-\otimes -$ commutes with pushouts in both variables, there are canonical maps of  $R[u]^{(p)}\otimes R[u]^{(q)}\to R[u]^{(p+q)}$. Since the tensor $-\otimes-$ commutes with sequential colimits in both variables, these maps induce the structure of a monoid on the colimit $\colim_nR[u]^{(n)}$. A proof that this monoid has indeed the universal property of $R[u]$ has been sketched in \cite{SS}. We obtain in Theorem \ref{SS} below a far-reaching generalisation of this result.

We shall now prove that, for each $n>0$, the map $R[u]^{(n-1)}\to R[u]^{(n)}$ is a (trivial) $\otimes$-cofibration whenever $u$ is a (trivial) cofibration. The considered map derives from $Y^{(n)}_-\to Y^{(n)}$ through a cobase change. Collecting all tensor factors $R$, the map $Y^{(n)}_-\to Y^{(n)}$ may be identified with an iterated pushout-product map along $u$, tensored with $R^{\otimes n+1}$. Therefore, $Y^{(n)}_-\to Y^{(n)}$ is a $\otimes$-cofibration and its cobase change $R[u]^{(n-1)}\to R[u]^{(n)}$ as well. If $u$ is a trivial cofibration, the iterated pushout-product map is a trivial cofibration, and its tensor product with $R^{\otimes n+1}$ is a couniversal weak equivalence by the monoid axiom; thus $R[u]^{(n-1)}\to R[u]^{(n)}$ is a trivial $\otimes$-cofibration. This proves that the free monoid monad is $\otimes$-admissible.

For the proof of (a) we consider for each $n>0$, the following commutative cube in $\Ee$
\begin{gather}\label{cube}\begin{diagram}[size=1.75em,silent,UO]&&Z^{(n)}_-&\rTo&&&S[u]^{(n-1)}\\&\ruTo&\vLine&&&\ruTo&\dTo\\Y^{(n)}_-&&&\rTo&R[u]^{(n-1)}&&\\
\dTo&&\dTo&&\dTo&&\\&&Z^{(n)}&\hLine&\VonH&\rTo&S[u]^{(n)}\\&\ruTo&&&&\ruTo&\\Y^{(n)}&\rTo&&&R[u]^{(n)}&&\end{diagram}\end{gather}
in which $Z^{(n)}_-\to Z^{(n)}$ is defined like $Y^{(n)}_-\to Y^{(n)}$ just replacing $R$ with $S$, and where we assume that $U_T(R)$ and $U_T(S)$ are cofibrant, and that $u$ has a cofibrant domain.

Front and back square of the cube are pushouts; in particular the right vertical maps are cofibrations since the left vertical maps are so by the pushout-product axiom. The natural transformation from front to back square is induced by tensor powers of $f:R\to S$. By induction, it suffices now to show that $R[u]^{(n)}\to S[u]^{(n)}$ is a weak equivalence whenever $R[u]^{(n-1)}\to S[u]^{(n-1)}$ is. Indeed, all objects of the cube (\ref{cube}) are cofibrant, the two left vertical maps are cofibrations, and the two left horizontal maps are weak equivalences. Proposition \ref{gluing}a implies then that $R[u]^{(n)}\to S[u]^{(n)}$ is a weak equivalence as required.

For the proof of (b) we consider the same cube assuming that $\Ee$ is strongly $h$-monoidal and that $f:R\to S$ is a weak equivalence of monoids. We have seen that $Y^{(n)}_-\to Y^{(n)}$ and $Z^{(n)}_-\to Z^{(n)}$ are $\otimes$-cofibrations and hence $h$-cofibrations by Proposition \ref{couniv0}.

Since $\Ee$ is strongly $h$-monoidal, the tensor power $f^{\otimes n+1}:R^{\otimes n+1}\to S^{\otimes n+1}$ is again a weak equivalence. Hence, for any vertex $(i_1,\dots,i_n)$ of the $n$-cube, the map \begin{gather}\label{vertex}R\otimes Y_{i_1}\otimes \cdots\otimes R\otimes Y_{i_n}\otimes R\to S\otimes Y_{i_1}\otimes \cdots\otimes S\otimes Y_{i_n}\otimes S\end{gather}
is a weak equivalence. Recollecting the tensor factors as above shows that the maps $Y^{(n)}\to Z^{(n)}$ and $Y^{(n)}_-\to Z^{(n)}_-$ are weak equivalences. Proposition \ref{gluing}b implies then that if $R[u]^{(n-1)}\to S[u]^{(n-1)}$ is a weak equivalence then so is $R[u]^{(n)}\to S[u]^{(n)}$.\end{proof}


\section{Diagram categories and Day convolution}

As a first application of our methods we observe that the class of compactly generated (strongly) $h$-monoidal categories $\Ee$ is closed under taking diagram categories over a small $\,\Ee$-enriched category $\CC$. More precisely, let $\Ee$ be a monoidal model category and $\CC$ be a small $\Ee$-enriched category. Let $[\CC,\Ee]$ be the category of $\Ee$-enriched functors and $\Ee$-natural transformations. Let $\CC_0$ be the set of objects of $\CC$, considered as a discrete $\Ee$-category. We have an inclusion $\Ee$-functor $i:\CC_0\to\CC.$ The category $[\CC_0,\Ee]\cong\Ee^{\CC_0}$ has an obvious product model structure.

There is a monad $i^*i_!$ on $[\CC_0,\Ee]$ where $i^*$ denotes the restriction functor and $i_!$ its left adjoint. The restriction functor $i^*$ is monadic, and the \emph{projective model structure} on $[\CC,\Ee]$ is by definition the model structure which is transferred from $[\CC_0,\Ee]$ along the adjunction $i_!:[\CC_0,\Ee]\leftrightarrows[\CC,\Ee]:i^*$ if such a transfer exists.

We shall call an object of $\Ee$ \emph{discrete} if it is a coproduct of copies of the unit of $\Ee$. Clearly, any tensor product of discrete objects is again discrete.

\begin{theorem}\label{projectivemodel}Let $\,\Ee$ be a compactly generated monoidal model category, and let $\CC$ be a small $\,\Ee$-enriched category. Then the projective model structure on $[\CC,\Ee]$ exists in each of the following three cases:

\begin{itemize}\item[(i)]all hom-objects of $\,\CC$ are discrete in $\Ee$;
\item[(ii)]all hom-objects of $\,\CC$ are cofibrant in $\Ee$;\item[(iii)]the monoid axiom holds in $\,\Ee$.\end{itemize}

The projective model structure on $[\CC,\Ee]$ is left proper if either $\,\Ee$ is $h$-monoidal (and hence (iii) holds), or if $\,\Ee$ is just left proper, but (i) or (ii) holds. If moreover $\,\CC$ is equipped with a symmetric monoidal structure, then $[\CC,\Ee]$ is a compactly generated monoidal model category with respect to Day's convolution product, and
\begin{itemize}
\item[(a)]the monoid axiom holds in $[\CC,\Ee]$ whenever it holds in $\,\Ee$;
\item[(b)]$[\CC,\Ee]$ is (strongly) $h$-monoidal whenever $\,\Ee$ is (strongly) $h$-monoidal;
\item[(c)]all objects in $[\CC,\Ee]$ are $h$-cofibrant whenever all objects in $\,\Ee$ are $h$-cofibrant.
\end{itemize}
\end{theorem}

\begin{proof}The existence and left properness of the projective model structure on $[\CC,\Ee]$ follows from Theorem \ref{transfer} if we prove that the monad $i^*i_!$ is $K$-admissible, where $K$ is the class of pointwise $\otimes$-cofibrations.  Note that the class of weak equivalences in $[\CC_0,\Ee]$ is $K$-perfect. Moreover, all objects of $[\CC_0,\Ee]$ are $K$-small, so that $[\CC_0,\Ee]$ is a $K$-compactly generated model category. Note also that the monad $i^*i_!$ is finitary because the right adjoint $i^*$ is left adjoint as well and thus preserves all colimits.

Now, for any object $X$  of $[\CC_0,\Ee]$, we have $(i_!X)(a)=\sqcup_{b\in \CC_0}\CC(b,a)\otimes X(b).$ Let $u:X\to Y$ be a morphism in $[\CC_0,\Ee]$ and let $\alpha:i_!X\to R$ be a morphism in $[\CC,\Ee]$. This defines for each $a\in\CC_0$ the following pushout
\begin{diagram}[small]\coprod_{b\in\CC}\CC(b,a)\otimes X(b)&\rTo^\alpha&R(a)\\\dTo & &\dTo^u\\  \coprod_{b\in\CC}\CC(b,a)\otimes Y(b) &\rTo&\NWpbk R[u,\alpha](a)
\end{diagram}
in $\Ee$. For the $K$-admissibilty of $i^*i_!$ we have to show that the right vertical map is a $\otimes$-cofibration (resp. weak equivalence) if $u:X\to Y$ is a cofibration (resp. trivial cofibration). This is obviously the case under assumptions (i) and (ii). For case (iii), note first that a pushout $u:R(a)\to R[u,\alpha](a)$ like above can be realised as a transfinite composition of pushouts of single maps $\CC(b,a)\otimes X(b)\to\CC(b,a)\otimes Y(b)$. For any cofibration $u$, such a pushout is a $\otimes$-cofibration, and hence a transfinite composition of them is again a $\otimes$-cofibration. For a trivial cofibration $u$, the analogous transfinite composition belongs to the monoidal saturation of the class of trivial cofibrations, and is therefore a weak equivalence under assumption (iii).

If $\Ee$ is left proper and $\CC$ satisfies (i) or (ii) then the left vertical map above is a cofibration, and left properness of $\Ee$ implies left properness of $[\CC,\Ee]$. Under assumption (iii) and assuming that $\Ee$ is $h$-monoidal, Proposition \ref{couniv0} shows that the left vertical map above is an $h$-cofibration which implies left properness of $[\CC,\Ee]$.

From now on we assume that $\CC$ is a symmetric monoidal category with tensor $$\odot:\CC\otimes\CC\to\CC$$ and we endow $[\CC,\Ee]$ with the Day convolution product. There is an external tensor product $\bar{\otimes}:[\CC,\Ee]\otimes[\CC,\Ee]\to [\CC\otimes \CC,\Ee]$ which is a Quillen functor of two variables with respect to the projective model structures on both sides, cf. Barwick \cite{Barwick}. Left Kan extension along the tensor $\odot:\CC\otimes\CC\to\CC$ also yields a left Quillen functor $\odot_!:[\CC\otimes\CC,\Ee]\to[\CC,\Ee]$. Therefore, the composite functor\begin{diagram}-\boxx-: [\CC,\Ee]\otimes [\CC,\Ee]&\rTo^{\bar{\otimes}}&[\CC\otimes \CC,\Ee]&\rTo^{\odot_!} &[\CC,\Ee]\end{diagram}which may be identified with the Day convolution product, is a left Quillen functor of two variables, hence $[\CC,\Ee]$ satisfies the pushout-product axiom. The unit axiom for $[\CC,\Ee]$ follows from the unit axiom for $\Ee$.

For the compact generation of $[\CC,\Ee]$, note first that the projective model structure on $[\CC,\Ee]$ is $K$-compactly generated for the saturated class $K$ of pointwise $\otimes$-cofibrations, since the weak equivalences of $[\CC,\Ee]$ are pointwise weak equivalences, and colimits in $[\CC,\Ee]$ are computed pointwise. Therefore, it suffices to show that each generating cofibration of $[\CC,\Ee]$ belongs to $K$, and that $K$ is stable under Day convolution $-\boxx Z$ with an arbitrary object $Z$. The aforementioned formula for the left adjoint $i_!$ shows that $i_!$ takes cofibrations in $[\CC_0,\Ee]$ to pointwise $\otimes$-cofibrations in $[\CC,\Ee]$. Observe furthermore that $X\boxx Z$ is a pointwise retract of $(i_!i^*X)\boxx(i_!i^*Z),$ and hence $f\boxx Z$ is a pointwise retract of $(i_!i^*f)\boxx (i_!i^*Z)$ for any map $f:X\to Y$ in $[\CC,\Ee]$. The latter morphism evaluated at $c\in \CC_0$ is given by
\begin{equation}\label{ccprd}\coprod_{a,b}\CC(a\odot b,c)\otimes X(a)\otimes Z(b) \to \coprod_{a,b}\CC(a\odot b,c)\otimes Y(a)\otimes Z(b)\end{equation}which is a $\otimes$-cofibration whenever $f:X\to Y$ is a pointwise $\otimes$-cofibration. Hence, the pointwise retract $f\boxx Z$ is also a pointwise $\otimes$-cofibration, as required.

For statement (a), it will now be enough to apply Corollary \ref{couniv1} and to show that for a trivial cofibration $f:X \to Y$, we get a couniversal weak equivalence $f\boxx Z: X\boxx Z\to Y\boxx Z$ in $[\CC,\Ee]$. For this, observe that like before $f\boxx Z$ is a pointwise retract of $(i_!i^*f)\boxx (i_!i^*Z).$ The latter evaluated at $c\in \CC_0$ is given by coproduct (\ref{ccprd}) above. Since the monoid axiom holds in $\Ee$, each component of this coproduct is as well a couniversal weak equivalence as well a $\otimes$-cofibration. Writing this coproduct as a transfinite composition of pushouts of its components shows (in virtue of Lemma \ref{transfinite}) that the coproduct itself is a couniversal weak equivalence. Since couniversal weak equivalences in $[\CC,\Ee]$ are pointwise couniversal weak equivalences and since they are closed under retract, $f\boxx Z$ is indeed a couniversal weak equivalence.

For statement (b), observe first that since colimits in $[\CC,\Ee]$ are computed pointwise, and since the weak equivalences in $[\CC,\Ee]$ are the pointwise weak equivalences, the $h$-cofibrations in $[\CC,\Ee]$ are precisely the pointwise $h$-cofibrations. Therefore, a similar argument as above (based on Proposition \ref{couniv0}) yields (b). Statement (c) follows easily from Lemma \ref{charhcof}ii.\end{proof}

\begin{remark}This theorem recovers and strengthens Theorem 4.4 and Corollary 4.8 of Dundas-{\O}stv{\ae}r-R\"ondigs \cite{DOR}. We do not talk about right properness here but right properness is preserved under any transfer.

If $\Ee$ possesses a sufficiently nice system of \emph{spheres} (with \emph{symmetries}) then the formalism of \cite{DOR} enables one to define \emph{(symmetric) spectra} in $\Ee$, as $\CC$-enriched functors on a certain $\Ee$-enriched category $\CC$ which satisfies assumption (ii) above. Therefore, there exists a levelwise projective model structure on (symmetric) spectra in any compactly generated monoidal model category $\Ee$ with nice system of spheres (with symmetries). This projective model structure is thus $h$-monoidal whenever $\Ee$ is. In this special case, $h$-monoidality could also be derived from Proposition \ref{twin1}, since there is a suitable \emph{injective} model structure on (symmetric) spectra witnessing the fact that $\CC$ is a (generalised) $\Ee$-enriched Reedy category.
\end{remark}

\begin{remark}Any one-object $\Ee$-enriched symmetric monoidal category $\CC$ can be viewed as a commutative monoid in $\Ee$ and vice-versa. In this case, the diagram category $[\CC,\Ee]$ (equipped with the Day convolution product) may be identified with the category of $\CC$-modules (equipped with the usual tensor product of $\CC$-modules). Theorem \ref{projectivemodel} for this special case recovers one of the results of Schwede-Shipley \cite{SS}.\end{remark}

\part{Algebras over tame polynomial monads}

In this second part we study algebras over polynomial monads and show that the techniques of Part 1 are applicable to them. Polynomial monads are intermediate between non-symmetric and symmetric coloured operads. They have remarkable properties which among others allow a thorough combinatorial analysis of free algebra extensions. Beside the prototypical example of the free monoid monad, most of the currently used notions of operads are expressable as algebras over polynomial monads. Part 3 treats these examples in more detail. The reader may wish to go forth and back between Parts 2 and 3 so as to have concrete examples at hand.

The main new result of this second part is a combinatorial condition under which a polynomial monad is relatively $\otimes$-adequate (resp. $\otimes$-adequate) whenever the ambient compactly generated monoidal model category is $h$-monoidal (resp. strongly $h$-monoidal). In particular we get a (relatively) left proper model structure on the algebras over this monad. Polynomial monads which satisfy this condition will be called {\it tame}. At the end of this second part we study the Quillen adjunction induced by a cartesian morphism of tame polynomial monads, and describe the total left derived functor of such an adjunction as a homotopy colimit.

There are other techniques for establishing the existence of a transferred model structure on algebras. One of the most popular and powerful methods, applicable to algebras over symmetric operads, goes back to a joint paper of the second author and Ieke Moerdijk \cite{BergerMoerdijk}. This method was generalised further in \cite{BergerMoerdijk1,Fresse,JY}. Since polynomial monads can be considered as a particular kind of coloured symmetric operad, it follows that under the Berger-Moerdijk conditions (the existence in $\Ee$ of a cocommutative interval and of a symmetric monoidal fibrant replacement functor), the category of algebras over any polynomial monad admits a transferred model structure. This is in particular the case for the category of chain complexes over a field of characteristic $0$, or the categories of simplicial sets, resp. compactly generated topological spaces.

These conditions are however not satisfied in all cases of interest, nor do they provide a clue for approaching the problem of left properness of transferred model structures. It makes therefore sense to consider the smaller class of tame polynomial monads for which a transfer exists under less restrictive conditions on $\Ee$, and for which the transferred model structures are at least relatively left proper.

\section{Cartesian monads and their internal algebra classifiers}\label{cartesianadjunction}

In this section we recall the theory of \emph{internal algebra classifiers} of the first author \cite{EHBat} including its recent developments (cf. \cite{W,W2}). This tool is fundamental for us because it enables us (cf. Section 7) to replace free algebra extensions with \emph{left Kan extensions} which are easier to handle. We formulate the theory for general cartesian monads, though later on we shall only apply it to polynomial monads.

The reader less familiar with $2$-category theory can skip the $2$-categorical Definitions \ref{intalg1} and \ref{intalg2}, and take Theorems \ref{bar1} and \ref{TST} as definitions. Indeed, as far as cartesian monads are concerned, the theory of internal algebra classifiers can be considered as a $2$-categorical characterisation of the well-known \emph{simplicial bar resolution} (cf. \cite{May}), the internal algebra classifier being a \emph{categorified} version of it.

It should however be observed that the whole theory applies to arbitrary $2$-monads and yields here internal algebra classifiers which in general are not anymore categorified bar resolutions (cf. \cite{W,W2}).

\subsection{Internal categories and their simplicial nerves}--\vspace{1ex}

Let $\CC$ be a category with pullbacks. An {\it internal category} $C$ in $\CC$ is a reflexive graph $C_1\pile{\rTo\\\lTo\\\rTo}C_0$ equipped with a unitary and associative multiplication $m:C_1\times_{C_0}C_1\to C_1$. This data assembles into a $2$-truncated simplicial object

{\unitlength=1mm

\begin{picture}(0,15)(-35,18.3)

\put(4,25){\makebox(0,0){\mbox{$C_0$}}}
\put(11,28.7){\makebox(0,0){\mbox{$\scriptstyle s$}}}
\put(11,23.8){\makebox(0,0){\mbox{$\scriptstyle t$}}}
\put(10.5,26.25){\makebox(0,0){\mbox{$\scriptstyle i$}}}
\put(14,27.2){\vector(-1,0){7}}
\put(14,22.5){\vector(-1,0){7}}
\put(7,25){\vector(1,0){7}}
\put(27,28){\vector(-1,0){7}}
\put(27,22){\vector(-1,0){7}}
\put(27,25){\vector(-1,0){7}}
\put(37,25){\makebox(0,0){\mbox{$C_1\times_{C_0}C_1$}}}
\put(24,29.5){\makebox(0,0){\mbox{$\scriptstyle pr_2$}}}
\put(24,23.4){\makebox(0,0){\mbox{$\scriptstyle pr_1$}}}
\put(24,26.4){\makebox(0,0){\mbox{$\scriptstyle m$}}}
\put(17,25){\makebox(0,0){\mbox{$C_1$}}}

\end{picture}}

\noindent enjoying the additional property that the following square\begin{diagram}[small]C_1\times_{C_0} C_1\times_{C_0}C_1&\rTo^{id_{C_1}\times_{C_0} m}&C_1\times_{C_0}C_1\\\dTo^{m\times_{C_0} id_{C_1}}&&\dTo_m\\C_1\times_{C_0}C_1&\rTo_m&C_1\end{diagram}commutes. An \emph{internal functor} between internal categories is a natural transformation of the associated $2$-truncated simplicial objects. Internal categories and internal functors form a category, denoted $\Cat(\CC).$ Each internal category $C$ extends to a simplicial object $C_\bullet$ in $\CC$, the so-called \emph{nerve} of $C$, by $$C_n=\overbrace{C_1\times_{C_0}\cdots\times_{C_0}C_1}^n,\quad n\geq 0,$$with obvious simplicial operators. The nerve functor $(\!-\!)_\bullet:\Cat(\CC)\to\CC^{\Delta^\op}$ is fully faithful and its essential image consists of those simplicial objects $C_\bullet$ of $\CC$ for which the \emph{Segal maps}$$C_n\longrightarrow \overbrace{C_1\times_{C_0}\cdots\times_{C_0}C_1}^n,\quad n\geq 0,$$are invertible. A simplicial object with invertible Segal maps will be called a \emph{strict Segal object}.  The nerve functor induces thus an equivalence of categories between internal categories and strict Segal objects. Accordingly, the $2$-truncation of a strict Segal object will be called its \emph{underlying internal category}.

Each internal category $C$ possesses an internal \emph{arrow category} $\Arr(C)$ with $\Arr(C)_0=C_1$ and $\Arr(C)_1=C_2\times_{C_1}C_2$. The latter object represents ``commuting squares'' in $C$ so that ``projecting in one direction'' yields the structural source/target maps $\Arr(C)_1\rightrightarrows\Arr(C)_0$ while ``projecting in the other direction'' yields two internal functors $s_C,t_C:\Arr(C)\rightrightarrows C$. For two internal functors $f,g:C\rightrightarrows D$, an \emph{internal natural transformation} $\phi:f\Rightarrow g$ is given by an internal functor $\phi:C\to\Arr(D)$ such that $f=s_D\phi$ and $g=t_D\phi$.

Internal categories, internal functors and internal natural transformations form a $2$-category, still denoted $\Cat(\CC)$. For $\CC=\Set$ we obtain the familiar $2$-category $\Cat=\Cat(\Set)$ of small categories, functors and natural transformations.

\subsection{Strict and lax morphisms of categorical $T$-algebras}\label{strictlax}--\vspace{1ex}

A natural transformation is called {\it cartesian} if all naturality squares are pullbacks. A \emph{monad} $T=(T,\mu,\eta)$ on a category $\CC$ with pullbacks is called \emph{cartesian} if $T$ preserves pullbacks, and as well the multiplication $\mu:T^2\Rightarrow T$ as well the unit $\eta:id_\CC\Rightarrow T$ of the monad are cartesian natural transformations.

Let $T$ be a cartesian monad on a finitely complete category $\CC.$ Since $T$ preserves pullbacks it induces a monad on the category $\Cat(\CC)$ of internal categories. Algebras for this extended monad (also denoted $T$) are called {\it categorical $T$-algebras}. As usual, categorical $T$-algebras can either be considered as $T$-algebras in internal categories, or as internal categories in $T$-algebras. According to the preceding subsection a categorical $T$-algebra can thus also be viewed as a strict Segal object in $T$-algebras. For instance, since the multiplication of $T$ is cartesian, any $T$-algebra $(X,\xi_X:T(X)\to X)$ has a simplicial bar resolution $B_\bullet(T,T,X)=T^{\bullet+1}(X)$ which is such a strict Segal object in $T$-algebras, cf. the proof of Theorem \ref{bar1}.

The monad $T$ on $\Cat(\CC)$ takes an internal natural transformation $\phi:f\Rightarrow g$ to an internal natural transformation $T(\phi):T(f)\Rightarrow T(g).$ Such a monad on a $2$-category gives rise to a so-called \emph{$2$-monad}. This allows the definition of two different notions of morphisms of categorical $T$-algebras. Beside the classical \emph{strict morphisms} there are also the \emph{lax morphisms},
cf. \cite{Lack2,W}. Let $(A, \xi_A),\, (B, \xi_B)$ be categorical $T$-algebras. A lax morphism $(A,\xi_A)\to(B,\xi_B)$ is given by a pair $(f,\phi)$ consisting of an internal functor $f:A\to B$ and an internal natural transformation

\scalebox{0.8} 
{
\begin{pspicture}(-4.5,-1)(1,1.3)
\psline[linewidth=0.022cm,arrowsize=0.1cm 2.0,arrowlength=1.4,arrowinset=0.4]{->}(0.6,0.75)(2,0.77)
\psline[linewidth=0.022cm,arrowsize=0.1cm 2.0,arrowlength=1.4,arrowinset=0.4]{->}(0.3,0.6)(0.3,-0.65)
\psline[linewidth=0.022cm,arrowsize=0.1cm 2.0,arrowlength=1.4,arrowinset=0.4]{->}(0.6,-0.85)(2.15,-0.85)
\psline[linewidth=0.022cm,arrowsize=0.1cm 2.0,arrowlength=1.4,arrowinset=0.4]{->}(2.45,0.6)(2.45,-0.65)
\usefont{T1}{ptm}{m}{n}
\rput(0.2,0.8){$T(A)$}
\usefont{T1}{ptm}{m}{n}
\rput(0.3,-0.8){$A$}
\usefont{T1}{ptm}{m}{n}
\rput(2.4,0.75){$T(B)$}
\usefont{T1}{ptm}{m}{n}
\rput(2.4,-0.85){$B$}
\psline[linewidth=0.026cm,arrowsize=0.033cm 1.59,arrowlength=1.42,arrowinset=0.51,doubleline=true,doublesep=0.03]{->}(1.45,0.1)(1.15,-0.3)
\usefont{T1}{ptm}{m}{n}
\rput(1.0,0.15){$\phi$}
\usefont{T1}{ptm}{m}{n}
\rput(1.35,1.0){$T(f)$}
\usefont{T1}{ptm}{m}{n}
\rput(1.35,-0.6){$f$}
\usefont{T1}{ptm}{m}{n}
\rput(0.05,0.0){$\xi_A$}
\usefont{T1}{ptm}{m}{n}
\rput(2.7,0.0){$\xi_B$}

\end{pspicture}
}

\noindent such that the following conditions hold:

\scalebox{0.8} 
{
\begin{pspicture}(-2,-2.2)(7.6948533,2)
\psline[linewidth=0.022cm,arrowsize=0.1cm 2.0,arrowlength=1.4,arrowinset=0.4]{->}(0.8,-0.15)(2.2,-0.15)
\psline[linewidth=0.022cm,arrowsize=0.1cm 2.0,arrowlength=1.4,arrowinset=0.4]{->}(0.45,-0.3)(0.45,-1.55)
\psline[linewidth=0.022cm,arrowsize=0.1cm 2.0,arrowlength=1.4,arrowinset=0.4]{->}(0.75,-1.75)(2.35,-1.75)
\psline[linewidth=0.022cm,arrowsize=0.1cm 2.0,arrowlength=1.4,arrowinset=0.4]{->}(2.6,-0.3)(2.6,-1.55)
\usefont{T1}{ptm}{m}{n}
\rput(0.4,-0.1){$T(A)$}
\usefont{T1}{ptm}{m}{n}
\rput(0.45,-1.75){$A$}
\usefont{T1}{ptm}{m}{n}
\rput(2.55,-0.1){$T(B)$}
\usefont{T1}{ptm}{m}{n}
\rput(2.55,-1.75){$B$}
\usefont{T1}{ptm}{m}{n}
\rput(0.2,-0.9){$\xi_A$}
\usefont{T1}{ptm}{m}{n}
\rput(2.9,-0.9){$\xi_B$}

\psline[linewidth=0.026cm,arrowsize=0.033cm 1.59,arrowlength=1.42,arrowinset=0.51,doubleline=true,doublesep=0.03]{->}(1.65,-0.8)(1.33,-1.2)
\usefont{T1}{ptm}{m}{n}
\rput(1.25,-0.75){$\phi$}
\usefont{T1}{ptm}{m}{n}
\rput(1.55,0.1){$T(f)$}
\usefont{T1}{ptm}{m}{n}
\rput(1.55,-1.5){$f$}
\psline[linewidth=0.022cm,arrowsize=0.1cm 2.0,arrowlength=1.4,arrowinset=0.4]{->}(0.8,1.5)(2.2,1.5)
\psline[linewidth=0.022cm,arrowsize=0.1cm 2.0,arrowlength=1.4,arrowinset=0.4]{->}(0.45,1.35)(0.5,0.1)
\psline[linewidth=0.022cm,arrowsize=0.1cm 2.0,arrowlength=1.4,arrowinset=0.4]{->}(2.6,1.35)(2.6,0.1)
\usefont{T1}{ptm}{m}{n}
\rput(0.45,1.55){$A$}
\usefont{T1}{ptm}{m}{n}
\rput(2.55,1.55){$B$}
\usefont{T1}{ptm}{m}{n}
\usefont{T1}{ptm}{m}{n}
\rput(1.4,1.7){$f$}
\usefont{T1}{ptm}{m}{n}
\rput(0.2,0.7){$\eta_A$}
\usefont{T1}{ptm}{m}{n}
\rput(2.9,0.7){$\eta_B$}
\usefont{T1}{ptm}{m}{n}
\rput(3.75,-0.15){=}
\psline[linewidth=0.022cm,arrowsize=0.1cm 2.0,arrowlength=1.4,arrowinset=0.4]{->}(5.2,-1.75)(6.75,-1.75)
\usefont{T1}{ptm}{m}{n}
\rput(4.9,-1.75){$A$}
\usefont{T1}{ptm}{m}{n}
\rput(7.0,-1.75){$B$}
\usefont{T1}{ptm}{m}{n}
\rput(5.95,-1.5){$f$}
\psline[linewidth=0.022cm,arrowsize=0.1cm 2.0,arrowlength=1.4,arrowinset=0.4]{->}(5.2,1.5)(6.6,1.5)
\psline[linewidth=0.022cm,arrowsize=0.1cm 2.0,arrowlength=1.4,arrowinset=0.4]{->}(4.9,1.3)(4.9,-1.5)
\psline[linewidth=0.022cm,arrowsize=0.1cm 2.0,arrowlength=1.4,arrowinset=0.4]{->}(7.0,1.3)(7.0,-1.5)
\usefont{T1}{ptm}{m}{n}
\rput(4.8,1.5){$A$}
\usefont{T1}{ptm}{m}{n}
\rput(7.0,1.5){$B$}
\usefont{T1}{ptm}{m}{n}
\usefont{T1}{ptm}{m}{n}
\rput(5.9,1.7){$f$}
\usefont{T1}{ptm}{m}{n}
\rput(4.6,-0.16){$id_A$}
\usefont{T1}{ptm}{m}{n}
\rput(7.35,-0.15){$id_B$}
\end{pspicture}
}

\noindent and

\scalebox{0.8} 
{
\begin{pspicture}(0,-4.5)(10.7,0.5)
\psline[linewidth=0.022cm,arrowsize=0.1cm 2.0,arrowlength=1.4,arrowinset=0.4]{->}(0.9,-1.3)(2.3,-2.3)
\psline[linewidth=0.022cm,arrowsize=0.1cm 2.0,arrowlength=1.4,arrowinset=0.4]{->}(0.45,-1.4)(0.45,-2.6)
\psline[linewidth=0.022cm,arrowsize=0.1cm 2.0,arrowlength=1.4,arrowinset=0.4]{->}(0.75,-3.15)(2.2,-4.15)
\psline[linewidth=0.022cm,arrowsize=0.1cm 2.0,arrowlength=1.4,arrowinset=0.4]{->}(2.55,-2.75)(2.55,-3.95)
\usefont{T1}{ptm}{m}{n}
\rput(0.4,-1.2){$T(A)$}
\usefont{T1}{ptm}{m}{n}
\rput(0.45,-2.85){$A$}
\usefont{T1}{ptm}{m}{n}
\rput(2.5,-2.5){$T(B)$}
\usefont{T1}{ptm}{m}{n}
\rput(2.55,-4.25){$B$}
\usefont{T1}{ptm}{m}{n}
\rput(0.25,-2.0){$\xi_A$}
\usefont{T1}{ptm}{m}{n}
\rput(2.85,-3.3){$\xi_B$}
\usefont{T1}{ptm}{m}{n}
\rput(4.15,-2.0){$\mu_B$}
\psline[linewidth=0.022cm,arrowsize=0.07cm 1.6,arrowlength=1.42,arrowinset=0.51,doubleline=true,doublesep=0.03]{->}(1.8,-2.75)(1.4,-3.1)
\usefont{T1}{ptm}{m}{n}
\rput(1.35,-2.65){$\phi$}
\usefont{T1}{ptm}{m}{n}
\rput(1.75,-1.55){$T(f)$}
\usefont{T1}{ptm}{m}{n}
\rput(1.3,-3.9){$f$}
\usefont{T1}{ptm}{m}{n}
\rput(3.9,-3.9){$\xi_B$}
\usefont{T1}{ptm}{m}{n}
\rput(3.9,-0.25){$T^2(f)$}
\psline[linewidth=0.022cm,arrowsize=0.1cm 2.0,arrowlength=1.4,arrowinset=0.4]{->}(4.0,-1.4)(2.8,-2.3)
\psline[linewidth=0.022cm,arrowsize=0.1cm 2.0,arrowlength=1.4,arrowinset=0.4]{->}(2.0,-0.15)(0.85,-0.9)
\psline[linewidth=0.022cm,arrowsize=0.1cm 2.0,arrowlength=1.4,arrowinset=0.4]{->}(3.15,-0.13)(4.2,-0.9)
\psline[linewidth=0.022cm,arrowsize=0.1cm 2.0,arrowlength=1.4,arrowinset=0.4]{->}(4.4,-1.45)(4.4,-2.65)
\usefont{T1}{ptm}{m}{n}
\rput(2.57,0.0){$T^2(A)$}
\usefont{T1}{ptm}{m}{n}
\rput(4.35,-1.17){$T^2(B)$}
\usefont{T1}{ptm}{m}{n}
\rput(4.3,-2.85){$T(B)$}
\usefont{T1}{ptm}{m}{n}
\rput(3.2,-1.7){$\mu_B$}
\usefont{T1}{ptm}{m}{n}
\rput(1.3,-0.3){$\mu_A$}
\psline[linewidth=0.022cm,arrowsize=0.1cm 2.0,arrowlength=1.4,arrowinset=0.4]{->}(4.35,-3.05)(3.0,-4.15)
\usefont{T1}{ptm}{m}{n}
\rput(5.33,-1.95){=}
\psline[linewidth=0.022cm,arrowsize=0.1cm 2.0,arrowlength=1.4,arrowinset=0.4]{->}(6.3,-1.4)(6.3,-2.6)
\psline[linewidth=0.022cm,arrowsize=0.1cm 2.0,arrowlength=1.4,arrowinset=0.4]{->}(6.7,-3.2)(8.0,-4.15)
\psline[linewidth=0.022cm,arrowsize=0.1cm 2.0,arrowlength=1.4,arrowinset=0.4]{->}(7.7,-2.15)(6.65,-2.85)
\usefont{T1}{ptm}{m}{n}
\rput(6.2,-1.2){$T(A)$}
\usefont{T1}{ptm}{m}{n}
\rput(6.35,-2.85){$A$}
\usefont{T1}{ptm}{m}{n}
\rput(8.2,-2.0){$T(A)$}
\usefont{T1}{ptm}{m}{n}
\rput(8.35,-4.25){$B$}
\psline[linewidth=0.026cm,arrowsize=0.033cm 1.59,arrowlength=1.42,arrowinset=0.51,doubleline=true,doublesep=0.03]{->}(8.55,-3.0)(8.15,-3.35)
\usefont{T1}{ptm}{m}{n}
\rput(8.0,-3.0){$\phi$}
\usefont{T1}{ptm}{m}{n}
\rput(6.9,-3.8){$f$}
\usefont{T1}{ptm}{m}{n}
\rput(6.9,-0.3){$\mu_A$}
\usefont{T1}{ptm}{m}{n}
\rput(7.1,-2.2){$\xi_A$}
\usefont{T1}{ptm}{m}{n}
\rput(9.5,-2.3){$T(f)$}
\usefont{T1}{ptm}{m}{n}
\rput(6.55,-1.85){$\xi_A$}
\usefont{T1}{ptm}{m}{n}
\rput(10.55,-1.95){$\mu_B$}
\usefont{T1}{ptm}{m}{n}
\rput(7.95,-1.0){$\mu_A$}
\psline[linewidth=0.022cm,arrowsize=0.1cm 2.0,arrowlength=1.4,arrowinset=0.4]{->}(8.7,-2.15)(9.7,-2.8)
\psline[linewidth=0.022cm,arrowsize=0.1cm 2.0,arrowlength=1.4,arrowinset=0.4]{->}(7.6,-0.15)(6.4,-0.95)
\psline[linewidth=0.022cm,arrowsize=0.1cm 2.0,arrowlength=1.4,arrowinset=0.4]{->}(8.8,-0.15)(9.9,-0.85)
\usefont{T1}{ptm}{m}{n}
\rput(8.2,0.0){$T^2(A)$}
\usefont{T1}{ptm}{m}{n}
\rput(10.25,-1.15){$T^2(B)$}
\usefont{T1}{ptm}{m}{n}
\rput(10.2,-2.9){$T(B)$}
\usefont{T1}{ptm}{m}{n}
\rput(9.5,-3.8){$\xi_B$}
\usefont{T1}{ptm}{m}{n}
\rput(9.6,-0.25){$T^2(f)$}
\psline[linewidth=0.022cm,arrowsize=0.1cm 2.0,arrowlength=1.4,arrowinset=0.4]{->}(9.8,-3.2)(8.65,-4.1)
\psline[linewidth=0.022cm,arrowsize=0.1cm 2.0,arrowlength=1.4,arrowinset=0.4]{->}(10.25,-1.35)(10.25,-2.65)
\psline[linewidth=0.022cm,arrowsize=0.1cm 2.0,arrowlength=1.4,arrowinset=0.4]{->}(8.2,-0.3)(8.2,-1.7)
\psline[linewidth=0.026cm,arrowsize=0.033cm 1.59,arrowlength=1.42,arrowinset=0.51,doubleline=true,doublesep=0.03]{->}(9.45,-1.45)(9.05,-1.8)
\usefont{T1}{ptm}{m}{n}
\rput(8.85,-1.4){$T(\phi)$}
\end{pspicture}
}

\noindent A $T$-natural transformation between two lax morphisms $(f,\phi),(g,\psi): A \rightrightarrows B$  is given by an internal natural transformation $\rho: f\Rightarrow g$ such that

\scalebox{0.8} 
{
\begin{pspicture}(-1.5,-1.85)(6.55,2)
\psline[linewidth=0.022cm,arrowsize=0.1cm 2.0,arrowlength=1.4,arrowinset=0.4]{->}(0.45,0.6)(0.45,-0.6)
\psline[linewidth=0.022cm,arrowsize=0.1cm 2.0,arrowlength=1.4,arrowinset=0.4]{->}(2.6,0.6)(2.6,-0.6)
\usefont{T1}{ptm}{m}{n}
\rput(0.2,0.0){$\xi_A$}
\usefont{T1}{ptm}{m}{n}
\rput(2.35,-0.1){$\xi_B$}
\usefont{T1}{ptm}{m}{n}
\rput(0.4,0.8){$T(A)$}
\usefont{T1}{ptm}{m}{n}
\rput(0.45,-0.8){$A$}
\usefont{T1}{ptm}{m}{n}
\rput(2.55,0.8){$T(B)$}
\usefont{T1}{ptm}{m}{n}
\rput(2.55,-0.8){$B$}
\psline[linewidth=0.026cm,arrowsize=0.033cm 1.59,arrowlength=1.42,arrowinset=0.51,doubleline=true,doublesep=0.03]{->}(1.75,-0.6)(1.45,-1.0)
\usefont{T1}{ptm}{m}{n}
\rput(1.3,-0.6){$\psi$}
\usefont{T1}{ptm}{m}{n}
\rput(1.5,1.65){$T(f)$}
\usefont{T1}{ptm}{m}{n}
\rput(1.55,-1.6){$g$}
\psbezier[linewidth=0.022,arrowsize=0.1cm 2.0,arrowlength=1.4,arrowinset=0.4]{->}(0.55,-1.0)(0.85,-1.55)(2.3,-1.5)(2.55,-0.95)
\psbezier[linewidth=0.022,arrowsize=0.1cm 2.0,arrowlength=1.4,arrowinset=0.4]{->}(0.55,0.6)(0.85,0.05)(2.3,0.1)(2.55,0.65)
\psbezier[linewidth=0.022,arrowsize=0.1cm 2.0,arrowlength=1.4,arrowinset=0.4]{->}(0.55,1.05)(0.85,1.6)(2.3,1.6)(2.5,1.0)
\psline[linewidth=0.026cm,arrowsize=0.033cm 1.59,arrowlength=1.42,arrowinset=0.51,doubleline=true,doublesep=0.03]{->}(1.75,1.0)(1.4,0.55)
\usefont{T1}{ptm}{m}{n}
\rput(1.3,1.0){$T(\rho)$}
\usefont{T1}{ptm}{m}{n}
\rput(1.6,0.0){$T(g)$}
\usefont{T1}{ptm}{m}{n}
\rput(3.35,-0.15){=}
\psline[linewidth=0.022cm,arrowsize=0.1cm 2.0,arrowlength=1.4,arrowinset=0.4]{->}(4.0,0.6)(4.0,-0.6)
\usefont{T1}{ptm}{m}{n}
\rput(3.95,0.8){$T(A)$}
\usefont{T1}{ptm}{m}{n}
\rput(4.0,-0.8){$A$}
\usefont{T1}{ptm}{m}{n}
\rput(6.1,0.8){$T(B)$}
\usefont{T1}{ptm}{m}{n}
\rput(6.1,-0.8){$B$}
\psline[linewidth=0.026cm,arrowsize=0.033cm 1.59,arrowlength=1.42,arrowinset=0.51,doubleline=true,doublesep=0.03]{->}(5.3,-0.7)(5.0,-1.1)
\usefont{T1}{ptm}{m}{n}
\rput(5.05,1.65){$T(f)$}
\usefont{T1}{ptm}{m}{n}
\rput(5.25,-1.6){$g$}
\psbezier[linewidth=0.022,arrowsize=0.1cm 2.0,arrowlength=1.4,arrowinset=0.4]{->}(4.1,-1.0)(4.4,-1.6)(5.9,-1.5)(6.1,-1.0)
\psbezier[linewidth=0.022,arrowsize=0.1cm 2.0,arrowlength=1.4,arrowinset=0.4]{->}(4.1,-0.6)(4.4,-0.05)(5.85,-0.1)(6.1,-0.65)
\psbezier[linewidth=0.022,arrowsize=0.1cm 2.0,arrowlength=1.4,arrowinset=0.4]{->}(4.1,1.0)(4.4,1.6)(5.85,1.55)(6.1,1.0)
\psline[linewidth=0.026cm,arrowsize=0.033cm 1.59,arrowlength=1.42,arrowinset=0.51,doubleline=true,doublesep=0.03]{->}(5.4,0.8)(5.05,0.35)
\usefont{T1}{ptm}{m}{n}
\rput(4.25,0.1){$\xi_A$}
\usefont{T1}{ptm}{m}{n}
\rput(6.45,0.1){$\xi_B$}
\usefont{T1}{ptm}{m}{n}
\rput(4.85,-0.7){$\rho$}
\usefont{T1}{ptm}{m}{n}
\rput(5.15,0.0){$f$}
\psline[linewidth=0.022cm,arrowsize=0.1cm 2.0,arrowlength=1.4,arrowinset=0.4]{->}(6.2,0.6)(6.2,-0.6)
\usefont{T1}{ptm}{m}{n}
\rput(4.9,0.75){$\phi$}
\end{pspicture}
}

For a cartesian monad $T$ on a category $\CC$ with pullbacks, categorical $T$-algebras, strict morphisms of categorical $T$-algebras and $T$-natural transformations form a $2$-category which we shall denote $\Alg_T(\Cat(\CC))$ or $\Cat(\Alg_T(\CC))$.

If $\CC$ has a terminal object $1$  and hence is finitely complete then $\Cat(\CC)$ as well has a terminal object and is finitely complete. An internal category $C$ is terminal if and only if $C_0$ and $C_1$ are terminal. A terminal internal category has a unique $T$-algebra structure; the latter promotes it to a terminal categorical  $T$-algebra. All terminal objects will be denoted $1$ hoping that this will cause no confusion.

\begin{defin}[{\cite{EHBat}}]\label{intalg1}Let $A$ be a categorical $T$-algebra for a cartesian monad $T$.

An \emph{internal $T$-algebra in $A$} is a \emph{lax morphism} of categorical $T$-algebras from the \emph{terminal} categorical $T$-algebra to $A$.

Internal $T$-algebras in $A$ and $T$-natural transformations form a category $\Int_T(A)$ and this construction extends to a $2$-functor $\Int_T:\Alg_T(\Cat(\CC))\to\Cat.$\end{defin}

\begin{theorem}[{\cite{EHBat}}]\label{bar1}The $2$-functor $\,\Int_T$ is representable. The representing categorical $T$-algebra $\HT$ in $\Alg_T(\Cat(\CC))$ is the underlying internal category
\begin{equation}\label{HT}\begin{diagram}T(1)&\pile{\lTo\\\rTo\\\lTo}&T^2(1)&\pile{\lTo\\\lTo\\\lTo}&T^3(1)\end{diagram} \end{equation}
\noindent of the simplicial bar resolution $B_\bullet(T,T,1)$ of the terminal $T$-algebra in $\Alg_T(\CC)$.
\end{theorem}

\noindent This categorical $T$-algebra $\HT$ will be called the \emph{internal algebra classifier} of $T$ because morphisms of categorical $T$-algebras $\HT\to A$ correspond one-to-one to lax morphisms of categorical $T$-algebras $1\to A$, i.e. to internal $T$-algebras in $A$.

\begin{proof}The free-forgetful adjunction $F_T:\CC\lra\Alg_T(\CC):U_T$ induces a simplicial bar resolution $B_\bullet(T,T,1)$ of the terminal $T$-algebra $1$ in $\Alg_T(\CC)$. Explicitly, $B_\bullet(T,T,1)=T^{\bullet+1}(1)$ with the usual simplicial operators, induced by multiplication and unit of $T$. Since $T$ has a cartesian multiplication, the following naturality square (where $!:T(1)\to 1$ denotes the unique $T$-algebra structure of $1$)\begin{diagram}[small]T^3(1)&\rTo^{T^2(!)}&T^2(1)\\\dTo^{\mu_ {T(1)}}&&\dTo_{\mu_1}\\T^2(1)&\rTo_{T(!)}&T(1)\end{diagram}
is a pullback in $\Alg_T(\CC)$, and hence the Segal map $T^3(1)\to T^2(1)\times_{T(1)}T^2(1)$ is invertible. Notice that this pullback square realises the identity of simplicial face operations $\partial^0\partial^2=\partial^1\partial^1:B_2(T,T,1)\to B_0(T,T,1)$. Similarly, the cartesianess of $\mu:T^2\Rightarrow T$ implies that the identity of simplicial face operations $\partial^0\partial^n=\partial^{n-1}\partial^0:B_n(T,T,1)\to B_{n-2}(T,T,1)$ is realised by a pullback square in $\Alg_T(\CC)$. It follows then by induction on $n$ that all higher Segal maps$$B_n(T,T,1)\longrightarrow \overbrace{B_1(T,T,1)\times_{B_0(T,T,1)}\cdots\times_{B_0(T,T,1)}B_1(T,T,1)}^n$$are invertible, so that (\ref{HT}) is the $2$-truncation of a strict Segal object in $T$-algebras, and therefore represents indeed a categorical $T$-algebra $\HT$.

It remains to be shown that $\HT$ has the asserted universal property, namely that the category of strict morphisms $\HT\to A$ is canonically isomorphic to the category of lax morphisms $1\to A$. Our proof closely follows Lack \cite[Section 2]{Lack2}.

A lax morphism of categorical $T$-algebras $1\to A$ consists of an internal functor $a:1\to A$ together with an internal natural transformation $\phi$ from $T(1)\stackrel{T(a)}{\to} T(A)\stackrel{\xi_A}{\to} A$ to $T(1)\to 1\stackrel{a}{\to} A$ fulfilling coherence conditions. Now, the existence of the $1$-cell $a:1\to A$ in $\Cat(\CC)$ amounts to the existence of the $1$-cell $\bar{a}:T(1)\stackrel{T(a)}{\to} T(A)\stackrel{\xi_A}{\to} A$ in $\Alg_T(\Cat(\CC))$. Moreover, the existence of the $2$-cell $\phi$ in $\Cat(\CC)$ amounts to the existence of the $2$-cell $\bar{\phi}=\xi_A\cdot T(\phi)$ in $\Alg_T(\Cat(\CC))$.

It can be checked that the coherence conditions of a lax morphism $1\to A$ translate under this correspondence into the coherence conditions of a morphism of categorical $T$-algebras $(\bar{a},\bar{\phi},\bar{\phi}\times_{\bar{a}}\bar{\phi}):\HT\to A$. This correspondence respects the category structures on both sides and is $2$-functorial in $A$, see \cite{Lack2} for details.\end{proof}

\begin{example}\label{Benabou}The previous theorem applies to the \emph{free monoid monad} $T$ which is a cartesian monad on $\Set.$ In this case, categorical $T$-algebras are strict monoidal small categories. B\'enabou \cite{Benabou} noticed that a lax monoidal functor from the terminal monoidal category to a monoidal category $A$ is the same as a monoid in $A.$ The category of internal $T$-algebras in $A$ is thus the category of monoids in $A.$

On the other hand, the internal algebra classifier $\HT$ is easily identified with a skeleton of the category of finite ordinals and order-preserving maps, also known as the \emph{augmented simplex category} $\Delta_+$. This is a strict monoidal category with unit the empty ordinal, and tensor given by the join of ordinals. This strict monoidal category contains a generic monoid (the first non-empty ordinal) and Theorem \ref{bar1} recovers the well-known fact that monoids in a (strict) monoidal category $A$ are the same as (strict) monoidal functors from $\Delta_+$ to $A$.\end{example}

\subsection{Monad morphisms}

Let $S$ (resp. $T$) be a finitary  monad on a cocomplete category $\DD$ (resp. $\CC$). For any functor $d: \CC \rightarrow \DD$ with left adjoint $c:\DD\to\CC$ the following three conditions are equivalent:

 \begin{enumerate}

\item There exists a  functor $d': Alg_T\rightarrow Alg_S$ such that $U_Sd'=dU_T$;

\item  There exists a natural transformation $\Psi: Sd\Rightarrow dT$ compatible with the multiplication and unit of $S$ and $T$;

\item\label{Phi} There exists a morphism of monads $\Phi: S\Rightarrow dTc$.\end{enumerate}

The equivalence between (1) and (2) is classical and does not require the existence of a left adjoint $c$. For the equivalence between (2) and (3), use unit $\eta:id_\DD\Rightarrow dc$ and counit $\epsilon:cd\Rightarrow id_\CC$ of the adjunction to define $\Phi=\Psi c\circ S\eta$, resp. $\Psi=dT\epsilon\circ\Phi d$. A $2$-categorical diagram chase shows that these two assignments are mutually inverse.

If these conditions are satisfied then by the adjoint lifting theorem the functor $d'$ has a left adjoint $c'$ such that the following square of adjoint functors commutes:
\vspace{13mm}

\begin{equation}\label{poladj} \end{equation}
{\unitlength=1mm

\begin{picture}(40,10)(-19,0)
\put(49,15){\shortstack{\mbox{$\scriptstyle U_T$}}}
\put(57,15){\shortstack{\mbox{$\scriptstyle F_T$}}}
\put(36,3){\shortstack{\mbox{$ c$}}}
\put(36,9){\shortstack{\mbox{$ d$}}}
\put(13,15){\shortstack{\mbox{$\scriptstyle U_S$}}}
\put(20,15){\shortstack{\mbox{$\scriptstyle F_S$}}}
\put(18,25){\makebox(0,0){\mbox{${\Alg_S}$}}}
\put(19,11){\vector(0,1){10}}
\put(17,21){\vector(0,-1){10}}
\put(26,24){\vector(1,0){24}}
\put(49,26){\vector(-1,0){24}}
\put(36,27){\shortstack{\mbox{$ d'$}}}
\put(36,21.5){\shortstack{\mbox{$ c'$}}}
\put(56,25){\makebox(0,0){\mbox{$\Alg_T$}}}
\put(54,21){\vector(0,-1){10}}
\put(56,11){\vector(0,1){10}}
\put(18,7){\makebox(0,0){\mbox{$\DD$}}}
\put(23,6){\vector(1,0){26}}
\put(49,8){\vector(-1,0){26}}
\put(56,7){\makebox(0,0){\mbox{$\CC$}}}
\end{picture}}

The natural transformation $\Phi:S\Rightarrow dTc$ yields (by adjunction) a natural transformation
$\Phi':cS \Rightarrow Tc$ and hence a natural transformation $T\Phi':TcS\Rightarrow T^2c$ which, after composition with $\mu c:T^2c\Rightarrow Tc$, gives rise to a natural transformation
\begin{gather*}\theta:TcS\Rightarrow Tc\end{gather*}inducing the structure of a \emph{right $S$-module} on the composite functor $Tc:\DD\to\CC$.

\begin{defin}\label{Phic}We will say that $\Phi:S\Rightarrow dTc$ is a cartesian monad morphism if $\Phi$ is a cartesian natural transformation and $c\dashv d$ is a cartesian adjunction  (i.e. unit and counit are cartesian natural transformations and $c$ preserves pullbacks).\end{defin}

\begin{lem}\label{polsq}For a natural transformation $\Phi:S\Rightarrow dTc$ between cartesian monads such that $c\dashv d$ is a cartesian adjunction, the following conditions are equivalent: \begin{enumerate}\item[(i)]the natural transformation $\Phi:S\Rightarrow dTc$ is cartesian; \item[(ii)]the natural transformation $\Phi':cS\Rightarrow Tc$ is cartesian;\item[(iii)]the natural transformation $\theta:TcS\Rightarrow Tc$ is cartesian.\end{enumerate}
 \end{lem}
\begin{proof}$\Phi$ is cartesian if and only if $\Phi'$ is because of the identities $\Phi'=(\epsilon Tc)(c\Phi)$ and $\Phi=(d\Phi')(\eta S)$ and the hypothesis that $c$ preserves pullbacks. Since $\mu:T^2\Rightarrow T$ is cartesian and $\theta=(\mu c)\Phi'$ we see that $\Phi'$ is cartesian if and only if $\theta$ is.\end{proof}

\begin{defin}[{\cite{EHBat}}]\label{intalg2}Let $\Phi:S\Rightarrow dTc$ be a cartesian monad morphism.

An \emph{internal $S$-algebra in a categorical $T$-algebra $A$} is a \emph{lax morphism} of categorical $S$-algebras from the \emph{terminal} categorical $S$-algebra to $d'(A)$.

Internal $S$-algebras in $A$ and $S$-natural transformations form a category $\Int_S(A)$. This construction extends to a $2$-functor
$$\Int_S: \Alg_T(\Cat(\CC))\to\Cat.$$ \end{defin}

\begin{theorem}[{\cite{EHBat}}]\label{TST}The $2$-functor $\,\Int_S$ is representable. The representing categorical $T$-algebra $\HS$ is the underlying internal category

\begin{equation}\label{objects}\begin{diagram}Tc(1)&\pile{\lTo\\\rTo\\\lTo}&TcS(1)&\pile{\lTo\\\lTo\\\lTo}&TcS^2(1)\end{diagram}\end{equation}

\noindent of the two-sided simplicial bar construction $B_\bullet(Tc,S,1)=TcS^\bullet(1)$ of the terminal $S$-algebra, where the right action of $S$ on $Tc$ is given by $\theta:TcS\Rightarrow Tc$.\end{theorem}

\noindent This categorical $T$-algebra $\HS$ will be called the \emph{internal $S$-algebra classifier} of $T$ because morphisms of categorical $T$-algebras $\HS\to A$ correspond one-to-one to lax morphisms of categorical $S$-algebras $1\to d'(A)$, i.e. to internal $S$-algebras in $A$.

\begin{proof}Since $\Phi:S\Rightarrow dTc$ is a cartesian monad morphism,  the right action $\theta:TcS\Rightarrow Tc$ is cartesian as well (cf. Lemma \ref{polsq}), so that $B_\bullet(Tc,S,1)$ is a strict Segal object of $\Alg_T(\CC)$ by the same inductive argument as in the proof of Theorem \ref{bar1}.

The universal property of the underlying categorical $T$-algebra $\HS$ follows now from the following adjunction argument. Lax morphisms $1\to d'(A)$ correspond by Theorem \ref{bar1} one-to-one to simplicial maps $B_\bullet(S,S,1)\to d'(A)_\bullet$; the latter correspond via the adjunction $c'\dashv d'$ to simplicial maps $c'B_\bullet(S,S,1)\to A_\bullet$. The simplicial isomorphism $c'B_\bullet(S,S,1)\cong B_\bullet(Tc,S,1)$ permits us to conclude.\end{proof}

\subsection{Internal left Kan extensions}For each categorical $T$-algebra $A$, the cartesian monad morphism $\Phi:S\Rightarrow dTc$ induces a functor $$\delta_A^{\Phi}: \Int_T(A)\to\Int_S(A),$$ taking an internal $T$-algebra $X:1\to A$ to the internal $S$-algebra $d'(X):1\to d'(A)$.

In good cases, the functor $\delta^\Phi_A$ admits a left adjoint $$\gamma^{\Phi}_A:\Int_S(A)\to\Int_T(A).$$It is one of the crucial observations of \cite{EHBat} that this left adjoint $\gamma^\Phi_A$ (if it exists) can be computed as an internal left Kan extension. Indeed, $\delta^{\Phi}_A$ may be identified with restriction $(\HPhi)^*$ along a certain internal functor of classifiers
$$\HPhi:\HS\to\HT$$ in the following way: represent an internal $T$-algebra $X$ in $A$ by $\tilde{X}:\HT\to A.$ Then the internal $S$-algebra $\delta^{\Phi}_A(X)$ is represented by the composite morphism\begin{diagram}[small]\HS&\rTo^\HPhi&\HT&\rTo^{\tilde{X}}&A.\end{diagram}Accordingly, the left adjoint $\gamma_A^\Phi$ may be identified with left Kan extension $(\HPhi)_!$ along the same internal functor. On objects, the internal functor $\HPhi:\HS\to\HT$ is given by $T(!):Tc(1)\to T(1),$ where $!:c(1)\to 1$ while on morphisms $\HPhi$ is given by $$TcS(1)\stackrel{T\Phi'_1}{\longrightarrow} T^2c(1) \ \stackrel{T^2(!)}{\longrightarrow} T^2(1).$$

\begin{pro}[{\cite{EHBat}}]\label{Kanextension}Let $A$ be a categorical $T$-algebra and $\Phi:S\Rightarrow dTc$ be a cartesian monad morphism.

The adjoint pair $\gamma_A^\Phi:\Int_S(A)\lra\Int_T(A):\delta_A^\Phi$ is represented by left Kan extension and restriction along the internal functor $\HPhi:\HS\to\HT$ in $T$-algebras.\end{pro}

\begin{proof}It suffices to show that $\delta_A^\Phi$ is represented by restriction along $\HPhi$.

The image $\delta_A^\Phi(X)$ of an internal $T$-algebra $X:1\to A$ is the lax morphism of categorical $S$-algebras $d'(X):1=d'(1)\to d'(A)$. By Theorem \ref{bar1}, such a lax morphism is represented by a simplicial map $B_\bullet(S,S,1)\to d'(A)_\bullet$. In virtue of the adjunction $c'\dashv d'$, the latter corresponds to a simplicial map $c'B_\bullet(S,S,1)\cong B_\bullet(Tc,S,1)\to A_\bullet$. It remains to be shown that this simplicial map factors through the simplicial map $B_\bullet(T,T,1)\to A_\bullet$ representing the internal $T$-algebra $X$.

The required simplicial map $c'B_\bullet(S,S,1)\to B_\bullet(T,T,1)$ is adjoint to a simplicial map $B_\bullet(S,S,1)\to d'B_\bullet(T,T,1)$. The latter derives from the natural transformation $\Psi:Sd\Rightarrow dT$ so that the former derives from its ``mate'' $\Phi':cS\Rightarrow Tc$. As underlying internal functor we get the functor $\HPhi:\HS\to\HT$ defined above.\end{proof}

\begin{defin}[{\cite{W2}}]\label{cocomplete}Let $(A,\xi_A)$ be a categorical $T$-algebra and $\xi:B\to C$ be a morphism of categorical $T$-algebras. The algebra $A$ is called \emph{cocomplete} with respect to $\xi$ if for any morphism of $T$-algebras $F:B\to A$ the following pointwise left Kan extension
{\unitlength=1mm

\begin{picture}(40,29)(0,0)
\put(48,15){\shortstack{\mbox{$\scriptstyle \xi$}}}
\put(51,25){\makebox(0,0){\mbox{$B$}}}
\put(51,21){\vector(0,-1){10}}
\put(51,7){\makebox(0,0){\mbox{$C$}}}
\put(70,25){\makebox(0,0){\mbox{$A$}}}
\put(54,25){\vector(1,0){14}}
\put(55,9){\vector(1,1){13}}
\put(60,26){\shortstack{\mbox{$\scriptstyle  F$}}}
\put(61,13){\shortstack{\mbox{$\scriptstyle  G$}}}
\put(58,16){\vector(0,-1){0.1}}
\put(57.85,19){\line(0,-1){2.2}}
\put(58.15,19){\line(0,-1){2.2}}
\put(55,17){\shortstack{\mbox{$\scriptstyle  \phi$}}}
\end{picture}}

\noindent exists in $\Cat(\CC)$ and the induced diagram
{\unitlength=1mm

\begin{picture}(40,30)(0,0)
\put(45,16){\shortstack{\mbox{$\scriptstyle T(\xi)$}}}
\put(51,25){\makebox(0,0){\mbox{$T(B)$}}}
\put(51,21){\vector(0,-1){10}}
\put(51,7){\makebox(0,0){\mbox{$T(C)$}}}
\put(72,25){\makebox(0,0){\mbox{$T(A)$}}}
\put(56,25){\vector(1,0){11}}
\put(56,10){\vector(1,1){12}}
\put(58,26){\shortstack{\mbox{$\scriptstyle  T({F})$}}}
\put(61,13){\shortstack{\mbox{$\scriptstyle  T(G)$}}}
\put(59,16){\vector(0,-1){0.1}}
\put(58.85,19){\line(0,-1){2.2}}
\put(59.15,19){\line(0,-1){2.2}}
\put(52,17){\shortstack{\mbox{$\scriptstyle  T(\phi)$}}}
\put(77,25){\vector(1,0){11}}
\put(90,25){\makebox(0,0){\mbox{$A$}}}
\put(80,26){\shortstack{\mbox{$\scriptstyle \xi_A$}}}
\end{picture}}

\noindent exhibits $\xi_A\cdot T(G)$ as the pointwise left Kan extension of $\xi_A\cdot T(F)$ in $\Cat(\CC).$\end{defin}

\begin{theorem}[{ \cite{W2}}]\label{polkan}
Let $A$ be a categorical $T$-algebra which is cocomplete with respect to $\HPhi:\HS\to\HT$ and let $\,Y$ be an internal $S$-algebra in $A$. Then the pointwise left Kan extension of $\,\widetilde{Y}$ along $\HPhi$ in $\Alg_T(\Cat(\CC))$  exists and its underlying functor is the pointwise left Kan extension of $\,U_T(\widetilde{Y})$ along $U_T(\HPhi)$.

In particular,$$U_T((\HPhi)_!(Y)) = \colim_{U_T(\HS)}U_T(\widetilde{Y})\quad\mathrm{in}\quad\Cat(\CC).$$\end{theorem}

\noindent A useful generalisation is the following relative version.  Let

  {\unitlength=1mm

\begin{picture}(40,17)(0,13)

\put(51,25){\makebox(0,0){\mbox{$R$}}}
\put(53,23){\vector(1,-1){5}}
\put(60.5,16){\makebox(0,0){\mbox{$T$}}}
\put(68,23){\vector(-1,-1){5}}
\put(70,25){\makebox(0,0){\mbox{$S$}}}
\put(54,25){\vector(1,0){14}}
\put(60,26){\shortstack{\mbox{$\scriptstyle  \Phi$}}}

\end{picture}}

\noindent be a commutative triangle of cartesian monad morphisms (cf. Definition \ref{Phic}). As above it induces a morphism of categorical $T$-algebras  $\HPhi: {\bf T}^{\tt R}\to {\bf T}^{\tt S}.$

\begin{theorem}[{\cite{W2}}]\label{RTS}
Let $A$ be a  categorical $T$-algebra which is cocomplete with respect to $\HPhi: {\bf T}^{\tt R}\to {\bf T}^{\tt S}$ and let $\,Y$ be an internal $R$-algebra in $A$. Then the pointwise left Kan extension of $\,\tilde{Y}$ along $\HPhi$ in $\,\Alg_T(\Cat(\CC))$  exists and its underlying functor is the pointwise left Kan extension of $U_T(\widetilde{Y})$ along $U_T(\HPhi)$.

In other words, the following diagram of adjoint functors commutes:

\begin{diagram}\Int_R(A)&\pile{\lTo^{ (\HPhi)^*}\\\rTo_{ (\HPhi)_!}}&\Int_S(A)\\\dTo^{U_T}&&\dTo_{U_T}\\[U_T({\bf T}^{\tt R}),U_T(A)]&\pile{\lTo^{U_T(\HPhi)^*}\\\rTo_{U_T(\HPhi)_!}}&[U_T({\bf T}^{\tt S}),U_T(A)]\end{diagram}\end{theorem}




\begin{remark}\label{units}Theorems \ref{polkan} and \ref{RTS} are valid more generally in a $2$-monadic setting, and a conceptual proof is based on an internalisation of Guitart's theory of exact squares \cite{Guitart}. A detailed exposition would lead us too far from our main purpose and we refer the reader to Weber \cite{W2}. Theorem 2.4.4 and Theorem 5.7.2 in \emph{loc. cit.} yield a comparatively short proof in the special case of cartesian monads. The two main ideas entering in the proof may be summarised as follows:

(a) If in the situation of Definition \ref{cocomplete} the additional assumption is made that the commuting square of internal functors\begin{diagram}[small]T(B)&\rTo^{T(\xi)}&T(C)\\\dTo^{\xi_B}&&\dTo_{\xi_C}\\B&\rTo_\xi&C\end{diagram}
is \emph{exact} (cf. \cite[Definition 2.4.1]{W2}) then the pair of pointwise left Kan extensions $$(T(G),G):(T(C),C)\to(T(A),A)$$ is a pseudomorphism of $T$-algebras and the forgetful functor $\Alg_T(\Cat(\CC))\to\Cat(\CC)$ preserves the pointwise left Kan extension of $F:B\to A$ along $\xi:B\to C$.

(b) Under the hypotheses of Theorem \ref{polkan} (and similarily for Theorem \ref{RTS}) the commuting square of internal functors\begin{diagram}[small]T(\HS)&\rTo^{T(\HPhi)}&T(\HT)\\\dTo^{\xi_{\HS}}&&\dTo_{\xi_{\HT}}\\\HS&\rTo_\HPhi&\HT\end{diagram}is indeed exact in $\Cat(\CC)$ by \cite[Proposition 4.3.4]{W2} since it is a \emph{pullback square} and the vertical functors are \emph{discrete fibrations}.

It is remarkable that the cartesianness of the unit of the monads $S$ and $T$ is not needed, neither for the existence of the internal algebra classifiers nor for the validity of Theorems \ref{polkan} and \ref{RTS}. We are however not aware of a single example of a pullback preserving monad, whose multiplication is cartesian and whose unit is not. So, we preferred to keep working with cartesian monads.\end{remark}

\section{Polynomial and tame polynomial monads}\label{polynomialmonad}

In this section we recall the definition of a polynomial monad in sets and of its associated coloured symmetric operad. We also introduce the new concept of a tame polynomial monad which will be crucial for us. For a nice and instructive account of polynomial functors we recommend Kock's article \cite{Kock} from which we shall borrow the idea of representing coloured bouquets as certain special polynomials. For a general treatment of polynomial monads in locally cartesian closed categories the reader may consult Gambino-Kock \cite{GK}. An extension to categories with pullbacks has been studied by Weber \cite{WPol}. Earlier appearances of polynomial functors can be found in the articles of Tambara \cite{T} and of Moerdijk-Palmgren \cite{MP}.

\subsection{Polynomial functors}For any set $I$ we denote $Set/I$ the comma category over $I.$ Objects of $\Set/I$ are mappings $\pi:X\to I,$ and morphisms of $\Set/I$ are commuting triangles over $I$. For each $i\in I$, the preimage $\pi^{-1}(i)$ will be called the \emph{fibre} of $\pi$ over $i.$ The mapping $\pi$ is completely determined by its fibres, and hence the category $\Set/I$ may be identified with the category of $I$-indexed families of sets $(X_i)_{i\in I}.$ This will be our favourite notation for the objects of $\Set/I.$

\begin{defin}A polynomial $P=(s,p,t)$ is a diagram in sets of the form \begin{diagram}J&\lTo^s&E&\rTo^p&B&\rTo^t&I\end{diagram}
A polynomial is of finite type if all fibres of the middle arrow $p$ are finite.\end{defin}

\noindent Each polynomial $P$ generates a functor between overcategories
$$\underline{P}:\Set/J \to \Set/I$$
which is defined as the composite functor
\begin{diagram}\Set/J&\rTo^{s^*}&\Set/E&\rTo^{p_*}&\Set/B&\rTo^{t_!}&\Set/I\end{diagram}
where $s^*$ is the pullback functor,
$$s^*(X)_e = X_{s(e)},$$
 $p_*$ is right adjoint to $p^*$,
 $$p_*(X)_b = \prod_{e\in p^{-1}(b)} X_e,$$
 and $t_!$ is left adjoint to $t^*$,
 $$t_!(X)_i = \coprod_{b\in t^{-1}(i)} X_b.$$
So, the functor $\underline{P}$ is given by the formula
\begin{equation}\label{PPP}  \underline{P}(X)_i = \coprod_{b\in t^{-1}(i)} \prod_{e\in p^{-1}(b)} X_{s(e)}, \end{equation}
which explains the name `polynomial' that is  a sum of products of formal variables.

Any functor $\underline{P}$ generated by a polynomial $P$ is called a {\it polynomial functor}. In particular, polynomial functors preserve \emph{connected limits}. This property \emph{characterises} polynomial functors from $\Set/J$ to $\Set/I .$ For this and other characterisations of polynomial functors we refer the reader to \cite{Kock,GK}. In particular, polynomial functors compose. The composite functor $\underline{P}\circ\underline{Q}$ is the polynomial functor $\underline{PQ}$ generated by an up to unique isomorphism uniquely determined polynomial $PQ$. Cartesian natural transformations of polynomial functors correspond bijectively to commutative  diagrams of the form\begin{diagram}J&\lTo^{s'}&E'\SEpbk&\rTo^{p'}&B'&\rTo^{t'}&I\\\dTo^{1_J}&&\dTo&&\dTo&&\dTo_{1_I}\\J&\lTo^s&E&\rTo^p&B&\rTo^t&I\end{diagram}
in which the horizontal lines are polynomials and the middle square is a pullback square. This defines a $2$-category  $\Poly$ with $0$-cells the overcategories $\Set/I$, with $1$-cells the polynomial functors, and with $2$-cells the cartesian natural transformations. We denote $\Poly(I)$ the category of polynomial endofunctors $\Set/I\to\Set/I$ and cartesian natural transformations. It is a monoidal category for composition of endofunctors.

\begin{defin}A polynomial monad is a monad in the $2$-category $\bf Poly.$\end{defin}
Hence a polynomial monad $T$ over $I$ is a monoid in $(\Poly(I),\circ)$. Each polynomial monad over $I$ is a cartesian monad on $\Set/I.$ The polynomial monad is finitary if and only if the generating polynomial is of finite type.

From now on we always assume that our polynomial monads are finitary.

\begin{remark}\label{operation}Let $T$ be the polynomial monad generated by the polynomial\begin{diagram}I&\lTo^s&E&\rTo^p&B&\rTo^t&I\end{diagram}

We shall outline how to define an \emph{$I$-coloured symmetric operad} $\Oo_T$ whose associated monad is $T$. Each element $b\in B$ comes equipped with a target $t(b)=i\in I,$ and a fibre $p^{-1}(b)\subset E.$ The elements $e\in p^{-1}(b)$ of the fibre have sources $s(e)\in I.$

A {\it $T$-operation} is a pair $(b,\sigma)$ consisting of an element $b\in B$ and a bijection $\sg:\{1,2,\dots,k\}\to p^{-1}(b)$. We shall refer to $\sg$ as a \emph{linear ordering of the fibre}. The $I$-coloured symmetric operad $\Oo_T$ associated to $T$ consists precisely of all $T$-operations. Such a $T$-operation $(b,\sg)$ belongs to $\Oo_T(i_1,\dots,i_k;i)$ if and only if $$t(b)=i\text{ and }(s(\sg(1)),s(\sg(2)),\ldots,s(\sg(k))=(i_1,\dots,i_k).$$ We shall write $s(b,\sg)=(i_1,\dots,i_k)$ and $t(b,\sg)=i$ and call them source and target of the $T$-operation $(b,\sg)$.

A {\it composite $T$-operation} consists of a list  $((b,\sg);(b_1,\sg_1),(b_2,\sg_2),\ldots,(b_k,\sg_k))$ such that $(t(b_1,\sg_1),\dots,t(b_k,\sg_k)) = s(b,\sg)$. The multiplication of the operad $\Oo_T$ is induced by the multiplication of the polynomial monad $T$ (cf. Remark \ref{bouquet} below) and associates to a composite $T$-operation a single $T$-operation $(b,\sg)((b_1,\sg_1),\dots,(b_k,\sg_k))$ with same target as $(b,\sg)$ and with source-list $(s(b_1,\sg_1),s(b_2,\sg_2),\ldots,s(b_k,\sg_k))$ linearly ordered in the obvious way. This multiplication satisfies the usual associativity, unitarity and equivariance constraints of an $I$-coloured symmetric operad.

We shall denote the composed $T$-operation by$$(b,\sg)((b_1,\sg_1),\dots,(b_k,\sg_k))=(b^\sg(b_1,\dots,b_k),\sg(\sg_1,\dots,\sg_k))$$ to remind that the linear ordering $\sg$ of the fibre of $b$ determines the way in which the individual $T$-operations $(b_i,\sg_i)$ are ``inserted'' into the leading $T$-operation $(b,\sg)$.\end{remark}

\begin{remark}\label{bouquet}
The process just described constitutes a functor $\Oo$ from the category of finitary polynomial monads over $I$ to the category of symmetric $I$-coloured operads in sets with freely acting symmetry groups. Kock \cite{Kock} and Szawiel-Zawadowski \cite{Zav} showed that this functor is an \emph{equivalence of categories}. For a deeper study of the relationship between polynomial monads and operads, the reader may consult the recent article \cite{WOper} of Weber.

Let us briefly recall the argument given by Kock in \cite{Kock}.

\begin{defin}[{\cite{Kock}}]\label{bouquetdef}An $I$-coloured bouquet of arity $k$ is a polynomial\begin{diagram}I&\lTo^s&\{1,2,\dots,k\}&\rTo^p&\{1\}&\rTo^t&I.\end{diagram}The latter will be represented by the $(k+1)$-tuple $(s(1),\dots,s(k);t(1))\in I^{k+1}$.

The full subcategory  of $\Poly(I)$ spanned by $I$-coloured bouquets will be denoted $\Bq(I)$. The associated nerve functor is denoted

$$\begin{array}{rccl}\Oo:\Poly(I)&\to &\Coll(I)&=\Set^{\Bq(I)^\op}\\P&\mapsto&\Oo_P&=\Hom_{\Poly(I)}(-,P)\end{array}$$\end{defin}

\noindent The subcategory $\Bq(I)$ is dense in $\Poly(I)$, i.e. the nerve functor is fully faithful. Moreover, $\Bq(I)$ is a groupoid: the symmetry group of a bouquet of arity $k$ may be identified with a certain subgroup of the symmetry group of $\{1,\dots,k\}$. The essential image of the nerve functor consists of those $I$-coloured collections in $\Coll(I)$ for which the automorphisms in $\Bq(I)$ act freely, cf. \cite[Theorem 2.4.10]{Kock}.

There is a substitutional $\circ$-product on $\Coll(I)$ for which the monoids are precisely the \emph{$I$-coloured symmetric operads} in sets, cf. the appendix of \cite{BergerMoerdijk1}, where the category of $I$-coloured bouquets $\Bq(I)$ is denoted $\mathbb{F}^\leq(I)$. It can be checked by hand that the nerve functor is a \emph{monoidal} functor $$\Oo:(\Poly(I),\circ)\to(\Coll(I),\circ)$$and therefore takes polynomial monads over $I$ to $I$-coloured symmetric operads. It follows from \cite[Theorem 2.2.12]{Kock} that this ``enhanced'' nerve functor induces an equivalence between the category of finitary polynomial monads over $I$ and the category of $I$-coloured symmetric operads with freely acting symmetry groups, cf. also \cite{Zav}.
\end{remark}

\subsection{Algebras over polynomial monads}The category of $T$-algebras $\Alg_T$ for a polynomial monad $T$ on $\Set/I$ coincides with the category of $\Oo_T$-algebras of the associated coloured symmetric operad $\Oo_T$. Explicitly a \emph{$T$-algebra} in sets is given by an $I$-indexed family of sets $(A_i)_{i\in I}$ together with structural maps$$m_{(b,\sg)}: A_{s(\sg(1))}\times\cdots\times A_{s(\sg(k))}\rightarrow A_{t(b)}$$for each operation $(b,\sg)$ of $T$. These structure maps satisfy the usual associativity, unitarity and equivariance conditions of an algebra for a coloured symmetric operad.

Given a cocomplete symmetric monoidal category  $(\Ee,\otimes,e)$, the strong symmetric monoidal functor $$\Set\to\Ee:X\mapsto\coprod_Xe$$ takes the coloured symmetric operad $\Oo_T$ to a coloured symmetric operad in $\Ee$ and thus defines a category $\Alg_T(\Ee)$ of $T$-algebras in $\Ee$. Explicitly, a $T$-algebra $A$ in $\Ee$ is an $I$-indexed family $(A_i)_{i\in I}$ of objects of $\Ee$ together with structural maps$$m_{(b,\sg)}: A_{s(\sg(1))}\otimes \cdots \otimes A_{s(\sg(k))} \rightarrow A_{t(b)}$$for each operation $(b,\sg)$ of $T$, subject to the same  associativity, unitarity and equivariance conditions as above.

\subsection{Internal $T$-algebras in cocomplete symmetric monoidal categories.}\label{cocompletesymm}

Since each polynomial monad $T$ is cartesian it generates a $2$-monad on the category $\Cat/{ I}$ where  $I$  is considered as a discrete category. The category of strict algebras for this $2$-monad is by definition the category $\Alg_T(\Cat)$ of categorical $T$-algebras. There is also a $2$-category of \emph{pseudo-$T$-algebras} associated to the $2$-monad $T$. A general strictification theorem of Power implies that any pseudo-$T$-algebra is equivalent to a strict $T$-algebra (cf. \cite{Power,Lack2,EHBat}). We tacitly apply this strictification whenever necessary.

A categorical $T$-algebra $(A_i)_{i\in I}$ is cocomplete with respect to morphisms between small categorical $T$-algebras in the sense of Definition \ref{cocomplete} if and only if each $A_i$ is a cocomplete category and the structure maps $m_{(b,\sg)}: A_{s(\sg(1))}\times\cdots\times A_{s(\sg(k))} \rightarrow A_{t(b)}$ preserve colimits in each variable.

Let $A$ be a categorical $T$-algebra. Then an internal $T$-algebra in $A$ can be explicitly given by a collection of objects $a_i\in A_i$ together with a morphism$$\mu_{(b,\sg)}:m_{(b,\sg)}(a_{s(\sigma(1))},\ldots, a_{s(\sigma(k))})\rightarrow a_{t(b)},$$for each operation $(b,\sigma),$ which satisfies obvious associativity, unitarity and equivariancy conditions. Here, $m_{(b,\sigma)}$ is the structure functor of $A.$

To any symmetric monoidal category $(\Ee,\otimes,e)$ we associate the categorical pseudo-$T$-algebra  $\Ee_T^{\bullet}$ with constant underlying collection $$\Ee_T^{\bullet}(i) = \Ee,\quad i\in { I}.$$Nullary $T$-operations act as unit $e: 1\rightarrow \Ee,$ unary $T$-operations act as identity and $T$-operations of arity $n\geq 2$ acts as iterated tensor product $\otimes^n.$ This pseudo-$T$-algebra $\Ee_T^{\bullet}$ is cocomplete with respect to morphisms between small pseudo-$T$-algebras if and only if $\Ee$ is cocomplete as a category, and moreover the tensor of $\Ee$ commutes with colimits in both variables, cf. \cite[Proposition 2.3.3]{W2}. Recall that the latter holds in particular if $\Ee$ is closed symmetric monoidal.

The assignment $\Ee\mapsto\Ee_T^{\bullet}$ is the right adjoint part of an adjunction between categorical $T$-algebras and symmetric monoidal categories. This adjunction is induced by a map $T\to P$ of polynomial monads in $\Cat$  where $P$ denotes the monad for symmetric monoidal categories. The existence of this adjunction provides a conceptual reason for the existence of an enrichment over $\Ee$ of the category of $T$-algebras. The interested reader may find more details in \cite{W,W2}. We will not pursue this point of view any further here. However, it will be essential for us to represent $T$-algebras in $\Ee$ as internal $T$-algebras in $\Ee_T^\bullet$, based on the following proposition.

\begin{pro}\label{internalinmonoidal}The category of $\,T$-algebras in $\Ee$ is isomorphic to the category of internal $T$-algebras in $\Ee_T^{\bullet}$.\end{pro}

\begin{proof}An internal $T$-algebra in $\Ee_T^{\bullet}$ is by definition a lax morphism of categorical $T$-algebras $1\to \Ee_T^{\bullet}$. Such a lax morphism is given by an $I$-collection $(X_i)_{i\in I}$ of objects of $\,\Ee$ together with coherence data fulfilling coherence conditions. The coherence data consists of a $2$-cell from $T(1)\to T(\Ee_T^{\bullet})\to\Ee_T^\bullet$ to $T(1)\to 1\to\Ee_T^\bullet$ which amounts precisely to $\Sigma_k$-equivariant structure maps in $\,\Ee$ $$X_{i_1}\otimes\cdots\otimes X_{i_k}\to X_i$$one for each orbit of $T$-operations $b\in T(1)$ of type $(i_1,\dots,i_k;i)$, cf. Remark \ref{operation} and Section \ref{graphical}. The unit and associativity constraints of this $2$-cell (cf. Section \ref{strictlax}) translate into the familiar unit and associativity constraints of an operad action.\end{proof}

\subsection{Representing $T$-algebras as internal functors $\HT\to\Ee_T^\bullet$}--\vspace{1ex}

According to Proposition \ref{internalinmonoidal} and Theorem \ref{bar1}, $T$-algebras in $\Ee$ can be represented as internal functors $\HT\to\Ee_T^\bullet$. By Theorem \ref{bar1}, the objects of the classifier $\HT$ are the elements of the free $T$-algebra $T(1)\cong\Oo_T(1)$. These elements correspond bijectively to orbits of $T$-operations (cf. Remark \ref{operation}), i.e. to elements of $B$.

A morphism from $b'$ to $b$ is given by an element of $T^2(1)$ with correct source and target. Since $T^2(1)\cong (\Oo_T\circ\Oo_T)(1)$, such an element corresponds to an orbit of composite $T$-operations $((b,\sigma);(b_1,\sg_1),\dots,(b_k,\sg_k))$ fulfilling $b'=b^\sg(b_1,\dots,b_k)$ in the notation of Remark \ref{operation}. It is important to observe that each such orbit can be represented by a $(k+1)$-tuple $(b;b_1,\dots,b_k)\in B^{k+1}$ together with a colour-preserving bijection between the source-list of $b$ and the list of targets of the individual $b_i$. This representation of the elements of $T^2(1)$ is unique up to a reordering of the $b_i$ going along with the appropriate change of bijection.

The unit of the polynomial monad $T$ defines an $I$-indexed collection $(1_i)_{i\in I}$ of special elements $1_i\in B$. The latter have singleton fibre with source and target $i\in I$. They induce two families of morphisms in $\HT$, namely the identities\begin{diagram}b&\rTo^{(b;1_{i_1},\dots,1_{i_k})}&b\end{diagram}and the  morphisms\begin{diagram}b&\rTo^{(1_i;b)}&1_i.\end{diagram}

Now, let $(X_i)_{i\in I}$ be the $I$-collection underlying the $T$-algebra $X$ in $\Ee$. Then the representing functor of categorical $T$-algebras
$$\tilde{X}:\HT\to\Ee_T^\bullet$$
is constructed as follows. We have $\tilde{X}(1_i)=X_i$ and $\tilde{X}(b)=X_{i_1}\otimes\cdots\otimes X_{i_k}$ where $(i_1,\dots,i_k)$ is the source-list of the fibre $p^{-1}(b)$  for a fixed linear ordering $\sg$. Then the map $\tilde{X}(b\to 1_i)$ in $\Ee^\bullet_T(i)$ represents the $T$-action $m_{(b,\sg)}:X_{i_1}\otimes\cdots\otimes X_{i_k}\to X_i$ on $X$, and the functoriality of $\tilde{X}:\HT\to\Ee_T^\bullet$ amounts precisely to the equivariance, associativity and unitarity constraints of this $T$-action.

\begin{remark}\label{convention}The $T$-algebra structure on the internal algebra classifier $\HT$ splits the latter into components ${\bf T}^{\tt T}_i$ spanned by those objects $b\in B$ with target $t(b)=i$. The restrictions of the representing functor $\tilde{X}:\HT\to\Ee_T^\bullet$ to these components will be denoted $\tilde{X}:{\bf T}^{\tt T}_i\to\Ee$. This amounts to removing the forgetful functor $U_T:\Alg_T(\Cat)\to \Cat$ from our notation.

In other words, we will identify the categorical $T$-algebra $\HT$ with its underlying $I$-collection of categories $({\bf T}^{\tt T}_i)_{i\in I}$, and will make no notational distinction between the representing functor $\tilde{X}:\HT\to\Ee^{\bf \bullet}_T$ and its components $\tilde{X}:{\bf T}^{\tt T}_i\to\Ee$. We hope this will cause no confusion.

The reader should observe that the object $1_i$ is terminal in its component ${\bf T}^{\tt T}_i$, so that the component $X_i$ of the $I$-collection $(X_i)_{i\in I}$ can be recovered as$$X_i=\tilde{X}(1_i)=\colim_{b\in{\bf T}^{\tt T}_i}\tilde{X}(b).$$\end{remark}

\begin{example}At this point, we should mention two examples which have been decisive for the elaboration of the whole theory. If $T$ is the free monoid monad (see Section \ref{SSM}) then $\HT$ is isomorphic to the augmented simplex category $\Delta_+$, cf. B\'enabou \cite{Benabou}, and we recover the fact, mentioned in Example \ref{Benabou}, that monoids in $\Ee$ correspond to strict monoidal functors $\Delta_+\to\Ee.$ If $T$ is the free symmetric operad monad (see Section \ref{symmetricoperads}) then $\HT$ is isomorphic to a category of labelled rooted planar trees ${\bf RTr}^{\tt RTr}$ which goes back to Ginzburg-Kapranov \cite[Section 1.2]{GinKap}, and which again has the characteristic property that symmetric operads in $\Ee$ correspond to certain functors ${\bf RTr}^{\tt RTr}\to\Ee$.
\end{example}

\subsection{Graphical formalism}\label{graphical}It is convenient to use some sort of \emph{graphical formalism} (cf. \cite{KJBM,MP}) to visualise objects and morphisms of $\HT.$  An object of $\HT$ (i.e. an orbit of $T$-operations) will be represented as a \emph{non-planar corolla}

{\unitlength=1mm

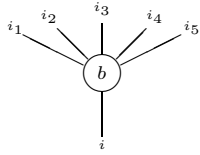
\begin{picture}(20,20)(-40,13)

\put(14,22.5){\circle{5}}
\put(2.5,28.5){\makebox(0,0){\mbox{$\scriptscriptstyle i_1$}}}
\put(26,28.5){\makebox(0,0){\mbox{$\scriptscriptstyle i_5$}}}
\put(7,30){\makebox(0,0){\mbox{$\scriptscriptstyle i_2$}}}
\put(21,30){\makebox(0,0){\mbox{$\scriptscriptstyle i_4$}}}
\put(14,31){\makebox(0,0){\mbox{$\scriptscriptstyle i_3$}}}
\put(16.5,23.5){\line(2,1){8}}
\put(11.5,23.5){\line(-2,1){8}}
\put(12.1,24.4){\line(-1,1){4}}
\put(14,25){\line(0,1){4}}
\put(15.9,24.4){\line(1,1){4}}
\put(14,20){\line(0,-1){6}}
\put(14,13){\makebox(0,0){\mbox{$\scriptscriptstyle i $}}}
\put(14,22.5){\makebox(0,0){\mbox{$\scriptstyle b $}}}

\end{picture}}

\noindent where $b\in B$, the incoming edges are decorated by the source-list of the fibre $s(p^{-1}(b))=(i_1,i_2,\dots,i_5)$ and the outgoing edge is decorated by the target $t(b)=i$. This orbit of $T$-operations $b$ is said to be of type $(i_1,\dots,i_5;i)$, cf. Definition \ref{bouquetdef}.

A morphism of $\HT$ (i.e. an orbit of composite $T$-operations) will be represented by a \emph{non-planar  bicorolla} (i.e. corolla of corollas) as for instance

{\unitlength=1mm

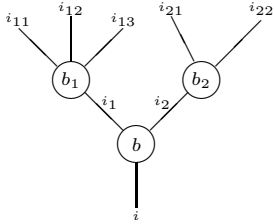
\begin{picture}(60,30)(-30,-4)

\begin{picture}(10,10)(2.5,8.5)

\put(14,22.5){\circle{5}}
\put(14,25){\makebox(0,0){\mbox{$ $}}}
\put(15.9,20.8){\line(1,-1){5}}
\put(12.1,24.4){\line(-1,1){5}}
\put(14,25){\line(0,1){6}}
\put(15.9,24.4){\line(1,1){5}}
\put(14,32){\makebox(0,0){\mbox{$\scriptscriptstyle  i_{12}$}}}
\put(21,30.5){\makebox(0,0){\mbox{$\scriptscriptstyle  i_{13}$}}}
\put(7,30.5){\makebox(0,0){\mbox{$\scriptscriptstyle  i_{11}$}}}
\put(14,22.5){\makebox(0,0){\mbox{$\scriptstyle b_1$}}}

\end{picture}

\begin{picture}(10,10)(5,17)

\put(14,22.5){\circle{5}}
\put(14,25){\makebox(0,0){\mbox{$ $}}}
\put(14,20){\line(0,-1){6}}
\put(15.9,24.4){\line(1,1){5}}
\put(14,22.5){\makebox(0,0){\mbox{$\scriptstyle b $}}}
\put(14,13){\makebox(0,0){\mbox{$\scriptscriptstyle i $}}}
\put(17.5,28){\makebox(0,0){\mbox{$\scriptscriptstyle i_2 $}}}
\put(10.5,28){\makebox(0,0){\mbox{$\scriptscriptstyle i_1 $}}}

\end{picture}

\begin{picture}(10,10)(7.5,8.5)

\put(14,22.5){\circle{5}}
\put(14,25){\makebox(0,0){\mbox{$ $}}}
\put(13,25){\line(-1,2){3}}
\put(15.9,24.4){\line(1,1){6}}
\put(22,32){\makebox(0,0){\mbox{$\scriptscriptstyle  i_{22}$}}}
\put(10,32){\makebox(0,0){\mbox{$\scriptscriptstyle  i_{21}$}}}
\put(14,22.5){\makebox(0,0){\mbox{$\scriptstyle b_2 $}}}

\end{picture}

\end{picture}}

\noindent where $b$ is of type $(i_1,i_2;i)$ and $b_1,b_2$ are respectively of type $(i_{11},i_{12},i_{13};i_1)$ and $(i_{21},i_{22};i_2)$. This bicorolla represents a morphism $b(b_1,b_2)\to b$ in $\HT$ with source $b(b_1,b_2)$ obtained by ``contracting'' inner edges, i.e. by inserting $b_1$ and $b_2$ into $b$ according to the multiplication of the polynomial monad $T$ and the source/target matching displayed in the bicorolla. The notation $b(b_1,b_2)\to b$ for morphisms in $\HT$ is slightly abusive, insofar as it does not explicitly mention the bijection between the list of targets of the $b_i$ and the source-list of $b$. The graphical representation incorporates this bijection and we tacitly assume that such a bijection is given.

\subsection{Cartesian morphisms of polynomial monads.}\label{cartesian} We need a more general notion of map between polynomials which includes the possibility of base-change. Let $S$ be a polynomial monad generated by  a polynomial
\begin{diagram}\label{polynomS} J&\lTo^{s'}&E'&\rTo^{p'}&B'&\rTo^{t'}&J.\end{diagram}
A cartesian morphism $\Phi=(\delta,\psi,\phi)$ from $S$ to $T$ is a commutative diagram
\begin{diagram}J&\lTo^{s'}&E'\SEpbk&\rTo^{p'}&B'&\rTo^{t'}&J\\\dTo^\delta&&\dTo^\psi&&\dTo^\phi&&\dTo^\delta\\I&\lTo^s&E&\rTo^p&B&\rTo^t&I\end{diagram}
in sets in which the middle square is a pullback, the horizontal lines generate the polynomial monads $S$ and $T$ and
 the diagram

  {\unitlength=1mm

\begin{picture}(40,30)(-19,0)

\put(48,15){\shortstack{\mbox{$\scriptstyle \phi$}}}
\put(37,8){\shortstack{\mbox{$\scriptstyle p$}}}
\put(20,15){\shortstack{\mbox{$\scriptstyle \psi$}}}
\put(6,16){\shortstack{\mbox{$\scriptstyle \delta\circ s'$}}}
\put(63,16){\shortstack{\mbox{$\scriptstyle \delta\circ t'$}}}
\put(23,25){\makebox(0,0){\mbox{${E'}$}}}
\put(20,23){\vector(-1,-1){14}}
\put(54,23){\vector(1,-1){14}}
\put(3,7){\makebox(0,0){\mbox{${I}$}}}
\put(70,7){\makebox(0,0){\mbox{${I}$}}}
\put(20,7){\vector(-1,0){14}}
\put(54,7){\vector(1,0){14}}
\put(23,21){\vector(0,-1){10}}
\put(26,25){\vector(1,0){22}}
\put(37,26){\shortstack{\mbox{$\scriptstyle p'$}}}
\put(51,25){\makebox(0,0){\mbox{$B'$}}}
\put(51,21){\vector(0,-1){10}}
\put(23,7){\makebox(0,0){\mbox{$E$}}}
\put(26,7){\vector(1,0){22}}
\put(51,7){\makebox(0,0){\mbox{$B$}}}

\end{picture}}

\noindent  generates a morphism of polynomial monads. The mapping $\delta: J\to I$ induces a cartesian adjunction $c\dashv d$ where $d = \delta^*:\Set/I \to\Set /J$ is the pullback functor and $c = \delta_!:\Set/J\to\Set/I$ is its left adjoint, so that for an object $Y$ of $\Set/J :$
\begin{equation}\label{ccc} c(Y)_i = \coprod_{j\in \delta^{-1}(i)} Y_j.\end{equation}

Then the equivalent conditions of Lemma \ref{polsq} are fulfilled and $\Phi$ generates a cartesian monad morphism $\Phi:S\Rightarrow dTc$ in the sense of Definition \ref{Phic}.

For a symmetric monoidal category $\Ee$ we get a restriction functor$$\delta_\Ee^{\Phi}: Alg_T(\Ee)\rightarrow Alg_S(\Ee).$$
Observe that $d'(\Ee_T^{\bullet}) = \Ee_S^{\bullet}$ so that an internal $S$-algebra in $\Ee_T^{\bullet}$ is the same as an ordinary $S$-algebra in $\Ee.$ Therefore, the restriction functor $\delta_\Ee^{\Phi}$ is induced by a functor of internal algebra classifiers $\HPhi:\HS\to\HT$, and its left adjoint $$\gamma_\Ee^\Phi:\Alg_S(\Ee)\to\Alg_T(\Ee)$$can be calculated as a left Kan extension along the same functor. To carry out such a program we need a description of the internal $S$-algebra classifier $\HS$ and of the canonical functor $\HPhi:\HS\to\HT$ in terms of the given map $\Phi$ between the generating polynomials.

\subsection{The category $\HS$}\label{TS}\vspace{1ex}

By Theorem \ref{TST}, the objects of $\HS$ are the elements of $Tc(1)$. According to (\ref{PPP}) and (\ref{ccc}), these elements can be understood as $J$-coloured $T$-operations. In order to distinguish them from the objects of $\HT$ we denote them by bold letters. A $J$-coloured $T$-operation $\bc$ is given by an element  $b\in B$ together with a colour $j\in \delta^{-1}(s(e))$ for each $e\in p^{-1}(b)$. The internal functor $\HPhi:\HS\rightarrow\HT$ replaces the $J$-colouring with an $I$-colouring by applying $\delta$. So, an object $\bc$ of $\HS$ is determined by its image $\HPhi(\bc)$ together with a compatible $J$-colouring.

In terms of our graphical formalism, such an object is represented by a non-planar corolla

 {\unitlength=1mm
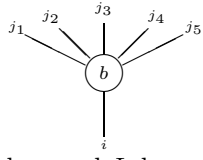
  \begin{picture}(20,20)(-40,13)
\put(14,22.5){\circle{5}}

\put(2.5,28.5){\makebox(0,0){\mbox{$\scriptscriptstyle j_1$}}}
\put(26,28.5){\makebox(0,0){\mbox{$\scriptscriptstyle j_5$}}}

\put(7,30){\makebox(0,0){\mbox{$\scriptscriptstyle j_2$}}}
\put(21,30){\makebox(0,0){\mbox{$\scriptscriptstyle j_4$}}}
\put(14,31){\makebox(0,0){\mbox{$\scriptscriptstyle j_3$}}}
\put(16.5,23.5){\line(2,1){8}}
\put(11.5,23.5){\line(-2,1){8}}
\put(12.1,24.4){\line(-1,1){4}}
\put(14,25){\line(0,1){4}}
\put(15.9,24.4){\line(1,1){4}}
\put(14,20){\line(0,-1){6}}

\put(14,13){\makebox(0,0){\mbox{$\scriptscriptstyle i $}}}
\put(14,22.5){\makebox(0,0){\mbox{$\scriptstyle b $}}}

\end{picture}}

\noindent with $J$-decorated incoming edges and $I$-decorated outgoing edge. The component $\theta_1$ of the right $S$-action $\theta:TcS\Rightarrow Tc$ is defined as follows, where we freely use our slightly abusive notation for composite operations, cf. Section \ref{graphical}. Note first that the elements of $TcS(1)$ can be understood as $J$-coloured $T$-operations $\bc$ together with a compatible family of $S$-operations $d_1,\dots,d_k\in B'$ in the sense that $(t'(d_1),\dots,t'(d_k))$ coincides with the $J$-colouring of $\bc$. Then $$\HPhi(\theta(\bc; d_1,\ldots, d_k)) = \HPhi(\bc)( \phi(d_1),\ldots,\phi(d_k)),$$ and the $J$-colouring of $\theta(\bc; d_1,\ldots, d_k)$ is inherited from (the sources of) the fibres of $d_1,\dots,d_k$ in an obvious way.

A morphism $\bc'\rightarrow \bc$ in $\HS$  is given by an element $(\bc;d_1,\ldots,d_k)$ of $TcS(1)$ such that $\bc'=\theta(\bc; d_1,\ldots, d_k).$
The effect of $\HPhi$ on a  morphism $\bc'\rightarrow \bc$ is obvious: if $\bc' =  \theta(\bc; d_1,\ldots, d_k)$ then the identity
$\HPhi(\bc') = \HPhi(\bc)(\phi(d_1),\ldots,\phi(d_k))$ represents a morphism $\HPhi(\bc')\to\HPhi(\bc)$ in $\HT$.

\subsection{Representing $S$-algebras as internal functors $\HS\to\Ee_T^\bullet$}\label{colimsec}--\vspace{1ex}

Let $X$ be an $S$-algebra in $\Ee$. By the universal property of $\HS$, the corresponding internal $S$-algebra in $\Ee_T^\bullet$ is represented by a morphism of $T$-algebras $\tilde{X}:\HS\to \Ee_T^{\bullet}$, which can be can be described as follows: Let $\bc$ be an object of $\HS$. Then
\begin{equation}\label{tilde}\tilde{X}(\bc)= X_{j_1}\otimes\cdots\otimes X_{j_k}\end{equation}where $(j_1,\dots,j_k)$ is the $J$-colouring of $\bc$. The morphism $\tilde{X}(\bc'\to\bc)$ in $\Ee$ is defined by the action of the $S$-operations $d_1,\dots,d_k$ on $X,$ where $\bc'=\theta(\bc; d_1,\ldots, d_k).$

The $T$-algebra structure on the internal algebra classifier $\HS$ splits the latter into components ${\bf T}^{\tt S}_i$ spanned by those objects $\bc$ with target $t(\bc)=i$. The restrictions of the representing functor $\tilde{X}:\HS\to\Ee_T^\bullet$ to these components will be denoted $\tilde{X}:{\bf T}^{\tt S}_i\to\Ee$. This amounts to removing the forgetful functor $U_T:\Alg_T(\Cat)\to \Cat$ from our notation.

In other words, we will identify the categorical $T$-algebra $\HS$ with its underlying $I$-collection of categories $({\bf T}^{\tt S}_i)_{i\in I}$, and will make no notational distinction between the representing functor $\tilde{X}:\HS\to\Ee^{\bf \bullet}_T$ and its components $\tilde{X}:{\bf T}^{\tt S}_i\to\Ee$. We hope this will cause no confusion.

\begin{theorem}\label{colim}Let $\Ee$ be a cocomplete, closed symmetric monoidal category and let $\Phi:S\Rightarrow dTc$ be a cartesian monad morphism between polynomial monads. Then restriction $\delta^\Phi_\Ee:\Alg_T(\Ee)\to\Alg_S(\Ee)$ has a left adjoint $\gamma_\Ee^\Phi:\Alg_S(\Ee)\to\Alg_T(\Ee).$

For any $S$-algebra $X$ in $\Ee$, the underlying $I$-collection of the $T$-algebra $\gamma_\Ee^{\Phi}(X)$  can be calculated as the following colimit:
\begin{equation}\label{polcolimit}\gamma_\Ee^{\Phi}(X)_i=(\HPhi)_!(\tilde{X})(1_i)=\colim_{\bc\in{\bf T}^{\tt S}_i}\tilde{X}(\bc)\quad(i\in I).\end{equation}\end{theorem}

\begin{proof}The first identification follows from Theorem \ref{polkan}, cf. Remark \ref{convention} for our notations and Section \ref{cocompletesymm} for applicability. The second identification just expresses that left Kan extension along ${\bf T}^{\tt S}_i\to 1$ (which calculates the colimit on the right) can be achieved in two steps: first left Kan extension along $(\HPhi)_i:{\bf T}^{\tt S}_i\to{\bf T}^{\tt T}_i$ then left Kan extension along ${\bf T}^{\tt T}_i\to 1$ (which calculates evaluation at $1_i$).\end{proof}

\subsection{Tame polynomial monads}\label{tame}Let $T$ be a finitary monad on a cocomplete category $\CC$. We denote by $T+1$ the finitary monad on $\CC\times\CC$ given by \begin{align*}(T+1)(X,Y)&=(TX,Y)\\(T+1)(\phi,\psi)&=(T\phi,\psi)\end{align*} with evident multiplication and unit. We get the following square of adjunctions\begin{gather}\label{adjoint}\begin{diagram}\Alg_T\times\CC&\pile{\lTo^{(id_{\Alg_T}\times U_T)\Delta_{\Alg_T}}\\\rTo_{-\vee F_T(-)}}&\Alg_T\\\dTo^{U_T\times id_\CC}\uTo_{F_T\times id_\CC}&&\dTo^{U_T}\uTo_{F_T}\\\CC\times\CC&\pile{\lTo^{\Delta_\CC}\\\rTo_{-\sqcup-}}&\CC\end{diagram}\end{gather}in which the right adjoints commute by definition.

If $\CC$ has pullbacks which commute with coproducts, and $T$ is a cartesian monad, then it is straightforward to verify that $T+1$ is a cartesian monad as well, and that the adjoint square (\ref{adjoint}) induces a cartesian morphism from $T+1$ to $T$ in the sense of Definition \ref{Phic}.

If $T$ is a polynomial monad on $\Set/I$ generated by the polynomial\begin{diagram}I&\lTo^s&E&\rTo^p&B&\rTo^t&I\end{diagram}then $T+1$ is a polynomial monad on $\Set/I\times\Set/I=\Set/(I\sqcup I)$ generated by \begin{diagram}I\sqcup I&\lTo^{s\sqcup 1_I}&E\sqcup I&\rTo^{p\sqcup 1_I}&B\sqcup I&\rTo^{t\sqcup 1_I}&I\sqcup I\end{diagram}

More precisely, the adjoint square (\ref{adjoint}) for a polynomial monad $T$ on $\Set/I$ is induced by the following cartesian morphism  of polynomials (cf. \ref{cartesian})\begin{diagram}I\sqcup I&\lTo^{s\sqcup 1_I}&E\sqcup I\SEpbk&\rTo^{p\sqcup 1_I}&B\sqcup I&\rTo^{t\sqcup 1_I}&I\sqcup I\\\dTo^{\nabla_I}&&\dTo^\psi&&\dTo^\phi&&\dTo^{\nabla_I}\\I&\lTo^s&E&\rTo^p&B&\rTo^t&I\end{diagram}in which $\nabla_I$ is the identity on each copy of $I$, and $\phi$ (resp. $\psi$) is the identity on $B$ (resp. $E$) and the unit  $\eta$ of $T$ on $I$.

\begin{defin}A semi-free coproduct of $T$-algebras is a coproduct $X\vee F_T(K)$ of a $T$-algebra $X$ and a free $T$-algebra $F_T(K)$.

A polynomial monad $T$ is said to be \emph{tame} if the internal classifier for semi-free coproducts $\cop$ is a coproduct of categories with terminal object.\end{defin}

\subsection{The category $\cop$ explicitly}\label{semifreeclassifier}
The generating polynomial of $T+1$\begin{diagram}I\sqcup I&\lTo^{s\sqcup 1_I}&E\sqcup I&\rTo^{p\sqcup 1_I}&B\sqcup I&\rTo^{t\sqcup 1_I}&I\sqcup I\end{diagram} can be described in terms of our graphical formalism. Its set of operations $B\sqcup I$ consists of corollas of two types

 {\unitlength=1mm
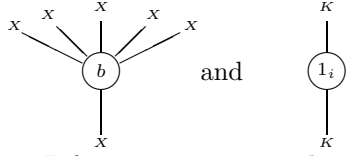
  \begin{picture}(20,20)(-25,13)
\put(14,22.5){\circle{5}}
\put(2.5,28.5){\makebox(0,0){\mbox{$\scriptscriptstyle X$}}}
\put(26,28.5){\makebox(0,0){\mbox{$\scriptscriptstyle X$}}}
\put(7,30){\makebox(0,0){\mbox{$\scriptscriptstyle X$}}}
\put(21,30){\makebox(0,0){\mbox{$\scriptscriptstyle X$}}}
\put(14,31){\makebox(0,0){\mbox{$\scriptscriptstyle X$}}}
\put(16.5,23.5){\line(2,1){8}}
\put(11.5,23.5){\line(-2,1){8}}
\put(12.1,24.4){\line(-1,1){4}}
\put(14,25){\line(0,1){4}}
\put(15.9,24.4){\line(1,1){4}}
\put(14,20){\line(0,-1){6}}
\put(14,13){\makebox(0,0){\mbox{$\scriptscriptstyle X $}}}
\put(14,22.5){\makebox(0,0){\mbox{$\scriptstyle b $}}}
\put(30,22.5){\makebox(0,0){\mbox{$\rm and$}}}
\put(44,22.5){\circle{5}}
\put(44,31){\makebox(0,0){\mbox{$\scriptscriptstyle K$}}}
\put(44,25){\line(0,1){4}}
\put(44,20){\line(0,-1){6}}
\put(44,13){\makebox(0,0){\mbox{$\scriptscriptstyle K $}}}
\put(44,22.5){\makebox(0,0){\mbox{$\scriptstyle 1_i $}}}

\end{picture}}

\noindent where $b\in B$, and $1_i\in B$ for $i\in I$ represents the unit of $B=T(1).$

According to Section \ref{TS} an object $\bc$ of $\cop$ is then represented by a corolla

 {\unitlength=1mm
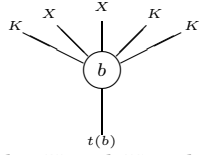
  \begin{picture}(20,20)(-40,13)
\put(14,22.5){\circle{5}}
\put(2.5,28.5){\makebox(0,0){\mbox{$\scriptscriptstyle K$}}}
\put(26,28.5){\makebox(0,0){\mbox{$\scriptscriptstyle K$}}}
\put(7,30){\makebox(0,0){\mbox{$\scriptscriptstyle X$}}}
\put(21,30){\makebox(0,0){\mbox{$\scriptscriptstyle K$}}}
\put(14,31){\makebox(0,0){\mbox{$\scriptscriptstyle X$}}}
\put(16.5,23.5){\line(2,1){8}}
\put(11.5,23.5){\line(-2,1){8}}
\put(12.1,24.4){\line(-1,1){4}}
\put(14,25){\line(0,1){4}}
\put(15.9,24.4){\line(1,1){4}}
\put(14,20){\line(0,-1){6}}
\put(14,13){\makebox(0,0){\mbox{$\scriptscriptstyle t(b) $}}}
\put(14,22.5){\makebox(0,0){\mbox{$\scriptstyle b $}}}
\end{picture}}

\noindent with incoming edges coloured by $X$ and $K$ and outgoing edge coloured by the target $t(b)$ of $b$.  The $X$-edges correspond to the operations on the $T$-algebra summand of the semi-free coproduct, while the $K$-edges correspond to the free summand.

A morphism $g:\bc'\to\bc$ in $\cop$ is given by a set of elements $b_1,\ldots,b_k\in B$, one for each $X$-coloured edge of $\bc,$ where the $1'$s correspond to $K$-edges and the $b_i'$s correspond to $X$-edges of $\bc.$   In the formalism of Section \ref{TS}  a typical morphism in $\cop$ is therefore a bicorolla

{\unitlength=1mm

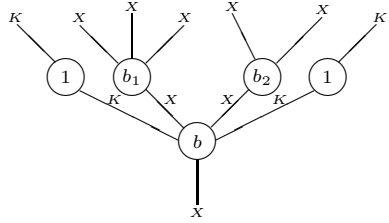
\begin{picture}(60,30)(-30,-4)

\begin{picture}(10,10)(2.5,8.5)
\put(14,22.5){\circle{5}}
\put(14,25){\makebox(0,0){\mbox{$ $}}}
\put(15.9,20.8){\line(1,-1){5}}
\put(12.1,24.4){\line(-1,1){5}}
\put(14,25){\line(0,1){6}}
\put(15.9,24.4){\line(1,1){5}}
\put(14,32){\makebox(0,0){\mbox{$\scriptscriptstyle  X$}}}
\put(21,30.5){\makebox(0,0){\mbox{$\scriptscriptstyle  X$}}}
\put(7,30.5){\makebox(0,0){\mbox{$\scriptscriptstyle  X$}}}
\put(14,22.5){\makebox(0,0){\mbox{$\scriptstyle b_1$}}}

\put(5.2,22.5){\circle{5}}
\put(7.1,20.7){\line(2,-1){13}}
\put(3.8,24.5){\line(-1,1){5}}
\put(40,22.5){\circle{5}}
\put(38.2,20.7){\line(-2,-1){13}}
\put(-1.5,30.5){\makebox(0,0){\mbox{$\scriptscriptstyle  K$}}}
\put(5.2,22.5){\makebox(0,0){\mbox{$\scriptstyle 1$}}}
\put(41.5,24.5){\line(1,1){5}}
\put(47,30.5){\makebox(0,0){\mbox{$\scriptscriptstyle  K$}}}
\put(40,22.5){\makebox(0,0){\mbox{$\scriptstyle 1$}}}
\end{picture}

\begin{picture}(10,10)(5,17)
\put(14,22.5){\circle{5}}

\put(10.5,28){\makebox(0,0){\mbox{$\scriptscriptstyle X$}}}
\put(18,28){\makebox(0,0){\mbox{$\scriptscriptstyle X$}}}
\put(3,28){\makebox(0,0){\mbox{$\scriptscriptstyle K$}}}
\put(25,28){\makebox(0,0){\mbox{$\scriptscriptstyle K$}}}
\put(14,20){\line(0,-1){6}}
\put(15.9,24.4){\line(1,1){5}}
\put(14,22.5){\makebox(0,0){\mbox{$\scriptstyle b $}}}
\put(14,13){\makebox(0,0){\mbox{$\scriptscriptstyle X $}}}
\end{picture}

\begin{picture}(10,10)(7.5,8.5)
\put(14,22.5){\circle{5}}
\put(14,25){\makebox(0,0){\mbox{$ $}}}
\put(13,25){\line(-1,2){3}}
\put(15.9,24.4){\line(1,1){6}}
\put(22,31.5){\makebox(0,0){\mbox{$\scriptscriptstyle  X$}}}
\put(10,32){\makebox(0,0){\mbox{$\scriptscriptstyle  X$}}}
\put(14,22.5){\makebox(0,0){\mbox{$\scriptstyle b_2 $}}}
\end{picture}

\end{picture}}

\noindent In this picture we tacitly assume that $1$ is the unit element $1_{s(e)}$ where $e$ belongs to the fibre $p^{-1}(b).$  The corolla representing the source of $g:\bc'\to\bc$ is obtained by contracting inner edges of the bicorolla according to the multiplication $b'= b(1,\ldots,b_1,1,\ldots, b_k,\ldots,1)$ of the polynomial monad $T,$ cf. Section \ref{graphical}.

\begin{remark}If $T$ is a tame polynomial monad then semi-free coproducts admit a functorial polynomial formula similar to formula (\ref{semfreemon}) of the introduction.  Indeed, Theorem \ref{colim} applied to the adjoint square (\ref{adjoint}), shows that the underlying object of the semi-free coproduct $X\vee F_T(K)$ is the colimit of a functor $\tilde{X}$ defined on $\cop$. If $\cop$ has a terminal object in each connected component, then the colimit of $\tilde{X}$ is the coproduct of the values of $\tilde{X}$ at these terminal objects. These values are tensor products of as many $X'$s and $K'$s, as there are $X$- resp. $K$-edges in the corollas representing the terminal objects of the different connected components of $\cop$.

More precisely, there is a (uniquely determined) polynomial functor$$P:\Set/I\times\Set/I \to\Set/I$$rendering the following diagram\begin{diagram}[small]\Alg_T\times\Set/I&\rTo^{(-)\vee F_T(-)}&\Alg_T\\\dTo^{U_T\times id_{\Set/I}}&&\dTo_{U_T}\\\Set/I\times\Set/I&\rTo^P&\Set/I\end{diagram}

\noindent commutative with generating polynomial given by
\begin{diagram}I\sqcup I&\lTo^s&\pi_0(\cop)^*&\rTo^p&\pi_0(\cop)&\rTo^t I.\end{diagram}
Here we identify the set $\pi_0(\cop)$ of connected components of $\cop$ with a representative set of objects of $\cop$ which are terminal in their component. Such an object of $\cop$ is represented by a corolla decorated by an element $b\in B$ with edges having colours $X$ or $K.$ The target of $b$ gives a map $t:\pi_0(\cop)\to I.$ The set $\pi_0(\cop)^*$ is the set of corollas as above with one edge marked. The map $p:\pi_0(\cop)^*\to \pi_0(\cop)$ simply forgets the marking. The source map $s:\pi_0(\cop)^*\to I\sqcup I$ returns the colour of the marked edge of $b$ and places it to the first copy of $I$ if the edge-colour is $X$ and to the second if the edge-colour is $K.$

See Sections \ref{SSM} and \ref{SSP} for explicit examples.

\end{remark}

\section{Free algebra extensions}\label{coprodandpushout}

In this central section we apply the theory of internal algebra classifiers to get an explicit formula for free algebra extensions over tame polynomial monads. This formula generalises the Schwede-Shipley formula \cite{SS} for free monoid extensions and involves a careful analysis of the internal classifier $\h$ for free algebra extensions over a tame polynomial monad $T$. We show in particular that a good behaviour of semi-free coproducts of $T$-algebras (the \emph{tameness} of $T$) is sufficient to express the underlying object of a free $T$-algebra extension as a sequential colimit of pushouts.

\subsection{Internal classifier for free algebra extensions}Let $T$ be a finitary monad on a cocomplete category $\CC$. Let ${\mathcal P}_{f,g}$ be the category whose objects are quintuples
$(X,K,L,g,f),$  where $X$ is a $T$-algebra, $K,L$  are objects in $\CC$ and $g:K\rightarrow U_T(X),\,f:K\rightarrow L$ are morphisms in $\CC.$ There is an obvious forgetful functor
$$\Uu_{f,g}:{\mathcal P}_{f,g} \rightarrow \CC\times\CC\times\CC,$$
taking the quintuple $(X,K,L,f,g)$ in $\Pp_{f,g}$ to the triple $(U_T(X),K,L)$ in $\CC\times\CC\times\CC$.

\begin{pro}\label{Cart}Let $T$ be a finitary monad on a cocomplete category $\CC$.

\begin{itemize} \item[(i)] The functor $\Uu_{f,g}$ is monadic and the induced monad $T_{f,g}$ is finitary;
\item[(ii)] There is a commutative square of adjunctions

\vspace{-1mm}
\begin{equation}\label{adjunctionsquare} \end{equation}
{\unitlength=1mm
\begin{picture}(40,19)(-15,-1)
\put(49,15){\shortstack{\mbox{$\scriptstyle U_T$}}}
\put(57,15){\shortstack{\mbox{$\scriptstyle F_T$}}}
\put(38,3){\shortstack{\mbox{$\scriptstyle \sqcup$}}}
\put(38,9){\shortstack{\mbox{$ \Delta$}}}
\put(11,15){\shortstack{\mbox{$\scriptstyle \Uu_{f,g}$}}}
\put(20,15){\shortstack{\mbox{$\scriptstyle \Ff_{f,g}$}}}
\put(18,25){\makebox(0,0){\mbox{${\mathcal P}_{f,g}$}}}
\put(19,11){\vector(0,1){10}}
\put(17,21){\vector(0,-1){10}}
\put(26,24){\vector(1,0){24}}
\put(49,26){\vector(-1,0){24}}
\put(38,27){\shortstack{\mbox{$ \Delta'$}}}
\put(38,21.5){\shortstack{\mbox{$\scriptstyle \sqcup'$}}}
\put(56,25){\makebox(0,0){\mbox{$\Alg_T$}}}
\put(54,21){\vector(0,-1){10}}
\put(56,11){\vector(0,1){10}}
\put(17,7){\makebox(0,0){\mbox{$ \CC\times\CC \times\CC$}}}
\put(31,6){\vector(1,0){16}}
\put(47,8){\vector(-1,0){16}}
\put(56,7){\makebox(0,0){\mbox{$\CC$}}}
\end{picture}}

\noindent in which  $\Delta:\CC\rightarrow\CC\times\CC\times\CC$ is the diagonal and $\Delta'$ is given by $$\Delta'(Y)= (Y,U_T(Y),U_T(Y),1_{U_T(Y)},1_{U_T(Y)}).$$

\item[(iii)]The left adjoint $\sqcup$ to $\Delta$ is given by coproduct in $\CC;$ the left adjoint $\sqcup'$ to $\Delta'$ is given by the following pushout in $\Alg_T$:

\begin{gather}\label{pushout}\begin{diagram}[small]F_T(K)&\rTo^{F_T(f)} & F_T(L) \\\dTo^{\hat{g}}&&\dTo_{}\\X&\rTo&\NWpbk \sqcup'(X,K,L,g,f)\end{diagram}\end{gather}

\noindent in which $\hat{g}$ is adjoint to $g$.
\item[(iv)]If $\CC$ has pullbacks which commute with coproducts and $T$ is a cartesian monad, then $T_{f,g}$ is a cartesian monad as well, and the adjoint square (\ref{adjunctionsquare}) induces a cartesian morphism from $T_{f,g}$ to $T$ in the sense of Definition \ref{Phic}.\item[(v)] If $T$ is a polynomial monad on $\Set/I$ then $T_{f,g}$ is a polynomial monad on $\Set/(I\sqcup I\sqcup I).$\end{itemize}

\end{pro}

\begin{proof}--

(i) The left adjoint $\Ff_{f,g}$ to $\Uu_{f,g}$ takes a triple $(X,K,L)$ in $\CC\times\CC\times\CC$ to the quintuple
$(F_T(X\sqcup K), K, K\sqcup L, j , i)$ in $\Pp_{f,g}$ where $i:K\to K\sqcup L$ is the coproduct injection and $j:K\to U_TF_T(X\sqcup K)$ is the composite of the coproduct injection $K\to X\sqcup K$ and the unit of the monad $T$ at $X\sqcup K$. It is then straightforward to check that $\Ff_{f,g}$ is indeed left adjoint to $\Uu_{f,g}$, and that $T_{f,g}=\Uu_{f,g}\Ff_{f,g}$ is finitary.

(ii) It is obvious that the square of right adjoints commutes. The  existence of the left adjoint $\sqcup'$ follows from the adjoint lifting theorem.

(iii) By adjointness, a map $(X,K,L,g,f)\to\Delta'(Y)$ in $\Pp_{f,g}$ corresponds one-to-one to a pair of $T$-algebra morphisms $(\phi:X\to Y,\,\psi:F_T(L)\to Y)$ such that $\phi \hat{g}=\psi F_T(f)$. The universal property of pushout (\ref{pushout}) then implies that this pair corresponds one-to-one to a $T$-algebra morphism $\sqcup'(X,K,L,g,f)\to Y$.

(iv) We have seen in (i) that $T_{f,g}(X,K,L)=(T(X\sqcup K),K,K\sqcup L)$ so that $T_{f,g}$ is a cartesian monad, since $T$ is a cartesian monad and moreover pullbacks commute with coproducts in $\CC$. It remains to be shown that Proposition \ref{polsq} applies, i.e. that $\Phi:T_{f,g}\to\Delta\circ T\circ\sqcup$ is a cartesian natural transformation. Indeed, all three components of$$\Phi_{(X,K,L)}:(T(X\sqcup K),K,K\sqcup L)\to(T(X\sqcup K\sqcup L),T(X\sqcup K\sqcup L),T(X\sqcup K\sqcup L))$$are cartesian natural transformations. The first component is obtained by applying $T$ to the coproduct injection $X\sqcup K\inc X\sqcup K\sqcup L$, the second component is obtained as the composite of the unit $K\to T(K)$ with $T(K\inc X\sqcup K\sqcup L)$, the third component as the composite $K\sqcup L\to T(K\sqcup L)\to T(X\sqcup K\sqcup L).$ In all three cases we can conclude, since the unit of $T$ is a cartesian natural transformation, $T$ preserves pullbacks, and pullbacks commute with coproducts in $\CC$.

(v) It is enough to show that $T_{f,g}$ preserves connected limits. But in $\Set/I$ connected limits commute with coproducts so that the explicit formula for $T_{f,g}$ yields the result.\end{proof}

In virtue of the preceding proposition, Theorem \ref{polkan} allows us to compute free $T$-algebra extensions in cocomplete categorical $T$-algebras as left Kan extensions along a map of categorical $T$-algebras $\h\rightarrow \HT$.

If $\Ee$ is a cocomplete, closed symmetric monoidal category and $T$ is a polynomial monad, Theorem \ref{colim} expresses free $T$-algebra extensions in $\Ee$ as colimits of certain $\Ee$-valued functors on the free $T$-algebra extension classifier $\h.$ It is therefore vital to get a better hold on the free $T$-algebra extension classifier $\h$. To this effect it will be convenient to introduce three other monads associated to $T$, which we shall denote by $T_f,\,T_g\text{ and }T+2$ respectively.

Let ${\mathcal P}_{f}$ be the category whose objects are quadruples
$(X,K,L,f),$ where $X$ is a $T$-algebra, $K,L$  are objects in $\CC$ and $f:K\rightarrow L$ is a morphism in $\CC.$

Let ${\mathcal P}_{g}$ be the category whose objects are quadruples $(X,K,L,g),$  where $X$ is a $T$-algebra, $K,L$  are objects in $\CC$ and $g:K\to U_T(X)$ is a morphism in $\CC$.

The obvious forgetful functors $\Uu_f:\Pp_f\to\CC\times\CC\times\CC$ and $\Uu_g:\Pp_g\to\CC\times\CC\times\CC$ are monadic yielding monads $T_f$ and $T_g$ for which there are propositions analogous to Proposition \ref{Cart}. We leave the details to the reader.

Finally, recall the monad $T+1$ from Section \ref{tame}. We put $T+2=(T+1)+1$ which is also a monad on $\CC\times\CC\times\CC$ as are $T_{f,g}$, $T_f$ and $T_g$.

There is a  commutative square of forgetful functors \begin{diagram}[small]\Pp_{f,g}&\rTo&\Pp_f\\\dTo&&\dTo\\\Pp_g&\rTo&\Alg_T\times\CC\times\CC\end{diagram}over $\CC\times\CC\times\CC$. All four forgetful functors have left adjoints so that we get a commutative square
of monad morphisms going in the opposite direction\begin{diagram}[small]T_{f,g}&\lTo&T_f\\\uTo&&\uTo\\T_g&\lTo&T+2\end{diagram}and augmented over $T$ via cartesian natural transformations. We thus obtain a commutative square of categorical $T$-algebra maps of the corresponding classifiers\begin{gather}\label{classifiersquare}\begin{diagram}[small]\h&\lTo&\hi\\\uTo&&\uTo\\\hj&\lTo&\ho\end{diagram}\end{gather}
\noindent which enables us to analyse the category structure of $\h$.\vspace{1ex}

Form now on we assume that $T$ is a polynomial monad on $\Set/I$. We have seen that the monad $T_{f,g}$ is then a polynomial monad on $\Set/(I\sqcup I\sqcup I)$. Similarly,

\begin{lem}\label{Tfg}For any polynomial monad $T$ on $\Set/I$, the monads $T+2,\,T_f,\,T_g$ are polynomial monads on $\Set/(I\sqcup I\sqcup I)$.

The internal classifiers $\h,\hi,\hj,\ho$ all have the same object-set, and diagram (\ref{classifiersquare}) identifies $\hi,\hj$ with subcategories of $\h$ which intersect in $\ho$ and which generate $\h$ as a category.\end{lem}

\begin{proof}The first assertion follows by a similar argument as for Proposition \ref{Cart}(v) from the explicit formulas for the monads $T+2$, $T_f$, $T_g$ given below.

According to Theorem \ref{TST}, the object-set of all four internal classifiers is $Tc(1)$ where $c:\CC\times\CC\times\CC\to\CC$ is given by the coproduct in $\CC$, while the morphism-sets are $TcS(1)$ where $S$ is one of the four monads $T_{f,g},T_f,T_g,T+2$. We have seen in Proposition \ref{Cart}(i) that $T_{f,g}(X,K,L)=(T(X\sqcup K),K,K\sqcup L)$. Similarly, we have $T_f(X,K,L)=(T(X),K,K\sqcup L)$ and $T_g(X,K,L)=(T(X\sqcup K),K,L)$ as well as $(T+2)(X,K,L)=(T(X),K,L)$. Evaluating these formulas for $X=K=L=1$, and using the fact that $c$ and $T$ are faithful functors, it follows that $\hi$ and $\hj$ are subcategories of $\h$ which intersect in $\ho$. Moreover, each morphism in $\h$ is the composite of morphisms in $\hi$ and $\hj$.\end{proof}

\subsection{The category $\h$ explicitly}\label{extensionclassifier}

We begin by describing the cartesian morphism $(\nabla_I,\phi,\psi)$ of polynomials (cf. Section \ref{cartesian})
\begin{diagram}I\sqcup I\sqcup I&\lTo^{s'}&E'\SEpbk&\rTo^{p'}&B'&\rTo^{t'}&I\sqcup I\sqcup I\\\dTo^{\nabla_I}&&\dTo^\psi&&\dTo^\phi&&\dTo^{\nabla_I}\\I&\lTo^s&E&\rTo^p&B&\rTo^t&I\end{diagram}which generates the cartesian monad morphism $\Phi:T_{f,g}\Rightarrow\Delta\circ T\circ\sqcup$ described in Proposition \ref{Cart}. We use our graphical formalism to represent the elements of $B'$, compare with Section \ref{semifreeclassifier}. The set $B'$ consists of corollas of the following types
\begin{enumerate} \item[(i)] for $b\in B, b\ne 1_i,$

 {\unitlength=1mm
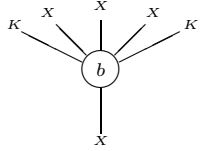
  \begin{picture}(20,20)(-30,13)
\put(14,22.5){\circle{5}}
\put(2.5,28.5){\makebox(0,0){\mbox{$\scriptscriptstyle K$}}}
\put(26,28.5){\makebox(0,0){\mbox{$\scriptscriptstyle K$}}}
\put(7,30){\makebox(0,0){\mbox{$\scriptscriptstyle X$}}}
\put(21,30){\makebox(0,0){\mbox{$\scriptscriptstyle X$}}}
\put(14,31){\makebox(0,0){\mbox{$\scriptscriptstyle X$}}}
\put(16.5,23.5){\line(2,1){8}}
\put(11.5,23.5){\line(-2,1){8}}
\put(12.1,24.4){\line(-1,1){4}}
\put(14,25){\line(0,1){4}}
\put(15.9,24.4){\line(1,1){4}}
\put(14,20){\line(0,-1){6}}
\put(14,22.5){\makebox(0,0){\mbox{$\scriptstyle b $}}}
\put(14,13){\makebox(0,0){\mbox{$\scriptscriptstyle X$}}}
\end{picture}}

\noindent (observe that a corolla of this type does not have $L$-coloured edges);

\item[(ii)] for $i\in I,$

 {\unitlength=1mm
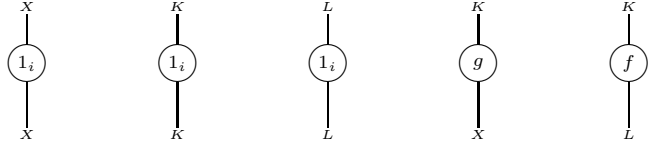
  \begin{picture}(20,20)(-20,13)
\put(-10,22.5){\circle{5}}
\put(-10,30){\makebox(0,0){\mbox{$\scriptscriptstyle X$}}}
\put(-10,13){\makebox(0,0){\mbox{$\scriptscriptstyle X$}}}
\put(-10,25){\line(0,1){4}}
\put(-10,20){\line(0,-1){6}}
\put(-10,22.5){\makebox(0,0){\mbox{$\scriptstyle 1_i $}}}
\put(10,22.5){\circle{5}}
\put(10,30){\makebox(0,0){\mbox{$\scriptscriptstyle K$}}}
\put(10,13){\makebox(0,0){\mbox{$\scriptscriptstyle K$}}}
\put(10,25){\line(0,1){4}}
\put(10,20){\line(0,-1){6}}
\put(10,22.5){\makebox(0,0){\mbox{$\scriptstyle 1_i $}}}
\put(30,22.5){\circle{5}}
\put(30,30){\makebox(0,0){\mbox{$\scriptscriptstyle L$}}}
\put(30,13){\makebox(0,0){\mbox{$\scriptscriptstyle L$}}}
\put(30,25){\line(0,1){4}}
\put(30,20){\line(0,-1){6}}
\put(30,22.5){\makebox(0,0){\mbox{$\scriptstyle 1_i $}}}
\put(50,22.5){\circle{5}}
\put(50,30){\makebox(0,0){\mbox{$\scriptscriptstyle K$}}}
\put(50,13){\makebox(0,0){\mbox{$\scriptscriptstyle X$}}}
\put(50,25){\line(0,1){4}}
\put(50,20){\line(0,-1){6}}
\put(50,22.5){\makebox(0,0){\mbox{$\scriptstyle g $}}}
\put(70,22.5){\circle{5}}
\put(70,30){\makebox(0,0){\mbox{$\scriptscriptstyle K$}}}
\put(70,13){\makebox(0,0){\mbox{$\scriptscriptstyle L$}}}
\put(70,25){\line(0,1){4}}
\put(70,20){\line(0,-1){6}}
\put(70,22.5){\makebox(0,0){\mbox{$\scriptstyle f $}}}
\end{picture}}

\noindent
the first three corollas represent identity  operations, the fourth corolla  represents the operation $g:K\to U_T(X)$ in a $T_{f,g}$-algebra, and the last one represents the operation $f:K\to L$ in a $T_{f,g}$-algebra.

\end{enumerate}

The set $E'$ is the set of such corollas equipped with a marking of one of the incoming edges. The map $p':E'\to B'$ forgets this additional marking. The source map $s':E'\to I\sqcup I\sqcup I$ is determined by the source of the decoration of the corolla together with the labeling of the marked edge. In a similar way, the target map $t':B'\to I\sqcup I\sqcup I$ is determined by the decoration of the corolla together with the labeling of the root edge. We leave it to the reader to check that this polynomial generates the monad $T_{f,g}$ constructed in Proposition \ref{Cart}. On the level of the generating polynomial the multiplication of $T_{f,g}$ is induced by the obvious substitution of corollas into corollas together with the following type of relations:

 {\unitlength=0.9mm
  \begin{picture}(20,20)(0,13)
\put(14,22.5){\circle{5}}
\put(2.5,28.5){\makebox(0,0){\mbox{$\scriptscriptstyle X$}}}
\put(26,28.5){\makebox(0,0){\mbox{$\scriptscriptstyle K$}}}
\put(7,30){\makebox(0,0){\mbox{$\scriptscriptstyle X$}}}
\put(21,30){\makebox(0,0){\mbox{$\scriptscriptstyle X$}}}
\put(14,31){\makebox(0,0){\mbox{$\scriptscriptstyle K$}}}
\put(16.5,23.5){\line(2,1){8}}
\put(11.5,23.5){\line(-2,1){8}}
\put(12.1,24.4){\line(-1,1){4}}
\put(14,25){\line(0,1){4}}
\put(15.9,24.4){\line(1,1){4}}
\put(14,20){\line(0,-1){6}}
\put(14,13){\makebox(0,0){\mbox{$\scriptscriptstyle X$}}}
\put(14,22.5){\makebox(0,0){\mbox{$\scriptstyle b $}}}
\put(33,22.5){\makebox(0,0){\mbox{$ \sim $}}}
\put(52,22.5){\circle{5}}
\put(40.5,28.5){\makebox(0,0){\mbox{$\scriptscriptstyle X$}}}
\put(64,28.5){\makebox(0,0){\mbox{$\scriptscriptstyle X$}}}
\put(45,30){\makebox(0,0){\mbox{$\scriptscriptstyle X$}}}
\put(59,30){\makebox(0,0){\mbox{$\scriptscriptstyle X$}}}
\put(52,31){\makebox(0,0){\mbox{$\scriptscriptstyle X$}}}
\put(54.5,23.5){\line(2,1){8}}
\put(49.5,23.5){\line(-2,1){8}}
\put(50.1,24.4){\line(-1,1){4}}
\put(52,25){\line(0,1){4}}
\put(53.9,24.4){\line(1,1){4}}
\put(52,20){\line(0,-1){6}}
\put(52,13){\makebox(0,0){\mbox{$\scriptscriptstyle X$}}}
\put(52,22.5){\makebox(0,0){\mbox{$\scriptstyle b $}}}
\put(65,22.5){\makebox(0,0){\mbox{$ \mbox{\Huge(}$}}}
\put(115,22.5){\makebox(0,0){\mbox{$ \mbox{\Huge)}$}}}
\put(70,22.5){\circle{5}}
\put(70,30){\makebox(0,0){\mbox{$\scriptscriptstyle X$}}}
\put(70,13){\makebox(0,0){\mbox{$\scriptscriptstyle X$}}}
\put(70,25){\line(0,1){4}}
\put(70,20){\line(0,-1){6}}
\put(70,22.5){\makebox(0,0){\mbox{$\scriptstyle 1_i $}}}
\put(75,20){\makebox(0,0){\mbox{$ \mbox{\Large ,}$}}}
\put(80,22.5){\circle{5}}
\put(80,30){\makebox(0,0){\mbox{$\scriptscriptstyle X$}}}
\put(80,13){\makebox(0,0){\mbox{$\scriptscriptstyle X$}}}
\put(80,25){\line(0,1){4}}
\put(80,20){\line(0,-1){6}}
\put(80,22.5){\makebox(0,0){\mbox{$\scriptstyle 1_i $}}}
\put(85,20){\makebox(0,0){\mbox{$ \mbox{\Large ,}$}}}
\put(90,22.5){\circle{5}}
\put(90,30){\makebox(0,0){\mbox{$\scriptscriptstyle K$}}}
\put(90,13){\makebox(0,0){\mbox{$\scriptscriptstyle X$}}}
\put(90,25){\line(0,1){4}}
\put(90,20){\line(0,-1){6}}
\put(90,22.5){\makebox(0,0){\mbox{$\scriptstyle g $}}}
\put(95,20){\makebox(0,0){\mbox{$ \mbox{\Large ,}$}}}
\put(100,22.5){\circle{5}}
\put(100,30){\makebox(0,0){\mbox{$\scriptscriptstyle X$}}}
\put(100,13){\makebox(0,0){\mbox{$\scriptscriptstyle X$}}}
\put(100,25){\line(0,1){4}}
\put(100,20){\line(0,-1){6}}
\put(100,22.5){\makebox(0,0){\mbox{$\scriptstyle 1_i $}}}
\put(105,20){\makebox(0,0){\mbox{$ \mbox{\Large ,}$}}}
\put(110,22.5){\circle{5}}
\put(110,30){\makebox(0,0){\mbox{$\scriptscriptstyle K$}}}
\put(110,13){\makebox(0,0){\mbox{$\scriptscriptstyle X$}}}
\put(110,25){\line(0,1){4}}
\put(110,20){\line(0,-1){6}}
\put(110,22.5){\makebox(0,0){\mbox{$\scriptstyle g $}}}
\end{picture}}

The two mappings $\phi:B'\to B$ and $\psi:E'\to E$ forget the edge-colourings of the corollas, and identify $f$ and $g$ with identity operations. This explicit presentation of the generating map $(\nabla_I,\phi,\psi)$ in conjunction with Section \ref{TS} yields the following description of the free $T$-algebra extension classifier $\h$, cf. Section \ref{semifreeclassifier}.

The objects of $\h$ are corollas decorated by the elements of $B=T(1)$ with incoming edges coloured by one of the three colours $X, K, L:$

 {\unitlength=1mm
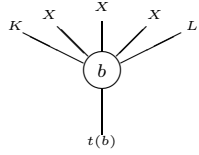
  \begin{picture}(20,20)(-40,13)
\put(14,22.5){\circle{5}}
\put(2.5,28.5){\makebox(0,0){\mbox{$\scriptscriptstyle K$}}}
\put(26,28.5){\makebox(0,0){\mbox{$\scriptscriptstyle L$}}}
\put(7,30){\makebox(0,0){\mbox{$\scriptscriptstyle X$}}}
\put(21,30){\makebox(0,0){\mbox{$\scriptscriptstyle X$}}}
\put(14,31){\makebox(0,0){\mbox{$\scriptscriptstyle X$}}}
\put(16.5,23.5){\line(2,1){8}}
\put(11.5,23.5){\line(-2,1){8}}
\put(12.1,24.4){\line(-1,1){4}}
\put(14,25){\line(0,1){4}}
\put(15.9,24.4){\line(1,1){4}}
\put(14,20){\line(0,-1){6}}
\put(14,22.5){\makebox(0,0){\mbox{$\scriptstyle b $}}}
\put(14,13){\makebox(0,0){\mbox{$\scriptscriptstyle t(b)$}}}
\end{picture}}

These incoming edges will be called  $X$-edges, $K$-edges or $L$-edges accordingly.

The morphisms of $\h$ can be described in terms of generators and relations. There are three types of generators. First, we have the generators coming from the $T$-algebra structure on $X$-coloured edges, with a similar description as in $\cop.$ The relations between these generators witness the relations between $T$-operations. The subcategory of $\h$ spanned by these generators coincides with $\ho$.

The next type of generators corresponds to the morphism $f:K\rightarrow L.$ Such a generator simply replaces a $K$-edge with an $L$-edge in the corolla. Generators of this kind will be called \emph{$F$-generators}. The category $\hi$ is precisely the subcategory of $\h$ generated by $\ho$ and $F$-generators.

Finally, we have generators corresponding to $g:K\rightarrow U_T(X).$ Such a generator replaces a $K$-edge with an $X$-edge. Generators of this kind will be called \emph{$G$-generators}. The category $\hj$ is precisely the subcategory of $\h$ generated by $\ho$ and $G$-generators.

The relations in $\h$ between the morphisms in $\ho$, the $F$-generators and the $G$-generators readily follow from the aforementioned description of $\h$. Most importantly for us, every span $b\stackrel{\phi}{\leftarrow} a\stackrel{\psi}{\rightarrow}a'$ in which $\phi$ is an $F$-generator (resp. $G$-generator) and  $\psi$ belongs to $\ho$, extends to a commutative square \begin{gather}\label{relation1}\begin{diagram}[small]a&\rTo^{\psi}&a'\\\dTo^{\phi}&&\dTo_{\phi'}\\b&\rTo_{\psi'}&b'\end{diagram}\end{gather}in which $\phi'$ is an $F$-generator (resp. $G$-generator) and $\psi'$ belongs to $\ho$. Indeed, $F$- (resp. $G$-)generators replace a $K$-edge by an $L$-edge (resp. $X$-edge), while the morphisms in $\ho$ only affect $X$-edges. So, we can either first apply $\psi$ and then replace the corresponding $K$-edge by an $L$-edge (resp. $X$-edge), or first apply $\phi$ and then apply the corresponding morphism in $\ho$.

A similar argument yields that every composite $a\stackrel{\phi_f}{\rightarrow} b\stackrel{\phi_g}{\rightarrow}c$ of an $F$-generator $\phi_f$ followed by a $G$-generator $\phi_g$ can be rewritten \begin{gather}\label{relation2}\begin{diagram}[small]a&\rTo^{\phi_f}&b\\\dTo^{\phi'_g}&&\dTo_{\phi_g}\\b'&\rTo_{\phi'_f}&c\end{diagram}\end{gather}as the composite of a $G$-generator $\phi'_g$ followed by an $F$-generator $\phi'_f$.

\subsection{A final subcategory of $\h$.}\label{finalsub}

Recall that a subcategory $\Aa$ of $\Bb$ is called \emph{final} if the inclusion functor $i:\Aa\inc\Bb$ is a final functor. This means that for each object $b$ of $\Bb$, the undercategory $b/\Aa$ is \emph{non-empty} and \emph{connected}. Final subcategories have the characteristic property that for functors $F:\Bb\to\Ee$ with cocomplete codomain, the canonical map $\colim_\Aa Fi \to\colim_\Bb F$ is an isomorphism.\vspace{1ex}

The existence of a terminal object in $\Bb$ is equivalent either to the existence of a right adjoint for the unique functor $\Bb\to 1$ to the terminal category $1$ or to the existence of an embedding of the terminal category $1$ as a full and final subcategory of $\Bb$. The following lemma is a ``several component'' version of this observation.

\begin{lem}\label{reflective}For any category $\Bb$ the following three conditions are equivalent:
\begin{itemize}\item[(i)]$\Bb$ is a coproduct of categories with terminal object;
\item[(ii)]$\Bb$ has a full subcategory which is discrete and final;
\item[(iii)]The connected component functor $\Bb\to \pi_0(\Bb)$ has a right adjoint.
\end{itemize}

In particular, if a reflective subcategory $\Aa$ of $\Bb$ satisfies one of the equivalent conditions (i)-(iii) then so does $\Bb$.\end{lem}

\begin{proof}The equivalence of conditions (i)-(iii) follows like in the connected case.

For the second assertion observe that the composite of two full and final inclusions is again a full and final inclusion, and that the inclusion of a reflective subcategory is a full right adjoint (and hence final) functor. Therefore $\Bb$ inherits property (ii) from $\Aa$.\end{proof}

\begin{lem}\label{final0}For any tame polynomial monad $T$, the categories $\ho,\hi,\hj$ are coproducts of categories with terminal object.
 \end{lem}

\begin{proof}The explicit descriptions of the categories $\cop$ and $\ho$ in Sections \ref{semifreeclassifier} and \ref{extensionclassifier} show that $\cop$ is the full subcategory of $\ho$ spanned by those corolla which do not contain any $L$-edges. This inclusion has a retraction $r:\ho\to\cop$ which preserves the arities of the corollas as well as the distinction between $X$- and non-$X$-edges, but takes $K$- and $L$-edges of a corolla in $\ho$ to $K$-edges of the image-corolla in $\cop$. Each morphism $g:\bc'\to\bc$ in $\ho$ induces a bijection between the $K$-edges of $\bc'$ and of $\bc$, and a bijection between the $L$-edges of $\bc'$ and of $\bc$. A morphism $g:\bc'\to\bc$ is completely determined by these two bijections and its image $r(g):r(\bc')\to r(\bc)$ in $\cop$. This implies that the restriction of $r$ to a connected component of $\ho$ is fully faithful and bijective on objects, and takes thus a connected component of $\ho$ isomorphically to the corresponding connected component of $\cop$. Therefore, since by tameness the category $\cop$ is a coproduct of categories with terminal object, the same holds for $\ho$.

The category $\hi$ contains a reflective subcategory isomorphic to $\cop$, namely the subcategory spanned by those corollas in $\hi$ which have only $X$- and $L$-edges. The reflection of an object of $\hi$ to this subcategory is given by successive applications of $F$-generators replacing all $K$-edges with $L$-edges. According to Lemma \ref{reflective}, this implies that $\hi$ is a coproduct of categories with terminal object.

The category $\hj$ also contains a reflective subcategory isomorphic to $\cop$, namely the subcategory spanned by those corollas in $\hj$ which have only $X$- and $L$-edges. This time, the $G$-generators define the reflection. According to Lemma \ref{reflective}, this again implies that $\hj$ is a coproduct of categories with terminal object.\end{proof}

Let $\tso$ be a final discrete full subcategory of $\ho$ obtained by choosing a terminal object in each connected component of $\ho$ (cf. Lemma \ref{final0}). We define $\ts$ to be the full subcategory of $\h$ spanned by the objects of $\tso$. In other words, the composite inclusion $\tso\inc\ho\inc\h$ can be written as
$$\tso\overset{H}{\inc}{\ts}\overset{E}{\inc}\h$$where $H$ is the identity on objects and $E$ is a full inclusion.

\begin{lem}\label{final}The category $\ts$ is a final subcategory of $\h$.\end{lem}

\begin{proof}Since the categories $\ho$ and $\h$ have the same objects, the very definition of $\tso$ implies that each object $a$ in $\h$ maps to a uniquely determined object $t_a$ in $\tso$ by a unique map $a\to t_a$ in $\ho$. Since $\tso$ is a subcategory of $\ts$, this shows that the undercategory $a/\ts$ contains at least $a\to t_a$ and is thus non-empty. We shall show that  in $a/\ts$ any object $a\to b$ is in the same connected component as $a\to t_a$ using an induction on the number of $K$-edges in the corolla representing $a.$

If $a$ has no $K$-edges then all morphisms $a\to b$ belong to $\ho$, so that $a/\ts$ contains just the canonical morphism $a\to t_a$, and there is nothing to prove. Assume now that $a$ has $p$ $K$-edges with $p>0$ and that we have already shown that $a/\ts$ is connected for all objects $a$ with less than $p$ $K$-edges.

Consider an arbitrary object $\phi:a\to b$ of $a/\ts$. In general, the morphisms in $\h$ are generated by the morphisms in $\ho$, the $F$-generators and the $G$-generators, see Section \ref{extensionclassifier}. If $\phi$ itself belongs to $\ho$ then, since $b$ is an object of $\tso$, one has necessarily $b=t_a$ and $\phi$ coincides with the canonical morphism $a\to t_a$.

If $\phi$ does not belong to $\ho$ then $\phi=\phi_2\phi_1$ where $\phi_1$ or $\phi_2$ is an $F$-generator or a $G$-generator. The relations (\ref{relation1}) and (\ref{relation2}) between the different generators of $\h$ imply that we can assume that it is $\phi_1$ which is a $F$- or $G$-generator. Since $F$- and $G$-generators decrease the number of $K$-edges, the codomain of $\phi_1:a\to a'$ has less than $p$ $K$-edges. Therefore, by induction hypothesis, the object $\phi_2:a'\to b$ is in the same connected component as $a'\to t_{a'}$ in $a'/\ts$. This implies that $\phi_2\phi_1:a\to a'\to b$ is in the same connected component as $a\to a'\to t_{a'}$ in $a/\ts$.

Applying the relation (\ref{relation1}) to the span
$t_a\leftarrow a\stackrel{\phi_1}{\to}a'$ we get a commutative square
\begin{diagram}[small]a&\rTo^{\phi_1}&a'\\\dTo^{}&&\dTo_{}\\t_a&\rTo_{\phi'_1}&b'\end{diagram}
where $a'\to b'$ is in $\ho.$ Since $b'$ is in the same connected component as $a'$ there is a canonical morphism $b'\to t_{a'}.$
Hence we have a commutative square in $\h$
\begin{diagram}[small]a&\rTo^{\phi_1}&a'\\\dTo&&\dTo\\t_a&\rTo&t_{a'}\end{diagram}
This shows that in $a/\ts$ the (arbitrarily chosen) object $\phi:a\to b$ is in the same connected component as the canonical object $a\to t_a$.\end{proof}

\subsection{Canonical filtration.}We say that an object $a$ of $\ts$ is of type $(p,q)$ if $a$ contains exactly $p$ $K$-edges and $q$ $L$-edges, and we call $p+q$ the \emph{degree} of $a$. We define $\ts^{(k)}$ (resp. $\wa^{(k)}$) to be the full subcategory of $\ts$ spanned by all objects of degree $\leq k$ (resp. of degree $k$). We define $\qa^{(k)}$ (resp. $\la^{(k)}$) to be the full subcategory of $\wa^{(k)}$ spanned by all objects of type $(p,k-p)$ such that $p\ne 0$ (resp. $p=0$).

Recall that the (categorical) {\it $k$-cube} is the category of subsets and inclusions of the set $\{1,\ldots,k\}.$ The {\it punctured $k$-cube} is the full subcategory of the $k$-cube spanned by the \emph{proper subsets} of  $\{1,\ldots,k\}$.

\begin{lem}\label{cubes}--

\begin{itemize}
\item[(i)]Each connected component of $\wa^{(k)}$ is isomorphic to a $k$-cube;
\item[(ii)]the category $\la^{(k)}$ is a final discrete full subcategory of $\wa^{(k)}.$
\item[(iii)]the category $\qa^{(k)}$ is isomorphic to a coproduct of punctured $k$-cubes.\end{itemize}\end{lem}

\begin{proof}Note first that the morphisms in $\wa^{(k)}$ cannot involve $G$-generators since the latter decrease the degree; therefore, they necessarily belong to $\hi$. Since the objects in $\ts$ are terminal in their connected component in $\ho$, the non-identity morphisms of $\wa^{(k)}$ cannot belong to $\ho$ either; therefore, all morphisms in $\wa^{(k)}$ are composites of $F$-generators. Recall that each $F$-generator replaces a $K$-edge with an $L$-edge. From this it readily follows that the connected components of $\wa^{(k)}$ are $k$-cubes. The assertions (ii) and (iii) are immediate consequences of (i).\end{proof}

\begin{theorem}\label{filtration}\label{SS}For any tame polynomial monad $T$ and any functor ${X}:\h \rightarrow \Ee$ with cocomplete codomain, the colimit of ${X}$ is a sequential colimit of pushouts in $\Ee$.

More precisely, for $P_k = \colim_{\ts^{(k)}} {X}|_{\ts^{(k)}}$, we get $$P=\colim_{\h} {X}\cong \colim_kP_k,$$ where the canonical map $P_{k-1}\to P_k$ is part of the following pushout square in $\Ee$

\begin{gather}\label{KLX}\begin{diagram}[small]Q_k&\rTo^{w_k}&L_k\\\dTo^{\alpha_k}&&\dTo\\P_{k-1}&\rTo&\NWpbk P_k\end{diagram}\end{gather}

 \noindent in which $Q_k$ (resp. $L_k$) is the colimit of the restriction of ${X}$ to $\qa^{(k)}$ (resp. $\la^{(k)}$).\end{theorem}

\begin{proof}The finality of $\ts$ implies $P=\colim_{\h} {X}\cong \colim_{\ts} {X}.$ It is obvious that $\ts \cong \colim_{k} \ts^{(k)}.$ Lemmas \ref{final} and \ref{Lack} then yield $P\cong\colim_{\ts}{X}|_{\ts}\cong\colim_kP_k.$

The inclusion $\qa^{(k)}\inc\wa^{(k)}$ induces the map $w_k:Q_k\rightarrow\colim_{\wa^{(k)}}\tilde{X}|_{\wa^{(k)}}\cong L_k$ where the last isomorphism is a consequence of Lemma \ref{cubes}(ii).

In order to construct the map $\alpha_k:Q_k\rightarrow P_{k-1}$ we shall realise $\ts^{(k-1)}$ as a final subcategory of a category $\overline{\qa}^{(k)}$ which contains $\qa^{(k)}$. The map $\alpha_k$ is then simply induced by the inclusion $\qa^{(k)}\inc\overline{\qa}^{(k)}$. This category $\overline{\qa}^{(k)}$ is by definition the full subcategory of $\ts^{(k)}$ spanned by the objects not contained in $\la^{(k)}$.

To prove that $\ts^{(k-1)}$ is a final subcategory of $\overline{\qa}^{(k)}$ note first that each object $a$ of $\overline{\qa}^{(k)}$ comes equipped with a canonical map  $\xi_a:a\to j(a)$, where $j(a)$ is terminal in the connected component of $a$ in $\hj$, and hence $\xi_a:a\rightarrow j(a)$ is the unique morphism in $\hj$ with codomain $j(a)$, cf. Lemma \ref{final0}. In particular, $j(a)$ belongs $\ts^{(k-1)}$ because $a$ contains at least one $K$-edge and $G$-generators decrease the degree.

It suffices now to show that each object $a\to c$ of $a/\ts^{(k-1)}$ lies in the same connected component as $\xi_a:a\to j(a)$. If $a\to c$ belongs to $\hj$ this holds trivially since $j(a)$ is terminal in its connected component in $\hj.$ Assume that $a\to c$ does not belong to $\hj.$
 We then factor $a\to c$ as $a\to b \to c$ where $a\to b$ is a composite of $F$-generators and $b\to c$ belongs to $\hj$; in other words, we perform all replacements (inside $a\to c$) of a $K$-edge with an $L$-edge first, and perform the replacements of a $K$-edge with an $X$-edge only afterwards. This is always possible due to relations between the generating morphisms of $\h$, cf. Section \ref{extensionclassifier}.

Since $b\to c$ belongs to $\hj$ and the codomain $c$ belongs to $\ts^{(k-1)}$, the domain $b$ cannot belong to $\la^{(k)}$ so that we have a canonical map $\xi_b:b\to j(b)$ whose codomain belongs to $\ts^{(k-1)}.$ Since $b\to c$ belongs to $\hj$ and $j(b)$ is terminal in its connected component of $\hj$ we have a factorisation of $\xi_b$ as $b \to c \to j(b).$  It thus suffices to construct a zig-zag in $a/\ts^{(k-1)}$ connecting $a\to b\stackrel{\xi_b}{\to} j(b)$ and $a\stackrel{\xi_a}{\to} j(a).$

 It  also suffices to assume that the morphism $a\to b$ is equal to one of the $F$-generators $f_v:a\to b$ which replaces a $K$-edge $v$ by a $L$-edge.
 Observe that the morphism $\xi_a:a\rightarrow j(a)$ can be  factorised as
$$a\stackrel{g}{\rightarrow} a' \stackrel{g_v}{\rightarrow} a''  \stackrel{m}{\rightarrow} j(a),$$
where $g$ is a composite of $G$-generators which replaces all $K$-edges by $X$-edges  except  the $K$-edge $v,$   $g_v$ is a $G$-generator which replaces $K$-edge $v$ by an $X$-edge, and $m$ belongs to $\ho.$

The morphism $\xi_b: b\rightarrow j(b)$ can also be factored as
$$b\stackrel{g'}{\rightarrow} b'   \stackrel{m'}{\rightarrow} j(b),$$
where $g'$ is a composite of $G$-generators which  replaces all $K$-edges by $X$-edges  and $m'$ belongs to $\ho.$ Moreover, the following diagram commutes
\begin{diagram}[small]a&\rTo^{g}&a'\\\dTo^{f_v}&&\dTo_{f'_v}\\b&\rTo_{g'}&b'\end{diagram}
where the morphism $f'_v$ is a $F$-generator which replaces the $K$-edge $v$ by an $L$-edge.

Since $F$-generators commute with the morphism from $\ho$ we obtain the following commutative diagram
\begin{diagram}[small]a'&\rTo^{m''}&a'''\\\dTo^{f'_v}&&\dTo_{f''_v}\\b'&\rTo_{m'}&j(b)\end{diagram}in which  $f''_v$ is an $F$-generator and $m''$ belongs to $\ho$. Observe, that $a'''$ belongs to $\ts^{k-1}$ since it is obtained  from $j(b)$ by replacing an $L$-edge by a $K$-edge.

Finally we observe, that $a,a',a'',a'''$ are in the same connected component of $\hj$ by construction and since $j(a)$ is terminal in this connected component we have a commutative diagram
\begin{diagram}[small]a'&\rTo^{g_v}&a''\\\dTo^{m''}&&\dTo_{}\\a'''&\rTo^{}&j(a)\end{diagram}

Putting all these morphisms together we get a commutative diagram in $a/\ts^{(k-1)}$
\begin{diagram}[small,silent,UO]                                        &&&&a&&&&\\
                                                               &                  &         & \ldTo (4,2) &\dTo &   \rdTo(4,2) &       &                    &     \\
                                                             b&   \rTo        & b'     &\lTo            &a'        &\rTo              & a''  &\rTo           &j(a) \\
                                                              & \rdTo(2,2) & \dTo &                  & \dTo   &                     &      & \ruTo(4,2) &    \\
                                                              &                  & j(b)   &\lTo           & a'''       &                     &        &           &\end{diagram}
which provides us with the necessary zigzag.\vspace{1ex}

It follows from the preceding discussion that diagram (\ref{KLX}) may be obtained by restricting $X:\h\to\Ee$ to the following commutative square of categories\begin{gather}\label{tamesquare}\begin{diagram}[small]\qa^{(k)}&\rTo&\wa^{(k)}\\\dTo&&\dTo\\\overline{\qa}^{(k)}&\rTo&\ts^{(k)}\end{diagram}\end{gather}and then computing the colimits of the corresponding restrictions. This yields commutativity of square (\ref{KLX}) as well as a canonical map $P_{k-1}\cup_{Q_k}L_k\to P_k$ in $\Ee$. It remains to be shown that the latter map is invertible, i.e. that square (\ref{KLX}) is a pushout diagram in $\Ee$.

A closer inspection of square (\ref{tamesquare}) reveals that it is a categorical pushout of a special kind: the category $\ts^{(k)}$ is obtained as the set-theoretical union of the categories $\overline{\qa}^{(k)}$ and $\wa^{(k)}$ along their common intersection $\qa^{(k)}$. Indeed, away from this intersection, there are no morphisms in $\ts^{(k)}$ between objects of $\overline{\qa}^{(k)}$ and objects of $\wa^{(k)}$. In virtue of Lemma \ref{Lack} this implies that (\ref{KLX}) is a pushout square in $\Ee$. \end{proof}

\begin{remark}
In our particular situation, the inverse of $P_{k-1}\cup_{Q_k}L_k\to P_k$ is obtained by gluing together (along $\qa^{(k)}$) the two colimit cocones ${X}|_{\overline{\qa}^{(k)}}\overset{\cdot}{\longrightarrow}P_{k-1}$ and ${X}|_{\wa^{(k)}}\overset{\cdot}{\longrightarrow}L_k$ and taking the colimit of ${X}$ over $\ts^{(k)}$.\end{remark}

We are indebted to Steve Lack for pointing out to us a proof of the following categorical fact.
Let $\DD$ be a small category and let $\Phi:\DD\to\Cat$ be a $\DD$-diagram of small categories. We denote by $\CC$ the colimit of this diagram, and by $\Phi(d)\to\CC$ the components of the corresponding colimit cocone. Let ${X}:\CC\to\Ee$ be a functor with cocomplete codomain. For each object $d$ of $\DD$ we consider the restriction ${X}_d: \Phi(d)\to\CC\to\Ee$ and its colimit $\colim_{\Phi(d)}{X}_d$ in $\Ee$. This defines a functor $\colim_{\Phi(-)}{X}_-: \DD\to\Ee.$
\begin{lem}\label{Lack}
The induced map $\colim_\DD\colim_{\Phi(-)}{X}_-\to\colim_\CC{X}$ is  an isomorphism in $\Ee$.\vspace{1ex}\end{lem}

\begin{proof} Consider the slice category $\Cat/\Ee.$ There is a functor $\gamma:\Cat/\Ee\to \Ee$ which takes the objects of the slice category to their colimit in $\Ee$. We have to show that $\gamma$ preserves colimits. We actually show that $\gamma$ has a right adjoint. For this we observe that $\gamma$ can be factored as follows:
$$\Cat/\Ee \stackrel{i}{\to}\Cat//\Ee\stackrel{\gamma'}{\longrightarrow}\Ee$$
where $\Cat//\Ee$ denotes the lax version of the slice category. In $\Cat//\Ee$ the morphisms from $F$ to $G$ are pairs $(f,\phi)$ defining triangles of the following form:
\begin{diagram}[small,UO]
\AA    &                          &       \rTo^f               &     &   \BB  \\
        \quad F &\rdTo(2,2)       &       \overset{\phi}{\Rightarrow}  &   \ldTo(2,2) &G\quad  \\
          &                         &       \Ee                   &      &
\end{diagram}

The functor $i$ is the obvious inclusion functor and $\gamma'$ is again given by taking objectwise the colimit.
Both functors $i$ and $\gamma'$ have right adjoints. The right adjoint of $\gamma'$ takes an object of $\Ee$ to the functor $1\to\Ee$ which picks up this object. The right adjoint of $i$ is the lax-limit functor  given by a Grothendieck type construction. For a functor $F:\AA\to\Ee$ this Grothendieck construction $P(F)$ is the category whose objects are pairs $(a,\alpha)$ consisting of an object $a$ of $\AA$ and a morphism $\alpha:F(a)\to x$ in $\Ee.$ The morphisms $(a,\alpha)\to (b,\beta)$ in $P(F)$ are pairs $(a\to a',\,x\to x')$ such that an obvious diagram commutes. There is a functor $P(F)\to\Ee$ which takes $(a,\alpha)$ to the codomain of $\alpha$ in $\Ee$. It is straightforward to check that this construction provides a right adjoint for $i.$ \end{proof}

\section{Admissibility of tame polynomial monads and Quillen adjunctions}

We are now ready to combine the results of Sections $2$, $5$, $6$ and $7$ so as to obtain the two main theorems of this article.

Throughout this section $\Ee$ denotes a monoidal model category and $I$ a set of colours. The category $\Ee/I$ is then a monoidal model category in which the cofibrations, weak equivalences and fibrations are defined pointwise, and the monoidal structure as well is defined pointwise.

\begin{theorem}\label{maintheorem}Let $(\Ee,\otimes,e)$ be a compactly generated monoidal model category and let $T$ be a tame polynomial monad on $\Set/I$. Then $\Ee/I$ is a compactly generated monoidal model category, and the monad on $\Ee/I$ induced by $T$ is
\begin{itemize}\item[(a)]relatively  $\otimes$-adequate if $\Ee$ satisfies the monoid axiom;
\item[(b)]$\otimes$-adequate if $\Ee$ is strongly $h$-monoidal.\end{itemize}
Therefore, the category of $T$-algebras in $\Ee/I$ admits a relatively left proper transferred model structure if $\Ee$ satisfies the monoid axiom (resp. a left proper transferred model structure, if $\Ee$ is strongly $h$-monoidal).\end{theorem}
\begin{proof}The cartesian monad morphism $\Phi$ of Proposition \ref{Cart}(iii) from $T_{f,g}$ to $T$ induces an adjunction$$\gamma_\Ee^\Phi:\Alg_{T_{f,g}}(\Ee)\lra\Alg_T(\Ee):\delta_\Ee^\Phi$$whose left adjoint takes the quintuple $(X,K,L,f,g)$ to the pushout $P$ in $\Ee$. Theorem \ref{colim} implies that the underlying object of $P$ can be calculated as the colimit of a functor $\tilde{X}:\h \rightarrow \Ee$. Theorem \ref{SS} permits us to realise this colimit as a sequential colimit of pushouts.

From a homotopical point of view it is essential that the sequential colimit presentation of $P$ is made up by pushouts of maps $w_k:Q_k\to L_k$ which are easy to calculate. Indeed, the functor $\tilde{X}:\h\to\Ee$ takes the values \begin{equation}\label{SSSR}\tilde{X}(a) = (\otimes_{v\in \chi(a)} X_v) \otimes (\otimes_{v\in \kappa(a)} K_v) \otimes (\otimes_{v\in\lambda(a)} L_v),\end{equation}cf. formula (\ref{tilde}). Here, $\chi(a)$ is the set of $X$-edges, $\kappa(a)$ is the set of $K$-edges, and $\lambda(a)$ is the set of $L$-edges in the corolla representing $a.$ The $F$-generators of $\h$ act via the map $f:K\to L$, the $G$-generators act via the map $g:K\to U_T(X)$, the morphisms in $\ho$ act via the $T$-algebra structure on $X$. Note that by Lemma \ref{cubes} the map $w_k:Q_k\to L_k$ is a coproduct of comparison maps. Each comparison map is obtained by taking the colimit over a punctured $k$-cube (a connected component of $\qa^{(k)}$) in which the edge-maps are tensor products $Y\otimes f_{v}$ where $f_{v}: K_{v} \rightarrow L_{v}.$

Now we can closely follow the proof of Theorem \ref{SSformonoids} which describes the special case of the free monoid monad. Indeed, formula (\ref{SSSR}) indicates that the only qualitative difference between a general tame polynomial monad and the free monoid monad lies in the fact that the map $w_k:Q_k\to L_k$ is in general a \emph{coproduct} of comparison maps while in the special case treated in Theorem \ref{SSformonoids} we have just a single comparison map. Since (trivial) cofibrations as well as weak equivalences between cofibrant objects are closed under arbitrary coproducts, this difference does not affect the argument establishing (a). However, to carry out the proof of (b) for a general tame polynomial monad, we need the additional property that the class of weak equivalences is closed under arbitrary coproducts. This follows from Proposition \ref{couniv0} because $\Ee$ is strongly $h$-monoidal, cf. Section \ref{strongunit}.\end{proof}

Let $\Phi:S\Rightarrow dTc$ be a cartesian monad morphism, cf. Definition \ref{Phic}. Assume that the categories of $S$- and $T$-algebras in $\Ee$ admit transferred model structures (e.g., $S$ and $T$ are tame polynomial and $\Ee$ is compactly generated $h$-monoidal). By Theorem \ref{colim}, restriction $\delta_\Ee^{\Phi}$ admits a left adjoint $\gamma_\Ee^{\Phi}$. The resulting adjoint pair$$\gamma_\Ee^\Phi:\Alg_S(\Ee)\lra\Alg_T(\Ee):\delta_\Ee^\Phi$$is a Quillen pair with respect to the transferred model structures on both sides. We shall see now that the total left derived functor $\LL\gamma_\Ee^\Phi$ has an explicit formula in terms of the combinatorial data defining $\Phi$, provided that the monoidal model category $\Ee$ admits a \emph{good realisation functor} $|\!-\!|_\Ee$ for simplicial objects.

One way to obtain such a realisation functor is to require that $\Ee$ is \emph{simplicially enriched} in such a way that simplicial hom $\Ee_\bullet(-,-)$ and internal hom $\uEe(-,-)$ are related by the following \emph{compatibility relation}
$$\Ee_\bullet(X,Y)\cong \Ee_\bullet(e,\uEe(X,Y)),$$and moreover $\Ee$ equipped with these simplicial hom's becomes a \emph{simplicial model category}, cf. \cite{Hirschhorn,Hovey}. The compatibility relation implies that the category $\Alg_{T}(\Ee)$ of $T$-algebras is simplicially enriched, and that the free-forgetful adjunction is a simplicial adjunction. The simplicial hom for the category of $T$-algebras is given by the usual formula involving a categorical end (see for instance \cite{SymBat}).

More generally, in order to have a good realisation functor for simplicial objects, it is enough to assume that $\Ee$ has a \emph{standard system of simplices} in the sense of Moerdijk and the second author, cf. \cite[Appendix A]{BergerMoerdijk0}. We then have the following derived version of Theorem \ref{colim}.

\begin{theorem}\label{hounder}\label{hocolim}Let $\Ee$ be a monoidal model category with a ``good'' realisation functor $|\!-\!|_\Ee$ for simplicial objects, and let $\Phi:S\Rightarrow dTc$ be a cartesian monad morphism between polynomial monads. Let $X$ be an $S$-algebra in $\Ee$ whose underlying $J$-collection is pointwise cofibrant. Then the $I$-collection underlying $\,\LL\gamma_\Ee^{\Phi}(X)$ can be calculated as the following homotopy colimit\begin{equation}\label{polhocolimit} \LL\gamma_\Ee^{\Phi}(X)_i = \hocolim_{\bc\in {\bf T}^{\tt S}_i}\widetilde{X}(\bc)\quad\quad(i\in I)\end{equation}where $\tilde{X}:\HS\to\Ee$ represents the $S$-algebra $X$, cf. Section \ref{colimsec}.\end{theorem}

\begin{proof}Let $B_\bullet(S,S,X)$ be the simplicial bar-construction of the $S$-algebra $X$. Its realisation $B(S,S,X)=|B_\bullet(S,S,X)|_\Ee$ is a cofibrant resolution of the $S$-algebra $X$ with respect to the transferred model structure on $\Alg_S(\Ee).$  This follows from a similar argument as for \cite[Theorem 5.5]{SymBat}, once we know that the following augmented cosimplicial object in $\Ee$

  \vspace{-3mm}\begin{equation}\label{aug}\end{equation}
  {\unitlength=1mm

\begin{picture}(0,4)(-40,17)

\put(-30,25){\makebox(0,0){\mbox{$X$}}}
\put(-25,25){\vector(1,0){7}}
\put(-10,25){\makebox(0,0){\mbox{$S(X)$}}}
\put(1,27){\vector(1,0){7}}
\put(1,23){\vector(1,0){7}}
\put(8,25){\vector(-1,0){7}}
\put(31,27){\vector(1,0){7}}
\put(31,23){\vector(1,0){7}}
\put(31,25){\vector(1,0){7}}
\put(50,25){\makebox(0,0){\mbox{$S^3(X)$}}}
\put(20,25){\makebox(0,0){\mbox{$S^2(X)$}}}
\put(78,25){\makebox(0,0){\mbox{$\ldots $}}}
\put(62,28){\vector(1,0){7}}
\put(62,24){\vector(1,0){7}}
\put(62,22){\vector(1,0){7}}
\put(62,26){\vector(1,0){7}}

\end{picture}}

\noindent is Reedy cofibrant. In this cosimplicial object the cofaces are generated by the unit of the monad $S$ and the codegeneracies are generated by the multiplication of $S$.

The unit $\eta:Id_J\to S$ is a cartesian map of polynomial monads. The category of $J$-collections in $\Ee$ can be identified with the category of $Id_J$-algebras so that $\eta$ induces the free-forgetful adjunction between $J$-collections and $S$-algebras, whence the monad $S=U_S F_S$ can be rewritten as $\delta_{\Ee}^{\eta}\gamma_{\Ee}^{\eta}$. Theorem \ref{polkan} implies that the adjoint pair $(\gamma_{\Ee}^\eta,\delta_{\Ee}^\eta)$ is represented by the adjoint pair $((\bf{S}^{\eta})_!,(\bf{S}^{\eta})^*)$ where $\bf{S}^\eta:{\bf S}^{\tt Id_J}\to{\bf S}^{\tt S}$ is the functor of classifiers induced by $\eta:Id_J\to S$, and $J$-collections (resp. $S$-algebras) are represented by functors ${\bf S}^{\tt Id_J}\to\Ee$ (resp. ${\bf S}^{\tt S}\to\Ee$).

The category ${\bf S}^{\tt Id_J}$ has the same objects as ${\bf S}^{\tt S}$ but only identity morphisms, and the functor $\bf{S}^\eta:{\bf S}^{\tt Id_J}\to{\bf S}^{\tt S}$ is identity on objects. The explicit formula for the left Kan extension $({\bf{S}^{\eta}})_!$ produces then for any $S$-algebra $X$, represented as a functor $\tilde{X}:{\bf S}^{\tt S}\to\Ee$, the following formula for the iteration of the monad $S$
$$\widetilde{S^k(X)}(d)=(({\bf{S}^{\eta}})_! ({\bf{S}^{\eta}})^*)^k(\widetilde{X})(d) = \coprod_{d\stackrel{}{\leftarrow}d_1\leftarrow \cdots\, \leftarrow  d_k}\tilde{X}(d_k),$$where the coproduct is taken over composable chains of morphisms in ${\bf S}^{\tt S}.$

Evaluating this formula at the terminal objects $1_j$ of ${\bf S}^{\tt S}_j$ we obtain
$$S^k(X)_j=(F_SU_S)^k(X)_{j}= \coprod_{d_1\leftarrow \cdots\, {\leftarrow} d_k} \tilde{X}(d_k) \quad(d_1,\ldots,d_k\in {\bf S}^{\tt S}_j).$$
The unit of the monad $U_SF_S$ is the canonical summand inclusion $X_j\to \coprod_d \tilde{X}(d)$ which takes $X_j$ to $\tilde{X}(1_j) = X_j.$
Therefore, the latching object of (\ref{aug}) in dimension $k$ is the coproduct of the $\tilde{X}(d_k)$ over degenerate $k$-simplices of the nerve of ${\bf S}^{\tt S}.$ Since the underlying $J$-collection of the $S$-algebra $X$ is pointwise cofibrant, all summands $\tilde{X}(d)$ are cofibrant, and hence the inclusion of the latchting object is a cofibration. It follows that $B(S,S,X)$ is a cofibrant resolution of the $S$-algebra $X$ so that the value of the total left derived functor $\LL\gamma_\Ee^{\Phi}$ can be calculated as $$\LL\gamma_\Ee^{\Phi}(X)=\gamma_\Ee^{\Phi}(B(S,S,X))=\gamma_\Ee^\Phi|B_\bullet(S,S,X)|_\Ee=|\gamma_\Ee^\Phi(B_\bullet(S,S,X))|_\Ee.$$The simplicial $T$-algebra $\gamma_\Ee^\Phi(B_\bullet(S,S,X))$ is isomorphic to

\vspace{0mm}\begin{equation}\label{barS}\end{equation}
   {\unitlength=1mm

\begin{picture}(0,0)(-30,17)

\put(-10,25){\makebox(0,0){\mbox{$\gamma_\Ee^{\Phi}(S(X))$}}}

\put(8,27){\vector(-1,0){7}}
\put(8,23){\vector(-1,0){7}}
\put(1,25){\vector(1,0){7}}
\put(38,27){\vector(-1,0){7}}
\put(38,23){\vector(-1,0){7}}
\put(38,25){\vector(-1,0){7}}
\put(50,25){\makebox(0,0){\mbox{$\gamma_\Ee^{\Phi}(S^3(X))$}}}
\put(20,25){\makebox(0,0){\mbox{$\gamma_\Ee^{\Phi}(S^2(X))$}}}
\put(78,25){\makebox(0,0){\mbox{$\ldots $}}}

\put(69,28){\vector(-1,0){7}}
\put(69,24){\vector(-1,0){7}}
\put(69,22){\vector(-1,0){7}}
\put(69,26){\vector(-1,0){7}}

\end{picture}}

\noindent and, by Theorem \ref{colim}, the underlying object in dimension $k$ is given by the formula
\begin{equation}\label{gam}\gamma_\Ee^{\Phi}(S^k(X))_i = \colim_{\bc\in{\bf T}^{\tt S}_i}\widetilde{S^k(X)}(\bc)\quad(i\in I),\end{equation}
where this time $S$-algebras are represented as functors from $\HS$ to $\Ee$.

To finish the proof we therefore need a presentation of the iterations of the comonad $F_SU_S$ on $Alg_S(\Ee)$ in terms of the \emph{relative classifier} $\HS$. For ease of notation, we will use the same letter $S$ for the   comonad $F_SU_S.$

Consider the following commutative triangle of polynomial monads
  {\unitlength=1mm

\begin{picture}(40,17)(0,13)

\put(50,25){\makebox(0,0){\mbox{$Id_J$}}}
\put(53,23){\vector(1,-1){5}}
\put(60.5,16){\makebox(0,0){\mbox{$T$}}}
\put(68,23){\vector(-1,-1){5}}
\put(70,25){\makebox(0,0){\mbox{$S$}}}
\put(54,25){\vector(1,0){14}}
\put(60,26){\shortstack{\mbox{$\scriptstyle  \eta$}}}
\put(67,19){\shortstack{\mbox{$\scriptstyle  \Phi$}}}
\end{picture}}

\noindent
and use Theorem \ref{RTS} to get
$$\widetilde{S(X)} =  \widetilde{\gamma_\Ee^{\eta}\delta_\Ee^{\eta}(X)}=
({\bf{T}^{\eta}})_! \widetilde{\delta_\Ee^{\eta}(X)} =
({\bf{T}^{\eta}})_! ({\bf{T}^{\eta}})^* (\widetilde{X})$$
and hence
$$\widetilde{S^k(X)} = (({\bf{T}^{\eta}})_! ({\bf{T}^{\eta}})^*)^k (\widetilde{X}).$$
The   canonical functor ${\bf T}^{\eta}: {\bf T}^{\tt Id_J}\to \HS$ is the inclusion of the discrete subcategory of  objects of $\HS$,  so that we  obtain (in a similar way as above) the formula
\begin{equation}\label{sam}\widetilde{S^k(X)}(\bc)  = \coprod\limits_{\bc\leftarrow\bc_1\leftarrow\cdots\leftarrow\bc_k}\tilde{X}(\bc_k)\end{equation} where this time the coproduct is over composable chains of morphisms in $\HS$.

Let $t_i:{\bf T}^{\tt S}_i\to 1$ be the unique functor to the terminal category. Putting formulas (\ref{gam}) and (\ref{sam}) together we obtain
\begin{eqnarray*}
\gamma_\Ee^{\Phi}(S^k(X))_i & = & \colim_{\bc\in{\bf T}^{\tt S}_i}\widetilde{S^k(X)}(\bc)  = (t_i)_! (({\bf{T}^{\eta}})_! ({\bf{T}^{\eta}})^*)^k (\widetilde{X}) \\
                                                      & =  &(t_i{\bf{T}^{\eta}})_! ({\bf{T}^{\eta}})^*(({\bf{T}^{\eta}})_! ({\bf{T}^{\eta}})^*)^{k-1}
                                                                  (\widetilde{X}) =    \coprod\limits_{\bc_1 \leftarrow\cdots\leftarrow\bc_{k}}\tilde{X}(\bc_{k}) .
\end{eqnarray*}

We conclude that the simplicial object $\gamma_\Ee^\Phi(B_\bullet(S,S,X))$ may be identified with the classical simplicial replacement of Bousfield-Kan for the functor $\tilde{X}:\HS\to\Ee$ representing the $S$-algebra $X$. By hypothesis, this functor is pointwise cofibrant so that the Bousfield-Kan simplicial replacement calculates the homotopy colimit of $\tilde{X}$ upon realisation.\end{proof}

\begin{remark}The simplicial $T$-algebra $\gamma_\Ee^\Phi(B_\bullet(S,S,X))$ is isomorphic to

  {\unitlength=1mm
  \begin{picture}(0,10)(-30,21)

\put(-10,25){\makebox(0,0){\mbox{$Tc(X)$}}}

\put(8,27){\vector(-1,0){7}}
\put(8,23){\vector(-1,0){7}}
\put(1,25){\vector(1,0){7}}
\put(38,27){\vector(-1,0){7}}
\put(38,23){\vector(-1,0){7}}
\put(38,25){\vector(-1,0){7}}
\put(50,25){\makebox(0,0){\mbox{$TcS^2(X)$}}}
\put(20,25){\makebox(0,0){\mbox{$TcS(X)$}}}
\put(78,25){\makebox(0,0){\mbox{$\ldots $}}}

\put(69,28){\vector(-1,0){7}}
\put(69,24){\vector(-1,0){7}}
\put(69,22){\vector(-1,0){7}}
\put(69,26){\vector(-1,0){7}}

\end{picture}}

\noindent where $c:\Ee/J\to\Ee/I$ is induced by the cartesian morphism of polynomial monads $\Phi$ from $S$ to $T$ (see Section \ref{cartesian}). In particular, we recognise here the two-sided simplicial bar-construction $B_\bullet(Tc,S,X)$, and hence we get the formula$$\LL\gamma_\Ee^\Phi(X)\cong|B_\bullet(Tc,S,X)|_\Ee.$$

\end{remark}

\begin{corol}\label{Nerve}The simplicial nerve $N(\HS)$ of the $S$-algebra classifier $\HS$ is a cofibrant simplicial $T$-algebra. In fact,
$$N(\HS) = {\mathbb L}\gamma_\Ee^{\Phi}(1)$$where $\Ee$ is the category of simplicial sets equipped with Quillen's model structure, and $1$ is the terminal simplicial $S$-algebra.\end{corol}

\begin{remark}Theorem \ref{hounder} and Corollary \ref{Nerve} generalise formulas of \cite{SymBat,EHBat} where the symmetrisation functor from $n$-operads to symmetric operads has been studied.  Theorem \ref{hounder} is one of the main tools in the study of stabilisation phenomena in \cite{BBC,BBC2}.\end{remark}

\begin{remark}Giansiracusa's formula \cite{G} for the derived modular envelope of a cyclic operad can be understood along similar lines. For the precise relationship with the present approach we refer the reader to \cite[Example 3.4.3]{W2}.\end{remark}

\part{Operads as algebras over polynomial monads}

Symmetric, non-symmetric, cyclic, modular operads, properads, PROP's, the higher operads of the first author, and other types of generalised operads (see \cite{Borisov-Manin, Markl, Loday, JY}) are examples of algebras over polynomial monads. In this part we study these examples in some detail and investigate whether the corresponding polynomial monad is tame or not. We pay a particular attention to the polynomial monad for $n$-operads for reasons explained in the introduction to this article. This special case turns out to be the most intricate one with respect to tameness.

In each case, the definition of the generating polynomial necessitates a rigorous definition of a certain class of graphs, together with the appropriate notion of \emph{graph insertion}. This graph insertion is responsible for the multiplication of the associated polynomial monad. We refer the reader to Part 4 for our terminology and conventions concerning graphs, trees and graph insertion.

\section{Operads based on contractible graphs}

\subsection{Diagram categories}As a first example of tame polynomial monad we consider `linear' monads. These are polynomial monads for which the middle map is the identity:\begin{diagram}I&\lTo^s&B&\rTo^{id}&B&\rTo^t I .\end{diagram}The category of linear monads is isomorphic to the category of small categories. If $C$ is a linear monad and $\CC$ the corresponding small category then the category of algebras over $C$  is the category of diagrams $[\CC,\Ee].$ To see that any linear monad is tame we need to compute the classifier ${\bf C}^{\tt C+1}.$ Let $i$ be an object of $\CC.$  The objects of  ${\bf C}^{\tt C+1}(i)$ are morphisms $f: j\to i$ coloured by two colours $X$ or $K.$ A (non-identity) morphism from $f:j_1\to i$ to $g:j_2\to i$ can exist only if $f$ and $g$ both have colour $X$ and there is a morphism $h:j_1\to j_2$ in $\CC$ such that $f = g\cdot h.$ In other words, the semi-free coproduct classifier ${\bf C}^{\tt C+1}(i)$ is isomorphic to a coproduct of overcategories $\CC/i$ together with a coproduct of as many terminal categories as there are non-identity morphisms in $\CC$. Since each overcategory $\CC/i$ has a terminal object we conclude that the linear monad $C$ is tame.

\subsection{Monoids and non-symmetric operads}\label{SSM}

Let $M$ be free monoid monad on $\Set$. It is a polynomial monad generated by the following polynomial\begin{diagram}1&\lTo^s&\LTrees^*&\rTo^p&\LTrees&\rTo^t 1\end{diagram}in which $\LTrees$ is the set of (isomorphism classes of finite) linear rooted trees, and $\LTrees^*$ is the set of (isomorphism classes of finite) linear rooted trees with one marked vertex. The mapping $p\,$ forgets the marking. The multiplication of $M$ is induced by insertion of linear trees into vertices of linear trees. The category ${\bf M}^{\tt M+1}$ can be described as follows, cf. Section \ref{semifreeclassifier}.

The objects of ${\bf M}^{\tt M+1}$ are corollas with vertex decorated by a linear tree and edges coloured by $X$ or $K.$ The edges of the corolla correspond to vertices of the decorating tree. Therefore, such a corolla can be considered as a linear tree with vertices labelled by $X$ or $K$. The morphisms of ${\bf M}^{\tt M+1}$ are generated by contractions of linear subtrees with $X$-coloured vertices to a single $X$-coloured vertex

   {\unitlength=0.9mm

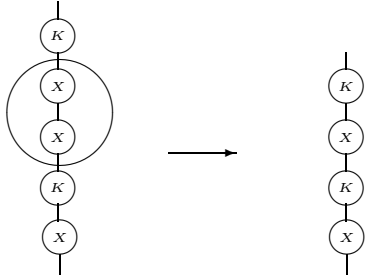
\begin{picture}(60,42)(-25,20)

\begin{picture}(10,10)(0,-11)

\put(14,22.5){\circle{5}}
\put(14,17.6){\line(0,1){2.5}}
\put(14,22.5){\makebox(0,0){\mbox{$\scriptscriptstyle K$}}}
\put(14,30){\circle{5}}
\put(14,27.5){\line(0,-1){2.5}}
\put(14,30){\makebox(0,0){\mbox{$\scriptscriptstyle X$}}}
\put(14,37.5){\circle{5}}
\put(14,35){\line(0,-1){2.5}}
\put(14,37.5){\makebox(0,0){\mbox{$\scriptscriptstyle X$}}}
\put(14,45){\circle{5}}
\put(14,42.5){\line(0,-1){2.5}}
\put(14,45){\makebox(0,0){\mbox{$\scriptscriptstyle K$}}}
\put(14,50){\line(0,-1){2.5}}

\end{picture}

\begin{picture}(10,10)(11,-3.7)

\put(14,22.5){\circle{5}}
\put(14,17){\line(0,1){3}}
\put(14,22.5){\makebox(0,0){\mbox{$\scriptscriptstyle X$}}}
\put(30,35){\vector(1,0){10}}
\put(14,41){\circle{15}}

\end{picture}

\begin{picture}(0,0)(-20,-11)

\put(14,22.5){\circle{5}}
\put(14,17.6){\line(0,1){2.5}}
\put(14,22.5){\makebox(0,0){\mbox{$\scriptscriptstyle K$}}}
\put(14,30){\circle{5}}
\put(14,27.5){\line(0,-1){2.5}}
\put(14,30){\makebox(0,0){\mbox{$\scriptscriptstyle X$}}}
\put(14,37.5){\circle{5}}
\put(14,35){\line(0,-1){2.5}}
\put(14,37.5){\makebox(0,0){\mbox{$\scriptscriptstyle K$}}}
\put(14,42.5){\line(0,-1){2.5}}

\end{picture}

\begin{picture}(0,0)(-18.8,-3.7)

\put(14,22.5){\circle{5}}
\put(14,17){\line(0,1){3}}
\put(14,22.5){\makebox(0,0){\mbox{$\scriptscriptstyle X$}}}

\end{picture}
\end{picture}}

\noindent and insertions of a single $X$-coloured vertex into an edge:

   {\unitlength=0.9mm

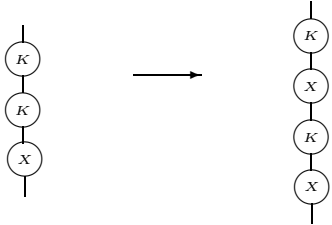
\begin{picture}(60,38)(-25,22)

\begin{picture}(10,10)(0,-15)

\put(14,22.5){\circle{5}}
\put(14,17.6){\line(0,1){2.5}}
\put(14,22.5){\makebox(0,0){\mbox{$\scriptscriptstyle K$}}}
\put(14,30){\circle{5}}
\put(14,27.5){\line(0,-1){2.5}}
\put(14,30){\makebox(0,0){\mbox{$\scriptscriptstyle K$}}}
\put(14,35){\line(0,-1){2.5}}

\end{picture}

\begin{picture}(10,10)(11,-7.7)

\put(14,22.5){\circle{5}}
\put(14,17){\line(0,1){3}}
\put(14,22.5){\makebox(0,0){\mbox{$\scriptscriptstyle X$}}}
\put(30,35){\vector(1,0){10}}

\end{picture}

\begin{picture}(0,0)(-20,-11)

\put(14,22.5){\circle{5}}
\put(14,17.6){\line(0,1){2.5}}
\put(14,22.5){\makebox(0,0){\mbox{$\scriptscriptstyle K$}}}
\put(14,30){\circle{5}}
\put(14,27.5){\line(0,-1){2.5}}
\put(14,30){\makebox(0,0){\mbox{$\scriptscriptstyle X$}}}
\put(14,37.5){\circle{5}}
\put(14,35){\line(0,-1){2.5}}
\put(14,37.5){\makebox(0,0){\mbox{$\scriptscriptstyle K$}}}
\put(14,42.5){\line(0,-1){2.5}}

\end{picture}

\begin{picture}(0,0)(-18.8,-3.7)

\put(14,22.5){\circle{5}}
\put(14,17){\line(0,1){3}}
\put(14,22.5){\makebox(0,0){\mbox{$\scriptscriptstyle X$}}}

\end{picture}

\end{picture}}

\vspace{1ex}
\noindent Obviously, every connected component of ${\bf M}^{\tt M+1}$ contains a terminal object which is a linear tree the vertices of which have alternating colours starting with $X$ and terminating with $X$:

   {\unitlength=0.9mm

\begin{picture}(60,40)(-45,22)

\begin{picture}(10,10)(0,-11)

\put(14,22.5){\circle{5}}
\put(14,17.6){\line(0,1){2.5}}
\put(14,22.5){\makebox(0,0){\mbox{$\scriptscriptstyle K$}}}
\put(14,30){\circle{5}}
\put(14,27.5){\line(0,-1){2.5}}
\put(14,30){\makebox(0,0){\mbox{$\scriptscriptstyle X$}}}
\put(14,37.5){\circle{5}}
\put(14,35){\line(0,-1){2.5}}
\put(14,37.5){\makebox(0,0){\mbox{$\scriptscriptstyle K$}}}
\put(14,45){\circle{5}}
\put(14,42.5){\line(0,-1){2.5}}
\put(14,45){\makebox(0,0){\mbox{$\scriptscriptstyle X$}}}
\put(14,50){\line(0,-1){2.5}}

\end{picture}

\begin{picture}(10,10)(11,-3.7)

\put(14,22.5){\circle{5}}
\put(14,17){\line(0,1){3}}
\put(14,22.5){\makebox(0,0){\mbox{$\scriptscriptstyle X$}}}

\end{picture}

\end{picture}}

\noindent Hence, the free monoid monad $M$ is tame and we obtain in particular formula (\ref{semfreemon}) of the introduction.

Moreover, the objects of the final subcategory $\ts$ of ${\bf M}^{\tt M_{f,g}}$ (cf. Lemma \ref{final}) are linear trees with vertices coloured by $X,K,L$ such that\begin{itemize}\item first and last vertex are coloured by $X;$\item  adjacent vertices have different colours;\item  no edge connects a $K$-vertex with an $L$-vertex.\end{itemize}

Let $X = (R,Y_0,Y_1,u,\alpha)$ be a $M_{f,g}$-algebra, i.e. the data for a pushout along a free map in the category of monoids (see the proof of Theorem \ref{SSformonoids}). Then, according to formula (\ref{SSSR}) the functor $\tilde{X}$ takes on a typical object of ${\bf \ts}$  the value$$R\otimes Y_{i_1}\otimes R\otimes Y_{i_2}\otimes\cdots \otimes Y_{i_k}\otimes R$$where  $(i_1,\dots,i_k)$ is a vertex of a punctured $k$-cube. We thus obtain the Schwede-Shipley formula (\ref{SSformula}) as a special instance of Theorem \ref{filtration} (cf. also \cite[pg. 10]{SS}).

The polynomial monad for non-symmetric operads monad is generated by the following polynomial
$$\begin{diagram}\NN_0&\lTo^s&\RTrees^*&\rTo^p&\RTrees&\rTo^t&\NN_0\end{diagram}$$
in which $\NN_0$ denotes the set of natural numbers including $0$; $\RTrees$ denotes the set of isomorphism classes of planar rooted trees. The elements of $\RTrees^*$ are elements of $\RTrees$ with an additional marking of one vertex. The mapping $p\,$ forgets this marking. The target map takes a planar tree to the cardinality of the set of its leaves and the source map takes a tree $S$ with marked vertex $v$ to the cardinality of the set of leaves of the corolla $cor_v(S)$ surrounding the vertex $v$ in $S$. The multiplication of the polynomial monad is induced by insertion of planar rooted trees into vertices of planar rooted trees.

The polynomial monad $O(1)$ for non-symmetric operads is tame for similar reasons as in the preceding example. The objects of ${\bf O(1)}^{\tt O(1)+1}$ are planar rooted trees the vertices of which are coloured by $X$ and $K.$  Morphisms in ${\bf O(1)}^{\tt O(1)+1}$ are generated by contractions of a subtree with $X$-coloured vertices to a single $X$-coloured vertex, and by insertion of a single $X$-coloured vertex into an edge. A typical terminal object in a connected component of ${\bf O(1)}^{\tt O(1)+1}$ is a planar rooted tree with vertices coloured by $X$ and $K$ such that adjacent vertices have different colours, and such that vertices incident to the root or to the leaves are $X$-coloured. For instance, a tree of the following form is terminal in its connected component:

   {\unitlength=1mm

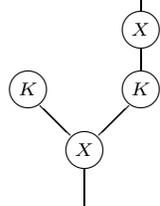
\begin{picture}(30,26)(0,19)

\begin{picture}(10,10)(-35,-2.6)

\put(14,22.5){\circle{5}}
\put(6.5,30.5){\circle{5}}
\put(21.5,30.5){\circle{5}}
\put(6.5,30.5){\makebox(0,0){\mbox{$\scriptstyle  K$}}}
\put(21.5,30.5){\makebox(0,0){\mbox{$\scriptstyle  K$}}}
\put(14,20){\line(0,-1){5}}
\put(12.1,24.4){\line(-1,1){4}}
\put(15.9,24.4){\line(1,1){4}}
\put(14,22.5){\makebox(0,0){\mbox{$\scriptstyle X$}}}
\put(21.5,44){\line(0,-1){3}}
\put(21.5,36){\line(0,-1){3}}
\put(21.5,38.5){\circle{5}}
\put(21.5,38.5){\makebox(0,0){\mbox{$\scriptstyle X$}}}

\end{picture}

\end{picture}}

There is also a corresponding description for the final subcategory $\ts$ of the free non-symmetric operad extension classifier ${\bf O(1)}^{\bf O(1)_{f,g}}$. As an instance of our Theorem \ref{filtration} we obtain Muro's formula \cite{FM} for free non-symmetric operad extensions.

\begin{remark}We can introduce a coloured version for the polynomial monads above. For this we need to use graphs whose edges are coloured by some set of colours $I.$ The coloured version of the free monoid monad gives  the monad for categories with  fixed object-set  $I.$ Algebras over this monad in a symmetric monoidal category $\Ee$  are precisely $\Ee$-enriched categories with object-set $I.$ Similarly, the coloured version of the monad for non-symmetric operads is the monad for multicategories with fixed object-set. These monads are tame by the same argument as for their single-coloured counterparts. A similar remark applies for all other polynomial monads of this article. We will thus not anymore mention the coloured versions.\end{remark}

\subsection{Symmetric operads}\label{SSP}\label{symmetricoperads}

The monad for symmetric operads \cite{BV,May} is generated by the following polynomial
$$\begin{diagram}\NN_0&\lTo^s&\ORTrees^*&\rTo^p&\ORTrees&\rTo^t&\NN_0\end{diagram}$$
in which $\ORTrees$ is the set of isomorphism classes of ordered rooted trees. Such an isomorphism class is represented by a planar rooted tree together with an ordering of its leaves. The structure maps of this polynomial monad are defined in a similar fashion as those of the polynomial monad for non-symmetric operads.

\begin{defin}\label{nd}A vertex $\,v$ of a rooted tree is called non-degenerate (resp. normal) if the set of incoming edges of $v$ is non-empty (resp. of cardinality at least $2$).

A tree is called \emph{regular} (resp. \emph{normal}) if its vertices are non-degenerate (resp. normal). A tree $T$ is called \emph{non-degenerate} if the set of its leaves is non-empty and any non-degenerate vertex belongs to a linear subtree containing a leaf of the tree.\end{defin}


There is a polynomial monad for {\it constant-free} symmetric operads. These are symmetric operads without constant operations. The generating polynomial is
$$\begin{diagram}\NN&\lTo^s&\ORTrees_{reg}^*&\rTo^p&\ORTrees_{reg}&\rTo^t&\NN\end{diagram}$$
where everything is defined as above except that we restrict to \emph{regular} ordered rooted trees.

One can also define a polynomial monad for {\it normal} symmetric operads. These are constant-free symmetric operads with a unique unary operation, i.e. the object of  operations of arity $1$ is the tensor unit $e$ of $\Ee.$ The generating polynomial of the monad for normal symmetric operad is
$$\begin{diagram}\NN&\lTo^s&\ORTrees_{nor}^*&\rTo^p&\ORTrees_{nor}&\rTo^t&\NN\end{diagram}$$
everything being defined as above except that we restrict to {\it normal} ordered rooted trees.

The polynomial monad for constant-free symmetric operads is tame. In an implicit manner, this was first observed by Getzler-Jones \cite[Section 1.5]{GJ}.  As in the non-symmetric case one can characterise the terminal object in each connected component of the internal algebra classifier as alternating trees with two colours $X$ and $K$  (see Section \ref{SSP}). The polynomial monads for normal symmetric operads are also tame. However, the polynomial monad for general symmetric operads is \emph{not} tame. The following tree

{\unitlength=0.9mm
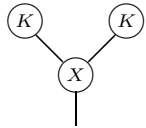
\begin{picture}(30,23)(-43,14)

\put(14,22.5){\circle{5}}
\put(6.5,30.5){\circle{5}}
\put(21.5,30.5){\circle{5}}
\put(6.5,30.5){\makebox(0,0){\mbox{$\scriptstyle  K$}}}
\put(21.5,30.5){\makebox(0,0){\mbox{$\scriptstyle  K$}}}
\put(14,20){\line(0,-1){5}}
\put(12.1,24.4){\line(-1,1){4}}
\put(15.9,24.4){\line(1,1){4}}
\put(14,22.5){\makebox(0,0){\mbox{$\scriptstyle X$}}}

\end{picture}}

\noindent is the only candidate for a terminal object in one of the connected components of corresponding internal algebra classifier but it has a non-trivial automorphism coming from a $\Z_2$-action on $X_2,$ and so, it can not be terminal. There is an obstruction for the existence of model structure on symmetric operads with coefficient in chain complexes over a field of positive characteristic,  similarly to (\ref{counterexample}), which was first described by Fresse \cite{Fresse}.

\begin{remark}\label{Cavigliared}Symmetric operads are often supposed to be \emph{$0$-reduced}, i.e. to have a unique constant (the monoidal unit $e$) in arity $0$. It seems unlikely that $0$-reduced symmetric operads are the algebras for a polynomial monad in sets. We are grateful to Giovanni Caviglia for having pointed this out. There is nevertheless a way of dealing with $0$-reduced symmetric operads from a polynomial monad point of view.

Let us call \emph{reduced symmetric operad} any algebra for the polynomial monad generated by the following polynomial

$$\begin{diagram}\NN&\lTo^s&\ORTrees_{nd}^*&\rTo^p&\ORTrees_{nd}&\rTo^t&\NN\end{diagram}$$ where $\ORTrees_{nd}$ is the set of isomorphism classes of \emph{non-degenerate} ordered rooted trees and $\ORTrees_{nd}^*$ denotes the set of elements of $\ORTrees_{nd}$ with one \emph{non-degenerate} vertex marked (cf. Definition \ref{nd}). The reader should observe that non-degenerate trees can have degenerate vertices. The structure maps of this polynomial are induced by substitution of non-degenerate trees into marked vertices of non-degenerate trees, much like above. 
In virtue of the existing input leaves in any non-degenerate tree, the polynomial monad for reduced symmetric operads is \emph{tame}. 

We denote by $\Ee/e$ the category of objects of $\Ee$ which are augmented over the monoidal unit $e$. The category of reduced symmetric operads in $\Ee/e$ contains then the category of $0$-reduced symmetric operads in $\Ee$ as a full subcategory, cf. Fresse \cite[Theorem 2.2.18]{Fresse2}. Reduced symmetric operads have the advantage over $0$-reduced symmetric operads to be algebras for a tame polynomial monad in sets. A similar observation applies to graphical PROP's versus Adams-Mac Lane PROP's, cf. Remark \ref{Giovanni}.\end{remark}

\subsection{Planar cyclic and cyclic operads}The generating polynomial of the monad for \emph{cyclic operads} is$$\begin{diagram}\NN&\lTo^s&\OTrees^*&\rTo^p&\OTrees&\rTo^t&\NN\end{diagram}$$where $\OTrees$ is the set of isomorphism classes of ordered (non-rooted) trees. In the generating polynomial of the monad for \emph{planar cyclic operads} the set $\OTrees$ has to be replaced with the set $\OPTrees$ of isomorphism classes of ordered planar trees. Neither of these three polynomial monads is tame for similar reasons as above. Nevertheless  we have

\begin{pro}The polynomial monad for normal (constant-free, reduced) cyclic (resp. planar cyclic) operads is tame.\end{pro}

\begin{proof}The terminal objects in the connected components of the internal algebra classifier can be characterised as alternating coloured trees much as in the case of the monad for non-symmetric operads.\end{proof}

\subsection{Dioperads and $\frac{1}{2}$PROP's}

Instead of going into details we refer the reader to \cite{Markl} for precise definitions. We just mention that the polynomial monad for dioperads is based on the class of contractible ordered graphs while the monad for $\frac{1}{2}$PROP's
is based on the class of those ordered contractible graphs, called $\frac{1}{2}$-graphs, which are obtained as two rooted trees glued together along the roots. Both of these polynomial monads are not tame as both contain the monad for symmetric operads as a submonad. Neverteless, if we define a \emph{normal} dioperad (resp. \emph{normal} $\frac{1}{2}$PROP) as one which has no operations of type $A(0,n)$ and no operations of type $A(n,0)$ for $n\geq 0$ while $A(1,1)=e$, then the following statement holds.

\begin{pro}The polynomial monads for normal dioperads and normal $\frac{1}{2}$PROP's are tame.\end{pro}

\noindent The proof is similar to the one for normal symmetric or normal cyclic operads.

\section{Operads based on general graphs}

\subsection{Modular operads}

The monad for \emph{modular operads} is generated by the polynomial$$\begin{diagram}\NN&\lTo^s&\OGraphs^*&\rTo^p&\OGraphs&\rTo^t&\NN\end{diagram}$$in which $\OGraphs$ denotes the set of isomorphism classes of ordered connected graphs (with non-empty set of vertices) and $\OGraphs^*$ is the set of such isomorphism classes with one vertex marked. The source and target maps and the composition operations are defined as above.
The polynomial monad for modular operads is not tame even if we restrict to  normal modular operads. Normality in this case means that modular operads do not have operations whose arities are corollas with less than three flags.  Indeed, in the corresponding internal algebra classifier there is a connected component which contains the following ordered bicoloured graph:

{\unitlength=0.9mm
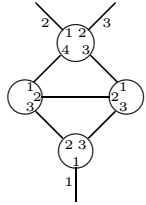
\begin{picture}(30,30)(-43,16)

\put(14,22.5){\circle{5}}
\put(6.5,30.5){\circle{5}}
\put(21.5,30.5){\circle{5}}
\put(7.3,32){\makebox(0,0){\mbox{$\scriptscriptstyle 1$}}}
\put(21,29){\makebox(0,0){\mbox{$\scriptscriptstyle 3$}}}
\put(8.3,30.5){\makebox(0,0){\mbox{$\scriptscriptstyle 2$}}}
\put(19.8,30.5){\makebox(0,0){\mbox{$\scriptscriptstyle 2$}}}
\put(21,32){\makebox(0,0){\mbox{$\scriptscriptstyle 1$}}}
\put(7.3,29){\makebox(0,0){\mbox{$\scriptscriptstyle 3$}}}
\put(9,30.5){\line(1,0){10}}
\put(14,20){\line(0,-1){5}}
\put(12.1,24.4){\line(-1,1){4}}
\put(15.9,24.4){\line(1,1){4}}
\put(7.8,32.7){\line(1,1){4}}
\put(12.1,40.4){\line(-1,1){4}}
\put(15.9,40.4){\line(1,1){4}}
\put(20.2,32.7){\line(-1,1){4}}
\put(14,38.5){\circle{5.5}}
\put(15,23.5){\makebox(0,0){\mbox{$\scriptscriptstyle 3 $}}}
\put(13,23.5){\makebox(0,0){\mbox{$\scriptscriptstyle 2 $}}}
\put(14.1,21){\makebox(0,0){\mbox{$\scriptscriptstyle 1 $}}}
\put(15,40){\makebox(0,0){\mbox{$\scriptscriptstyle 2 $}}}
\put(13,40){\makebox(0,0){\mbox{$\scriptscriptstyle 1 $}}}
\put(12.5,37.5){\makebox(0,0){\mbox{$\scriptscriptstyle 4 $}}}
\put(15.5,37.5){\makebox(0,0){\mbox{$\scriptscriptstyle 3 $}}}
\put(13,18){\makebox(0,0){\mbox{$\scriptscriptstyle 1 $}}}
\put(9.5,41.5){\makebox(0,0){\mbox{$\scriptscriptstyle 2 $}}}
\put(18.6,41.5){\makebox(0,0){\mbox{$\scriptscriptstyle 3 $}}}

\end{picture}}

\noindent in which top and bottom vertices are $X$-coloured and the middle vertices are $K$-coloured. The isomorphism class of this graph can not be contracted further inside the internal algebra classifier but it admits a non-trivial automorphism which consists of the following renumbering of the edges incoming in into $X$-vertices:

{\unitlength=0.9mm

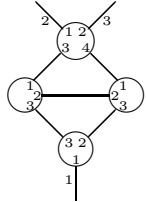
\begin{picture}(30,30)(-43,16)

\put(14,22.5){\circle{5}}
\put(6.5,30.5){\circle{5}}
\put(21.5,30.5){\circle{5}}
\put(7.3,32){\makebox(0,0){\mbox{$\scriptscriptstyle 1$}}}
\put(21,29){\makebox(0,0){\mbox{$\scriptscriptstyle 3$}}}
\put(8.3,30.5){\makebox(0,0){\mbox{$\scriptscriptstyle 2$}}}
\put(19.8,30.5){\makebox(0,0){\mbox{$\scriptscriptstyle 2$}}}
\put(21,32){\makebox(0,0){\mbox{$\scriptscriptstyle 1$}}}
\put(7.3,29){\makebox(0,0){\mbox{$\scriptscriptstyle 3$}}}
\put(9,30.5){\line(1,0){10}}
\put(14,20){\line(0,-1){5}}
\put(12.1,24.4){\line(-1,1){4}}
\put(15.9,24.4){\line(1,1){4}}
\put(7.8,32.7){\line(1,1){4}}
\put(12.1,40.4){\line(-1,1){4}}
\put(15.9,40.4){\line(1,1){4}}
\put(20.2,32.7){\line(-1,1){4}}
\put(14,38.5){\circle{5.5}}
\put(15,23.5){\makebox(0,0){\mbox{$\scriptscriptstyle 2 $}}}
\put(13,23.5){\makebox(0,0){\mbox{$\scriptscriptstyle 3 $}}}
\put(14.1,21){\makebox(0,0){\mbox{$\scriptscriptstyle 1 $}}}
\put(15,40){\makebox(0,0){\mbox{$\scriptscriptstyle 2 $}}}
\put(13,40){\makebox(0,0){\mbox{$\scriptscriptstyle 1 $}}}
\put(12.5,37.5){\makebox(0,0){\mbox{$\scriptscriptstyle 3 $}}}
\put(15.5,37.5){\makebox(0,0){\mbox{$\scriptscriptstyle 4 $}}}
\put(13,18){\makebox(0,0){\mbox{$\scriptscriptstyle 1 $}}}
\put(9.5,41.5){\makebox(0,0){\mbox{$\scriptscriptstyle 2 $}}}
\put(18.6,41.5){\makebox(0,0){\mbox{$\scriptscriptstyle 3 $}}}

\end{picture}}

There is thus no transferred model structure on modular operads under the general assumptions of our Main Theorem \ref{maintheorem}.

\begin{remark}The usual definition of modular operad uses stable graphs \cite{GetzK,Markl}, i.e. graphs decorated by genus and satisfying a stability condition. There is a polynomial monad for this version of modular operad, which is not tame by exactly the same argument as above for non-decorated graphs.\end{remark}

These negative results do not exclude the existence of a transferred model structure on the algebras under some more restrictive conditions on the monoidal model category $\Ee$. The following proposition illustrates in which way specific properties of $\Ee$ can be used to establish the existence of a transfer even for algebras over non-tame polynomial monads.

\begin{pro}The monad for modular operads is $\otimes$-admissible in the monoidal model category $\Ch({\bf k})$ of chain complexes over a field $\bf k$ of characteristic $0$.\end{pro}
\begin{proof}The generating trivial cofibrations are of the form $0\to D$ where $D$ is the chain complex $\ldots\leftarrow 0 \leftarrow {\bf k}\stackrel{id}{\leftarrow}{\bf k}\leftarrow 0\leftarrow \ldots.$ It is thus enough to consider semi-free coproducts of modular operads. The underlying object of a semi-free coproduct $X\vee F(K)$ has a direct summand equal to $X$ with canonical injection $X\to X\vee F(K)$. We have to show that this injection is a quasi-isomorphism for acyclic $K$.

The classifier ${\bf Mod}^{\tt Mod +1}$ contains a final subcategory the objects of which are isomorphism classes of ordered graphs with vertices coloured by $X$ and $K,$ such that internal edges only connect vertices of different colours. Since no further contractions of such graphs are possible the morphisms of this final subcategory are generated by insertion of corollas into $X$-vertices. Therefore the final subcategory is equivalent to a coproduct of finite groups. The semi-free coproduct is the colimit of a functor on this subcategory which assigns to each graph a tensor product of as many $X$'s and $K$'s as there are $X$- and $K$-vertices in the graph. The morphisms of the classifier act by permuting $K$-factors and through the modular operad action on $X$-factors. It follows that the colimit of this functor splits into one component which corresponds to the image of the canonical injection $X\to X\vee F(K)$ and other components which correspond to graphs with at least one $K$-vertex. It thus suffices to show that the latter components are acyclic.

Since our chain complex $K$ is a chain complex of ${\bf k}$-vector spaces, acyclicity of $K$ implies contractibility of $K$. Therefore, the functor on the classifier takes those graphs which have $K$-vertices to contractible chain complexes. Over each connected component of the final subcategory (with $K$-vertices) the functor can thus be considered as a chain complex in the abelian category of diagrams over this connected component. As we have seen the latter is equivalent to the category of representations of a finite group. Since ${\bf k}$ has characteristic $0$ this representation category is semisimple so that every chain complex in it is cofibrant. Since taking the colimit is a left Quillen functor we conclude that the colimit of $F$ is acyclic on each connected component containing $K$-vertices, as required.\end{proof}

\subsection{Properads, PROP's}

The generating polynomials of the monad for properads and PROP's are defined in complete analogy with the preceding section, by specifying the appropriate insertional class of graphs. We refer the reader to Markl \cite{Markl} and Johnson-Yau \cite{JY} for an explicit link with the original definitions of properads by Vallette \cite{V} and of PROP's by Adams-Mac Lane. The insertional class of graphs for PROP's (properads) consists of all (connected) directed loop-free graphs.

In the \emph{normal} versions there are no operations of type $A(0,n)$ and no operations of type $A(n,0)$ for $n\geq 0$, while $A(1,1)=e.$

\begin{remark}\label{Giovanni}Our definition of PROP as an algebra of a polynomial monad is slightly weaker than the classical definition as a strict symmetric monoidal category whose set of objects is the commutative monoid of natural numbers.
The difference between the classical and our graphical definition of PROP concerns the structure of vertical composition. In classical PROP's there are two a priori different compositions, horizontal and vertical:
$$ \circ_h: A(n,0)\times A(0,m)\to  A(n,m) \ , \   \circ_v: A(n,0)\times A(0,m)\to  A(n,m). $$
In particular $A(0,0)$ carries two multiplications which satisfy the middle interchange relation and thus make $A(0,0)$ a commutative monoid by the classical Eckmann-Hilton argument. Therefore, the symmetric group action in the coloured operad for classical PROP's is not free so that this operad does not correspond to any polynomial monad. In graphical PROP's however the vertical composition $   \circ_v: A(n,k)\times A(k,m)\to  A(n,m)  $ is represented by the directed graph

 \begin{center}
\scalebox{1} 
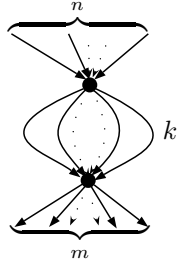
{
\begin{pspicture}(0,-1.70)(8.48,1.70)
\psline[linewidth=0.02cm,arrowsize=0.05cm 3.0,arrowlength=1.4,arrowinset=0.0]{->}(3.29,1.28)(4.29,0.67)
\psline[linewidth=0.02cm,arrowsize=0.05cm 3.0,arrowlength=1.4,arrowinset=0.0]{->}(4.05,1.28)(4.33,0.67)
\psline[linewidth=0.02cm,arrowsize=0.05cm 3.0,arrowlength=1.4,arrowinset=0.0]{->}(5.11,1.28)(4.39,0.67)
\psdots[dotsize=0.2](4.33,0.60)
\psdots[dotsize=0.2](4.31,-0.67)
\psline[linewidth=0.02cm,arrowsize=0.05cm 3.0,arrowlength=1.4,arrowinset=0.0]{<-}(3.33,-1.29)(4.25,-0.71)
\psline[linewidth=0.02cm,arrowsize=0.05cm 3.0,arrowlength=1.4,arrowinset=0.0]{<-}(5.11,-1.29)(4.39,-0.73)
\psline[linewidth=0.02cm,arrowsize=0.05cm 3.0,arrowlength=1.4,arrowinset=0.0]{<-}(3.79,-1.29)(4.27,-0.75)
\psline[linewidth=0.02cm,arrowsize=0.05cm 3.0,arrowlength=1.4,arrowinset=0.0]{<-}(4.69,-1.31)(4.35,-0.75)
\psbezier[linewidth=0.02,arrowsize=0.05cm 3.0,arrowlength=1.4,arrowinset=0.0]{->}(4.41,0.55)(5.77,-0.09)(4.99,-0.31)(4.39,-0.67)
\psbezier[linewidth=0.02,arrowsize=0.05cm 3.0,arrowlength=1.4,arrowinset=0.0]{->}(4.25,0.55)(2.79,-0.07)(3.72,-0.27)(4.21,-0.63)
\psbezier[linewidth=0.02,arrowsize=0.05cm 3.0,arrowlength=1.4,arrowinset=0.0]{->}(4.29,0.55)(3.75,0.00)(3.85,-0.11)(4.29,-0.63)
\psbezier[linewidth=0.02,arrowsize=0.05cm 3.0,arrowlength=1.4,arrowinset=0.0]{->}(4.35,0.55)(4.95,0.12)(4.71,-0.25)(4.35,-0.61)
\psline[linewidth=0.02cm,linestyle=dotted,dotsep=0.16cm,arrowsize=0.05cm 2.0,arrowlength=1.4,arrowinset=0.4]{->}(4.31,1.28)(4.33,0.70)
\psline[linewidth=0.02cm,linestyle=dotted,dotsep=0.16cm,arrowsize=0.05cm 2.0,arrowlength=1.4,arrowinset=0.4]{->}(4.61,1.28)(4.37,0.70)
\psline[linewidth=0.02cm,linestyle=dotted,dotsep=0.16cm,arrowsize=0.05cm 2.0,arrowlength=1.4,arrowinset=0.4]{->}(4.29,-0.77)(4.07,-1.27)
\psline[linewidth=0.02cm,linestyle=dotted,dotsep=0.16cm,arrowsize=0.05cm 2.0,arrowlength=1.4,arrowinset=0.4]{->}(4.31,-0.77)(4.45,-1.27)
\psbezier[linewidth=0.02,linestyle=dotted,dotsep=0.16cm,arrowsize=0.05cm 2.0,arrowlength=1.4,arrowinset=0.4]{->}(4.31,0.52)(3.99,-0.03)(4.11,-0.21)(4.33,-0.63)
\psbezier[linewidth=0.02,linestyle=dotted,dotsep=0.16cm,arrowsize=0.05cm 2.0,arrowlength=1.4,arrowinset=0.4]{->}(4.35,0.56)(4.63,0.00)(4.55,-0.21)(4.33,-0.59)
\usefont{T1}{ptm}{m}{n}
\rput(4.17,1.51){$\overbrace{ \ \ \ \ \ \ \ \ \ \ \ \ \ \ \ \ \ \ \ \ \ }^{\Large n} $}
\usefont{T1}{ptm}{m}{n}
\rput(5.40,0.00){$\Large k$}
\usefont{T1}{ptm}{m}{n}
\rput(4.20,-1.48){$\underbrace{\ \ \ \ \ \ \ \ \ \ \ \ \ \ \ \ \ \ \ \ \ }_{\Large m}$}
\end{pspicture}
}
\end{center}
\noindent If $k = 0$ there is no such graph representing vertical composition; rather there is only a graph with two connected components which represents horizontal composition. Therefore, in the case of graphical PROP's $A(0,0)$ carries only one composition (the horizontal) and is thus not necessarily a commutative monoid. It is not hard to see, however, that the coloured operad for classical PROP's is a canonical quotient of the coloured operad for graphical PROP's, and that for normalised PROP's there is no difference at all between the classical and our definition.\end{remark}

\begin{pro}The polynomial monads for (normal) properads and PROP's are not tame.\end{pro}

\begin{proof}This negative result follows from the fact that both category of PROP's and properads contain the category of symmetric operads. In the normal case observe that the following graph lives in one of the connected components of ${\bf T}^{\tt T+1}$:

{\unitlength=0.9mm

\begin{picture}(30,30)(-43,16)

\put(14,22.5){\circle{5}}
\put(6.5,30.5){\circle{5}}
\put(21.5,30.5){\circle{5}}
\put(7.3,32){\makebox(0,0){\mbox{$\scriptscriptstyle 1$}}}
\put(21,29){\makebox(0,0){\mbox{$\scriptscriptstyle 2$}}}
\put(21,32){\makebox(0,0){\mbox{$\scriptscriptstyle 1$}}}
\put(7.3,29){\makebox(0,0){\mbox{$\scriptscriptstyle 2$}}}
\put(14,20){\vector(0,-1){5}}
\put(8.1,28.4){\vector(1,-1){4}}
\put(19.9,28.4){\vector(-1,-1){4}}
\put(11.8,36.7){\vector(-1,-1){4}}
\put(14,45.3){\vector(0,-1){4}}
\put(16.2,36.7){\vector(1,-1){4}}
\put(14,38.5){\circle{5.5}}
\put(15,23.5){\makebox(0,0){\mbox{$\scriptscriptstyle 3 $}}}
\put(13,23.5){\makebox(0,0){\mbox{$\scriptscriptstyle 2 $}}}
\put(14.1,21){\makebox(0,0){\mbox{$\scriptscriptstyle 1 $}}}
\put(14,40){\makebox(0,0){\mbox{$\scriptscriptstyle 1 $}}}
\put(12.5,37.5){\makebox(0,0){\mbox{$\scriptscriptstyle 2 $}}}
\put(15.5,37.5){\makebox(0,0){\mbox{$\scriptscriptstyle 3 $}}}
\put(13,18){\makebox(0,0){\mbox{$\scriptscriptstyle 2 $}}}
\put(13,43){\makebox(0,0){\mbox{$\scriptscriptstyle 1 $}}}

\end{picture}}

\noindent in which top and bottom vertices are $X$-couloured and the middle vertices are $K$-colour. The isomorphism class of this graph can not be contracted further in ${\bf T}^{\tt T+1}$ but it admits a non-trivial automorphism consisting of the following renumbering of the edges incoming into the $X$-vertices:

{\unitlength=0.9mm
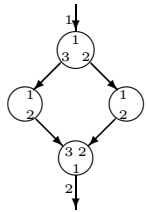
\begin{picture}(30,30)(-43,16)

\put(14,22.5){\circle{5}}
\put(6.5,30.5){\circle{5}}
\put(21.5,30.5){\circle{5}}
\put(7.3,32){\makebox(0,0){\mbox{$\scriptscriptstyle 1$}}}
\put(21,29){\makebox(0,0){\mbox{$\scriptscriptstyle 2$}}}
\put(21,32){\makebox(0,0){\mbox{$\scriptscriptstyle 1$}}}
\put(7.3,29){\makebox(0,0){\mbox{$\scriptscriptstyle 2$}}}
\put(14,20){\vector(0,-1){5}}
\put(8.1,28.4){\vector(1,-1){4}}
\put(19.9,28.4){\vector(-1,-1){4}}
\put(11.8,36.7){\vector(-1,-1){4}}
\put(14,45.3){\vector(0,-1){4}}
\put(16.2,36.7){\vector(1,-1){4}}
\put(14,38.5){\circle{5.5}}
\put(15,23.5){\makebox(0,0){\mbox{$\scriptscriptstyle 2$}}}
\put(13,23.5){\makebox(0,0){\mbox{$\scriptscriptstyle  3$}}}
\put(14.1,21){\makebox(0,0){\mbox{$\scriptscriptstyle 1 $}}}
\put(14,40){\makebox(0,0){\mbox{$\scriptscriptstyle 1 $}}}
\put(12.5,37.5){\makebox(0,0){\mbox{$\scriptscriptstyle 3 $}}}
\put(15.5,37.5){\makebox(0,0){\mbox{$\scriptscriptstyle 2 $}}}
\put(13,18){\makebox(0,0){\mbox{$\scriptscriptstyle 2 $}}}
\put(13,43){\makebox(0,0){\mbox{$\scriptscriptstyle 1 $}}}

\end{picture}}

\end{proof}

\subsection{Wheeled operads, properads and  PROP's}

There is also a ``wheeled'' version of the notions of operad, properad and PROP, due to Markl-Merkulov-Shadrin \cite{Merk}. The polynomial monad for the wheeled version is defined by allowing certain loops in the insertional class of graphs of the ``non-wheeled'' version. For instance, the insertional class of graphs for wheeled operads contains rooted trees as well as graphs obtained from rooted trees by identifying the root with one of the leaves of the tree (see \cite{Merk} for details).

\begin{pro}The polynomial monads for wheeled normal properads and PROP's are not tame.\end{pro}

\begin{proof} The classifier ${\bf T}^{\tt T +1}$ for wheeled version of properads and PROP's  contains more objects and admit more contractions than in nonwheeled version since directed loops are allowed (see \cite[Remark after Def.2.1.8]{Merk}).
In wheeled PROP's case  we can multiply $X$-vertices (something which is not allowed in properads) and contract any edges between $X$-vertices (which is not allowed in ``unwheeled'' PROP's). Doing this operation we end up with a graph which has only one $X$-vertex connected to $K$-vertices. Nevertheless the following graph, which does not admit any further contraction, has a non-trivial automorphism in ${\bf T}^{\tt T +1}$:

{\unitlength=0.9mm

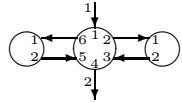
\begin{picture}(30,15)(-43,23)

\put(4,30.5){\circle{5}}
\put(24,30.5){\circle{5}}
\put(14,30.5){\circle{7}}
\put(5.1,31.7){\makebox(0,0){\mbox{$\scriptscriptstyle 1$}}}
\put(23,29.3){\makebox(0,0){\mbox{$\scriptscriptstyle 2$}}}
\put(23,31.7){\makebox(0,0){\mbox{$\scriptscriptstyle 1$}}}
\put(5.1,29.3){\makebox(0,0){\mbox{$\scriptscriptstyle 2$}}}
\put(6.1,29){\vector(1,0){5.2}}
\put(22,29){\vector(-1,0){5.2}}
\put(11.4,32){\vector(-1,0){5.2}}
\put(14,37.3){\vector(0,-1){3.8}}
\put(14,27.4){\vector(0,-1){4.5}}
\put(16.8,32){\vector(1,0){5.2}}
\put(13,25.5){\makebox(0,0){\mbox{$\scriptscriptstyle 2 $}}}
\put(13,36.5){\makebox(0,0){\mbox{$\scriptscriptstyle 1 $}}}
\put(15.8,31.7){\makebox(0,0){\mbox{$\scriptscriptstyle 2$}}}
\put(15.8,29.3){\makebox(0,0){\mbox{$\scriptscriptstyle 3$}}}
\put(12.2,31.7){\makebox(0,0){\mbox{$\scriptscriptstyle 6$}}}
\put(12.2,29.3){\makebox(0,0){\mbox{$\scriptscriptstyle 5$}}}
\put(14.1,32.5){\makebox(0,0){\mbox{$\scriptscriptstyle 1$}}}
\put(14,28.3){\makebox(0,0){\mbox{$\scriptscriptstyle 4 $}}}

\end{picture}}

\noindent In this graph the central vertex is $X$-coloured  and all the other vertices are $K$-coloured. The non-trivial automorphism is generated by a substitution which interchanges $2$ and $6$, and $5$ and $3$ inside the $X$-vertex.\end{proof}

On the positive side we have the following statement that  tame polynomial monads  based on noncontractible graphs do exist.

\begin{pro}The polynomial monads for normal (constant-free, reduced) wheeled operads are tame.\end{pro}
We leave the proof to the interested reader as it is very much analogous to the case of symmetric operads.

\section{Baez-Dolan $+$-construction for polynomial monads}

With any polynomial monad $T$ one can associate another polynomial monad $T^+$, the so-called Baez-Dolan $+$-construction of $T$, see \cite{BD2,Kock, KJBM, Leinster}. If $T$ is generated by the polynomial $P$\begin{diagram}I&\lTo^s&E&\rTo^p&B&\rTo^t&I\end{diagram}then the generating polynomial $P^+$ for  $T^{+}$ is \begin{diagram}B&\lTo^{s^+}&tree^*(P)&\rTo^{p^+}&tree(P)&\rTo^{t^+}&B\end{diagram}where $tree(P)$ is the set of $P$-trees, and $tree^*(P)$ is the set of $P$-trees with a marked vertex. Recall \cite{Kock, KJBM} that a \emph{$P$-tree} is an isomorphism class of rooted trees whose edges are coloured by the elements of $I$, and whose vertices are decorated by $T$-operations in the sense of Remark \ref{operation}, in such a way that the sources and the target of these $T$-operations coincide with the given edge-colouring of the $P$-tree.

The source map $s^+$ of the generating polynomial of $T^{+}$ returns the $T$-operation decorating the marked vertex, the middle map $p^+$ forgets marking, and the target map $t^+$ computes the composite $T$-operation of the whole $P$-tree.

The monad $M$ for monoids is the $+$-construction of the identity monad on $Set.$ The monad $O(1)$ for non-symmetric operads is the $+$-construction of the monad $M$. In general, the characteristic property of the $+$-construction is that \emph{$T^+$-algebras} can be identified with \emph{cartesian monads over $T$} (i.e. $T$-operads in the terminology of Leinster \cite{Leinster}). For instance, non-symmetric operads can be identified with cartesian monads over $M$. In particular, there is a natural notion of \emph{algebra} over a $T^+$-algebra, namely an algebra for the corresponding cartesian monad over $T$.

\begin{theorem}For any polynomial monad $T$, the polynomial monad  $T^+$ is tame.\end{theorem}
\begin{proof}The objects of $({\bf T^{+}})^{\scriptstyle \tt T^+ +1}$ are $P$-trees with an additional colouring of their vertices by two colours $X$ and $K.$ The morphisms are generated by contraction to corolla of subtrees all of whose vertices are $X$-coloured, or introducing a new $X$-vertex on an edge coloured by an $i\in I.$ The decorating element in this vertex is $1_i\in B.$ One can always contract an object in $({\bf T^{+}})^{\scriptstyle \tt T^+ +1}$ to an object in which the vertex colours alternate starting with $X$ and ending with $X$ (cf. Sections \ref{SSM} and \ref{SSP}). Associativity and unitality of the multiplication of $T$ imply that such a contraction represents a unique morphism in $({\bf T^{+}})^{\scriptstyle \tt T^+ +1}.$ Hence, these objects are terminal in their connected components.\end{proof}

This theorem permits the definition of \emph{homotopy $T$-algebras} for any polynomial monad $T$ and choice of monoidal model category $\Ee$ fulfilling the hypotheses of Theorem \ref{maintheorem}. Indeed, the terminal object $Ass^T$ of the category of $T^+$-algebras corresponds to the identity morphism of cartesian monads  $id:T\to T,$ so that the category of algebras  of $Ass^T$ is isomorphic to the category of algebras of $T. $ Since $T^+$ is tame polynomial we can consider a cofibrant resolution $cAss^T_\Ee\to Ass_\Ee^T$ in the category $\Alg_{T^+}(\Ee)$ of $T^+$-algebras in $\Ee$. The $cAss^T_\Ee$-algebras should then be considered as \emph{homotopy $T$-algebras} in $\Ee$.

For instance, if $M$ is the polynomial monad for monoids then $Ass^M$ induces the classical non-symmetric operad $Ass^M_\Ee$ for monoids in $\Ee$, and $cAss^M_\Ee$ is a cofibrant resolution in the category of non-symmetric operads in $\Ee$, i.e. an \emph{$A_\infty$-operad}. Similarly, the cofibrant operad $cAss^{M^+}_\Ee$ is a coloured operad in $\Ee$ whose algebras can be considered as (non-symmetric) ``homotopy operads'' in $\Ee$.

One can therefore, for any polynomial monad $T$ and any monoidal model category $\Ee$ fulfilling the hypotheses of Theorem \ref{maintheorem}, embed the category of $T$-algebras in $\Ee$ in a larger category of homotopy $T$-algebras. We conjecture that in virtue of the cofibrancy of $cAss^T_\Ee$ in $\Alg_{T^+}(\Ee)$ there exists a transferred model structure on the category of homotopy $T$-algebras under very mild additional hypotheses on $\Ee$.

\section{Higher operads}\label{higheroperads}

For the convenience of the reader we recall here the definition of the higher operads of the first author. In particular, we describe them as algebras over a polynomial monad, closely following \cite{SymBat, EHBat}. This subsumes Section \ref{SSM} since $0$-operads are monoids, and $1$-operads are non-symmetric operads. We also define the monads for various interesting subcategories of the category of $n$-operads. The monads for constant-free, reduced and normal $n$-operads are tame polynomial, while the monad for general $n$-operads is polynomial, but not tame polynomial if $n\geq 2$.

\subsection{Complementary relations and  $n$-ordinals}-\vspace{1ex}

For all what follows $n$ denotes a fixed positive integer.

\begin{defin}\label{ordinals}An \emph{$n$-ordered set} $X$ is a set equipped with $n$ binary \emph{antireflexive} relations  $<_0, \ldots, <_{n-1}$ which are \emph{complementary} in the following sense:
\begin{itemize}\item[(i)]for every pair $(a,b)$ of distinct elements of $X$ there is one and only one relation $<_p$ such that
 $a<_p b$ or $b<_p a.$\end{itemize}

\noindent An \emph{$n$-ordinal} $T$ is a finite $n$-ordered set such that for any $a,b,c$ in $T$
  \begin{itemize}\item[(ii)] if $a<_p b$ and $b<_q c$ then $a<_{min(p,q)} c .$
 \end{itemize}

\noindent A \emph{map of $n$-ordered sets} $\sigma:X\to Y$ is a mapping of the underlying sets such that the relation $a<_p b$ in $X$ implies one of the following three types of relations in $Y$:
 \begin{enumerate}
\item $\sigma(a) <_q \sigma(b)$ for $p\leq q\quad$ or
\item $\sigma(a)= \sigma(b)\quad$ or
\item $\sigma(b) <_q \sigma(a)$ for $p<q.$
\end{enumerate}\end{defin}

\noindent The category of $n$-ordered sets and maps between them will be denoted $\mathrm{Rel}(n)$. The full subcategory of $\mathrm{Rel}(n)$ spanned by the $n$-ordinals will be denoted $\Ord(n)$. There is an obvious forgetful functor $|\!-\!|:\mathrm{Rel}(n)\rightarrow\FinSet$ which forgets the relations. The \emph{cardinality}   of an $n$-ordered set $X$ is the cardinality of its underlying set $|X|$.

As there are no non-trivial automorphisms in $\Ord(n)$ we will assume that each isomorphism class of $n$-ordinals contains a single element, i.e. $\Ord(n)$ is \emph{skeletal}.

Each $n$-ordinal can be represented as a pruned planar rooted tree with $n$ levels (\emph{pruned $n$-tree} for short), cf. \cite[Theorem 2.1]{SymBat}. For example, the $2$-ordinal
\[
0 <_0 2,\ 0 <_0 3,\ 0 <_0 4,\ 1 <_0 2,\  1<_0 3,\ 1 <_0 4,\
0 <_1 1,\ 2 <_1 3,\ 2 <_1 4,\ 3 <_1 4,
\]
is represented by the following pruned tree
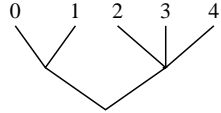
\begin{figure}[h]
\[
\psscalebox{.4 .4} 
{
\begin{pspicture}(0,-1.7822068)(6.81,1.78)
\psline[linecolor=black, linewidth=0.05](1.06,-0.35)(3.06,-1.75)(5.06,-0.35)(4.96,-0.35)
\psline[linecolor=black, linewidth=0.05](0.06,1.04)(1.06,-0.35)(2.06,1.04)(2.06,1.04)
\psline[linecolor=black, linewidth=0.05](3.46,1.04)(5.06,-0.35)(6.66,1.04)(6.66,1.04)
\psline[linecolor=black, linewidth=0.05](5.06,1.04)(5.06,-0.35)(5.06,-0.35)
\usefont{T1}{ptm}{m}{n}
\rput(0.06,1.54){\Huge 0}
\usefont{T1}{ptm}{m}{n}
\rput(2.06,1.54){\Huge 1}
\usefont{T1}{ptm}{m}{n}
\rput(3.46,1.54){\Huge 2}
\usefont{T1}{ptm}{m}{n}
\rput(5.06,1.54){\Huge 3}
\usefont{T1}{ptm}{m}{n}
\rput(6.66,1.54){\Huge 4}
\end{pspicture}
}\]
\caption{\label{symbol}A pruned $2$-tree.}
\end{figure}

The initial $n$-ordinal $z^nU_0$ has empty underlying set and its representing pruned $n$-tree is degenerate: it has no edges but consists only of the root at level $0$. The terminal $n$-ordinal $U_n$ is represented by a linear tree with $n$ levels.

\begin{defin}\label{domination}An $n$-ordered set $X$ is said to \emph{dominate} an $n$-ordered set $\,Y$ if there is a map of $n$-ordered sets $X\to Y$ inducing the identity on underlying sets.\end{defin}

We also would like to consider the limiting case of $\infty$-ordinals.

\begin{defin}Let $T$ be a finite set equipped with a sequence of  binary antireflexive  complimentary relations  $<_0,<_{-1} \ldots, <_p , <_{p-1} \ldots   $ for all integers $p\le 0.$ The set
  $T$  is called an $\infty$-ordinal if these relations satisfy:
  \begin{itemize}
 \item \  $a<_p b $ \ and \ $b<_q c$ \ implies\
 $a<_{min(p,q)} c .$
 \end{itemize}
 \end{defin}

The definition of morphism between $\infty$-ordinals coincides with the definition of morphism between $n$-ordinals for finite $n.$
The category $Ord(\infty)$ denotes {\it the skeletal category of $\infty$-ordinals.}

For an $n$-ordinal $R\ ,\ k\le n $ we consider its  {\it vertical suspension} $S(R)$ which is an $(n+1)$-ordinal with the underlying set $R ,$ and the order $<_m $ equal the  order $<_{m-1}$ on $R$ while $<_0$ is empty.

For example, a vertical suspension of the $2$-ordinal from Figure \ref{symbol} is the $3$-ordinal

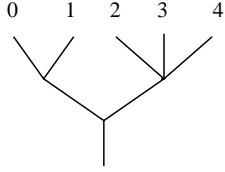
\[
\scalebox{0.4} 
{
\begin{pspicture}(0,-2.795)(7.23,2.81)
\psline[linewidth=0.05](1.18,0.13)(3.18,-1.26)(5.18,0.13)(5.08,0.13)
\psline[linewidth=0.05](0.18,1.53)(1.18,0.13)(2.18,1.53)
\psline[linewidth=0.05](3.58,1.53)(5.18,0.13)(6.78,1.53)
\psline[linewidth=0.05](5.18,1.63)(5.18,0.11)
\usefont{T1}{ptm}{m}{n}
\rput(0.15,2.43){\Huge 0}
\usefont{T1}{ptm}{m}{n}
\rput(2.08,2.43){\Huge 1}
\usefont{T1}{ptm}{m}{n}
\rput(3.56,2.43){\Huge 2}
\usefont{T1}{ptm}{m}{n}
\rput(5.13,2.43){\Huge 3}
\usefont{T1}{ptm}{m}{n}
\rput(6.98,2.43){\Huge 4}
\psline[linewidth=0.05](3.18,-1.25)(3.18,-2.77)
\end{pspicture}
}\]

The suspension operation induces functors
$S_p:Ord(n)\rightarrow Ord(n+1), \ 0\le p \le n.$ We also define an $\infty$-vertical suspension functor $Ord(n)\rightarrow Ord(\infty)$ as follows. For an $n$-ordinal $T$ its $\infty$-suspension is an $\infty$-ordinal $S^{\infty}T$ whose underlying set is the same as the underlying set of $T$ and $a<_p b$ in $S^{\infty}T$ if
$a<_{n+p-1} b$ in $T .$ It is not hard to see that the sequence
$$ Ord(0)\stackrel{S}{\longrightarrow} Ord(1) \stackrel{S}{\longrightarrow} Ord(2) \longrightarrow \ldots \stackrel{S}{\longrightarrow} Ord(n) \longrightarrow \ldots \stackrel{S^{\infty}}{\longrightarrow} Ord(\infty),$$
exhibits $Ord(\infty)$ as a colimit of $Ord(n) .$\vspace{1ex}

\subsection{Fox-Neuwirth stratification of configuration spaces and $n$-ordinals}\label{FNord}

Recall that the moduli space of configurations of $k$ ordered, pairwise distinct points in $\RR^n$ admits a stratification which goes back to Fox-Neuwirth.
Consider the following configuration space
$$F(\RR^n,k)= \{(x_1,\ldots, x_k)\in (\RR^n)^k \ | \  x_i \ne x_j  \  \mbox{if} \  i\ne j \ \}$$
and denote $\stackrel{\scriptscriptstyle o \  \ \ \ \ \ \  }{S^{n-p-1}_\pm}$ the open $(n-p-1)$-hemispheres defined by
\[\stackrel{\scriptscriptstyle o \  \ \ \ \ \ \  }{S^{n-p-1}_{+}}=\left\{x\in \RR^n \left \vert \ \begin{array}[c]{ll}
  x_1^2+\ldots+x_n^2 = 1   \\
   x_{p+1}> 0 \ \mbox{and} \ x_i = 0  \ \mbox{for} \  1\le i \le p
  \end{array} \right. \right\} \]
and
 \[  \stackrel{\scriptscriptstyle o \  \ \ \ \ \ \  }{S^{n-p-1}_{-}} = \left\{ x \in \RR^n \left \vert \ \begin{array}[c]{ll}
  x_1^2+\ldots+x_n^2 = 1   \\
   x_{p+1}< 0 \ \mbox{and} \ x_i = 0  \ \mbox{for} \  1\le i \le p
  \end{array} \right. \right\} \raisebox{-3mm}{.} \]
Let $u_{ij}: F(\RR^n,k)\rightarrow S^{n-1}$ be the function
$$u_{ij}(x_1,\ldots,x_k) =  \frac{x_j - x_i}{||x_j-x_i||} $$
The Fox-Neuwirth cell corresponding to an $n$-ordinal $T$ of cardinality $k$ is
\[ FN_T =  \left\{ x\in F(\RR^n,k) \left |
 \ \begin{array}[c]{ll}
  \ u_{ij}(x)  \in \stackrel{\scriptscriptstyle o \  \ \ \ \ \ \  }{S^{n-p-1}_{+}}&  \ \mbox{\rm if} \ i<_p j  \ \mbox{\rm in} \ T  \\
  &  \\
  \ u_{ij}(x)  \in \stackrel{\scriptscriptstyle o \  \ \ \ \ \ \  }{S^{n-p-1}_{-}}&  \ \mbox{\rm if} \ j<_p i  \ \mbox{\rm in} \ T
  \end{array}\right.  \right\} \raisebox{-7mm}{.}\]
For instance, the Fox-Neuwirth cell which corresponds to the $2$-ordinal below

{\unitlength=1mm

\begin{picture}(60,15)(-40,45)

\begin{picture}(10,5)(0,-28)

\put(5,22){\makebox(0,0){\mbox{$S \ \ =$}}}
\put(12.1,24.4){\line(-1,1){4}}
\put(12.1,24.4){\line(1,1){4}}
\put(12.1,24.4){\line(2,-1){11.2}}

\end{picture}

\begin{picture}(10,10)(-0.1,-28)
\put(12.1,24.4){\line(-1,1){4}}
\put(12.1,24.4){\line(1,1){4}}
\put(12.1,24.4){\line(0,-1){5.8}}
\end{picture}

\begin{picture}(10,10)(-0.2,-28)

\put(12.1,24.4){\line(-1,1){4}}
\put(12.1,24.4){\line(1,1){4}}
\put(12.1,24.4){\line(-2,-1){11.2}}
\end{picture}
\end{picture}}

\noindent  consists of configurations of points which seat on three parallel lines in the prescribed order

\

\scalebox{0.6} 
{
\begin{pspicture}(-8,-1.76)(4.96,1.76)
\usefont{T1}{ptm}{m}{n}
\rput(0.05,-0.97){\Large 1}
\usefont{T1}{ptm}{m}{n}
\rput(2.28,0.28){\Large 3}
\usefont{T1}{ptm}{m}{n}
\rput(4.50,-0.53){\Large 5}
\usefont{T1}{ptm}{m}{n}
\rput(0.09,0.88){\Large 2}
\usefont{T1}{ptm}{m}{n}
\rput(2.27,1.30){\Large 4}
\usefont{T1}{ptm}{m}{n}
\rput(4.49,0.66){\Large 6}
\psline[linewidth=0.04cm](0.48,1.74)(0.50,-1.7)
\psline[linewidth=0.04cm](2.70,1.72)(2.68,-1.74)
\psline[linewidth=0.04cm](4.88,1.74)(4.88,-1.72)
\psdots[dotsize=0.2](0.48,0.9)
\psdots[dotsize=0.2](0.50,-0.92)
\psdots[dotsize=0.2](2.68,0.3)
\psdots[dotsize=0.2](2.70,1.34)
\psdots[dotsize=0.2](4.88,0.7)
\psdots[dotsize=0.2](4.88,-0.52)
\end{pspicture}
}

Each Fox-Neuwirth cell is a convex subspace of $(\RR^n)^k$, open in its closure, and we have a stratification
$$F(\RR^n,k)= \bigcup_{|T|\cong  \{1,\ldots, k\} ,  \pi \in S_k} \pi FN_T.$$
Here $\pi FN_T$ is the space obtained from $FN_T$ by renumbering points according to the permutation $\pi .$

The {domination relation} of Definition \ref{domination} induces a partial ordering of the set of $n$-ordinal structures on a fixed set $\{1,\dots,k\}$ of cardinality $k$. Using Fox-Neuwirth stratification we can  show that the nerve of this poset is homotopy equivalent to  $F(\RR^n,k)$, cf. \cite[Remark 2.2]{SymBat}, \cite[Theorem 5.1]{LocBat}.

\subsection{Fibres and total order}Let $\sigma:T\rightarrow S$ be a map of $n$-ordinals. For each $i\in|S|,$ the set-theoretical fibre $|\sigma|^{-1}(i)\subset|T|$ inherits from $T$ the structure of an $n$-ordinal. This $n$-ordinal will be denoted $\sg^{-1}(i)$ and called the \emph{fibre} of $\sg$ at $i$.

Observe that the underlying set $|S|$ is totally ordered by the relation
 $$a<b \ \ \mbox{if there exists} \ p\in \{0,\ldots,n-1\} \text{ such that } a<_p b.$$
We call this the \emph{total order} on the underlying set $|S|.$ The fibres of a map of $n$-ordinals $\sigma:T\rightarrow S$ will accordingly be represented by an ordered list $(T_0,\ldots,T_k)$ of $n$-ordinals, the ordering being induced by the total order on $|S|$.

Analogously, any two composable maps of $n$-ordinals\begin{diagram}[small]T&\rTo^\sg&S&\rTo^\omega&R\end{diagram}induce a family of maps of $n$-ordinals \begin{diagram}[small](\omega\sg)^{-1}(i)\to\omega^{-1}(i)\end{diagram}indexed by $i\in|R|$. We thus have a list of fibres $(T_0,\dots,T_k)$ for the composite map $\omega\sg$, a list of fibres $(S_0,\dots,S_k)$ for $\omega$, and a list of fibres $(T_i^0,\dots,T_i^{m_i})$ for each map  $(\omega\sg)^{-1}(i)\to\omega^{-1}(i)$ where $i\in|R|$. These notations will be used in the definition of an $n$-operad below.

\begin{defin}A map of $n$-ordinals is a \emph{quasibijection} (resp. \emph{order-preserving}) if it induces a bijection between the underlying sets (resp., preserves the total orders of the underlying sets, i.e. only possibilities $(1)$ and $(2)$ of Definition \ref{ordinals} occur).\end{defin}

\begin{defin}An $n$-collection in a symmetric monoidal category $\Ee$ is a family $(A_T)_{T\in\Ord(n)}$ of objects of $\Ee$ indexed by $n$-ordinals. The category of $n$-collections and levelwise morphisms in $\Ee$ will be denoted by $Coll_n(\Ee)$.\end{defin}

We now recall the definition of a pruned $(n-1)$-terminal $n$-operad \cite{SymBat}. Since we do not need other types of $n$-operads we will call them simply $n$-operads.

\begin{defin}\label{defnoper}An \emph{$n$-operad} in $\Ee$ is
an $n$-collection $(A_T)_{T\in\Ord(n)}$ in $\Ee$ equipped with the following structure:\vspace{1ex}

- a morphism $\epsilon: e \rightarrow  A_{U_n}$ (unit);

- a morphism $m_{\sigma}:A_S\otimes A_{T_0}\otimes \dots \otimes A_{T_k}
 \rightarrow A_T$ (multiplication) for each map of $n$-ordinals $\sigma:T \rightarrow S$.\vspace{1ex}

\noindent They must satisfy the following identities:\vspace{1ex}

- for any composite map of $n$-ordinals\begin{diagram}[small]T&\rTo^\sg&S&\rTo^\omega&R\end{diagram}
the associativity diagram

{\unitlength=1mm

\begin{picture}(300,45)(2,0)

\put(20,35){\makebox(0,0){\mbox{$\scriptstyle A_R\otimes A_{S_{\bullet}}\otimes A_{T_0^{\bullet}}\otimes \dots\otimes A_{T_i^{\bullet}}\otimes\dots\otimes A_{T_k^{\bullet}}$}}}
\put(20,31){\vector(0,-1){12}}
\put(94,31){\vector(0,-1){12}}
\put(88,35){\makebox(0,0){\mbox{$\scriptstyle A_R\otimes A_{S_{0}}\otimes A_{T_1^{\bullet}}\otimes\dots\otimes A_{S_{i}}\otimes A_{T_i^{\bullet}}\otimes\dots\otimes A_{S_{k}}\otimes A_{T_k^{\bullet}}$ }}}
\put(50,35){\makebox(0,0){\mbox{$\scriptstyle \cong $}}}
\put(20,15){\makebox(0,0){\mbox{$\scriptstyle A_S\otimes A_{T_1^{\bullet}}\otimes \dots\otimes A_{T_i^{\bullet}}\otimes  \dots \otimes A_{T_k^{\bullet}}$}}}
\put(94,15){\makebox(0,0){\mbox{$\scriptstyle A_R\otimes A_{T_{\bullet}}$}}}
\put(60,5){\makebox(0,0){\mbox{$ \scriptstyle A_T$}}}
\put(35,11){\vector(4,-1){19}}
\put(85,11){\vector(-4,-1){19}}

\end{picture}}

\noindent commutes,
where $$A_{S_{\bullet}}=A_{S_0}\otimes\dots\otimes A_{S_k},$$
$$A_{T_{i}^{\bullet}}=A_{T_i^0} \otimes \dots\otimes A_{T_i^{m_i}}$$
and $$ A_{T_{\bullet}}=A_{T_0}\otimes \dots\otimes A_{T_k};$$

- for an identity $\sigma = id : T\rightarrow T$ the diagram

{\unitlength=1mm
\begin{picture}(50,25)(30,2)

\put(97,20){\vector(-1,0){20}}
\put(60,17){\vector(0,-1){8}}
\put(60,20){\makebox(0,0){\mbox{\small$A_T\otimes A_{U_n}\otimes \dots \otimes A_{U_n}$}}}
\put(114,20){\makebox(0,0){\mbox{\small$A_T\otimes{e}\otimes \dots \otimes {e}$}}}
\put(60,5){\makebox(0,0){\mbox{\small$A_T$}}}
\put(105,15){\vector(-4,-1){30}}
\put(90,9){\makebox(0,0){\mbox{\small$id$}}}

\end{picture}}

\noindent commutes;

- for the unique morphism $T\rightarrow U_n$ the diagram

{\unitlength=1mm
\begin{picture}(50,25)(30,2)

\put(87,20){\vector(-1,0){15}}
\put(60,17){\vector(0,-1){8}}
\put(60,20){\makebox(0,0){\mbox{\small$A_{U_n}\otimes A_T$}}}
\put(98,20){\makebox(0,0){\mbox{\small$e \otimes A_T$}}}
\put(60,5){\makebox(0,0){\mbox{\small$A_T$}}}
\put(95,17){\vector(-3,-1){25}}
\put(84,11){\makebox(0,0){\mbox{\small$id$}}}

\end{picture}}

\noindent commutes.

\end{defin}

The notion of $n$-operad morphism is obvious and we have a category $O_n(\Ee)$ of $n$-operads.
The forgetful functor $$U_n: O_n(\Ee)\rightarrow Coll_n(\Ee)$$ has a left adjoint and is monadic whenever $\Ee$ is cocomplete.

\subsection{Constant-free, reduced and normal $n$-operads}

\begin{defin}An $n$-ordinal is called regular if it is not the initial  $n$-ordinal.
An $n$-ordinal is called normal if it is neither the initial nor the terminal $n$-ordinal.\end{defin}

There is a category structure on $ROrd(n)$ which we will call the category of {\it regular $n$-ordinals}. The morphisms are those maps of $n$-ordinals which are \emph{surjective} on the underlying sets. This forces the fibres to be regular again.

\begin{defin}A regular $n$-collection in $\Ee$ is a family $(A_T)_{T\in ROrd(n)}$ of objects of $\Ee$ indexed by the set $ROrd(n)$ of regular (i.e. nonempty) $n$-ordinals.

A normal $n$-collection in $\Ee$ is a family $(A_T)_{T\in ROrd(n)}$ of objects of $\Ee$ indexed by the set $NOrd(n)$ of normal $n$-ordinals.
\end{defin}

A  {\it constant-free} $n$-operad is defined in a similar way as an $n$-operad by using regular $n$-collections
and maps of regular $n$-ordinals. This defines the category  of constant-free  $n$-operads $CFO_n(\Ee)$ together with a forgetful functor
$$CFU_n: CFO_n(\Ee)\rightarrow RColl_n(\Ee),$$
 where $CFColl_n(\Ee)$ is the category of constant-free $n$-collections.
This functor has a left adjoint and is monadic whenever $\Ee$ is cocomplete.

Analogously we have a category of \emph{normal $n$-ordinals} $NOrd(n)$ with morphisms being surjective morphisms between normal $n$-ordinals. In the list of fibres of a map of normal $n$-ordinals we include only those fibres which are not equal to the terminal $n$-ordinal.  We have the corresponding categories of {\it normal $n$-collections} $NColl_n(\Ee)$ and {\it normal $n$-operads} $NO_n(\Ee).$ The latter can be considered as the full subcategory of $CFO_n(\Ee)$ consisting of constant-free $n$-operads $A$ such that $A_{U_n}=e.$ The forgetful functor $$NU_n: NO_n(\Ee)\rightarrow NColl_n(\Ee),$$is again monadic whenever $\Ee$ is cocomplete.

The category of {\it reduced $n$-operads} is the category of constant-free $n$-operads with additional structure.
This structure is the structure of a contravariant functor on the category of regular $n$-ordinals and their \emph{injective} order-preserving maps.
Similar to the case of reduced symmetric operads, the category of reduced $n$-operads contains the category of those $n$-operads $A$ such that $A_{z^nU_0}=e$
as a full subcategory, cf. Remark \ref{Cavigliared}.

\subsection{The polynomial monad for $n$-operads}\label{noperads}

\begin{pro}\label{ON} The monad $O(n)$ for $n$-operads is generated by the polynomial
\begin{diagram}{\tt Ord(n)}&\lTo^s&\nPTrees^*&\rTo^p&\nPTrees_{}&\rTo^t&{\tt Ord(n)}\end{diagram}
where
\begin{itemize} \item $\tt Ord(n)$ is the set of isomorphism classes of $n$-ordinals; \item $\nPTrees$ is the set of isomorphism classes of $n$-planar trees; \item $\nPTrees^*$ is the set of elements $\tau\in\nPTrees$ with one vertex marked. \end{itemize}
The structure maps are defined as follows: \begin{itemize}\item The middle map forgets the marking; \item The target map associates to $\tau$ the  $n$-ordinal of its leaves; \item The source map associates to a marked vertex $\tau$ the $n$-ordinal $\tau_v$ decorating the marked vertex $v$. \end{itemize}  The multiplication of the monad is induced by insertion of $n$-planar trees into vertices of   $n$-planar trees and the unit of the monad assigns to an $n$-ordinal $T$ a corolla decorated by $T$ with $T$ as an $n$-ordinal  structure on its leaves.\end{pro}

\begin{proof}We give a sketch of the proof. More details can be found in \cite{SymBat,EHBat}. First, we can show that the data above determine a polynomial monad $O(n).$
Second, any algebra $A$ of $O(n)$ has the structure of an $n$-operad.  The unit of this operad is given by an $n$-planar rooted tree $L_0$ (see (\ref{specialgraph}) for notations) whose target is the terminal $n$-ordinal $U_n$ and whose set of sources is empty because $L_0$ does not have vertices.
To define a multiplication in  $A$ we will associate an $n$-planar tree $[\sigma]$ with each morphism of $n$-ordinals $\sigma:T\to S.$
The set of vertices $\{v_S,v_{T_i},\ldots,v_{T_k}\}$ of the tree $[\sigma]$ is in one-to-one correspondence with the set $\{S, T_1,\ldots T_k\},$ where $T_1,\ldots,T_k$ are fibres of $\sigma.$ The outgoing edge of the  vertex $v_S$ is the root of the tree $[\sigma].$  The elements of $|S|$ are incoming edges of this vertex with its $n$-ordinal structure. The other vertices are all above  $v_S$ and the outgoing edge of $v_{T_i}$ is the $i$-th incoming edge of $v_S.$ The leaves attached to the vertex $v_{T_i}$ correspond to the elements of $|T_i|$ and this set has $T_i$ as its $n$-ordinal structure. Finally, the set of leaves is equipped with the $n$-ordinal structure of $T.$

\scalebox{0.7} 
{
\begin{pspicture}(-2,-1.41)(9.65,2.44)
\psline[linewidth=0.022cm](0.28,1.04)(0.19,0.66)
\psline[linewidth=0.022cm](0.58,0.30)(0.58,0.66)
\psline[linewidth=0.022cm](1.16,1.04)(0.94,0.62)
\psline[linewidth=0.022cm](0.19,0.66)(0.6,0.30)
\psline[linewidth=0.022cm](0.6,0.32)(0.94,0.64)
\psline[linewidth=0.022cm](3.03,1.13)(3.37,0.77)
\psline[linewidth=0.022cm](3.68,1.13)(3.37,0.77)
\psline[linewidth=0.022cm](3.37,0.77)(3.37,0.33)
\psline[linewidth=0.022cm,arrowsize=0.153cm 2.0,arrowlength=1.2,arrowinset=0.6]{->}(1.51,0.72)(2.51,0.72)
\usefont{T1}{ptm}{m}{n}
\rput(5.65,1.72){1}
\usefont{T1}{ptm}{m}{n}
\rput(6.38,1.72){3}
\psline[linewidth=0.022cm](7.27,-0.54)(7.55,-0.90)
\psline[linewidth=0.022cm](7.81,-0.54)(7.55,-0.90)
\psline[linewidth=0.022cm](7.55,-0.90)(7.55,-1.34)
\pscircle[linewidth=0.022,dimen=outer](6.36,0.50){0.61}
\pscircle[linewidth=0.022,dimen=outer](7.54,-0.91){0.61}
\psline[linewidth=0.022cm](7.57,-1.50)(7.57,-1.90)
\psline[linewidth=0.022cm](6.04,1.02)(5.68,1.52)
\psline[linewidth=0.022cm](6.71,0.99)(7.0,1.54)
\psline[linewidth=0.022cm](7.94,-0.45)(8.38,0.04)
\psline[linewidth=0.022cm](6.73,0.01)(7.17,-0.42)
\psline[linewidth=0.022cm](6.12,0.86)(6.12,0.46)
\psline[linewidth=0.022cm](6.66,0.84)(6.65,0.47)
\psline[linewidth=0.022cm](6.12,0.46)(6.41,0.09)
\psline[linewidth=0.022cm](6.65,0.47)(6.39,0.09)
\psline[linewidth=0.022cm](0.0,1.04)(0.2,0.66)
\psline[linewidth=0.022cm](0.44,1.04)(0.58,0.66)
\psline[linewidth=0.022cm](0.58,0.66)(0.72,1.04)
\psline[linewidth=0.022cm](0.88,1.04)(0.94,0.62)
\usefont{T1}{ptm}{m}{n}
\rput(1.98,1.03){$\sigma$}
\psline[linewidth=0.022cm](6.4,0.82)(6.4,0.12)
\psline[linewidth=0.022cm](6.39,1.53)(6.39,1.11)
\usefont{T1}{ptm}{m}{n}
\rput(7.08,1.72){5}
\usefont{T1}{ptm}{m}{n}
\rput(8.05,1.70){2}
\usefont{T1}{ptm}{m}{n}
\rput(8.80,1.70){4}
\pscircle[linewidth=0.022,dimen=outer](8.76,0.48){0.61}
\psline[linewidth=0.022cm](8.44,1.00)(8.08,1.50)
\psline[linewidth=0.022cm](9.11,0.98)(9.4,1.52)
\psline[linewidth=0.022cm](8.52,0.84)(8.52,0.44)
\psline[linewidth=0.022cm](9.06,0.82)(9.05,0.45)
\psline[linewidth=0.022cm](8.52,0.44)(8.81,0.07)
\psline[linewidth=0.022cm](9.05,0.45)(8.79,0.07)
\psline[linewidth=0.022cm](8.8,0.80)(8.8,0.10)
\psline[linewidth=0.022cm](8.79,1.51)(8.79,1.09)
\usefont{T1}{ptm}{m}{n}
\rput(9.49,1.70){6}
\usefont{T1}{ptm}{m}{n}
\usefont{T1}{ptm}{m}{n}
\rput(2,-0.5){$\sigma(1)=1, \ \sigma(2)=2, \  \sigma(3) =1, $}
\rput(2,-1){$\sigma(4)=2, \ \sigma(5)=1, \  \sigma(6) =2. $}
\usefont{T1}{ptm}{m}{n}
\rput(6,-0.5){$[\sigma]   \  \   = $}
\end{pspicture}
}
\begin{figure}[h]
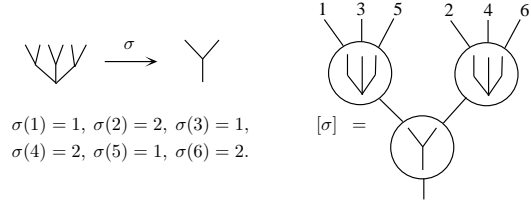
\caption{{\small A $2$-ordinal and its corresponding $2$-planar tree.}}\label{sigma}\end{figure}

\noindent The element $[\sigma]$ then produces the multiplication$$m_{\sigma}:A_S\times A_{T_0}\times \dots \times A_{T_k}\rightarrow A_T$$in the algebra $A.$ All axioms follow from the associativity and unitality of the operation of insertion of $n$-planar trees.

Conversely, if $B$ is an $n$-operad one can associate with any $n$-planar tree $\tau$ a composite map out of
the product of all $B_{\tau_v}$ for all vertices of $\tau$ to the set $B_{t(\tau)}$ using induction from top to the bottom of the tree (see \cite[Lemma 3.2]{SymBat}).\end{proof}

\subsection{The polynomial monads for normal and constant-free $n$-operads}

The monad for normal $n$-operads has been computed in \cite[Theorems 3.1 and 4.1]{SymBat}. We show now that it is a polynomial monad.

\begin{pro}\label{NON}The monad for normal $n$-operads is generated by the polynomial
\begin{diagram}{\tt NOrd(n)}&\lTo^s&\nPTrees_{nor}^*&\rTo^p&\nPTrees_{nor}&\rTo^t&{\tt NOrd(n)}\end{diagram}
where
\begin{itemize}\item $\tt NOrd(n)$ is the  set of isomorphism classes  of  normal $n$-ordinals; \item $\nPTrees_{nor}$ is the set of isomorphism classes of normal $n$-planar trees;\item $\nPTrees_{nor}^*$ is the set of elements of $\,\nPTrees_{nor}$ with one vertex marked. \end{itemize}
The structure maps of this polynomial monad $NO(n)$ are defined in the same way as for the polynomial monad for $n$-operads.\end{pro}

\begin{proof}The proof is analogous to the proof of Proposition \ref{ON}. The only difference is that we construct a normal $n$-planar tree $[\sigma]$ only for surjective maps of normal $n$-ordinals. In this construction we also do not introduce vertices for those fibres of $\sigma$ which are isomorphic to the terminal $n$-ordinal.\end{proof}

\vspace{1ex}

 \scalebox{0.8} 
{
\begin{pspicture}(-2,-2.26)(8.18,2.29)
\begin{pspicture}(1,0)(0,0)
\psline[linewidth=0.022cm](0.30,1.19)(0.30,0.77)
\psline[linewidth=0.022cm](0.70,1.19)(0.70,0.41)
\psline[linewidth=0.022cm](1.08,1.19)(1.08,0.77)
\psline[linewidth=0.022cm](0.30,0.79)(0.70,0.41)
\psline[linewidth=0.022cm](0.70,0.41)(1.08,0.79)
\end{pspicture}
\psline[linewidth=0.022cm](1.88,1.22)(2.16,0.86)
\psline[linewidth=0.022cm](2.42,1.22)(2.16,0.86)
\psline[linewidth=0.022cm](2.16,0.86)(2.16,0.42)
\psline[linewidth=0.022cm,arrowsize=0.153cm 2.0,arrowlength=1.2,arrowinset=0.6]{->}(0.6,0.75)(1.6,0.75)
\usefont{T1}{ptm}{m}{n}
\rput(1.1,1){$\sigma$}
\usefont{T1}{ptm}{m}{n}
\rput(0,-0.5){$\sigma(1)=1, \ \sigma(2)=2, \  \sigma(3) =1$}
\usefont{T1}{ptm}{m}{n}
\rput(5,-0.5){$[\sigma]   \  \   = $}
\rput(5.16,2.07){\large 1}
\usefont{T1}{ptm}{m}{n}
\rput(6.62,2.07){\large 3}
\usefont{T1}{ptm}{m}{n}
\rput(8.02,0.77){\large 2}
\psline[linewidth=0.022cm](6.98,-0.19)(7.26,-0.55)
\psline[linewidth=0.022cm](7.52,-0.19)(7.26,-0.55)
\psline[linewidth=0.022cm](7.26,-0.55)(7.26,-0.99)
\pscircle[linewidth=0.022,dimen=outer](6.07,0.85){0.61}
\pscircle[linewidth=0.022,dimen=outer](7.25,-0.56){0.61}
\psline[linewidth=0.022cm](7.28,-1.15)(7.28,-1.55)
\psline[linewidth=0.022cm](5.70,1.34)(5.34,1.80)
\psline[linewidth=0.022cm](6.42,1.34)(6.70,1.82)
\psline[linewidth=0.022cm](7.62,-0.07)(7.98,0.50)
\psline[linewidth=0.022cm](6.44,0.36)(6.88,-0.07)
\psline[linewidth=0.022cm](5.78,1.24)(5.78,0.82)
\psline[linewidth=0.022cm](6.36,1.24)(6.36,0.82)
\psline[linewidth=0.022cm](5.78,0.82)(6.12,0.44)
\psline[linewidth=0.022cm](6.36,0.82)(6.10,0.44)
\end{pspicture}}
\begin{figure}[h]\caption{{\small A normal $2$-ordinal and its corresponding normal $2$-planar tree.}}\label{sigma2}\end{figure}

\begin{pro}\label{CFON}The monad for constant-free $n$-operads is generated by the polynomial

\begin{diagram}{\tt ROrd(n)}&\lTo^s&\nPTrees_{reg}^*&\rTo^p&\nPTrees_{reg}&\rTo^t&{\tt ROrd(n)}\end{diagram}
\noindent where
\begin{itemize} \item $\tt ROrd(n)$ is the  set of isomorphism classes of regular $n$-ordinals; \item $\nPTrees_{reg}$ is the set of  isomorphism classes of regular $n$-planar trees; \item $\nPTrees_{reg}^*$ is the set of elements of $\,\nPTrees_{reg}$ with one vertex marked. \end{itemize}
The structure maps of this polynomial monad $CFO(n)$ are analogous to those of the polynomial monad for $n$-operads.\end{pro}

\begin{proof}The proof is again similar to the proof of Proposition \ref{ON}. The  difference is that we construct a regular $n$-planar tree $[\sigma]$ only for surjective maps of regular $n$-ordinals.\end{proof}

\begin{defin}An $n$-planar tree is called non-degenerate if its underlying tree is non-degenerate in the sense of Definition \ref{nd}. A vertex $v$ of an $n$-planar tree $\tau$ is called non-degenerate if the corolla $cor_v(\tau)$ is non-degenerate in the tree underlying $\tau$.\end{defin}

\begin{pro}\label{RON}The monad for reduced $n$-operads is generated by the polynomial
\begin{diagram}{\tt ROrd(n)}&\lTo^s&\nPTrees_{nd}^*&\rTo^p&\nPTrees_{nd}&\rTo^t&{\tt ROrd(n)}\end{diagram}
\noindent where
\begin{itemize}\item $\tt ROrd(n)$ is the set of isomorphism classes of regular $n$-ordinals;\item $\nPTrees_{nd}$ is the set of isomorphism classes of non-degenerate $n$-planar trees;\item $\nPTrees_{nd}^*$ is the set of elements of $\,\nPTrees_{nd}$ with one non-degenerate vertex marked.\end{itemize}
The structure maps of this polynomial monad $RO(n)$ are analogous to those of the polynomial monad for $n$-operads.\end{pro}

\begin{proof}Everything is similar to the other three cases except that contraction of trees may involve dropping degenerate vertices. As a result, the underlying category of this polynomial monad is the category of regular $n$-ordinals and their injections.\end{proof}

We now prove that the polynomial monads $NO(n),CFO(n)$ and $RO(n)$ are tame, while the monad $O(n)$ is not tame for $n\geq 2$. We construct an obstruction for the existence of transferred model structure in the latter case. Our proof will be based on a combinatorial lemma about directed categories.

\begin{defin}A small category $C$ is called \emph{directed} if there is a function $\dim$ (called dimension function) on objects of this category to an ordinal $\lambda$  such that any non-identity morphism strictly increases the dimension.\end{defin}

\begin{lem}\label{dl}Let $C$ be a finite directed category with a set of generating morphisms $G$ which satisfies the following two conditions:
\begin{enumerate}\item[(i)]Any two parallel  generators in $C$ are equal;\item[(ii)]Any span of  generators$$\omega\stackrel{\phi}{\leftarrow}\tau\stackrel{\psi}{\to}\upsilon$$ in $C$ can be completed to a commutative square by a cospan of generators (or identities)$$\upsilon\stackrel{\phi^*}{\to} \varsigma\stackrel{\psi^*}{\leftarrow}\omega.$$\end{enumerate}Then there is a unique terminal object in each connected component of $C.$\end{lem}

\begin{proof}We first prove that each connected component of $C$ has a unique weakly terminal object, i.e. an object such that there exists at least one morphism to it from any other object of the same connected component. We use an induction on the number $k$ of objects in $C.$  If $k=1$ the unique object in $C$ is weakly terminal.  Assume now that we know that the statement is true for any for $k\le m-1.$ Observe that given a zig-zag of morphisms in $C$
$$a_0\stackrel{}{\leftarrow} a_1\stackrel{}{\rightarrow} a_2 \leftarrow \ldots \ldots \rightarrow  a_{n-2}\stackrel{}{\leftarrow} a_{n-1}\stackrel{}{\rightarrow} a_n$$
 one can replace it by a zig-zag
$$a_0\stackrel{}{\rightarrow} b_1 \stackrel{}{\leftarrow } a_2 \leftarrow \ldots \ldots \rightarrow  a_{n-2}\stackrel{}{\rightarrow} b_{n-1}\stackrel{}{\leftarrow} a_n.$$
Doing the same for the zig-zag of $b$'s and continuing we come to the conclusion that for any two objects $c,c'$ of the same connected component of $C$ one can find an object $c''$ and a cospan $$c\rightarrow  c''\leftarrow c'.$$

We can now assume that there is only one connected component in $C,$ otherwise the statement for $k=m$ follows immediately from the inductive hypothesis.  Let $L$ be the minimum of the function $\dim$ on $C.$ Consider the full subcategory $C'$ consisting of objects $a$ such that $\dim(a)>L.$ Then $C'$ is obviously connected and satisfies our inductive hypothesis. Therefore, it contains a weakly terminal object $t$. If an object $a$ does not belong to $C'$ then there must be a span
$$a\rightarrow b\leftarrow t$$ where $b\in C'.$ In this span $b\leftarrow t$ must be an identity, otherwise $\dim(b)>\dim(t)$ and $t$ would not be weakly terminal. So, we have found a map from any object of $C$ to $t.$ This weakly terminal object is obviously unique.

The next step of the proof is to show that the weakly terminal object $t$ is actually terminal in its connected component.
We use an induction on $\dim$  to prove  that there is at most one morphism  to $t.$  Indeed, the statement is true  for
all objects $a$ such that $\dim(a)\ge \dim(t).$ Now, suppose we know that the morphism is unique for all objects $a$ such that $\dim(a)\ge m.$ Let $k$ be the maximal integer such that  $k<m$ and there exists an object $b$ such that $\dim(b)=k.$ Let $\dim(b)=k.$
and  $f,g:b\rightarrow t.$  One can factorise $f=f_1\cdot f_2$ and $g=g_1\cdot g_2$
where $f_1$ and $g_1$ are generators. Now, we can complete the cospan
$$a_1 \stackrel{f_1}\leftarrow b \stackrel{g_1}{\rightarrow} a_2$$
to a span $$a_1\rightarrow c\leftarrow a_2.$$
If $a_1 \ne a_2$ then we can use our inductive hypothesis and, therefore, there is only one morphism from $c$ to $t$ and we finished the proof.  If $a_1 = a_2$ then $f_1=g_1$ and $f_2 = g_2 $ by the inductive hypothesis.\end{proof}

\subsection{The normal $n$-operad classifier $\nn$}According to formula (\ref{HT}) and Proposition \ref{NON} the $n$-collection whose $S$-th term is the set of normal $S$-dominated $n$-planar trees is the object of objects of the classifier $\nn.$ The morphisms of $\nn$ are generated by insertions of $n$-planar trees into vertices of another $n$-planar tree. This allows us to describe the morphisms explicitly as certain contractions of ``forests'' in an $n$-planar tree.

Let $\tau$ be an $n$-planar tree. {\it An $n$-ordered subtree $\tau'$ of $\tau$} is a plain subtree of $\tau$ such that for each of its vertices the set of incoming edges carries the obvious $n$-ordinal structure induced from $\tau$.

A $n$-ordered subtree $\tau'$ is called an {\it $n$-planar subtree} if the set of leaves of $\tau'$ is equipped with an $n$-ordinal structure (called {\it target $n$-ordinal of $\tau'$}) such that
 \begin{itemize}\item[(i)] $\tau'$ is an $n$-planar tree;\item[(ii)]the corresponding contraction $C_{(\tau,\tau')}$ (\ref{CVT}) is a map of $n$-ordinals.
 \end{itemize}
{\it A forest of $n$-planar subtrees in $\tau$} is a set of $n$-planar subtrees of $\tau$ whose sets of vertices are pairwise disjoint. The source of such a forest is the tree $\tau$ itself; the target is the $n$-planar tree obtained from $\tau$ by contracting each subtree of the forest to a corolla, and decorating the corresponding vertex by the target $n$-ordinal of the contracted subtree.

The category $\nn$ is generated by contractions of single $n$-planar subtrees. There are two types of relations:
\begin{itemize}\item[$(\bullet\bullet)$]
 Given a forest consisting of two disjoint $n$-planar subtrees the morphism of contraction of this forest factorises as two consecutive contractions of its subtrees in any of the two possible orders;
\item[$(\hspace{-0.5mm}\odot\hspace{-2.25mm}\bullet)$]Given an $n$-planar subtree which contains itself an $n$-planar subtree, the morphism of contraction of the bigger subtree factorises as contraction of the smaller subtree followed by contraction of the result of the first contraction.  \end{itemize}

For an example of a generator see the left hand side tree on Figure \ref{4to7}. The target $2$-ordinal of the subtree shown by a dash line is the $2$-ordinal in the root of the contracted tree. Observe that the total order of this $2$-ordinal structure on the leaves is different from the linear order induced by planar structure of the subtree.

\begin{lem}\label{parallel}Any two parallel generators in $\nn$ are equal.\end{lem}

\begin{proof}For the proof it is enough to see that due to the planarity and normality of the underlying trees source and target of a generator permit to reconstruct entirely the $n$-planar subtree whose contraction defines the generator. Indeed, we know the decorating $n$-ordinals of the vertices of this subtree and, hence, its $n$-ordinal structure. The decoration of the vertex of the target tree gives us the $n$-ordinal structure on the set of leaves of this subtree, so that we reconstructed the entire contracted $n$-planar subtree.\end{proof}

\subsection{Diamond generated by a span in $\nn$}\label{diamond}

Let now $\tau$ be an $n$-planar tree and let $\tau',\tau''$ be two $n$-planar subtrees of $\tau.$ These two subtrees generate a span of morphisms in $\nn$
 \begin{equation}\label{dspan}\omega\stackrel{\phi}{\leftarrow}\tau\stackrel{\psi}{\to}\upsilon\end{equation}
 We now describe  a construction   which produces a cospan in $\nn$ $$ \upsilon\stackrel{\phi^*}{\to} \varsigma\stackrel{\psi^*}{\leftarrow} \omega$$  making the square
 \begin{diagram}[small]\tau&\rTo^{\psi}&\upsilon\\\dTo^{\phi}&&\dTo_{\phi^*}\\\omega&\rTo_{\psi^*}&\varsigma\end{diagram}commutative.
 We call this construction {\it diamond  generated by the span (\ref{dspan})}.

We consider two cases:
\begin{itemize}\item[(i)] $\tau'$ and $\tau''$ have no common vertices so they form an $n$-planar subforest  $\tau\cup \tau'$ in $\tau$ ;
\item[(ii)] $\tau'$ and $\tau''$ have at least one common vertices, so $\tau\cup \tau'$ is a partially $n$-ordered subtree $\tau'''$ in $\tau.$
\end{itemize}

Case (i) is easy since in this case we can take as $\varsigma$ the result of contraction of the subforest $\tau\cup \tau'.$ The resulting square commutes because of the relation of type ($\bullet\bullet$) in $\nn.$

For the case (ii) we will now complete the partial $n$-ordered subtree structure on $\tau'''$ to an $n$-planar subtree structure in $\tau$ in such a way that $\tau'$ and $\tau''$ both are $n$-planar subtrees of $\tau'''.$  If such a structure on $\tau'''$ exists then we can take as $\varsigma$ the result of contraction of  $\tau'''$ and we will have a morphism $\delta: \tau\to \varsigma.$ We also get a morphism $\phi^*:\upsilon \to \varsigma$ because $\tau'$ is a subtree of $\tau'''$ and relation of type $(\hspace{-0.5mm}\odot\hspace{-2.25mm}\bullet)$ in $\nn$ shows that $\delta = \psi\cdot\phi^*.$ Analogously we have $\delta = \phi\cdot\psi^*$ and the result follows.

To provide $\tau'''$ with an $n$-planar structure we need to equip the set of leaves $L(\tau''')$ with an $n$-ordinal structure and check that such an equipment satisfies the necessary condition. Observe that $L(\tau''')$ is a subset of $L(\tau'')\cup L(\tau').$ Let $$C_{(\tau,\tau')}:L(\tau)\to L(\tau') \ , \ C_{(\tau,\tau'')}:L(\tau)\to L(\tau'')$$ be the corresponding contraction functions. They determine a unique function $$C: L(\tau)\to L(\tau')\cup L(\tau'').$$   For each $h\in L(\tau''')\subset L(\tau')\cup F(\tau'')$ let us choose
 a  leave $i(h)\in C^{-1}(h).$ We have a subset  $$\{i(h) | h\in  L(\tau''')\}\subset L(\tau)$$ and we equip it with the $n$-ordinal structure induced from $F(\tau).$ The tree $\tau'''$ is an $n$-planar subtree of $\tau.$ It is also clear that $\tau'$ and $\tau''$ are $n$-planar subtrees of $\tau'''$ which finishes the construction.

\begin{lem}\label{directed}For each $n$-ordinal $T$ the category $\nnt$ is a finite directed category.\end{lem}

\begin{proof}The finiteness of $\nnt$ is clear. To prove the existence of an increasing dimension-function, let us construct an antireflexive and transitive relation on the objects of $\nnt$ which reflects the morphisms in $\nnt$. We will write $(\tau,T)  \ll  (\tau',T)$ if
\emph{either} the underlying rooted tree of $\tau'$ is obtained from the underlying rooted tree of $\tau$ by contraction of some internal edges, \emph{or} $\tau$ and $\tau'$ have isomorphic underlying rooted trees but $d(\tau) < d(\tau').$

Here $d(\tau) = \sum_{v\in \tau} d(\tau_v)$, where for an $n$-ordinal $S$ the dimension $d(S)$ is the number of edges in the level-tree representation minus $n-1.$ This dimension $d(S)$ of the $n$-ordinal $S$ is actually the geometric dimension of the Fox-Neuwirth cell defined by $S$ (cf. Section \ref{FNord}), while the resulting $d(\tau)$ is the geometric dimension of the Getzler-Jones cell defined by the $n$-planar tree $\tau$ (cf. \cite{SymBat}). It follows that each generator $f: \tau \to \tau'$ satisfies $\tau \ll \tau',$ because such a generator \emph{either} contracts some internal edges of the underlying tree  \emph{or} comes from a quasibijection, in which case the dimension of the domain tree is strictly less than the dimension of codomain tree (a quasibijection corresponds to an inclusion of a Fox-Neuwirth cell into the boundary of another). We now get by induction on the number of elements that any finite set equipped with a transitive antireflexive relation possesses a dimension-function for its elements which strictly increases along this relation.\end{proof}

\subsection{The classifier \copn}Objects of the classifier $\copn$ consist of  $n$-planar planar trees with an additional decoration of each vertex by
colours $X$ or $ K.$ We call such trees {\it coloured $n$-planar trees.} The vertices are called $X$-vertices or $K$-vertices according to their colours.
{\unitlength=1mm

\begin{picture}(60,60)(-20,-5)

\begin{picture}(10,10)(0,0)

\put(14,22.5){\circle{5}}
\put(7,30.5){\makebox(0,0){\mbox{$\scriptstyle  10$}}}
\put(21,30.5){\makebox(0,0){\mbox{$\scriptstyle  1$}}}
\put(14,32){\makebox(0,0){\mbox{$\scriptstyle  7$}}}
\put(15.9,20.8){\line(1,-1){5}}
\put(12.1,24.4){\line(-1,1){5}}
\put(14,25){\line(0,1){6}}
\put(15.9,24.4){\line(1,1){5}}
\put(14,22.5){\makebox(0,0){\mbox{$\scriptstyle X_t $}}}

\end{picture}

\begin{picture}(10,10)(2.5,8.5)

\put(14,22.5){\circle{5}}
\put(14,25){\makebox(0,0){\mbox{$ $}}}
\put(15.9,20.8){\line(1,-1){5}}
\put(12.1,24.4){\line(-1,1){5}}
\put(14,25){\line(0,1){6}}
\put(15.9,24.4){\line(1,1){5}}
\put(14,32){\makebox(0,0){\mbox{$\scriptstyle  12$}}}
\put(21,30.5){\makebox(0,0){\mbox{$\scriptstyle  11$}}}
\put(14,22.5){\makebox(0,0){\mbox{$\scriptstyle K_v $}}}

\end{picture}

\begin{picture}(10,10)(5,17)

\put(14,22.5){\circle{5}}
\put(14,25){\makebox(0,0){\mbox{$ $}}}
\put(14,20){\line(0,-1){6}}
\put(15.9,24.4){\line(1,1){5}}
\put(14,22.5){\makebox(0,0){\mbox{$\scriptstyle K_t $}}}

\end{picture}

\begin{picture}(10,10)(7.5,8.5)

\put(14,22.5){\circle{5}}
\put(14,25){\makebox(0,0){\mbox{$ $}}}
\put(13,25){\line(-1,2){3}}
\put(15.9,24.4){\line(1,1){8.4}}
\put(14,22.5){\makebox(0,0){\mbox{$\scriptstyle X_s $}}}

\end{picture}

\begin{picture}(10,10)(22.75,-2.5)

\put(14,22.5){\circle{5}}
\put(7,29.5){\makebox(0,0){\mbox{$\scriptstyle  5$}}}
\put(20,29.5){\makebox(0,0){\mbox{$\scriptstyle  8$}}}
\put(14,32){\makebox(0,0){\mbox{$\scriptstyle  3$}}}
\put(12.1,24.4){\line(-1,1){4}}
\put(14,25){\line(0,1){6}}
\put(15.9,24.4){\line(1,1){4}}
\put(14,22.5){\makebox(0,0){\mbox{$\scriptstyle X_p $}}}
\end{picture}

\begin{picture}(10,10)(18.3,-4)

\put(14,22.5){\circle{5}}
\put(14,25){\makebox(0,0){\mbox{$ $}}}
\put(12.1,24.4){\line(-1,1){4}}
\put(14,25){\line(0,1){4}}
\put(15.9,24.4){\line(1,1){4}}
\put(7,29.5){\makebox(0,0){\mbox{$\scriptstyle  4$}}}
\put(20,29.5){\makebox(0,0){\mbox{$\scriptstyle  13$}}}
\put(14,22.5){\makebox(0,0){\mbox{$\scriptstyle K_v $}}}

\end{picture}

\begin{picture}(10,10)(29.5,-13)

\put(14,22.5){\circle{5}}
\put(7,30.5){\makebox(0,0){\mbox{$\scriptstyle  2$}}}
\put(21,30.5){\makebox(0,0){\mbox{$\scriptstyle  6$}}}
\put(14,32){\makebox(0,0){\mbox{$\scriptstyle  9$}}}
\put(12.1,24.4){\line(-1,1){5}}
\put(14,25){\line(0,1){6}}
\put(15.9,24.4){\line(1,1){5}}
\put(14,22.5){\makebox(0,0){\mbox{$\scriptstyle X_w $}}}

\end{picture}

\end{picture}}
\begin{figure}[h]
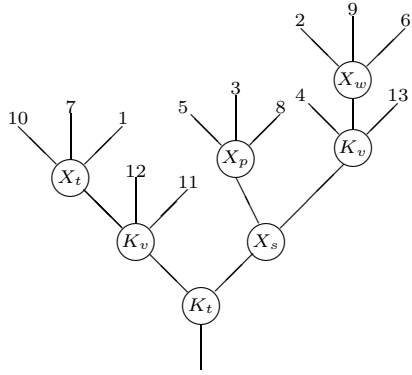
\caption{{\small Typical coloured $n$-planar tree. Here $v,w,p,t,s$ are $n$-ordinals decorating vertices.}}\end{figure}

\noindent The morphisms in $\copn$ are generated by contractions of $n$-planar subtrees all of whose vertices are $X$-vertices. The resulting new vertex after such a  contraction gets colour $X.$  Relations between morphisms are the same as in $\nn.$

\begin{theorem}\label{T01}The polynomial monad $NO(n)$ is tame.\end{theorem}

\begin{proof}We check that the classifier $(\copn)_T$ satisfies the conditions of Lemma \ref{dl}. Indeed, we use the same dimension-function as in Lemma \ref{directed}. Property (i) of the Diamond Lemma follows from Lemma \ref{parallel}. Property (ii) follows from the fact that in such a span of generators only $n$-planar subtrees with $X$-vertices are involved and so we can repeat the construction of the diamond from \ref{diamond} verbatim.\end{proof}

\begin{theorem}\label{T02}The polynomial monads $CFO(n)$  and $RO(n)$ are tame.\end{theorem}

\begin{proof}The category $\copn$ is a subcategory of  ${\bf CFO(n)}^{\tt CFO(n)+1}$ for non-terminal $n$-ordinals. It is not difficult to construct a terminal object in the connected component of ${\bf CFO(n)}^{\tt CFO(n)+1}$ out of the terminal object in the corresponding connected component of $\copn.$ Indeed, let $t\in \copn$ be such an object. We construct an object $t'$ in  ${\bf CFO(n)}^{\tt CFO(n)+1}$ as follows: for any two $K$-vertices in $t$ connected by an edge, or for a leave or root attached to a $K$-vertex, we replace this edge or leave with a linear tree with one vertex. We assign the colour $X$ to this new vertex. We have a morphism from $t$ to $t'$ generated by the unit of the $n$-operad.

 {\unitlength=0.9mm

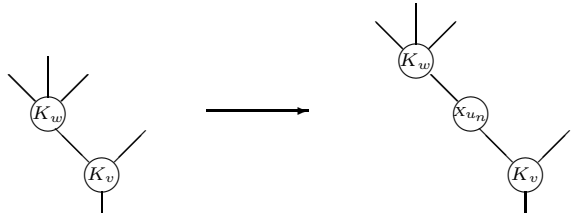
\begin{picture}(60,35)(-35,25)

\begin{picture}(10,10)(14.3,-11)
\put(14,22.5){\circle{5}}
\put(14,25){\makebox(0,0){\mbox{$ $}}}
\put(12.1,24.4){\line(-1,1){4.9}}
\put(14,17){\line(0,1){3}}
\put(15.9,24.4){\line(1,1){4.6}}
\put(14,22.5){\makebox(0,0){\mbox{$\scriptstyle K_v $}}}

\end{picture}

\begin{picture}(10,10)(33.5,-20)

\put(14,22.5){\circle{5}}
\put(12.1,24.4){\line(-1,1){4}}
\put(14,25){\line(0,1){6}}
\put(15.9,24.4){\line(1,1){4}}
\put(14,22.5){\makebox(0,0){\mbox{$\scriptstyle K_w $}}}

\end{picture}

\begin{picture}(10,10)(-40,0)

\begin{picture}(10,10)(14.3,-11)

\put(14,22.5){\circle{5}}
\put(14,25){\makebox(0,0){\mbox{$ $}}}
\put(12.1,24.4){\line(-1,1){4.9}}
\put(14,17){\line(0,1){3}}
\put(15.9,24.4){\line(1,1){4.6}}
\put(14,22.5){\makebox(0,0){\mbox{$\scriptstyle K_v $}}}

\end{picture}

\begin{picture}(10,10)(41.8,-27.8)

\put(14,22.5){\circle{5}}
\put(12.1,24.4){\line(-1,1){4}}
\put(14,25){\line(0,1){6}}
\put(15.9,24.4){\line(1,1){4}}
\put(14,22.5){\makebox(0,0){\mbox{$\scriptstyle K_w $}}}

\end{picture}

\begin{picture}(10,10)(45,-20)

\put(14,22.5){\circle{5}}
\put(12.1,24.4){\line(-1,1){4}}
\put(-25,23){\vector(1,0){15}}
\put(14,22.5){\makebox(0,0){\mbox{$\scriptscriptstyle X_{\hspace{-0.2mm}\scriptscriptstyle u_{\hspace{-0.2mm}n}} $}}}

\end{picture}

\end{picture}

\end{picture}}

\noindent For $S=U_n$ a typical terminal object in the connected component is a linear tree whose vertices are decorated by $U_n$ and whose colours are alternating between $X$ and $K,$ starting with $X$ and ending with $X.$

   {\unitlength=0.9mm

\begin{picture}(60,40)(-45,25)

\begin{picture}(10,10)(0,-11)

\put(14,22.5){\circle{5}}
\put(14,17.6){\line(0,1){2.5}}
\put(14,22.5){\makebox(0,0){\mbox{$\scriptscriptstyle K_{\hspace{-0.2mm}\scriptscriptstyle u_{\hspace{-0.2mm}n}} $}}}
\put(14,30){\circle{5}}
\put(14,27.5){\line(0,-1){2.5}}
\put(14,30){\makebox(0,0){\mbox{$\scriptscriptstyle X_{\hspace{-0.2mm}\scriptscriptstyle u_{\hspace{-0.2mm}n}} $}}}
\put(14,37.5){\circle{5}}
\put(14,35){\line(0,-1){2.5}}
\put(14,37.5){\makebox(0,0){\mbox{$\scriptscriptstyle K_{\hspace{-0.2mm}\scriptscriptstyle u_{\hspace{-0.2mm}n}} $}}}
\put(14,45){\circle{5}}
\put(14,42.5){\line(0,-1){2.5}}
\put(14,45){\makebox(0,0){\mbox{$\scriptscriptstyle X_{\hspace{-0.2mm}\scriptscriptstyle u_{\hspace{-0.2mm}n}} $}}}
\put(14,50){\line(0,-1){2.5}}

\end{picture}

\begin{picture}(10,10)(11,-3.7)

\put(14,22.5){\circle{5}}
\put(14,17){\line(0,1){3}}
\put(14,22.5){\makebox(0,0){\mbox{$\scriptscriptstyle X_{\hspace{-0.2mm}\scriptscriptstyle u_{\hspace{-0.2mm}n}} $}}}
\end{picture}

\end{picture}}

\noindent We leave the proof that these objects are terminal in their connected components of ${\bf CFO(n)}^{\tt CFO(n)+1}$ as an exercise.

The classifier ${\bf RO(n)}^{\tt RO(n)+1}$ for reduced $n$-operads contains more objects than the classifier ${\bf CFO(n)}^{\tt CFO(n)+1}$ because
trees with stumps are allowed. Nevertheless it also contains morphisms of `dropping' stumps. So objects which are terminal in their connected component of ${\bf CFO(n)}^{\tt CFO(n)+1}$ are also terminal in their connected components of ${\bf RO(n)}^{\tt RO(n)+1}.$
\end{proof}

\begin{remark}\label{Tambad}The difficulties with the case of $n$-operads ($n\ge2$) in comparison with monoids and non-symmetric operads (and many other operad types) are closely related to the existence of the so-called ``bad'' cells in the Fulton-MacPherson compactification  of real configurations spaces, discovered by Tamarkin \cite{SymBat,Kon}. For instance, the object of $(\copn)_S$ ($n=2$), represented in Figure \ref{Tam}, is terminal in its connected component.

   {\unitlength=1mm

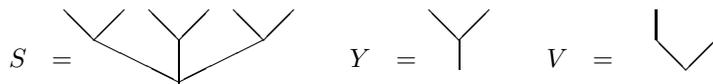
\begin{picture}(60,25)(-10,35)

\begin{picture}(10,5)(0,-28)

\put(5,22){\makebox(0,0){\mbox{$S \ \ =$}}}
\put(12.1,24.4){\line(-1,1){4}}
\put(12.1,24.4){\line(1,1){4}}
\put(12.1,24.4){\line(2,-1){11.2}}

\end{picture}

\begin{picture}(10,10)(-0.1,-28)
\put(12.1,24.4){\line(-1,1){4}}
\put(12.1,24.4){\line(1,1){4}}
\put(12.1,24.4){\line(0,-1){5.8}}
\end{picture}

\begin{picture}(10,10)(-0.2,-28)

\put(12.1,24.4){\line(-1,1){4}}
\put(12.1,24.4){\line(1,1){4}}
\put(12.1,24.4){\line(-2,-1){11.2}}
\end{picture}

\begin{picture}(10,10)(-15,-28)

\put(2,22){\makebox(0,0){\mbox{$Y \ \ =$}}}
\put(12.1,24.4){\line(-1,1){4}}
\put(12.1,24.4){\line(1,1){4}}
\put(12.1,24.4){\line(0,-1){4}}

\end{picture}

\begin{picture}(10,10)(-34,-28)

\put(-2,22){\makebox(0,0){\mbox{$V \ \ =$}}}
\put(8.1,28.4){\line(0,-1){4}}
\put(16.1,28.4){\line(0,-1){4}}
\put(8.1,24.4){\line(1,-1){4}}
\put(16.1,24.4){\line(-1,-1){4}}

\end{picture}

\end{picture}}


  {\unitlength=1mm

\begin{picture}(60,40)(-20,-10)

\begin{picture}(10,10)(-8.8,-2.6)

\put(14,22.5){\circle{5}}
\put(8,29.5){\makebox(0,0){\mbox{$\scriptstyle  3$}}}
\put(20,29.5){\makebox(0,0){\mbox{$\scriptstyle  5$}}}
\put(12.1,24.4){\line(-1,1){4}}
\put(15.9,24.4){\line(1,1){4}}
\put(14,22.5){\makebox(0,0){\mbox{$\scriptstyle K_V $}}}

\end{picture}

\begin{picture}(10,10)(2.5,8.5)

\put(14,22.5){\circle{5}}
\put(15.9,20.8){\line(1,-1){5}}
\put(12.1,24.4){\line(-1,1){5}}
\put(14,25){\line(0,1){6}}
\put(7,30.5){\makebox(0,0){\mbox{$\scriptstyle  1$}}}
\put(14,22.5){\makebox(0,0){\mbox{$\scriptstyle X_V $}}}

\end{picture}

\begin{picture}(10,10)(5,17)

\put(14,22.5){\circle{5}}
\put(14,25){\makebox(0,0){\mbox{$ $}}}
\put(14,20){\line(0,-1){6}}
\put(15.9,24.4){\line(1,1){5}}
\put(14,22.5){\makebox(0,0){\mbox{$\scriptstyle X_Y $}}}

\end{picture}

\begin{picture}(10,10)(7.5,8.5)

\put(14,22.5){\circle{5}}
\put(21.5,30.5){\makebox(0,0){\mbox{$\scriptstyle 6 $}}}
\put(13.7,25){\line(0,1){6}}
\put(15.9,24.4){\line(1,1){5}}
\put(14,22.5){\makebox(0,0){\mbox{$\scriptstyle X_V $}}}

\end{picture}

\begin{picture}(10,10)(19,-2.5)

\put(14,22.5){\circle{5}}
\put(7,29.5){\makebox(0,0){\mbox{$\scriptstyle  2$}}}
\put(20,29.5){\makebox(0,0){\mbox{$\scriptstyle  4$}}}
\put(12.1,24.4){\line(-1,1){4}}
\put(15.9,24.4){\line(1,1){4}}
\put(14,22.5){\makebox(0,0){\mbox{$\scriptstyle K_V $}}}

\end{picture}

\end{picture}}\vspace{-10mm}

\begin{figure}[h]
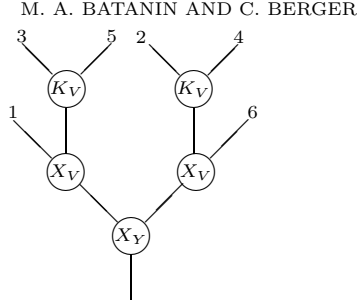
\caption{{\small Terminal object in a connected component of $\copn$ which corresponds to a Tamarkin bad cell.}} \label{Tam}\end{figure}

\noindent It was shown in \cite{SymBat} that this figure corresponds exactly to the first bad cell in the Fulton-MacPherson operad ${\bf fm}^2.$ One can show that certain cells of maximal dimensions in ${\bf fm}^n$ correspond to terminal objects in $\copn$ but in general, we don't know an explicit characterisation of such cells and as a consequence an explicit formula for semi-free coproducts of $n$-operads for $n\ge 2.$ \end{remark}

\subsection{The polynomial monad $O(n)$ is not tame for  $n\ge 2$}\label{counterexample}

Some connected components of ${\bf O(n)}^{\tt O(n) +1}$ do not have terminal objects but do contain weakly terminal objects. This is because we have to include trees with stumps in the description of the monad for $n$-operads. As a consequence colimits over ${\bf O(n)}^{\tt O(n) +1}$ are more complicated than those appearing in the monads for constant-free and normal $n$-operads. Weak equivalences may not be preserved by these colimits, and this creates an obstruction for the existence of transferred model structure on $n$-operads in an arbitrary monoidal model category $\Ee$ even if the latter satisfies the hypothesis of Theorem \ref{maintheorem}.

More precisely, for $n=2$, the category $({\bf O(n)}^{\tt O(n)+1})_{z^2U_0}$ has as objects all coloured $2$-planar trees without
leaves. All such trees which contain  exactly two $K$-vertices without leaves form a full connected subcategory of $({\bf O(n)}^{\tt O(n)+1})_{z^2U_0}$ This subcategory contains a final subcategory consisting of two objects

 {\unitlength=1mm

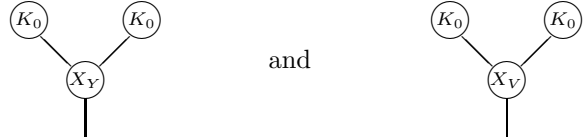
\begin{picture}(30,30)(0,10)

\begin{picture}(10,10)(-8.8,-2.6)

\put(14,22.5){\circle{5}}
\put(6.5,30.5){\circle{5}}
\put(21.5,30.5){\circle{5}}
\put(6.5,30.5){\makebox(0,0){\mbox{$\scriptstyle  K_0$}}}
\put(21.5,30.5){\makebox(0,0){\mbox{$\scriptstyle  K_0$}}}
\put(14,20){\line(0,-1){5}}
\put(12.1,24.4){\line(-1,1){4}}
\put(15.9,24.4){\line(1,1){4}}
\put(14,22.5){\makebox(0,0){\mbox{$\scriptstyle X_Y $}}}

\end{picture}

\put(40,28){\makebox(0,0){\mbox{and}}}

\begin{picture}(10,10)(-54.8,-2.6)

\put(14,22.5){\circle{5}}
\put(6.5,30.5){\circle{5}}
\put(21.5,30.5){\circle{5}}
\put(6.5,30.5){\makebox(0,0){\mbox{$\scriptstyle  K_0$}}}
\put(21.5,30.5){\makebox(0,0){\mbox{$\scriptstyle  K_0$}}}
\put(14,20){\line(0,-1){5}}
\put(12.1,24.4){\line(-1,1){4}}
\put(15.9,24.4){\line(1,1){4}}
\put(14,22.5){\makebox(0,0){\mbox{$\scriptstyle X_V $}}}

\end{picture}

\end{picture}}

\noindent where $Y,V$ are $2$-ordinals as in (\ref{Tambad}). There are exactly two non-trivial morphisms between these objects generated by two quasibijections $\sigma_0$ and $\sigma_1$ from $V$ to $Y.$ It follows that the $z^2U_0$-component of the semi-free coproduct of $2$-operads contains a summand isomorphic to the coequaliser of
$$ X_{Y}\otimes K_0\otimes K_0 \coeq X_{V}\otimes  K_0\otimes K_0,$$
where $$d=X(\sigma_0)\otimes id_{K_0\otimes K_0}$$
 $$d' = X(\sigma_1)\otimes \tau_{K_0\otimes K_0},$$ and $\tau$ is the symmetry morphism in $\Ee.$

Let now  $\Ee$ be the category of chain complexes over a field $\bf k$ of characteristic two. Let $K$ be a non-reduced  $2$-collection such that $K_0 = C,$ where $C$ is an acyclic cofibrant object in $\Ee$ and $K_T =0$ for all $T\ne 0.$ Let $X$ be a $2$-operad with $X_T = {\bf k}$ for all $T\in Ord(2).$  The aforementioned calculation shows that the coproduct $F(K)\vee X$ contains a summand $C\otimes C/\Z_2$  which is not necessarily an acyclic complex, moreover, it is not hard to find a $C$ such that  $0$-homology of $C\otimes C/\Z_2$ will be ${\bf k}\oplus {\bf k}.$ Hence, $X\to F(K)\vee X$ can not be a weak equivalence. Nevertheless, if the category of $2$-operads in $\Ee$ admitted a transferred model structure this map would have to be a weak equivalence because trivial cofibrations are closed under pushouts. Similar obstructions for the existence of a model structure on $n$-operads exist for any $n\ge 2.$

\part{Graphs, trees and graph insertion}

In this last part, we give formal definitions of graphs, trees and graph insertion. Graphs have been used in an essential way in defining the polynomial monads of Part 3 whose algebras describe different types of operads. Indeed, each class of graphs, which is suitably \emph{closed under graph insertion}, gives rise to a polynomial monad on sets whose multiplication is directly induced by graph insertion. The graph insertional origin of the monad multiplication is responsible for the substitutional structure of the associated algebras. This explains somewhat why these algebras are ``generalised operads''.

Most important for us are the graphical properties of the so-called \emph{$n$-planar trees}. These are higher-order generalisations of the well-known linear ($n=0$) and planar ($n=1$) rooted trees. Algebras for the polynomial monad defined by the class of $n$-planar trees are precisely pruned $(n-1)$-terminal $n$-operads of the first author \cite{SymBat,EHBat}, cf. Section \ref{higheroperads}.

Our notions of graph and graph insertion are equivalent to those used in the recent book by Johnson-Yau \cite{JY}, though closer in spirit to the Feynman graphs of Joyal-Kock \cite{JK}. We are grateful to Mark Johnson and Donald Yau for pointing out that our previously used definitions omitted the ``free-living loops'' which naturally appear whenever ``operadic algebras'' are equipped with ``traces''.

\section{Graphs and graph insertion}

\begin{defin}A graph is a finite category $G$ with three kinds of objects, called $v$-objects, $f$-objects and $e$-objects respectively.

Among non-identity arrows in $G$ are only allowed arrows with source an $f$-object and target either a $v$- or an $e$-object such that the following two axioms hold:\begin{itemize}\item[(i)]each $e$-object is the target of precisely two non-identity arrows;\item[(ii)]each $f$-object is the source of at least one and at most two non-identity arrows, among which one at least has an $e$-object as target; if both targets are $e$-objects, the arrows must be parallel.\end{itemize}

A morphism of graphs $G\to G'$ is a functor of the underlying categories which takes $v$-, $f$- resp. $e$-objects of $\,G$ to $v$-, $f$- resp. $e$-objects of $\,G'$.\end{defin}

Observe that this definition allows the set of $f$-objects to be empty, in which case the set of $e$-objects is empty as well. For ease of terminology, the $v$-, $f$- resp. $e$-objects of $G$ will be called the \emph{vertices}, \emph{flags} resp. \emph{edges} of the graph $G$. The reader should consider the categorical structure of $G$ as describing the incidence relations between these three kinds of objects of $G$.

Each edge $e$ of $G$ comes equipped with a unique cospan $f_1\to e \leftarrow f_2.$ We say that the two flags $f_1$ and $f_2$ are {\it adjacent}. A flag $f$ is {\it free} if $f$ is the source of exactly one non-identity arrow. For a vertex $v$ of $G$ the sources of the incoming arrows are the {\it flags attached to} $v$, or \emph{$v$-flags}. Each $v$-flag is the source of exactly one arrow with target $v.$ An edge $e$ is {\it internal} if the cospan $f_1\to e \leftarrow f_2$ either fulfills $f_1=f_2$ or extends to a zigzag$$v_1 \leftarrow f_1 \to e \leftarrow f_2 \to v_2.$$

\noindent Non-internal edges will be called {\it external.}

A {\it corolla} is a graph with a unique vertex and only external edges.  For each $n\geq 0$ there is up to isomorphism a unique corolla with $n$ external edges.

A {\it free-living edge} (resp. {\it free-living loop}) is a graph with no vertices and one external (resp. internal) edge.

A {\it pointed loop} is a graph with one vertex and one internal edge.

Let $G,G'$ be two graphs, let $v$ be a vertex of $G$ and let $\rho$ be a bijection between the set of free flags of $G'$ and the set of $v$-flags of $G$. Then {\it the insertion of $G'$ into $G$ along $\rho$} is the graph $G\circ_\rho G'$ defined as follows.

Let $G\setminus v$ be the graph obtained by removing from $G$ vertex $v$ as well as all arrows with target $v$. Let $f(G')$ (resp. $e(G')$) be the discrete subcategory of $G'$ containing only the \emph{free flags} (resp. \emph{external edges}) of $G'$. Let  $\overline{G'\setminus e(G')}$  be the category obtained by removing from $G'$ all external edges $e$ as well as all arrows to such $e$, and by identifying each free flag with its adjacent.

Both categories $f(G')$ and $\overline{G'\setminus e(G')}$ are not graphs in our sense, but there is an obvious functor
(composite of inclusion and quotient)  $f(G')\to\overline{G'\setminus e(G')}.$ The bijection $\rho$ induces a functor $f(G')\to G'\setminus v$ which takes a free flag of $G'$ to its image under $\rho$ in $G\setminus v$. We define $G\circ_{\rho} G'$ to be the categorical pushout

\begin{diagram}[small]f(G')&\rTo &\overline{G'\setminus e(G')}\\\dTo^\rho & &\dTo \\  G\setminus v &\rTo&\NWpbk    G\circ_{\rho} G'\end{diagram}
which is easily seen to be a graph in our sense. This graph insertion has obvious associativity and commutativity properties. Moreover, the corollas serve as right units. Observe that $G'$ can be considered as a subgraph of $G\circ_{\rho}G'.$

Insertion of free-living edges \emph{removes} vertices. More precisely, insertion of a free-living edge into the unique vertex of a pointed loop (resp. corolla with two external edges) yields a free-living loop (resp. free-living edge).

If $G''$ is obtained by insertion of a graph $G'$ into a vertex of a graph $G$ then we will say dually that $G$ is the result of {\it contracting $G'$ inside $G''.$ }

\subsection{Trees, forests and other special graphs}\label{specialgraph}

The realisation of a graph $G$ is the geometric realisation of (the simplicial nerve of) the underlying category.

A graph $G$ is {\it connected} (resp. a {\it tree}) if the realisation of $G$ is connected (resp. contractible). A graph $G$ is a {\it forest} if it is a finite coproduct of trees.

A {\it rooted tree} is a tree with a distinguished external edge called {\it root;} the other external edges of the tree are called \emph{leaves} or \emph{input edges}. In a rooted tree $T$ the corolla $cor_v(T)$ attached to a vertex $v$ has a canonical structure of rooted tree. The input edges (resp. root) of this ``local'' tree $cor_v(T)$ will be called the incoming edges (resp. outgoing edge) of $v$. Morphisms of rooted trees are morphisms of the underlying graphs which preserve the outgoing edges of the vertices.

A graph is called a {\it rooted forest} if it is a finite coproduct of rooted trees.

\noindent There are some types of graphs which we would like to give separate names:

-- A \emph{linear tree} on $n$ vertices $L_n$ is a tree with $n$ vertices and $n+1$ edges two of which are external. In particular, $L_0$ denotes the \emph{free-living edge}.

-- A \emph{linear graph} on $n$ vertices $[1,n]$ is a tree with $n$ vertices and $n-1$ internal edges and no external edges. We will call the two vertices with only one edge attached the {\it boundary vertices} of $[1,n]$. A \emph{path} between two vertices $v_1,v_2$ of a graph $G$ is a graph morphism $[1,n]\to G$ taking the boundary vertices to $v_1,v_2.$

A non-empty graph $G$ without free-living edges/loops is connected (resp. a tree) if and only if for any ordered pair of vertices there is a (unique) path between them.

\subsection{Ordered graphs and ordered trees}\label{orderedgraphs}

A graph is said to be \emph{partially ordered} if the set of its free flags is linearly ordered. A graph $G$ is said to be \emph{ordered} if $G$ as well as the corollas $cor_v(G)$ for the vertices $v$ of $G$ are partially ordered.

In particular, an \emph{ordered corolla} with vertex $v$ carries two linear orderings: a linear ordering of the set of its free flags and a linear ordering the set of its $v$-flags:

{\unitlength=0.75mm
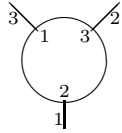
\begin{picture}(0,25)(-60,11)

\put(14,22.5){\circle{14.8}}
\put(4,35){\makebox(0,0){\mbox{$ $}}}
\put(14,37){\makebox(0,0){\mbox{$ $}}}
\put(24,35){\makebox(0,0){\mbox{$ $}}}
\put(10.5,26.8){\makebox(0,0){\mbox{$\scriptstyle 1 $}}}
\put(17.7,26.8){\makebox(0,0){\mbox{$\scriptstyle 3 $}}}
\put(14,17){\makebox(0,0){\mbox{$\scriptstyle 2 $}}}
\put(5,30){\makebox(0,0){\mbox{$\scriptstyle 3 $}}}
\put(23,30){\makebox(0,0){\mbox{$\scriptstyle 2 $}}}
\put(13,12){\makebox(0,0){\mbox{$\scriptstyle 1 $}}}
\put(14,15.4){\line(0,-1){5}}
\put(9.3,27.7){\line(-1,1){5}}
\put(18.8,27.8){\line(1,1){5}}

\end{picture}}

An isomorphism between ordered graphs is an isomorphism of the underlying graphs which preserves all orderings.

An {\it ordered rooted tree} is a rooted tree which is ordered as a tree in such a way that the external root and the roots (i.e. outgoing edges) of the corollas are the first elements of the respective linear orderings. This implies that we can forget about the external root and the roots of corollas and keep only the linear orderings of the input edges of the tree and of the incoming edges for each vertex of the tree.

An {\it ordered rooted forest} is a finite coproduct of ordered rooted trees.

Ordered graphs admit an operation of graph insertion which depends only on compatibility conditions between graphs. The bijection $\rho$ between the free flags of  $G'$ and the $v$-flags of $G$ is uniquely determined by the linear orderings once the cardinalities are the same. We therefore can speak unambigously about the insertion of $G'$ into the vertex $v$ of $G$, the result of which shall be denoted $G'\circ_v G.$ The unit for this ordered graph insertion is given by those ordered corollas for which the linear ordering of the free flags coincides with linear ordering of the $v$-flags.

We can easily check that the subcategories of ordered trees (forests), ordered rooted trees (forests) are closed under graph insertion and, hence, induce a well defined graph insertion on isomorphism classes of the corresponding groupoids.

Each isomorphism class of \emph{ordered} rooted trees has a unique representative by a \emph{planar} rooted tree equipped with a linear ordering of its input edges.

\subsection{Directed graphs and directed trees}

A {\it directed graph} is a graph with a chosen arrow in each span $f_1\to e\leftarrow f_2$ for each edge $ e$ of the graph.  Such a choice amounts to the choice of an orientation for this edge. Morphisms of directed graphs are required to preserve these chosen arrows.

In a directed graph the $v$-flags of a vertex $v$ are subdivided into incoming and outgoing flags. The same is true for free flags.   {\it A directed ordered graph} is a directed graph with a linear ordering of all outgoing and a linear ordering of all incoming free flags, as well as linear orderings of incoming $v$-flags and linear orderings of outgoing $v$-flags for each vertex $v$. Directed graph-insertion is defined like in the non-directed case assuming in addition that the bijection $\rho$ is \emph{orientation-reversing}.

Any rooted tree admits two canonical orientations from top to bottom or vice versa. We always orient a rooted tree from top to bottom.  We also have a directed version of the linear graph $[1,n]$ with the direction going from $p$ to $p-1.$

{\it A loop} in a directed graph $G$ is any map of directed graphs $s:[1,n]\to G, n\ge 2,$ for which $s(1)=s(n)$ or any map from the free living loop $l$ to $G.$

A  {\it loop-free graph} is a directed graph which has no loops. Loop-free graphs are closed under graph-insertion.

\section{Planar trees and $n$-planar trees}

\subsection{Planar trees}

A subgraph $G'$ of a graph $G$ is called \emph{plain} if for each vertex $v$ of $G'$, the cardinalities of the set of $v$-flags are the same in $G'$ and in $G$. For any pair $(T,T')$ consisting of a tree $T$ and plain subtree $T'$ of $T$ there is a well-defined function, called {\it contraction}
\begin{equation}\label{CVT} C_{(T,T')}:e(T)\to e(T')\end{equation}
taking external edges of $T$ to external edges of $T'$. This function is constructed as follows. For each external edge $e$ of $T$ there exists a unique $n\ge 0$ and a unique injective  map $\gamma_e$ from the rooted tree $L^{rt}_n$ to $T$ with the following three properties:
\begin{itemize}
\item[(i)] $\gamma_e$ takes
the unique input edge of $L^{rt}_n$ to $e;$
\item[(ii)]  $\gamma_e$ takes
the root $\mathbf{rt}$ of $L^{rt}_n$ to an external edge of $T';$
\item[(iii)] the image of $\gamma_e$ does not contain any vertices of $T'$.
\end{itemize}
We then define $C_{(T,T')}(e) = \gamma_e(\mathbf{rt}).$\vspace{1ex}

We introduce a special notation if $T'$ is the corolla $cor_v(T)$ of a vertex $v$ of $T$, namely
\begin{equation}\label{CV} C_v=C_{(T,cor_v(T))}:e(T)\to e(cor_v(T))\end{equation}and call it \emph{$v$-contraction}. If $T$ is an ordered tree then $e(T)$ and $e(cor_v(T))$ are linearly ordered. We will say that the ordered tree $T$ is {\it planar} if for each vertex $v$ of $T$ the $v$-contraction preserves the linear orders up to cyclic permutation, which means that the $v$-contraction becomes an order-preserving map after cyclic permutation of the set of external edges of $cor_v(T).$

If $T$ is a rooted tree then the sets $e(T)$ and $e(cor_v(T))$ are pointed by the respective roots and the $v$-contraction is a map of pointed sets which restricts away from the roots. By abuse of notation we consider the $v$-contractions of a rooted tree as restricted to the set of input edges. Accordingly, an ordered rooted tree is planar if and only if for each vertex $v$ the $v$-contraction is order-preserving.

The subcategories of planar and planar rooted trees are closed under graph insertion and there is a graph insertion on isomorphism classes of the corresponding groupoids.

\subsection{Higher planar rooted trees}

The notion of $n$-planar rooted tree generalises linear and planar rooted trees.  A {\it partially  $n$-ordered rooted tree  $T$} is a rooted tree equipped with the structure of $n$-ordinal on the set of incoming edges of each vertex $v$ of $T.$ An  {\it $n$-ordered rooted tree  $T$} is a partially $n$-ordered rooted tree equipped with an  $n$-ordinal structure  on the set of free flags of $T$. Since each $n$-ordinal induces a linear order on its underlying set, such a tree has a canonical structure of ordered rooted tree.

We say that $T$ is an {\it $n$-planar tree} if for each $v$ the contraction function $C_v$ is a map of $n$-ordinals.  For instance, $0$-planar trees are linear rooted trees and $1$-planar trees are planar rooted trees as previously defined. Each isomorphism class of partially $n$-ordered rooted trees contains a unique planar representative. Therefore, an $n$-planar tree $\tau$ is a planar rooted tree with the following additional structure:

\begin{itemize}\item Each vertex $v$ is decorated by an $n$-ordinal $\tau_v$ whose underlying linear ordered set coincides with the set of incoming edges of $v;$
\item The leaves are labelled by natural numbers from $1$ to $p$ where $p$ is the number of leaves of $\tau;$
\end{itemize}
Such a {\it labelled decorated tree} becomes an $n$-planar tree if we fix as well an $n$-ordinal $S$ with underlying set $|S|=\{1,\ldots,p\}$ such that the following compatibility condition holds:

Recall (cf. \cite{SymBat}) that the set $L(\tau)$ of leaves of an $n$-ordered tree $\tau$ has a canonical structure of $n$-ordered set. Let $k,l\in L(\tau)$ be leaves of $\tau$ and let $v(k,l)$ be the upmost vertex of $\tau$ which lies below $k$ and $l$ in $\tau$. The shortest edge-path in $\tau$ which begins with $k$ (resp. $l$) and ends at $v(k,l)$ determines a unique incoming edge of $v(k,l)$, and hence a unique element $e_k$ (resp. $e_l$) of $|\tau_{v(k,l)}|$. By definition we have  $k<_r l$ in $L(\tau)$ precisely when $e_k<_r e_l$ in the $n$-ordinal $\tau_{v(k,l)}$.

The $n$-planarity of the pair $(\tau,S)$ amounts then to the requirement that $S$ \emph{dominates} the $n$-ordered set $L(\tau)$ in the sense of Definition \ref{domination}. This ``planar'' description of $n$-planar trees is quite efficient. For example, the decorated tree $\tau$ on the left hand side

\scalebox{0.9} 
{
\begin{pspicture}(-2,-5.1)(6.43,3)
\usefont{T1}{ptm}{m}{n}
\rput(0.00,2.02){\large 1}
\usefont{T1}{ptm}{m}{n}
\rput(1.06,2.04){\large 3}
\usefont{T1}{ptm}{m}{n}
\rput(6.27,2.02){\large 5}
\usefont{T1}{ptm}{m}{n}
\rput(0.68,2.04){\large 2}
\usefont{T1}{ptm}{m}{n}
\rput(5.35,2.04){\large 4}
\usefont{T1}{ptm}{m}{n}
\rput(3.96,2.04){\large 7}
\psline[linewidth=0.022cm](0.31,1.19)(0.31,0.78)
\psline[linewidth=0.022cm](0.71,1.19)(0.71,0.41)
\psline[linewidth=0.022cm](1.08,1.19)(1.08,0.79)
\psline[linewidth=0.022cm](0.30,0.79)(0.72,0.41)
\psline[linewidth=0.022cm](0.70,0.41)(1.08,0.79)
\psline[linewidth=0.022cm](5.50,1.23)(5.78,0.87)
\psline[linewidth=0.022cm](6.04,1.23)(5.78,0.87)
\psline[linewidth=0.022cm](5.78,0.87)(5.78,0.43)
\psline[linewidth=0.022cm](2.98,1.23)(3.00,0.83)
\psline[linewidth=0.022cm](3.62,1.23)(3.62,0.85)
\psline[linewidth=0.022cm](3.00,0.83)(3.32,0.47)
\psline[linewidth=0.022cm](3.62,0.85)(3.32,0.45)
\pscircle[linewidth=0.022,dimen=outer](5.73,0.86){0.59}
\pscircle[linewidth=0.022,dimen=outer](3.27,0.88){0.61}
\pscircle[linewidth=0.022,dimen=outer](0.67,0.88){0.61}
\psline[linewidth=0.022cm](2.90,0.41)(2.62,0.13)
\psline[linewidth=0.022cm](3.64,1.37)(3.96,1.81)
\psline[linewidth=0.022cm](2.9,1.37)(2.58,1.83)
\psline[linewidth=0.022cm](6.06,1.35)(6.24,1.79)
\psline[linewidth=0.022cm](5.40,1.33)(5.40,1.79)
\psline[linewidth=0.022](0.26,1.31)(0.06,1.79)
\psline[linewidth=0.022cm](0.68,1.49)(0.68,1.81)
\psline[linewidth=0.022cm](1.06,1.33)(1.08,1.79)
\usefont{T1}{ptm}{m}{n}
\rput(2.58,2.06){\large 6}
\psline[linewidth=0.022cm](3.46,-0.88)(3.74,-1.24)
\psline[linewidth=0.022cm](4.00,-0.88)(3.74,-1.24)
\psline[linewidth=0.022cm](3.74,-1.24)(3.74,-1.68)
\pscircle[linewidth=0.022,dimen=outer](2.23,-0.31){0.61}
\pscircle[linewidth=0.022,dimen=outer](3.73,-1.25){0.61}
\psline[linewidth=0.022cm](3.76,-1.84)(3.76,-2.24)
\psline[linewidth=0.022cm](1.74,0.01)(1.08,0.43)
\psline[linewidth=0.022cm](4.12,-0.78)(5.68,0.29)
\psline[linewidth=0.022cm](2.62,-0.76)(3.18,-1.02)
\psline[linewidth=0.022cm](1.94,0.07)(1.94,-0.34)
\psline[linewidth=0.022cm](2.52,0.07)(2.52,-0.34)
\psline[linewidth=0.022cm](1.94,-0.34)(2.28,-0.72)
\psline[linewidth=0.022cm](2.52,-0.34)(2.26,-0.72)
\end{pspicture}
\begin{pspicture}(-2.5,-4.26)(8.18,0.29)
\psline[linewidth=0.022cm](1.6,0.01)(1.60,-0.61)
\psline[linewidth=0.022cm](1.60,0.00)(1.88,0.59)
\psline[linewidth=0.022cm](1.60,0.00)(1.34,0.59)
\psline[linewidth=0.022cm](2.8,0.00)(2.8,0.59)
\psline[linewidth=0.022cm](2.30,-0.01)(2.3,0.59)
\psline[linewidth=0.022cm](0.0,0.01)(0.0,0.61)
\psline[linewidth=0.022cm](0.52,0.00)(0.52,0.61)
\psline[linewidth=0.022cm](0.98,0.00)(0.98,0.61)
\psline[linewidth=0.022cm](1.60,-0.61)(2.3,0.00)
\psline[linewidth=0.022cm](0.0,0.01)(1.6,-0.60)
\psline[linewidth=0.022cm](2.8,0.00)(1.6,-0.60)
\psline[linewidth=0.022cm](1.62,-0.60)(0.52,0.00)
\psline[linewidth=0.022cm](1.62,-0.60)(0.98,0.00)
\end{pspicture}
}\vspace{-32mm}
\begin{figure}[h]
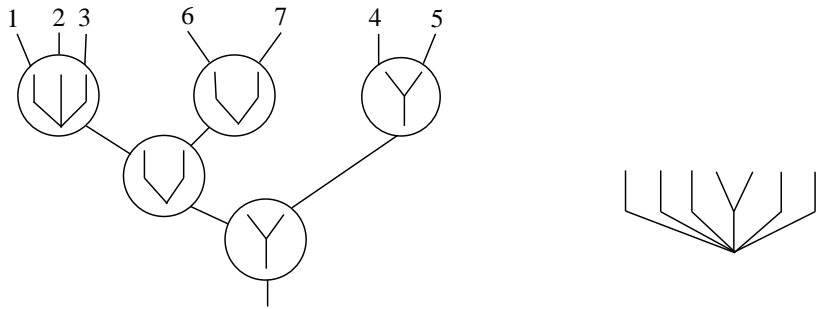
\caption{{\small A $2$-planar tree as a labelled decorated tree with dominating $2$-ordinal.}} \label{1to7}\end{figure}

\noindent acquires an $n$-planar structure if we just add that the induced $2$-ordered set $L(\tau)$ is dominated by the $2$-ordinal as depicted on the right hand side.

\subsection{Regular and normal  $n$-planar trees and the geometry of Fulton-MacPherson operad}

An $n$-planar tree is called {\it regular} if all its decorated ordinals are regular.
An $n$-planar tree is called {\it normal} if all its decorated ordinals are normal.

Normal $n$-planar trees are closely related to the  geometry of the Getzler-Jones decomposition of the Fulton-MacPherson operad
of point configurations in $\RR^n$ \cite{SymBat,GJ,Marklbad}. Recall that the quotient of   $F(\RR^n,k)$ by the Lie group generated by dilatations and translations is called the moduli space of configurations ${\rm Mod}^n(k).$ This moduli space has the same homotopy type as $F(\RR^n,k).$ The  Fox-Neuwirth stratification descends to a stratification of   ${\rm Mod}^n(k)$. Applying Fulton-MacPherson compactification to this stratification we obtain the $k$-th space of the Fulton-MacPherson operad together with its Getzler-Jones decomposition \cite{GJ}. The cells of this decomposition are in one-to-one correspondence with labelled decorated trees \cite{SymBat}. For instance,  the labelled decorated tree of Figure \ref{sigma} corresponds to the following $6$-dimensional Getzler-Jones cell (the so-called `bad' cell of Tamarkin \cite{SymBat,Marklbad}):

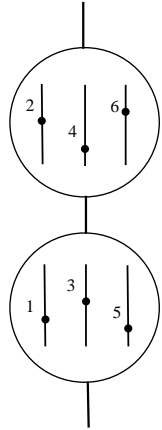
\scalebox{0.6} 
{
\begin{pspicture}(-8,-4.74)(8.36,4.74)
\pscircle[linewidth=0.03,dimen=outer](1.67,2.05){1.67}
\pscircle[linewidth=0.03,dimen=outer](1.67,-2.06){1.67}
\psline[linewidth=0.05cm](1.64,3.68)(1.64,4.72)
\psline[linewidth=0.05cm](1.7,-0.42)(1.7,0.42)
\psline[linewidth=0.05cm](1.74,-3.7)(1.76,-4.72)
\psline[linewidth=0.04cm](0.78,-1.1)(0.8,-2.92)
\psline[linewidth=0.04cm](1.7,-1.1)(1.7,-2.92)
\psline[linewidth=0.04cm](2.62,-1.12)(2.64,-2.92)
\psline[linewidth=0.04cm](0.72,2.9)(0.74,1.12)
\psline[linewidth=0.04cm](1.68,2.88)(1.68,1.08)
\psline[linewidth=0.04cm](2.58,2.88)(2.58,1.1)
\psdots[dotsize=0.18](0.8,-2.32)
\psdots[dotsize=0.18](1.7,-1.92)
\psdots[dotsize=0.18](2.64,-2.52)
\psdots[dotsize=0.18](0.72,2.08)
\psdots[dotsize=0.18](1.68,1.46)
\psdots[dotsize=0.18](2.58,2.28)
\usefont{T1}{ptm}{m}{n}
\rput(0.46257812,-2.035){\large 1}
\usefont{T1}{ptm}{m}{n}
\rput(1.3763516,-1.575){\large 3}
\usefont{T1}{ptm}{m}{n}
\rput(2.3926954,-2.155){\large 5}
\usefont{T1}{ptm}{m}{n}
\rput(0.46488282,2.425){\large 2}
\usefont{T1}{ptm}{m}{n}
\rput(1.4079297,1.865){\large 4}
\usefont{T1}{ptm}{m}{n}
\rput(2.3404298,2.405){\large 6}
\end{pspicture}
}

A normal $n$-planar tree $\tau$ contains extra information about the corresponding Getzler-Jones cell.  Namely, the $n$-ordinal $S$ on the leaves of $\tau$ expresses that the boundary of the closure of the Fox-Neuwirth cell $FN_S$ has a non-empty intersection with the Getzler-Jones cell represented by $\tau$, i.e. the latter lives in the $S$-space of the Getzler-Jones $n$-operad constructed by the first author in \cite{SymBat}. For example, the tree of Figure \ref{sigma} is a normal $2$-planar tree with leaf $2$-ordinal given by the $2$-ordinal $S$. Example \ref{FNord} tells us that the $6$-dimensional Getzler-Jones cell above contains part of the boundary of the $6$-dimensional Fox-Neuwirth cell $FN_S \subset {\rm Mod}^2(6)$ (this part of the boundary is actually a $5$-dimensional contractible manifold with corners  \cite{SymBat,Marklbad}).

\subsection{Insertion and contraction of $n$-planar trees}

Since $n$-planar trees are in particular ordered rooted trees, graph insertion of an $n$-planar rooted tree inside a vertex of another $n$-planar rooted tree is well defined. The result of this graph insertion is in general just an $n$-ordered rooted tree unless we require a compatibility condition of the corresponding $n$-ordinal structures. This compatibility condition is the following: an $n$-planar tree $\tau'$ inserts into a vertex $v$ of an $n$-planar tree $\tau$ if the bijection $\rho_v$ induced by the $n$-order structures of $\tau$ and $\tau'$ (see \ref{orderedgraphs}) is an isomorphism between the leaf $n$-ordinal of $\tau'$ and the corolla $n$-ordinal $\tau_v.$  It is now easy to check that the resulting tree has a canonical $n$-planar structure. This operation preserves isomorphism classes of $n$-planar trees. Normal and regular $n$-planar trees are closed under graph insertion and so are their isomorphism classes. For instance, graph insertion of the $2$-planar tree of Figure \ref{sigma2} as shown in Figure \ref{3to7}

\scalebox{0.9} 
{
\begin{pspicture}(-2,-4)(8.18,3)
\usefont{T1}{ptm}{m}{n}
\rput(0.00,2.05){\large 1}
\usefont{T1}{ptm}{m}{n}
\rput(1.06,2.07){\large 3}
\usefont{T1}{ptm}{m}{n}
\rput(2.81,2.05){\large 5}
\usefont{T1}{ptm}{m}{n}
\rput(0.68,2.07){\large 2}
\usefont{T1}{ptm}{m}{n}
\rput(1.45,2.07){\large 4}
\usefont{T1}{ptm}{m}{n}
\rput(4.10,2.05){\large 7}
\psline[linewidth=0.022cm](0.30,1.22)(0.30,0.80)
\psline[linewidth=0.022cm](0.70,1.22)(0.70,0.44)
\psline[linewidth=0.022cm](1.08,1.22)(1.08,0.80)
\psline[linewidth=0.022cm](0.30,0.82)(0.72,0.44)
\psline[linewidth=0.022cm](0.70,0.44)(1.08,0.82)
\psline[linewidth=0.022cm](1.88,1.22)(2.16,0.86)
\psline[linewidth=0.022cm](2.42,1.22)(2.16,0.86)
\psline[linewidth=0.022cm](2.16,0.86)(2.16,0.42)
\psline[linewidth=0.022cm](3.26,1.22)(3.28,0.82)
\psline[linewidth=0.022cm](3.90,1.22)(3.90,0.84)
\psline[linewidth=0.022cm](3.28,0.82)(3.60,0.46)
\psline[linewidth=0.022cm](3.90,0.84)(3.60,0.44)
\psline[linewidth=0.022cm](1.84,-0.23)(1.84,-0.66)
\psline[linewidth=0.022cm](2.20,-0.23)(2.20,-1.03)
\psline[linewidth=0.022cm](2.52,-0.25)(2.52,-0.66)
\psline[linewidth=0.022cm](1.84,-0.65)(2.18,-1.05)
\psline[linewidth=0.022cm](2.18,-1.07)(2.52,-0.65)
\pscircle[linewidth=0.022,dimen=outer](3.59,0.89){0.59}
\pscircle[linewidth=0.022,dimen=outer](2.15,0.89){0.61}
\pscircle[linewidth=0.022,dimen=outer](0.67,0.91){0.61}
\pscircle[linewidth=0.022,dimen=outer](2.15,-0.58){0.61}
\psline[linewidth=0.022cm](2.18,-1.17)(2.18,-1.57)
\psline[linewidth=0.022cm](2.16,0.30)(2.16,0.02)
\psline[linewidth=0.022cm](2.52,1.38)(2.84,1.82)
\psline[linewidth=0.022cm](1.78,1.36)(1.48,1.82)
\psline[linewidth=0.022cm](3.92,1.38)(4.10,1.82)
\psline[linewidth=0.022cm](3.26,1.36)(3.26,1.82)
\psline[linewidth=0.022cm](0.26,1.34)(0.06,1.82)
\psline[linewidth=0.022cm](0.68,1.52)(0.68,1.84)
\psline[linewidth=0.022cm](1.06,1.36)(1.08,1.82)
\psline[linewidth=0.022cm](2.64,-0.23)(3.50,0.32)
\psline[linewidth=0.022cm](0.82,0.34)(1.66,-0.21)
\usefont{T1}{ptm}{m}{n}
\rput(3.26,2.05){\large 6}
\usefont{T1}{ptm}{m}{n}
\rput(5.16,2.07){\large 1}
\usefont{T1}{ptm}{m}{n}
\rput(6.62,2.07){\large 3}
\usefont{T1}{ptm}{m}{n}
\rput(8.02,0.77){\large 2}
\psline[linewidth=0.022cm](6.98,-0.19)(7.26,-0.55)
\psline[linewidth=0.022cm](7.52,-0.19)(7.26,-0.55)
\psline[linewidth=0.022cm](7.26,-0.55)(7.26,-0.99)
\pscircle[linewidth=0.022,dimen=outer](6.07,0.85){0.61}
\pscircle[linewidth=0.022,dimen=outer](7.25,-0.56){0.61}
\psline[linewidth=0.022cm](7.28,-1.15)(7.28,-1.55)
\psline[linewidth=0.022cm](5.70,1.34)(5.34,1.80)
\psline[linewidth=0.022cm](6.42,1.34)(6.70,1.82)
\psline[linewidth=0.022cm](7.62,-0.07)(7.98,0.50)
\psline[linewidth=0.022cm](6.44,0.36)(6.88,-0.07)
\psline[linewidth=0.022cm](5.78,1.24)(5.78,0.82)
\psline[linewidth=0.022cm](6.36,1.24)(6.36,0.82)
\psline[linewidth=0.022cm](5.78,0.82)(6.12,0.44)
\psline[linewidth=0.022cm](6.36,0.82)(6.10,0.44)
\psbezier[linewidth=0.022,linestyle=dashed,dash=0.16cm 0.16cm](5.18,1.14)(5.20,0.15)(3.39,-0.32)(2.70,-0.81)
\psbezier[linewidth=0.022,linestyle=dashed,dash=0.16cm 0.16cm](7.66,-1.47)(7.14,-2.25)(3.58,-0.55)(2.66,-0.83)
\end{pspicture}
}\vspace{-20mm}
\begin{figure}[h]
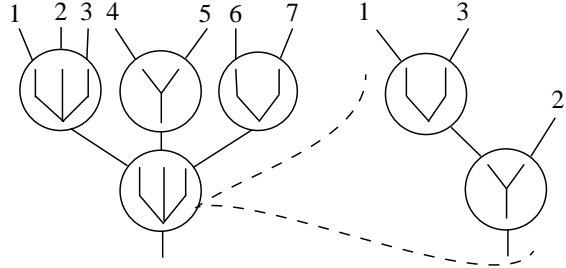
\caption{{\small Insertion of a $2$-planar tree to a vertex of another $2$-planar tree.}} \label{3to7}\end{figure}

\noindent yields the $2$-planar tree of Figure \ref{1to7}.

If an $n$-planar tree $\tau''$ is obtained by an insertion of an $n$-planar tree  $\tau'$ into a vertex of an $n$-planar tree  $\tau$   we will say that $\tau$ is the result of {\it contracting $\tau'$ inside $\tau''.$} For instance, the graph insertion of Figure \ref{3to7} corresponds to the following contraction:

\scalebox{0.9} 
{
\begin{pspicture}(0,-2.8)(6.43,3)
\usefont{T1}{ptm}{m}{n}
\rput(0.00,2.02){\large 1}
\usefont{T1}{ptm}{m}{n}
\rput(1.06,2.04){\large 3}
\usefont{T1}{ptm}{m}{n}
\rput(6.27,2.02){\large 5}
\usefont{T1}{ptm}{m}{n}
\rput(0.68,2.04){\large 2}
\usefont{T1}{ptm}{m}{n}
\rput(5.35,2.04){\large 4}
\usefont{T1}{ptm}{m}{n}
\rput(3.96,2.04){\large 7}
\psline[linewidth=0.022cm](0.30,1.19)(0.30,0.77)
\psline[linewidth=0.022cm](0.70,1.19)(0.70,0.41)
\psline[linewidth=0.022cm](1.08,1.19)(1.08,0.77)
\psline[linewidth=0.022cm](0.30,0.79)(0.72,0.41)
\psline[linewidth=0.022cm](0.70,0.41)(1.08,0.79)
\psline[linewidth=0.022cm](5.50,1.23)(5.78,0.87)
\psline[linewidth=0.022cm](6.04,1.23)(5.78,0.87)
\psline[linewidth=0.022cm](5.78,0.87)(5.78,0.43)
\psline[linewidth=0.022cm](2.98,1.23)(3.00,0.83)
\psline[linewidth=0.022cm](3.62,1.23)(3.62,0.85)
\psline[linewidth=0.022cm](3.00,0.83)(3.32,0.47)
\psline[linewidth=0.022cm](3.62,0.85)(3.32,0.45)
\pscircle[linewidth=0.022,dimen=outer](5.73,0.86){0.59}
\pscircle[linewidth=0.022,dimen=outer](3.27,0.88){0.61}
\pscircle[linewidth=0.022,dimen=outer](0.67,0.88){0.61}
\psline[linewidth=0.022cm](2.90,0.41)(2.62,0.13)
\psline[linewidth=0.022cm](3.64,1.37)(3.96,1.81)
\psline[linewidth=0.022cm](2.88,1.37)(2.58,1.83)
\psline[linewidth=0.022cm](6.06,1.35)(6.24,1.79)
\psline[linewidth=0.022cm](5.40,1.33)(5.40,1.79)
\psline[linewidth=0.022cm](0.26,1.31)(0.06,1.79)
\psline[linewidth=0.022cm](0.68,1.49)(0.68,1.81)
\psline[linewidth=0.022cm](1.06,1.33)(1.08,1.79)
\usefont{T1}{ptm}{m}{n}
\rput(2.58,2.06){\large 6}
\psline[linewidth=0.022cm](3.46,-0.88)(3.74,-1.24)
\psline[linewidth=0.022cm](4.00,-0.88)(3.74,-1.24)
\psline[linewidth=0.022cm](3.74,-1.24)(3.74,-1.68)
\pscircle[linewidth=0.022,dimen=outer](2.23,-0.31){0.61}
\pscircle[linewidth=0.022,dimen=outer](3.73,-1.25){0.61}
\psline[linewidth=0.022cm](3.76,-1.84)(3.76,-2.24)
\psline[linewidth=0.022cm](1.74,0.01)(1.08,0.43)
\psline[linewidth=0.022cm](4.12,-0.78)(5.68,0.29)
\psline[linewidth=0.022cm](2.62,-0.76)(3.18,-1.02)
\psline[linewidth=0.022cm](1.94,0.07)(1.94,-0.34)
\psline[linewidth=0.022cm](2.52,0.07)(2.52,-0.34)
\psline[linewidth=0.022cm](1.94,-0.34)(2.28,-0.72)
\psline[linewidth=0.022cm](2.52,-0.34)(2.26,-0.72)
\psline[linewidth=0.022cm,arrowsize=0.153cm 2.0,arrowlength=1.2,arrowinset=0.6]{->}(6.5,0)(8.3,0)
\rput{-120}(-0.8,1){\psellipse[linewidth=0.02,linestyle=dashed,dash=0.16cm 0.16cm,dimen=outer](-0.3,4.1)(0.9,2)}\end{pspicture}
\begin{pspicture}(-2,-2.26)(8.18,2.29)
\usefont{T1}{ptm}{m}{n}
\rput(0.00,2.05){\large 1}
\usefont{T1}{ptm}{m}{n}
\rput(1.06,2.07){\large 3}
\usefont{T1}{ptm}{m}{n}
\rput(2.81,2.05){\large 5}
\usefont{T1}{ptm}{m}{n}
\rput(0.68,2.07){\large 2}
\usefont{T1}{ptm}{m}{n}
\rput(1.45,2.07){\large 4}
\usefont{T1}{ptm}{m}{n}
\rput(4.10,2.05){\large 7}
\psline[linewidth=0.022cm](0.3,1.25)(0.3,0.8)
\psline[linewidth=0.022cm](0.7,1.25)(0.7,0.45)
\psline[linewidth=0.022cm](1.08,1.22)(1.08,0.80)
\psline[linewidth=0.022cm](0.30,0.82)(0.72,0.44)
\psline[linewidth=0.022cm](0.70,0.44)(1.08,0.82)
\psline[linewidth=0.022cm](1.88,1.22)(2.16,0.86)
\psline[linewidth=0.022cm](2.42,1.22)(2.16,0.86)
\psline[linewidth=0.022cm](2.16,0.86)(2.16,0.42)
\psline[linewidth=0.022cm](3.26,1.22)(3.28,0.82)
\psline[linewidth=0.022cm](3.90,1.22)(3.90,0.84)
\psline[linewidth=0.022cm](3.28,0.82)(3.60,0.46)
\psline[linewidth=0.022cm](3.90,0.84)(3.60,0.44)
\psline[linewidth=0.022cm](1.84,-0.23)(1.84,-0.66)
\psline[linewidth=0.022cm](2.20,-0.23)(2.20,-1.03)
\psline[linewidth=0.022cm](2.52,-0.25)(2.52,-0.66)
\psline[linewidth=0.022cm](1.84,-0.65)(2.18,-1.05)
\psline[linewidth=0.022cm](2.18,-1.07)(2.52,-0.65)
\pscircle[linewidth=0.022,dimen=outer](3.59,0.89){0.59}
\pscircle[linewidth=0.022,dimen=outer](2.15,0.89){0.61}
\pscircle[linewidth=0.022,dimen=outer](0.67,0.91){0.61}
\pscircle[linewidth=0.022,dimen=outer](2.15,-0.58){0.61}
\psline[linewidth=0.022cm](2.18,-1.17)(2.18,-1.57)
\psline[linewidth=0.022cm](2.16,0.30)(2.16,0.02)
\psline[linewidth=0.022cm](2.52,1.38)(2.84,1.82)
\psline[linewidth=0.022cm](1.78,1.36)(1.48,1.82)
\psline[linewidth=0.022cm](3.92,1.38)(4.10,1.82)
\psline[linewidth=0.022cm](3.26,1.36)(3.26,1.82)
\psline[linewidth=0.022cm](0.26,1.34)(0.06,1.82)
\psline[linewidth=0.022cm](0.68,1.52)(0.68,1.84)
\psline[linewidth=0.022cm](1.06,1.36)(1.08,1.82)
\psline[linewidth=0.022cm](2.64,-0.23)(3.50,0.32)
\psline[linewidth=0.022cm](0.82,0.34)(1.66,-0.21)
\usefont{T1}{ptm}{m}{n}
\rput(3.26,2.05){\large 6}
\end{pspicture}
}
\begin{figure}[h]
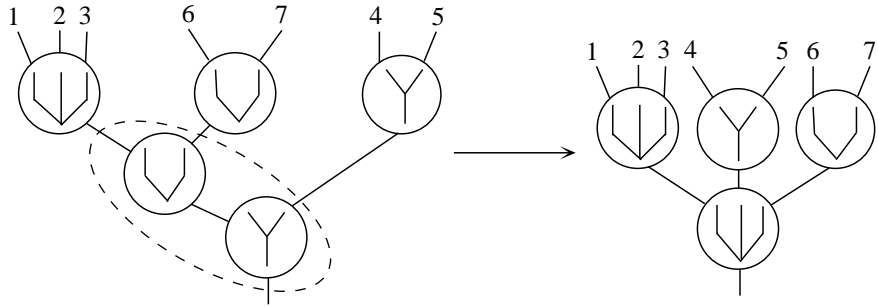
\caption{{\small Contraction of a $2$-planar tree.}} \label{4to7}\end{figure}

\subsection{Insertion of normal $n$-planar trees and multiplication in Fulton-MacPherson operad}

\label{subst} Insertion of normal $n$-planar trees is closely related to the multiplication of the Fulton-MacPherson operad \cite{SymBat}. In fact, it can be explained in terms of infinitesimal substitutions of point configurations. For the example of Figure \ref{3to7} we have two infinitesimal configurations in the Fulton-MacPherson operad

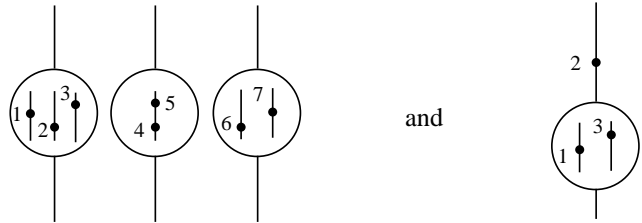
\scalebox{1} 
{
\begin{pspicture}(-1.5,-2)(8.38,2)
\usefont{T1}{ptm}{m}{n}
\rput(0.10171875,-0.02){\footnotesize 1}
\usefont{T1}{ptm}{m}{n}
\rput(0.74623436,0.26){\footnotesize 3}
\usefont{T1}{ptm}{m}{n}
\rput(2.1417968,0.12){\footnotesize 5}
\usefont{T1}{ptm}{m}{n}
\rput(0.44992188,-0.24){\footnotesize 2}
\psline[linewidth=0.022cm](0.28,0.28)(0.28,-0.38)
\pscircle[linewidth=0.022,dimen=outer](0.59,-0.01){0.59}
\psline[linewidth=0.022cm](0.6,1.42)(0.6,0.58)
\psline[linewidth=0.022cm](0.6,-0.58)(0.6,-1.46)
\psline[linewidth=0.022cm](0.6,0.28)(0.6,-0.38)
\psline[linewidth=0.022cm](0.88,0.26)(0.88,-0.4)
\psdots[dotsize=0.12](0.28,-0.02)
\psdots[dotsize=0.12](0.6,-0.2)
\psdots[dotsize=0.12](0.88,0.1)
\usefont{T1}{ptm}{m}{n}
\rput(1.7319531,-0.24){\footnotesize 4}
\pscircle[linewidth=0.022,dimen=outer](1.93,-0.01){0.59}
\psline[linewidth=0.022cm](1.94,1.42)(1.94,0.58)
\psline[linewidth=0.022cm](1.94,-0.58)(1.94,-1.46)
\psline[linewidth=0.022cm](1.94,0.28)(1.94,-0.38)
\psdots[dotsize=0.12](1.94,0.12)
\psdots[dotsize=0.12](1.94,-0.2)
\usefont{T1}{ptm}{m}{n}
\rput(3.3129532,0.24){\footnotesize 7}
\usefont{T1}{ptm}{m}{n}
\rput(2.886953,-0.14){\footnotesize 6}
\pscircle[linewidth=0.022,dimen=outer](3.29,-0.01){0.59}
\psline[linewidth=0.022cm](3.3,1.42)(3.3,0.58)
\psline[linewidth=0.022cm](3.3,-0.58)(3.3,-1.46)
\psline[linewidth=0.022cm](3.08,0.3)(3.08,-0.36)
\psline[linewidth=0.022cm](3.5,0.32)(3.5,-0.34)
\psdots[dotsize=0.12](3.08,-0.2)
\psdots[dotsize=0.12](3.5,0)
\usefont{T1}{ptm}{m}{n}
\rput(5.5140233,-0.075){and}
\usefont{T1}{ptm}{m}{n}
\rput(7.8262343,-0.18){\footnotesize 3}
\usefont{T1}{ptm}{m}{n}
\rput(7.3617187,-0.58){\footnotesize 1}
\pscircle[linewidth=0.022,dimen=outer](7.79,-0.45){0.59}
\psline[linewidth=0.022cm](7.8,1.46)(7.8,0.14)
\psline[linewidth=0.022cm](7.8,-1.02)(7.8,-1.42)
\psline[linewidth=0.022cm](7.58,-0.14)(7.58,-0.8)
\psline[linewidth=0.022cm](8.0,-0.12)(8.0,-0.78)
\psdots[dotsize=0.12](7.58,-0.5)
\psdots[dotsize=0.12](8.0,-0.3)
\psdots[dotsize=0.12](7.8,0.66)
\usefont{T1}{ptm}{m}{n}
\rput(7.529922,0.66){\footnotesize 2}
\end{pspicture}
}

\noindent Then we substitute the configuration on the first line to the  point $1$ in the second picture, the second line to the point $2$ in the second picture and the third line to the point $3$ in the second picture. The resulting configuration is therefore

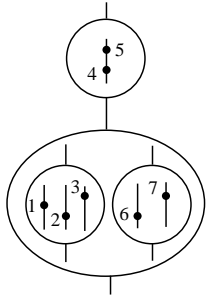
\scalebox{0.9} 
{
\begin{pspicture}(-5,-2.8)(2.94,2.5)
\usefont{T1}{ptm}{m}{n}
\rput(0.38171875,-0.84){\footnotesize 1}
\usefont{T1}{ptm}{m}{n}
\rput(1.0262344,-0.56){\footnotesize 3}
\usefont{T1}{ptm}{m}{n}
\rput(1.6817969,1.46){\footnotesize 5}
\usefont{T1}{ptm}{m}{n}
\rput(0.7299219,-1.06){\footnotesize 2}
\psline[linewidth=0.022cm](0.56,-0.54)(0.56,-1.2)
\pscircle[linewidth=0.022,dimen=outer](0.87,-0.83){0.59}
\psline[linewidth=0.022cm](0.88,0.06)(0.88,-0.24)
\psline[linewidth=0.022cm](0.88,-1.4)(0.88,-1.68)
\psline[linewidth=0.022cm](0.88,-0.54)(0.88,-1.2)
\psline[linewidth=0.022cm](1.16,-0.56)(1.16,-1.22)
\psdots[dotsize=0.12](0.56,-0.84)
\psdots[dotsize=0.12](0.88,-1.0)
\psdots[dotsize=0.12](1.16,-0.7)
\usefont{T1}{ptm}{m}{n}
\rput(1.2719531,1.1){\footnotesize 4}
\pscircle[linewidth=0.022,dimen=outer](1.47,1.33){0.59}
\psline[linewidth=0.022cm](1.48,2.16)(1.48,1.92)
\psline[linewidth=0.022cm](1.48,0.76)(1.48,0.28)
\psline[linewidth=0.022cm](1.48,1.62)(1.48,0.96)
\psdots[dotsize=0.12](1.48,1.46)
\psdots[dotsize=0.12](1.48,1.16)
\usefont{T1}{ptm}{m}{n}
\rput(2.1729531,-0.58){\footnotesize 7}
\usefont{T1}{ptm}{m}{n}
\rput(1.7469531,-0.96){\footnotesize 6}
\pscircle[linewidth=0.022,dimen=outer](2.15,-0.83){0.59}
\psline[linewidth=0.022cm](2.16,0.04)(2.16,-0.24)
\psline[linewidth=0.022cm](2.16,-1.4)(2.16,-1.66)
\psline[linewidth=0.022cm](1.94,-0.52)(1.94,-1.18)
\psline[linewidth=0.022cm](2.36,-0.5)(2.36,-1.16)
\psdots[dotsize=0.12](1.94,-1.0)
\psdots[dotsize=0.12](2.36,-0.7)
\psellipse[linewidth=0.022,dimen=outer](1.47,-0.81)(1.47,1.09)
\psline[linewidth=0.022cm](1.54,-2.16)(1.54,-1.88)
\end{pspicture}
}

\noindent which corresponds exactly to the tree of Figure \ref{1to7}.

\section{Other definitions of graphs}

Since there are many different treatments of graphs in the literature we include several of them  and describe briefly the relationship with our definition. The order is not chronological nor it is related to the importance and popularity.

\subsection{Feynman graphs of Joyal-Kock}

In \cite{JK} Joyal and Kock define a Feynman graph $\Gamma$
as a span in finite sets:$$E\stackrel{s}{\leftarrow} H\stackrel{t}{\to} V$$equipped with a fixed point free involution $\iota:E\to E$ such that $s$ is an injection. The elements of $V$ are called vertices, the elements of $H$ half-edges and the elements of $E$ oriented edges.

To any Feynman graph $\Gamma$ one can assign a graph $G$ in our sense as follows. The set of vertices of $G$ is the set  $V.$ The set of flags of $G$ is the set $E.$ The set of edges of $G$ is the set of orbits of $\iota.$ There is a unique morphism from a flag $h$ to a $v\in V$ if $h = s(\bar{h}).$ The target $v$ of this map is $t(\bar{h}).$ For any flag there is a unique map from it to its orbit. One checks that this construction produces a graph in our sense.

\begin{pro}The set of isomorphism classes of Feynman graphs embeds into the set of isomorphism classes of graphs in our sense. A graph in our sense is not a Feynman graph if and only if it contains a connected component isomorphic to a free living loop.\end{pro}

\begin{proof}If a graph $G$ does not contain free-living loops we can reconstruct a Feynman graph $\Gamma$ as follows. The set of vertices of $\Gamma$ will be equal to the set of vertices of $G.$ The set of half-edges will be the set of non-free flags of $G.$ The set of oriented edges will be equal to the set of all flags of $G$ and the involution exchanges the adjacent flags. Since $G$ does not contain free-living loop this involution is fixed-point free.\end{proof}

\subsection{Getzler-Kapranov graphs}

In \cite{Borisov-Manin,Cos,GetzK} a slightly different definition of a graph has been used (attributed to Kontsevich-Manin by Getzler and Kapranov) : a graph $\Gamma$ is a map of sets $\pi:H\to V$ together with an involution $\sigma:H\to H, \sigma^2 = 1.$ We can construct a Feynman graph  $\Gamma$ in the sense of Joyal-Kock as follows.
The set of vertices of  $\Gamma$ is $V.$ The set of flags of ${\Gamma}$ is $H$ and $t= \pi.$ The set of oriented edges of $\Gamma $ is  equal to $H\sqcup H^{\sigma}$ where $H^{\sigma}$ is the set of fixed points of $\sigma$ (so we simply add to $H$ one extra point for each fixed point of $\sigma$) and $s$ is the first  coprojection. The involution $\iota$ maps each fixed point $h$ to its copy in $H^{\sigma}$ and coincides with $\sigma(h)$ if  $h$ is not a fixed point of $\sigma.$ This shows that the set of isomorphism classes of Getzler-Kapranov graphs is a subset of the set of isomorphism classes of Feynman graphs. The difference is in the treatment of graphs without vertices which exist in the latter but not in the former setting.

\subsection{Johnson-Yau graphs}This type of graph was introduced in \cite{Markl,Merk} and studied by Johnson-Yau \cite{JY} paying special attention to free-living edges and loops.

\begin{defin}[{\cite{JY}}]A generalised graph $\mathbb{G}$ is a finite set $Flag(\mathbb{G})$ with involution $\iota ,$ together with a partition and choice of an isolated cell $\mathbb{G}_0$ such that there is a fixed-point free involution $\pi $ on the set of fixed points of $\iota $ within $\mathbb{G}_0.$\end{defin}

Here a partition is a presentation of a finite set as a finite coproduct of some finite sets called cells (empty cells are allowed).
An isolated cell of a partition with an involution is a cell invariant under the involution.

For each Johnson-Yau graph we can construct a  graph $G$ in our sense as follows. A vertex of $G$ is a cell of $\mathbb{G}$ which is not an exceptional cell of $\mathbb{G}_0.$ The  flags of $G$ are $(Flag(\mathbb{G})\setminus \mathbb{G}_0)\sqcup (Flag(\mathbb{G})\setminus \mathbb{G}_0)^{\iota}$ as well as all orbits of $\iota$ on $\mathbb{G}_0.$ For $h\in Flag(\mathbb{G})\setminus \mathbb{G}_0$  which is not a fixed point of $\iota$ we have an edge $e(h)$ such that $e(h) = e(\iota h)$ and a unique morphism in $G$  from $h$ to $e.$ For a fixed point $h$ of $\tau$ in $(Flag(\mathbb{G})\setminus \mathbb{G}_0)$ its adjacent belongs to $(Flag(\mathbb{G})\setminus \mathbb{G}_0)^{\iota}$ and we have a morphism from $h$ and its adjacent to $e.$

For a non-trivial orbit $h$ of $\iota$ in $\mathbb{G}_0$ there are an edge in $G$ and  exactly two morphisms from $h$ to this edge in $G.$ Finally, for each orbit of $\pi$ there is one edge in $G$ and a morphism from each element of this orbit to this edge.

\begin{pro}This construction determines a bijection between isomorphism clas-ses of graphs in our sense and isomorphism classes of Johnson-Yau graphs.\end{pro}

\subsection{Joyal-Street graphs}According to Joyal-Street \cite{JS} a {\it topological graph} is a Hausdorff space $G$ with a discrete closed subset $G_0$ (the set of vertices of $G$) such that $G\setminus G_0$ is a $1$-manifold without boundary (empty $1$-manifold is allowed).

A {\it graph with boundary} is a pair $(\Gamma,\partial \Gamma)$ where $\Gamma$ is a compact topological graph and $\partial \Gamma$ is a subset of its set of vertices such that each $x\in \partial \Gamma$ is of degree $1.$ The last thing means that the space $V\setminus x$ has one connected component, where $V$ is a sufficiently small connected neighbourhood of $x.$ Morphisms between graphs with boundary are homeomorphisms taking vertices to vertices and boundary points to boundary points.

\begin{pro}\label{geometricrealisation}Geometric realisation provides a bijection between isomorphism classes of graphs in our sense and isomorphism classes of Joyal-Street graphs with boundary.\end{pro}

\begin{proof}Geometric realisation produces a topological graph with boundary where the boundary points correspond to free flags. For the inverse correspondence we assume that each closed edge of  $\Gamma$ is parametrised by a bijective continuous function from the interval $[0,1]$  so that it makes sense to consider $t$-points, $t\in[0,1]$, on the edges of $\Gamma$.

For $(\Gamma,\partial \Gamma)$ we construct a categorical graph $G$ by taking as its vertices the set of vertices of $\Gamma$ minus the set $\partial \Gamma$ of boundary points. The set of edges of $G$  is the set of $1/2$-points on the edges of the topological graph $\Gamma.$ The set of adjacent flags is the set of $1/4$- and $3/4$-points on those edges of $\Gamma$ which connect two non-boundary points. For an edge which connects a boundary point to a non-boundary point we take the non-boundary point as a flag and the middle point of the interval between the non-boundary point and the $1/2$-point as its adjacent. If an edge connects two non-boundary points then we take those points as flags of $G.$ The morphisms in $G$ are obvious from the following example:

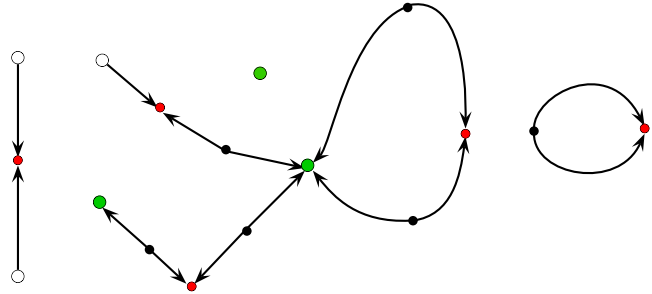
\scalebox{0.7} 
{
\begin{pspicture}(-3,-3.2)(12.141295,3)
\definecolor{color382b}{rgb}{0.0,0.8,0.0}
\definecolor{color383b}{rgb}{0.2,0.8,0.0}
\psline[linewidth=0.04cm,arrowsize=0.153cm 2.0,arrowlength=1.4,arrowinset=0.4]{->}(1.75,1.55)(2.8,0.65)
\psline[linewidth=0.04cm,arrowsize=0.153cm 2.0,arrowlength=1.4,arrowinset=0.4]{->}(4.05,-0.15)(2.9,0.55)
\psline[linewidth=0.04cm,arrowsize=0.153cm 2.0,arrowlength=1.4,arrowinset=0.4]{->}(4.1,-0.18)(5.6,-0.5)
\psbezier[linewidth=0.04,arrowsize=0.153cm 2.0,arrowlength=1.4,arrowinset=0.4]{->}(7.5,2.55)(6.45,2.1)(6.1,-0.05)(5.75,-0.4)
\psbezier[linewidth=0.04,arrowsize=0.153cm 2.0,arrowlength=1.4,arrowinset=0.4]{->}(7.6,2.6)(8.75,2.8)(8.65,0.5)(8.65,0.25)
\psbezier[linewidth=0.04,arrowsize=0.153cm 2.0,arrowlength=1.4,arrowinset=0.4]{->}(7.65,-1.5)(8.4,-1.4)(8.55,-0.65)(8.65,0.1)
\psbezier[linewidth=0.04,arrowsize=0.153cm 2.0,arrowlength=1.4,arrowinset=0.4]{->}(7.6,-1.5)(6.6,-1.55)(6.15,-1.05)(5.75,-0.55)
\psline[linewidth=0.04cm,arrowsize=0.153cm 2.0,arrowlength=1.4,arrowinset=0.4]{->}(4.50,-1.65)(5.6,-0.55)
\psline[linewidth=0.04cm,arrowsize=0.153cm 2.0,arrowlength=1.4,arrowinset=0.4]{->}(4.45,-1.7)(3.5,-2.7)
\psline[linewidth=0.04cm,arrowsize=0.153cm 2.0,arrowlength=1.4,arrowinset=0.4]{->}(2.65,-2.05)(3.35,-2.7)
\psline[linewidth=0.04cm,arrowsize=0.153cm 2.0,arrowlength=1.4,arrowinset=0.4]{->}(2.65,-2.05)(1.75,-1.25)
\psline[linewidth=0.04cm,arrowsize=0.153cm 2.0,arrowlength=1.4,arrowinset=0.4]{->}(0.15,-2.5)(0.15,-0.45)
\psline[linewidth=0.04cm,arrowsize=0.153cm 2.0,arrowlength=1.4,arrowinset=0.4]{->}(0.15,1.55)(0.15,-0.3)
\psbezier[linewidth=0.04,arrowsize=0.153cm 2.0,arrowlength=1.4,arrowinset=0.4]{->}(9.95,0.25)(9.95,0.9)(11.35,1.65)(12.05,0.3)
\psbezier[linewidth=0.04,arrowsize=0.153cm 2.0,arrowlength=1.4,arrowinset=0.4]{->}(9.95,0.15)(9.95,-0.65)(11.55,-0.95)(12.05,0.15)
\psdots[dotsize=0.25,fillstyle=solid,fillcolor=color382b,dotstyle=o](5.65,-0.45)
\psdots[dotsize=0.25,fillstyle=solid,fillcolor=color383b,dotstyle=o](4.75,1.3)
\psdots[dotsize=0.25,fillstyle=solid,fillcolor=color382b,dotstyle=o](1.7,-1.15)
\psdots[dotsize=0.25,fillstyle=solid,dotstyle=o](1.75,1.55)
\psdots[dotsize=0.18,fillstyle=solid,fillcolor=red,dotstyle=o](3.45,-2.75)
\psdots[dotsize=0.18,fillstyle=solid,fillcolor=red,dotstyle=o](12.05,0.25)
\psdots[dotsize=0.18,fillstyle=solid,fillcolor=red,dotstyle=o](2.85,0.65)
\psdots[dotsize=0.18,fillstyle=solid,fillcolor=red,dotangle=90.05,dotstyle=o](0.15,-0.35)
\psdots[dotsize=0.18,fillstyle=solid,fillcolor=red,dotstyle=o](8.65,0.15)
\psdots[dotsize=0.18](7.65,-1.5)
\psdots[dotsize=0.18](7.55,2.55)
\psdots[dotsize=0.18](9.95,0.2)
\psdots[dotsize=0.18](4.1,-0.15)
\psdots[dotsize=0.18](2.65,-2.05)
\psdots[dotsize=0.18](4.5,-1.7)
\psdots[dotsize=0.25,fillstyle=solid,dotangle=90.033676,dotstyle=o](0.15,-2.56)
\psdots[dotsize=0.25,fillstyle=solid,dotangle=90.033676,dotstyle=o](0.15,1.6)
\end{pspicture}
}

\noindent In this picture large dots correspond to the vertices of the topological graph and green dots correspond to the vertices of the categorical graphs, white dots correspond to the boundary points and the free flags of the categorical graph, black dots correspond to the non-free flags of the categorical graph and red dots correspond to its edges.\end{proof}

\noindent Proposition \ref{geometricrealisation} justifies the representation of graphs and trees used in this paper.\vspace{2ex}

\vspace{1ex}

{\bf\noindent Acknowledgements.} We wish to express our gratitude to G. Caviglia, D.-C. Cisinski, B. Fresse, E. Getzler, M. Johnson, A. Joyal, J. Kock, S. Lack, A. Lazarev, G. Maltsiniotis, M. Markl, F. Muro, S. Schwede, R. Street, B. Vallette and D. Yau for many useful comments and conversations. Special thanks go to Mark Weber for numerous enlightening discussions about polynomial monads and $2$-dimensional monad theory, and to the referee whose careful reading helped us to eliminate several inaccuracies and to ameliorate considerably the presentation.

The authors gratefully acknowledge the financial support of Scott Russel Johnson Memorial Foundation, Max Plank Institut f\"{u}r Mathematik in Bonn, Newton Institute of Mathematical Science in Cambridge, the Australian Research Council (grants DP0558372, DP130101172) and the French Agence Nationale de Recherche (grant HOGT).

\renewcommand{\refname}{Bibliography.}

\end{document}